\documentclass{memo-l}
\usepackage{amsmath}
\usepackage{amssymb}
\usepackage{multirow}
\usepackage[pdfencoding=auto]{hyperref}
\usepackage[all]{xy}
\SelectTips{eu}{}
\usepackage[mathscr]{eucal}
\usepackage{mathtools}
\usepackage{mathrsfs}
\usepackage{dsfont}
\makeindex
\allowdisplaybreaks
\numberwithin{equation}{section}
\numberwithin{figure}{section}

\newtheorem{lemma}[equation]{Lemma}
\newtheorem{theorem}[equation]{Theorem}
\newtheorem*{theorem*}{Theorem}
\newenvironment{theoremqed}{\pushQED{\qed}\begin{theorem}}{\popQED\end{theorem}}
\newtheorem{prop}[equation]{Proposition}
\newenvironment{propqed}{\pushQED{\qed}\begin{prop}}{\popQED\end{prop}}
\newtheorem{cor}[equation]{Corollary}
\newenvironment{corqed}{\pushQED{\qed}\begin{cor}}{\popQED\end{cor}}
\newtheorem{conj}[equation]{Conjecture}

\theoremstyle{definition}
\newtheorem{question}[equation]{Question}
\newtheorem{defn}[equation]{Definition}
\newtheorem{eg}[equation]{Example}
\newtheorem{egs}[equation]{Examples}
\newtheorem{rk}[equation]{Remark}

\renewcommand{\ge}{\geqslant}
\renewcommand{\le}{\leqslant}

\newcommand{\bul}{{\scriptstyle\bullet}}
\newcommand{\Dick}{E}
\newcommand{\core}{\mathsf{core}}
\newcommand{\cc}{\mathsf{c}}
\newcommand{\cg}{\succ}
\newcommand{\cge}{\succcurlyeq}
\newcommand{\cl}{\prec}
\newcommand{\cle}{\preccurlyeq}
\newcommand{\da}{\downarrow}
\newcommand{\ep}{\varepsilon}

\newcommand{\Ext}{\mathsf{Ext}}
\newcommand{\half}{{\textstyle\frac{1}{2}}}
\newcommand{\Hom}{\mathsf{Hom}}
\newcommand{\Ker}{\mathsf{Ker}}
\newcommand{\im}{\mathsf{Im}}
\newcommand{\ind}{\mathsf{ind}}
\newcommand{\Lin}{\mathscr{L}}
\newcommand{\mmod}{\mathsf{mod}}
\newcommand{\npj}{\gamma}
\newcommand{\one}{\mathds{1}}
\newcommand{\phm}{\phantom{-}}
\newcommand{\Pic}{\mathsf{T}}
\newcommand{\PPic}{\mathbb{T}}
\newcommand{\proj}{\mathsf{proj}}
\newcommand{\Proj}{\mathsf{Proj}\,}
\newcommand{\res}{\mathsf{res}\hspace{.4mm}}
\newcommand{\Soc}{\mathsf{Soc}}
\newcommand{\Spec}{\mathsf{Spec}}
\newcommand{\Struct}{\Delta}
\newcommand{\tilt}{{\mathsf{tilt}}}
\newcommand{\Tr}{\mathsf{Tr}}
\newcommand{\vect}{\mathsf{vect}}
\newcommand{\zero}{\star}

\newcommand{\bC}{\mathbb{C}}
\newcommand{\bF}{\mathbb{F}}
\newcommand{\bN}{\mathbb{N}}
\newcommand{\bP}{\mathbb{P}}
\newcommand{\bQ}{\mathbb{Q}}
\newcommand{\bR}{\mathbb{R}}
\newcommand{\bZ}{\mathbb{Z}}
\newcommand{\bi}{{\mathsf{i}\mkern1mu}}
\newcommand{\cH}{\mathcal{H}}
\newcommand{\cO}{\mathcal{O}}
\newcommand{\cV}{\mathcal{V}}
\newcommand{\fa}{\mathfrak{a}}

\newcommand{\fH}{\mathfrak{H}}
\newcommand{\fI}{\mathfrak{I}}
\newcommand{\fJ}{\mathfrak{J}}
\newcommand{\fm}{\mathfrak{m}}

\newcommand{\fp}{\mathfrak{p}}
\newcommand{\fX}{\mathfrak{X}}
\newcommand{\fY}{\mathfrak{Y}}

\title{Modular Representation
  Theory and \\ Commutative Banach Algebras
}

\author{David J. Benson}
\address{Institute of Mathematics, University of Aberdeen, Aberdeen
  AB24 3UE, Scotland, United Kingdom}

\begin{document}

\begin{abstract}
In a recent paper of Benson and Symonds,
a new invariant was introduced for modular representations of a finite
group. An interpretation was given as a spectral radius with respect
to a Banach algebra completion of the representation ring.
Our purpose here is to take these notions further, and investigate
the structure of the resulting Banach algebras.
Some of the material in that paper is repeated
here in greater generality, and for clarity of exposition.

We give an axiomatic definition of an abstract representation ring, 
and representation ideal. The completion is then a commutative
Banach algebra, and the techniques of Gelfand from the 1940s are
applied in order to study the space of algebra homomorphisms to $\bC$.
One surprising consequence of this investigation 
is that the Jacobson radical and the nil
radical of a (complexified) representation ring always coincide.

These notes are intended for representation theorists. So background
material on commutative Banach algebras is given in detail, whereas
representation theoretic background is more condensed.
\end{abstract}



\maketitle

\tableofcontents

\chapter*{Preface}

The purpose of these notes is to study the asymptotics of the direct
sum decomposition of tensor products and tensor powers of finite dimensional
representations of a finite group in prime characteristic. This takes
us into the world of commutative Banach algebras and their spectral
theory.

This study grew out of the author's joint work with
Symonds~\cite{Benson/Symonds:2020a}. In the introduction to that paper, 
the example was given of 
the two dimensional indecomposable module $M$ for the cyclic group of
order five in characteristic five.%
\index{cyclic!group of order $5$}\index{group!cyclic of order $5$}
In this example, although the
dimension of $M^{\otimes n}$ is $2^n$, the dimension of the
non-projective part of $M^{\otimes n}$ is roughly $\tau^n$, where
$\tau$ is the golden ratio\index{tau@$\tau$}\index{golden ratio} 
(see Section~\ref{se:Z5}). This led us to define $\npj(M)$ to be
$\tau$ in this example, and to interpret it as the reciprocal of the radius of
convergence of the corresponding generating function. 
This, in turn, is interpreted as the spectral radius of $[M]$ as
an element of the representation ring, with respect to a suitable norm.

Let $a(G)$\index{a@$a(G)$, $a_\bC(G)$} be the representation ring, or 
Green ring\index{Green ring} of a finite
group $G$ over a field $k$ of characteristic $p$. This has a free
$\bZ$-basis consisting of the isomorphism classes of 
finitely generated indecomposable $kG$-modules, with addition and
multiplication coming from direct sum and tensor product.

The complexification $a_\bC(G)=\bC\otimes_\bZ a(G)$ is a 
commutative normed algebra, with the norm coming from dimension.
Its completion $\hat a(G)$ is a commutative Banach algebra which
forms our basic object of study. We look at various quotients of the
form $\hat a_\fX(G)=\hat a(G)/\hat a(G,\fX)$\index{aa@$\hat a_{\fX}(G)$} 
where $\fX$ is an ideal of
indecomposable $kG$-modules, and examine invariants of modules
coming from spectral radius in these quotients.  The purpose of this
is to obtain information about the asymptotic behaviour of the tensor
powers of a finitely generated $kG$-module.

With the aim of being applicable in a wider set of circumstances, we
have formulated the definitions and main theorems in terms of abstract representation rings.
The axioms are set up in Definition~\ref{def:repring}. In particular,
a representation ring $\fa$\index{a@$\fa$} 
comes equipped with a free $\bZ$-basis
$\{x_i\mid i\in\fI\}$\index{indexing set $\fI$}\index{I@$\fI$}
as an additive group, which is supposed to be
thought of as the basis of indecomposable modules in the case of
$a(G)$. One of the unexpected consequences of studying the
completion of $\fa$ in this generality is the proof in
Theorem~\ref{th:J=nil} that the Jacobson radical and the nilradical of
the complexification $\fa_\bC$ coincide. 

\begin{theorem*}
If $\fa$ is a representation ring then the Jacobson radical and the
nil radical of $\fa_\bC$ are equal.
\end{theorem*}

If $x$ is a positive element of a representation ring $\fa$, 
and $\fX\subset\fI$\index{X@$\fX$}
is a representation ideal, let $\cc_n^\fX(x)$ be the
dimension of the $\fX$-core\index{core} of $x^n$ (these concepts are defined in
Definitions~\ref{def:positive} and~\ref{def:repideal}). Then in
Section~\ref{se:npj} we define\index{gamma invariant $\npj_\fX(x)$}
\[ \npj_\fX(x)=\limsup_{n\to\infty}\sqrt[n]{\cc_n^\fX(x)}. \]
The following theorem, which to some extent parallels Theorem~1.2
of~\cite{Benson/Symonds:2020a}, 
summarises some of the properties of this
invariant.

\begin{theorem*}
The invariant $\npj_\fX(x)$ has the following properties:
\begin{enumerate}
\item We have $\displaystyle
\npj_\fX(x)=\lim_{n\to\infty}\sqrt[n]{\cc_n^\fX(x)}=
\inf_{n\ge 1}\sqrt[n]{\cc_n^\fX(x)}$. 
\item We have $0\le \npj_\fX(x) \le \dim x$.
\item We have $\npj_\fX(x)=\dim x$ if and only if $\core_\fX(x^n)=x^n$
  for all $n\ge 0$.
\item We have $\npj_\fX(x)=0$ if and only if $x\in\langle\fX\rangle$,
  otherwise $\npj_\fX(x)\ge 1$.
\item If $1\le\npj_\fX(x)<\sqrt{2}$ then $x$ is $\fX$-endotrivial.
\item If $x$ is not $\fX$-endotrivial and $\npj_\fX(x)=\sqrt{2}$ then $xx^*x\equiv 2x
  \pmod{\langle\fX\rangle}$.
\item We have $\npj_\fX(x^*)=\npj_\fX(x)$.
\item We have $\max\{\npj_\fX(x),\npj_\fX(y)\}\le \npj_\fX(x+y)\le
  \npj_\fX(x)+\npj_\fX(y)$.
\item If $a,b\ge 0$ we have $\npj_\fX(a+bx)=a+b\npj_\fX(x)$.
\item We have $\npj_\fX(xy)\le\npj_\fX(x)\npj_\fX(y)$.
\item We have $\npj_\fX(x^m)=\npj_\fX(x)^m$.
\item If $\fX\subseteq\fY$ we have $\npj_\fY(x)\le \npj_\fX(x)$.
\end{enumerate}
\end{theorem*}
The proofs of the various parts of this theorem may be found in the
following places: part~(i)~in Theorem~\ref{th:inf}, (ii)~in 
Lemma~\ref{le:gamma-bounds}, (iii)~in Lemma~\ref{le:gamma<dim},
(iv)~in Lemma~\ref{le:gamma=0}, (v)~and~(vi) in
Theorem~\ref{th:sqrt2}, (vii)~in Lemma~\ref{le:npjx*}, (viii)~in
Theorem~\ref{th:gamma-subadditive},  (ix)~in Theorem~\ref{th:a+bx}, 
(x)~in Lemma~\ref{le:gamma(xy)}, (xi)~in Lemma~\ref{le:gamma(x^m)}, 
and~(xii)~in Lemma~\ref{le:gammaXinY}.

In the case where $\fa=a(G)$ and $\fX=\fX_\proj$, which was the case
studied in~\cite{Benson/Symonds:2020a}, we can say more.

\begin{theorem*}
Let $G$ be a finite group and $k$ a field of characteristic $p$.
Let $\fa=a(G)$ be the representation ring of $kG$-modules, 
and $\fX=\fX_\proj$ be the representation ideal of projective
modules. Then the following properties hold
in addition to those discussed above, where we write $\npj_G$
for the invariant $\npj_{\fX_{\proj}}$ applied to $\fa=a(G)$.
\begin{enumerate}
\setcounter{enumi}{12}
\item $\npj_G(M)=\dim(M)$ if and only if there is an element of $G$ of
  order $p$ acting trivially on $M$.
\item We have $\npj_G(M)=\max_{E\le G}\npj_E(M)$, where the maximum is
  taken over elementary abelian subgroups $E\le G$.
\item If $M$ is endotrivial then $\npj_G(M)=1$. In particular,
  combining this with {\rm (v)}, there
  are no modules $M$ with $1<\npj_G(M)<\sqrt{2}$.
\item If $M$ is a two dimensional faithful indecomposable module for
  an elementary abelian group
  $E\cong (\bZ/p)^r$ then $\npj_E(M)=2\cos(\pi/p^r)$.
In particular, if $p=2$ and $r=2$ we have $\npj_E(M)=\sqrt{2}$, and if
$p=5$ and $r=1$ we have $\npj_E(M)=\tau$, the golden
ratio.\index{golden ratio}\index{tau@$\tau$}
\end{enumerate}
\end{theorem*}

The main theme of these notes is that the invariant $\npj_\fX(x)$
described above may
be interpreted as a spectral radius of the element $x$ in a quotient
of a completion of the complexified representation ring $\fa_\bC$. This completion is a
commutative Banach algebra, whose structure we shall investigate.

We put a norm on $\fa_\bC=\bC\otimes_\bZ \fa$\index{a@$\fa_{\bC}$} by setting
\[ \left\|\sum_{i\in\fI} a_ix_i\right\| = \sum_{i\in\fI}|a_i|\dim x_i. \]
This makes $\fa_\bC$ a commutative normed algebra with identity 
element $\one$.\index{$\one$} Its completion
$\hat\fa$ is therefore a commutative Banach algebra.

If $\fX$ is a representation ideal in $\fa$ then the linear span
$\langle\fX\rangle_\bC$ is an ideal in $\fa_\bC$, and its closure
$\widehat{\langle\fX\rangle}_\bC$ is an ideal in $\hat\fa$. The
quotient norm on 
$\hat\fa_\fX=\hat\fa/\widehat{\langle\fX\rangle}_\bC\cong\widehat{\fa/\langle\fX\rangle}_\bC$ 
is given by
\[ \left\|\sum_{i\in\fI}a_ix_i\right\|_\fX = \sum_{i\in\fI}|a_i|\dim\core_\fX(x_i)=
  \sum_{i\in\fI\setminus\fX}|a_i|\dim x_i. \]
After introducing the background material on commutative Banach
algebras in Chapter~\ref{ch:Banach}, we investigate these quotients $\hat\fa_\fX$ of
$\hat\fa$ in Chapter~\ref{ch:complete-repring}.

In terms of these quotients, if $x\in\fa_{\cge 0}$ then the invariant $\npj_\fX(x)$ is
interpreted as the spectral radius of the image of $x$ in
$\hat\fa_\fX$ (Theorem~\ref{th:spec-radius}). 
This enables us to relate $\npj_\fX(x)$ to the species
$s\colon\fa\to\bC$ of representation rings 
introduced by Benson and Parker~\cite{Benson/Parker:1984a}
and studied in~\cite{Benson:1984b,Fan:1991a,Green:1962a,
Haeberle:2008a,Hussein/Radwan:1997a,Hussein/Radwan:2007a,
O'Reilly:1965a,
Webb:1984a,Webb:1986a,Witherspoon:1996a}.

A \emph{species}\index{species} of $\fa$ is a ring homomorphism 
$s\colon\fa\to\bC$. Such a ring homomorphism extends uniquely to a
$\bC$-algebra homomorphism $s\colon\fa_\bC\to\bC$, which we also call a
species. The species which are continuous with respect to the 
norm on $\fa_\bC$ are the 
\emph{dimension bounded}\index{dimension!bounded species}%
\index{species!dimension bounded}
species, namely the ones that
satisfy $|s(x)|\le \dim x$ for every $x\in\fa_{\cge 0}$.
Such a species extends uniquely to a species $s\colon \hat\fa\to \bC$.
All species of $\hat\fa$ are of this form, and are automatically continuous.

If $\fX$ is a representation ideal in $\fa$ then the species which
vanish on $\fX$ and extend to species of $\hat\fa_\fX$ are the 
$\fX$-\emph{core bounded}\index{core!bounded species}%
\index{species!core bounded} 
ones, namely those that satisfy $|s(x)|\le\dim\core_\fX(x)$ for every
$x\in\fa_{\cge 0}$.

We may now apply a theorem of Gelfand relating the spectral radius to
the species.

\begin{theorem*}
Let $x\in\fa_{\cge 0}$ and $\fX$ be a representation ideal in
$\fa$. Then $\npj_\fX(x)$ is equal to the supremum of $|s(x)|$, where
$s$ runs over the $\fX$-core bounded species $s\colon \fa\to \bC$.
\end{theorem*}

The set of species of $\hat\fa_\fX$ is topologised with the weak*
topology\index{weak* topology},\index{topology, weak*} described in
Section~\ref{se:struct-space}, to form a compact Hausdorff topological
space called the \emph{structure space}\index{structure!space} 
$\Struct_\fX(\fa)$.\index{DeltaXa@$\Struct_\fX(\fa)$} If
$\fY\subseteq\fX$ are represetation ideals then every $\fX$-core
bounded species of $\fa$ is $\fY$-core bounded, and this induces a
homeomorphism identifying $\Struct_\fX(\fa)$ with a closed subset of
$\Struct_\fY(\fa)$.\bigskip

Now in some ways the Banach algebra $\hat\fa_\fX$ looks like
the group algebra $\ell^1(\Gamma)$ of a discrete abelian group $\Gamma$. This
analogy is strongest when $\fX=\fX_{\max}$,\index{X@$\fX_{\max}$} the largest representation
ideal in $\fa$, consisting of those $i$ for which $[x_ix_{i^*}:\one]>0$.
This is because in this quotient, $x_i^*$ acts as a
sort of partial inverse for $x_i$. In particular, we construct a
Hilbert space $H(\fa)$ on which $\hat\fa_{\max}=\hat\fa_{\fX_{\max}}$
acts, in such a way that the action of $x^*$ is the adjoint of the
action of $x$. This gives us an injective map of Banach $*$-algebras $\hat\fa_{\max}\to
\Lin(H(\fa))$, the bounded operators\index{bounded!operator} 
on $H(\fa)$. The crucial inequality 
giving boundedness is Theorem~\ref{th:|xy|}, which turns out to be
quite tricky to prove.
This says that for $x\in \fa_\bC$ and $y\in H(\fa)$ we have
\[ |xy|\le\|x\|_{\max}|y|. \]
We let $C^*_{\max}(\fa)$ denote
the closure of the image of $\hat\fa_{\max}\to \Lin(H(\fa))$. This is a
$C^*$-algebra, and is the completion of $\fa_{\max}$ with respect to
the sup norm\index{sup!norm}
\[ \|x\|_{\sup}=\sup_{|y|=1}|xy|. \]
 Some consequences of this construction include 
the statement that there are no non-zero
 quasinilpotent elements in $\hat\fa_{\max}$, and a better understanding of
 idempotents. 

In Chapter~\ref{ch:fingrp}, we specialise to the situation where $G$
is a finite group, $k$ is a field of characteristic $p$, and $a(G)$ is the representation
ring of $kG$-modules. In this case, we have further structure coming
from restriction and induction, elementary abelian subgroups, and
Adams psi operations. In particular, in the case $\fX=\fX_\proj$, we
show that the invariant $\npj(M)$ is detected on elementary abelian
subgroups. One consequence of this is that there cannot exist a module
$M$ with $1<\npj(M)<\sqrt 2$. 

In Chapter~\ref{ch:eg}, we illustrate this
situation with some groups and modules where we can make explicit
computations. In particular, in case $G$ is a cyclic group of order
$p$ and $M$ is the indecomposable two-dimensional module, we have
$\npj(M)=2\cos(\pi/p)$. In Section~\ref{se:SL2q} we 
prove the following.

\begin{theorem*} 
If $G=SL(2,q)$ with $q$ a power
of $p$, amd $M$ is the two-dimensional natural module, then
$\npj(M)=2\cos(\pi/q)$. 
\end{theorem*}

The proof uses tilting theory. We conjecture that for representations
of finite groups, if
$\npj(M)<2$ then $\npj(M)=2\cos(\pi/q)$ for $q\ge 2$ an integer.

We illustrate
the structure of the space $\Struct(G)$ using the only cases where we
can make complete computations, namely $G=V_4$, the Klein four group,
and $G=A_4$, the alternating group of degree four, in characteristic two.
\bigskip

{\sc Notation.}
We shall often want to use $i$ as an index, so we write 
$\bi$\index{i@$\bi$} for the complex number $\sqrt{-1}$.\bigskip

\noindent
{\bf Acknowledgement.}
I thank Peter Symonds for getting me started on this
project, and Rob Archbold for carefully reading portions of this
manuscript. Part of this work was carried out while I was
resident at
the Mathematical Sciences Research Institute in Berkeley, 
California during the Spring 2018 semester, supported
by the National Science Foundation under Grant No.\ DMS-1440140.
Part was also carried out while I was resident at the Isaac Newton Institute for 
Mathematical Sciences in Cambridge, UK, during the 
Spring 2020 semester programme 
``Groups, representations and applications: new perspectives,'' 
supported by EPSRC grant EP/R014604/1.

\chapter{Abstract representation rings}\label{ch:abstract}

\section{Axioms for representation rings}

In this section we formulate the properties of representation rings that we shall
need. Our fundamental model is the modular representation ring, or
Green ring $a(G)$ of a finite group $G$ over a field
$k$ of characteristic $p$, but there are many 
other examples. Motivation for the
properties~(i)--(v) comes from this example, see
Proposition~\ref{pr:a(G)repring}.
In case the field $k$ is algebraically closed, the stronger 
property (ii$'$) below holds, and $a(G)$ is a closed representation ring.

\begin{defn}\label{def:repring}
A \emph{representation ring}\index{representation!ring}
consists of a commutative ring $\fa$\index{a@$\fa$} whose 
additive group is a free abelian group with a specified basis
consisting of symbols $x_i$\index{x@$x_i$, $x_0$} with $i$ in an indexing set
$\fI$.\index{indexing set $\fI$|textbf}\index{I@$\fI$|textbf}
The identity element $x_0=\one$\index{$\one$} of $\fa$ is
one of the basis elements, corresponding to
$0\in\fI$.\index{$0\in\fI$} 
Multiplication is given by
$x_ix_j = \sum_{k\in\fI} c_{i,j,k}x_k$ where the
\emph{structure constants}\index{structure!constants}
$c_{i,j,k}$ are non-negative integers, and 
given $i,j\in\fI$, there are only finitely many $k$ with $c_{i,j,k}\ne
0$. If $x=\sum_{i\in\fI} a_i x_i$ is an element of $\fa$ then we write $[x:x_i]$
for $a_i$, the \emph{multiplicity}\index{multiplicity} of
$x_i$ in $x$. Thus we have $x=\sum_{i\in\fI} [x:x_i]x_i$ and
$x_ix_j=\sum_{k\in\fI}[x_ix_j:x_k]x_k$.
\begin{enumerate}
\item There is an involutive permutation $i \mapsto i^*$ of the indexing set
$\fI$, which induces an involutive automorphism of $\fa$ sending
$x=\sum_{i\in\fI}a_ix_i$ to $x^*=\sum_{i\in\fI} a_i x_{i^*}$.\smallskip
\item
 If $[x_ix_j:\one]>0$ then $j=i^*$.\smallskip
\item
 If $i\in\fI$ satisfies $c_{i,i^*,0}= 0$ then
\[ \sum_{j\in\fI} c_{i,i^*,j}c_{j,i,i} \ge 2. \]
In other words, if $[x_ix_{i^*}:\one]=0$ then $[x_ix_{i^*}x_i:x_i]\ge 2$.\smallskip
\item
 There is a dimension function\index{dimension!function}
\[ \dim\colon \fa \to \bZ, \]
which is
a ring homomorphism with the property that for each $i\in \fI$,
\[ \dim(x_i)=\dim(x_{i^*})>0. \]
Thus $\dim(x)=\dim(x^*)$ for all $x\in \fa$.
\item
There is a non-zero element
$\rho\in\fa$\index{rho@$\rho$} which is a 
non-negative linear combination of the basis elements,
with the property that for all
$x\in\fa$, $x\rho  = (\dim x)\rho$.\newline
[See also Remark~\ref{rk:axiom(v)}.]
\end{enumerate}
A \emph{closed representation ring}%
\index{closed representation ring}\index{representation!ring!closed} 
is a representation ring satisfying a stronger version of (ii):
\begin{enumerate}
\item[(ii$'$)] If $[x_ix_j:\one]>0$ then $j=i^*$ and
  $[x_ix_j:\one]=1$.
\end{enumerate}
\end{defn}

\begin{egs}\label{eg:repring}
The modular representation ring of a finite group is the motivating
example, and this will be discussed at length in
Chapters~\ref{ch:fingrp} and~\ref{ch:eg}. 

The smallest a representation ring can be is $\bZ$ with just one basis
element, $x_0=1$. The element $\rho$ can be any positive integer, and
the dimension function $\bZ\to\bZ$ is the identity map. If $\rho=1$,
this is the representation ring of the trivial group.

Other natural examples include the
representation rings of finite group
schemes and finite supergroup schemes.

The following examples of representation rings 
are somewhat artificial, but will be used later to
illustrate some ideas. 
\begin{enumerate}
\item
Choose an integer $d\ge 2$, and 
let 
\[ \fa = \bZ[u,u^{-1}] \oplus \bZ \]
as a sum of a subring and an ideal,  
where the ideal summand $\bZ$ is spanned by $\rho$. Multiplication is given by
$u\rho=u^{-1}\rho=\rho$, $\rho^2=d\rho$. The index set $\fI$ is
$\bZ\cup\{\infty\}$ where $x_n=u^n$ for $n\in\bZ$ and $x_\infty=\rho$.
The involutive permutation on $\fI$ sends $n$ to $-n$ and fixes
$\infty$. The dimension function is given by $\dim u^n=1$,
$\dim\rho=d$.
\item
Choose an integer $d\ge 2$, and let
\[ \fa =\bZ[u,v]/((uv-1)(u-d),(uv-1)(v-d)). \]
The index set $\fI$ is again $\bZ\cup\{\infty\}$. The basis is given
by $x_0=1$,  $x_n=u^n$ and $x_{-n}=v^n$ for $n>0$, and
$x_\infty=\rho=uv-1$. Again the involutive permutation on $\fI$ sends
$n$ to $-n$ and fixes $\infty$. The dimension function is given by
\[ \dim u^n = \dim v^n=d^n,\qquad \dim\rho=d^2-1. \]
This example will be used in Section~\ref{se:endotriv} as an
illustration of the difference between the big endotrivial group and the
small endotrivial group.
\item
Let $\fa=\bZ[v]/(v^3-2v^2)$ with basis $x_0=1$, $x_1=v$,
$x_2=v^2-v$, $\rho=x_1+x_2$, and $x_1^*=x_2$, $x_2^*=x_1$. 
Then $x_1^2=x_1x_2=x_2^2=\rho$, and
$\dim x_1=\dim x_2 = 2$. Note that in this ring we have
$(x_2-x_1)^2=0$, so the nil radical\index{nil radical} of $\fa$ is non-zero.
\end{enumerate}
\end{egs}

\begin{rk}\label{rk:axiom(v)}
Examples which are not covered by our axioms include the
representation ring of a compact Lie group, as studied for example by 
Segal~\cite{Segal:1968b}. This example satisfies all but axiom~(v) of
Definition~\ref{def:repring}. Anything we do in this work that does
not mention projectives does not depend
on this axiom, and there are arguments for deleting it, but we have
chosen to retain it. The element $\rho$ in that axiom plays the role
of the regular representation.\index{regular representation}

There is an elementary way of enhancing such a 
representation ring in such a way that axiom~(v) holds, but this
method is not terribly satisfactory. Namely, given
$\fa$ satisfying all but this axiom, and an integer $n\ge 2$, 
then we endow $\fa\oplus\bZ$ with the
structure of a representation ring with one more basis element $\rho$
satisfying $\rho^*=\rho$, $x\rho=(\dim x)\rho$ for $x\in\fa$, 
$\dim\rho=n$, and $\rho^2=n\rho$. Note that $[\rho\rho^*:\one]=0$ and
$[\rho\rho^*\rho:\rho]=n^2\ge 2$, so that axioms~(ii) and~(iii) are satisfied.

Our argument for retaining axiom~(v), however, is that our focus will
be on studying representation ideals in representation rings, and
examples such as the representation ring of a compact Lie group have
no non-zero representation ideals. So most of what we do here is vacuous
in that case. Axiom~(v) does not ensure that there are non-zero
representation ideals, but if there are none, then the representation
ring is finite dimensional and semisimple, and looks very much like
the ordinary character ring of a finite group. We shall call this case
an \emph{ordinary representation ring},%
\index{ordinary representation!ring}\index{representation!ring!ordinary}
and the contrary case a 
\emph{modular representation ring},%
\index{modular representation ring}\index{representation!ring!modular}
see Definition~\ref{def:modular}. In a modular representation ring,
there is a unique maximal representation ideal and a 
unique minimal representation ideal, see 
Proposition~\ref{pr:max-min}.
\end{rk}

\begin{rk}
Algebras with distinguished bases and non-negative structure constants
have been studied in other representation theoretic contexts, with
different goals in mind. For example, Arad and
Blau~\cite{Arad/Blau:1991a} introduced the concept of a 
\emph{table algebra}\index{table algebra} in order to study
decompositions of products of ordinary characters or conjugacy classes
of finite groups. They allow non-negative real structure constants,
but they assume that the basis is finite, and their axioms imply that
$[x_ix^*_i:\one]>0$ for all basis elements $x_i$. So their definition
is closer to our definition of \emph{ordinary representation ring},
see Definition~\ref{def:modular} below. 

Lusztig~\cite{Lusztig:1987a} introduced the concept of a 
\emph{based ring}\index{based ring} for the study of Grothendieck
groups of Hecke algebras.\index{Hecke algebra} These do not have to be
commutative, but they do satisfy $[x_ix^*_i:\one]>0$. This subject is
further studied in Etingof, Gelaki, Nikshych and 
Ostrik~\cite{Etingof/Gelaki/Nikshych/Ostrik:2015a} with fusion
categories\index{fusion category} in mind; their Frobenius--Perron
dimension\index{Frobenius!P@--Perron dimenson} is related to our gamma
invariant.

Other notions include the cellular algebras of Graham and
Lehrer~\cite{Graham/Lehrer:1996a}, generalised to 
Green's concept of a 
\emph{tabular algebra}\index{tabular algebra}~\cite{Green:2002a}.
These are defined over $\bZ[v,v^{-1}]$ and are geared towards
Temperley--Lieb algebras\index{Temperley--Lieb algebra}, Hecke
algebras, Ariki--Koike algebras,\index{Ariki--Koike algebra} and so on.
\end{rk}

The following lemmas give some of the more elementary consequences of
the definitions. For this purpose, we introduce one more definition.

\begin{defn}\label{def:a_R}
If $R$ is a commutative ring, we write $\fa_R$\index{a@$\fa_{R}$} for the ring
$R\otimes_\bZ\fa$ obtained by extending scalars to $R$. 
Elements of $\fa_\bC$ are finite sums $\sum_{i\in\fI}a_ix_i$, with $a_i\in R$.
We shall mostly be interested in the case $R=\bC$,\index{a@$\fa_{\bC}$} the complex numbers, 
but other extensions of scalars will occasionally be considered. 
In case $R=\bC$, if $x=\sum_{i\in\fI}a_i x_i$ we define
$x^*=\sum_{i\in\fI}\bar{a}_i x_i^*$.
\end{defn}

\begin{lemma}
We have $\rho=\rho^*$.
\end{lemma}
\begin{proof}
By property~(i), $\rho^*$ also satisfies property~(v).
So by properties~(iv) and~(v) we have
$\dim(\rho^*)\rho=\rho^*\rho=\dim(\rho)\rho^*$, and also
$\dim(\rho^*)=\dim(\rho)>0$.
\end{proof}

\begin{lemma}\label{le:xixi*xi}
For all $i\in\fI$ we have $[x_ix_{i^*}x_i:x_i]>0$.
\end{lemma}
\begin{proof}
 If $[x_ix_{i^*}:\one]=0$ then this follows from property (iii).
  If $[x_ix_{i^*}:\one]>0$ then $[x_ix_{i^*}x_i:x_i]\ge
  [x_ix_{i^*}:\one]>0$.
\end{proof}

\begin{lemma}
The following are equivalent for a basis element $x_i$:
\begin{enumerate}
\item $\dim x_i = 1$.
\item $x_ix_{i^*}=\one$.
\item $x_i$ is invertible in $\fa$.
\end{enumerate}
\end{lemma}
\begin{proof}
(i) $\Rightarrow$ (ii):  
If $\dim x_i=1$ then $\dim x_ix_{i^*}=1$ and $\dim x_ix_{i^*}x_i=1$.
In particular,
$[x_ix_{i^*}:\one]$ is either zero or one. If it
were zero then by Definition~\ref{def:repring}\,(iii) we would have
$[x_ix_{i^*}x_i:x_i]\ge 2$, contradicting $\dim x_ix_{i^*}x_i=1$,
and so we have $[x_ix_{i^*}:\one]=1$. We have
$\dim(x_ix_{i^*}-\one)=0$, and hence $x_ix_{i^*}=\one$.

(ii) $\Rightarrow$ (iii) is obvious.

(iii) $\Rightarrow$ (i): If $x_i$ is invertible in $\fa$ then 
$\dim x_i$ is invertible in $\bZ$, hence equal to one.
\end{proof}

The following is the analogue of Proposition~2.2 of~\cite{Benson/Carlson:1986a}.

\begin{lemma}\label{le:BCideal}
  Let $i,j,k\in\fI$ with $[x_ix_j:x_k]>0$. If $[x_ix_{i^*}:\one]=0$ then
  $[x_kx_k^*:\one]=0$.
\end{lemma}
\begin{proof}
  It follows from property (ii) that if $[x_ix_{i^*}:\one]=0$ then
  for all $x\in\fa$ we have $[x_ix:\one]=0$. In particular,
  $[x_ix_jx_{k^*}:\one]=0$, and since $[x_ix_j:x_k]>0$ it follows that
  $[x_kx_{k^*}:\one]=0$.
\end{proof}

\begin{defn}\label{def:species}
If $\fa$ is a representation ring, a \emph{species}\index{species} of
$\fa$ is a ring homomorphism $s\colon\fa\to\bC$. A species extends
uniquely to a $\bC$-algebra homomorphism $s\colon \fa_\bC\to\bC$,\index{a@$\fa_{\bC}$}
which we also call a species.
\end{defn}

\begin{lemma}\label{le:species-indep}
Any set of species of $\fa$ is linearly independent.
\end{lemma}
\begin{proof}
Let $s_1,\dots,s_n\colon \fa \to \bC$ satisfy a non-trivial linear relation
\[ \lambda_1 s_1+\dots +\lambda_n s_n = 0 \]
with $n$ as small as possible. Choose $x\in \fa$ such that $s_1(x)\ne
s_2(x)$. Then for all $y\in \fa$ we have
\[ \lambda_1 s_1(xy)+ \lambda_2s_2(xy)+\dots +\lambda_n s_n(xy) = 0, \]
and hence
\[ \lambda_1 s_1(x)s_1(y) + \lambda_2s_2(x)s_2(y)+ \dots + \lambda_n s_n(x)s_n(y) = 0. \]
But also
\[ \lambda_1 s_1(x)s_1(y)+\lambda_2s_1(x)s_2(y) \dots + \lambda_n s_1(x) s_n(y) = 0. \]
Subtracting gives a shorter non-trivial linear relation, contradicting
the minimality of $n$.
\end{proof}

\section{Ordinary representation theory}

\begin{defn}\label{def:positive}
An element $x\in\fa$ is
\emph{non-negative}\index{non-negative element} if each $[x:x_i]\ge 0$,
and \emph{positive}\index{positive element} if, 
in addition, $x\ne 0$. We write $x\cge y$ or $y\cle x$
if $x-y$ is non-negative and $x\cg y$ or $y\cl x$
if $x-y$ is positive. We write
$\fa_{\cge 0}$\index{a@$\fa_{\cge 0}$, $\fa_{\cg 0}$}
for the set of non-negative elements, and $\fa_{\cg 0}$
for the set of positive elements.
\end{defn}

\begin{defn}\label{def:proj}
A basis element $x_i$ of a representation ring $\fa$ is said to be
\emph{projective indecomposable}\index{projective!indecomposable} if
$[\rho:x_i]>0$.  The number of projective indecomposables is finite.

An element $x=\sum_{i\in\fI} a_ix_i$ is said to be
\emph{virtually projective}\index{virtually projective element} 
if  $a_i\ne 0$ implies that $x_i$ is projective. If in addition $x\cge 0$
then we say that $x$ is \emph{projective}.\index{projective!element}
\end{defn}

\begin{lemma}\label{le:proj-ideal}\ 

\begin{enumerate}
\item If $x\in\fa_{\cge 0}$ and $y$ is projective then $xy$ is
  projective.
\item If $x\in\fa$ and $y$ is virtually projective then $xy$ is
  virtually projective.
\end{enumerate}
\end{lemma}
\begin{proof}
(i) It suffices to consider the case $x=x_i$, $y=x_j$ with $x_j$
projective. We have $x_i\rho=(\dim x_i)\rho$, so $x_i\rho$ is a 
non-negative linear combination of projective basis elements, and
hence so is $x_ix_j$.

(ii) follows from (i).
\end{proof}

The following theorem will have a generalisation involving commutative Banach
algebras in Theorem~\ref{th:no-qn}. The final statement of the 
theorem should also be contrasted with
Example~\ref{eg:repring}\,(iii).

\begin{theorem}[Ordinary representation theory]\label{th:ordinary}%
\index{ordinary representation!theory}
Suppose that $\one$ is projective in a representation
ring $\fa$. Then every element of $\fa$ is virtually projective, 
and the additive group of $\fa$ has finite rank. 
If $x\in\fa_\bC$ such that $xx^*=0$ then $x=0$.
There are no non-zero nilpotent elements in 
$\fa_\bC$.\index{a@$\fa_{\bC}$} 
\end{theorem}
\begin{proof}
If $\one$ is projective then for every basis element
$x_i$ we have
\[ (\dim x_i)[\rho:x_i]=[x_i\rho:x_i]\ge [\rho:\one]>0 \]
and so every $x_i$ is a projective indecomposable. Thus every element of
$\fa$ is virtually projective, and $\fa$ has finite rank.
Next, using Definition~\ref{def:repring}\,(ii) and the fact that
$\rho=\sum_j[\rho:x_j]x_j$, we have 
\[ [\rho:x_{i^*}][x_ix_{i^*}:\one]=\sum_j[\rho:x_j][x_ix_j:\one]
=[x_i\rho:\one]=(\dim x_i)[\rho:\one]>0 \]
and so $[x_ix_{i^*}:\one]>0$. Set $n_i=[x_ix_{i^*}:\one]$.\index{n@$n_i=[x_ix_{i^*}:\one]$}
Then 
\[ xx^*=\sum_i|a_i|^2x_ix_{i^*}+\sum_{i\ne j}a_i\bar a_j x_ix_{j^*} \]
and the coefficient of $\one$ in this is $\sum n_i|a_i|^2$. This is zero
if and only if $x=0$, so $xx^*=0$ implies $x=0$. If $x^2=0$ then
$(xx^*)(xx^*)^*=x^2x^{*\,2}=0$, so $xx^*=0$ and hence $x=0$.
\end{proof}

\begin{cor}\label{co:ordinary}
Suppose that $\one$ is projective in $\fa$. Then
as a $\bC$-algebra, $\fa_\bC$ 
is a direct sum of a finite number of copies
of $\bC$, and elements are separated by species
$s\colon \fa_\bC\to \bC$.
\end{cor}
\begin{proof}
It follows from Theorem~\ref{th:ordinary} that $\fa_\bC$ is 
finite dimensional and semisimple.\index{semisimple}
Every finite dimensional semisimple commutative algebra over $\bC$
has this property.
\end{proof}

\begin{defn}\label{def:modular}
Because of Theorem~\ref{th:ordinary} and Corollary~\ref{co:ordinary}, 
we shall say that $\fa$ is an \emph{ordinary representation ring}%
\index{ordinary representation!ring|textbf}\index{representation!ring!ordinary|textbf} 
if $\one$ is projective in $\fa$, and that $\fa$ is a
\emph{modular representation ring}%
\index{modular representation ring|textbf}\index{representation!ring!modular|textbf}
if $\one$ is not a projective indecomposable in $\fa$.

The \emph{character table}\index{character table} of an ordinary
representation ring $\fa$ is the square table of complex numbers whose rows
are indexed by the index set $\fI$ and whose columns are indexed by the species 
$s\colon \fa\to \bC$. The \emph{character} of an element
$x=\sum_{i\in\fI}a_ix_i\in\fa$ is the corresponding linear combination
of the rows of the character table. Elements of $\fa$ are determined
by their characters, and ring operations on elements correspond to
pointwise operations on the corresponding characters.
\end{defn}

The following property of ordinary representation rings is familiar
from representation theory of finite groups. The usual proof in that
case is that the inverse of an eigenvalue of a matrix of finite order is
its complex conjugate.

\begin{theorem}\label{th:ordinary=>symmetric}
If $s$ is a species of an ordinary representation ring $\fa$ 
then for all $x\in\fa_\bC$ we have $s(x^*)=\overline{s(x)}$.
\end{theorem}
\begin{proof}
Define a species $\bar{s}^*\colon \fa\to \bC$ by
$\bar{s}^*(x)=\overline{s(x^*)}$. By Corollary~\ref{co:ordinary}, we may
choose a
primitive idempotent element $e\in \fa_\bC$ such that $s(e)=1$, and
$s'(e)=0$ for all other species $s'$. If $\bar{s}^*\ne s$ then
$\bar{s}^*(e)=0$ and so $s(e^*)=0$. So for all species $s'$ of
$\fa_\bC$ we have $s'(ee^*)=0$, and hence $ee^*=0$. By
Theorem~\ref{th:ordinary}, this implies that $e=0$, which is a
contradiction. Hence $\bar{s}^*=s$, which then implies that for all
$x\in\fa_\bC$ we have $s(x^*)=\overline{s(x)}$.
\end{proof}

\section{Representation ideals and cores}

\begin{lemma}
If $x\in \fa_{\cge 0}$ then $\dim x \ge 0$;
if $x\in\fa_{\cg 0}$ then $\dim x > 0$.
\end{lemma}
\begin{proof}
It follows from Definition~\ref{def:repring}\,(iii) that if
$x=\sum_{i\in\fI}a_ix_i$ then
\begin{equation*}
\dim x = \sum_{i\in\fI}a_i\dim x_i,
\end{equation*}
and that this is $\ge 0$ if all $a_i\ge 0$, and $>0$ if in addition
some $a_i>0$.
\end{proof}

\begin{defn}\label{def:repideal}
A \emph{representation ideal}\index{representation!ideal}
$\fX$\index{X@$\fX$|textbf} of a
representation ring $\fa$ is a proper subset $\fX\subset \fI$ with the
following properties:
\begin{enumerate}
\item If $i\in \fX$, $j\in \fI$ and $k\in \fI\setminus \fX$ then
$c_{i,j,k}= 0$. Equivalently, if $i\in \fX$ and there exists $x\in\fa$
such that $[x_ix:x_k]\ne 0$ then $k\in\fX$.
\item If $i\in\fX$ then $i^*\in\fX$.
\end{enumerate}

The linear span $\langle\fX\rangle$
of a representation ideal $\fX$ is an ideal in $\fa$.
We write $\langle\fX\rangle_{\cge 0}$ for
$\langle\fX\rangle\cap\fa_{\cge 0}$ and
$\langle\fX\rangle_{\cg 0}$ for
$\langle\fX\rangle\cap\fa_{\cg 0}$.
We write $\fa_\fX$\index{a@$\fa_{\fX}$} 
for the quotient $\fa/\langle\fX\rangle$, and
$\fa_{\bC,\fX}$\index{a@$\fa_{\bC,\fX}$} for 
$\bC\otimes_\bZ\fa_{\fX}\cong\fa_\bC/\langle\fX\rangle_\bC$.

If $\fX$ is a representation ideal in $\fa$ and $x\in
\fa_{\cge 0}$ then we can write $x=x'+x''$ where
$x'=\sum_{i\in\fI\setminus\fX}a_ix_i$ and $x''=\sum_{i\in\fX}a_ix_i$.
We define the $\fX$-\emph{core}\index{core|textbf}\index{X@$\fX$-core} of $x$ to be $x'$,
and we denote it by $\core_\fX(x)$. In the case $\fX=\fX_\proj$, we
omit the subscript and just write
$\core(x)$.\index{core@$\core_\fX(x)$, $\core(x)$}
\end{defn}

\begin{lemma}\label{le:core}
Let $\fY\subseteq\fX$ be representation ideals in a representation ring $\fa$
and let $x,y,z\in\fa_{\cge 0}$.
\begin{enumerate}
\item The product $xy\cge 0$.
\item We have $\core_\fX(x)\cle\core_\fY(x)\cle x$, and in
particular we have
\[ \dim \core_\fX(x)\le \dim\core_\fY(x)\le \dim x. \]
\item If $x\cle y$ then $\dim x \le \dim y$ and  $\core_\fX(x)\cle
\core_\fX(y)$. In particular, we have
\[ \dim\core_\fX(x)\le\dim\core_\fX(y). \]
\item We have $\core_\fX(xy)=\core_\fX(\core_\fX(x)\core_\fX(y))\cle
\core_\fX(x)\core_\fX(y)$.
\item If $y\cle x$ and $x\in \langle \fX\rangle$ then $y\in\langle
\fX\rangle$.
\item If $y\cle x$ then $yz\cle xz$.
\item $\core_\fX(x^*)=\core_\fX(x)^*$.
\end{enumerate}
\end{lemma}
\begin{proof}
Parts (i), (ii) and (iii) follow from the fact that the structure constants
$c_{i,j,k}$ of $\fa$ are non-negative.
  
(iv)  The equality follows from the definitions of core and representation ideal,
and the inequality follows from (ii).
  
(v)  For each $i\in\fI$, $[x:x_i]\ge[y:x_i]$.
If $[y:x_i]>0$ then $[x:x_i]>0$ and so $i\in\fX$.

(vi) $xz-yz=(x-y)z$ is a product of elements of
$\fa_{\cge 0}$, and is
hence in $\fa_{\cge 0}$ by (i).

(vii) This follows from part (ii) of Definition~\ref{def:repideal}.
\end{proof}

\begin{lemma}\label{le:xx*x}
If $x\in\fa_{\cge 0}$ then $xx^*x\cge x$. If
also $[xx^*:\one]=0$ then $xx^*x\cge 2x$.
\end{lemma}
\begin{proof}
  It follows from Lemma~\ref{le:xixi*xi} that for each $i\in\fI$ we have
  $[xx^*x:x_i]\ge [x:x_i]$. If $[xx^*:\one]=0$ then by
  Definition~\ref{def:repring}\,(iii) we have $[xx^*x:x_i]\ge 2[x:x_i]$.
\end{proof}
  
\begin{prop}\label{pr:nilp}
  Let $\fX$ be a representation ideal in a representation ring $\fa$,
  and let $x\in\fa_{\cge 0}$.
  \begin{enumerate}
    \item If $xx^*\in\langle \fX\rangle$ then $x\in\langle \fX\rangle$.
    \item If $x^n\in\langle\fX\rangle$ for some $n>0$ then
      $x\in\langle\fX\rangle$.
 \end{enumerate}   
\end{prop}
\begin{proof}
  (i) If $xx^*\in\langle \fX\rangle$ then $xx^*x\in\langle \fX\rangle$.
  By Lemma~\ref{le:xx*x} we have $xx^*x\cge x$, and so by
  Lemma~\ref{le:core}\,(v) we have $x\in\langle\fX\rangle$.

(ii)  
We may suppose that $n\ge 2$.
If $x^n\in\langle\fX\rangle$ then
$x^nx^*=x^{n-2}(xx^*x)\in\langle\fX\rangle$.
By Lemma~\ref{le:xx*x} we have $xx^*x\cge x$ and so
by Lemma~\ref{le:core}\,(vi) we have $x^nx^*\cge x^{n-1}$.
Applying Lemma~\ref{le:core}\,(v), we have
$x^{n-1}\in\langle\fX\rangle$.
Now apply induction on $n$.
\end{proof}

\begin{defn}
Let $\fX$ be a representation ideal in a representation ring $\fa$.
We say that $x\in\fa_{\cge 0}$ is 
\emph{indecomposable modulo $\fX$}\index{indecomposable modulo $\fX$}
if $x-x_i\in\langle\fX\rangle$ for
some $i\in\fI\setminus\fX$. 
\end{defn}

\begin{lemma}\label{le:xy=0}
If $x,y\in\fa_{\cge 0}$ and $xy=0$ then either $x=0$ or $y=0$.
\end{lemma}
\begin{proof}
If $xy=0$ then $\dim(x)\dim(y)=\dim(xy)=0$ and so either $\dim(x)=0$
or $\dim(y)=0$. Hence $x=0$ or $y=0$.
\end{proof}

An ordinary representation ring
has no non-empty representation ideals, since any non-empty representation ideal
would have to contain $\rho$, and then it would have to contain all projective
indecomposables including $\one$, whereas representation ideals have
to be proper subsets of the basis.
The following definition gives some examples of representation ideals in modular
representation rings.

\begin{defn}\label{def:max-proj}
Let $\fa$ be a modular representation ring.
\begin{enumerate}
\item We write $\fX_{\max}$\index{X@$\fX_{\max}$|textbf}
for the subset $\{i\in\fI\mid[x_ix_{i^*}:\one]=0\}$ of $\fI$.
\item We write $\fX_\proj$\index{X@$\fX_{\proj}$}
  for the subset consisting of those $i\in\fI$ for which $x_i$ is
  projective, see Definition~\ref{def:proj}.
\end{enumerate}
We write $\fa_{\max}$ and 
$\fa_\proj$\index{a@$\fa_{\max}$|textbf}\index{a@$\fa_{\proj}$}
for the quotients $\fa/\langle\fX_{\max}\rangle$ and 
$\fa/\langle\fX_\proj\rangle$.
\end{defn}

Part (i) of the following proposition is the analogue of Lemma~2.5 of~\cite{Benson/Carlson:1986a}.

\begin{prop}\label{pr:max-min}
Let $\fa$ be a modular representation ring.
\begin{enumerate}
\item The subset $\fX_{\max}\subseteq\fI$ is the unique maximal proper
subset of $\fa$ that is a representation
ideal.\index{maximal!representation ideal}
\item The subset $\fX_\proj\subseteq\fI$ is the unique minimal
  non-empty subset of $\fa$ that is a representation
  ideal.\index{minimal!representation ideal}
\end{enumerate}
\end{prop}
\begin{proof}
(i) It follows from Lemma~\ref{le:BCideal} that $\fX_{\max}$ is a
representation ideal. It is a proper subset of $\fI$ since $[\one.\one:\one]=1$.
To see that $\fX_{\max}$ is the unique maximal proper representation ideal,
let $\fX$ be a representation ideal containing an element $i$ such
that $c_{i,i^*,0}>0$. Then $x_ix_{i^*}\in\langle\fX\rangle$ and so
$0\in\fX$ and $x_0=\one\in\langle\fX\rangle$. Thus $\fX=\fI$.

(ii) It follows from Lemma~\ref{le:proj-ideal} that $\fX_\proj$ is a
non-empty representation ideal. Conversely,
if $\fX$ is a non-empty representation ideal and $i\in\fX$ then
$x_i\rho=(\dim x_i)\rho$. If $x_j$ is projective then
$[\rho:x_j]>0$ and so $[x_i\rho:x_j]>0$. Thus $\fX_\proj\subseteq\fX$.
\end{proof}

\begin{cor}\label{co:zero-divisor}
If $x,y\in\fa_{\cge 0}$ and $xy\in\langle\fX_{\max}\rangle$ then
either $x\in\langle\fX_{\max}\rangle$ or
$y\in\langle\fX_{\max}\rangle$.
\end{cor}
\begin{proof}
If $xy\in\langle\fX_{\max}\rangle$ then by
Proposition~\ref{pr:max-min}\,(i), we have 
$xyy^*\in\langle\fX_{\max}\rangle$. If
$y\not\in\langle\fX_{\max}\rangle$ then by definition of $\fX_{\max}$
we have $[yy^*:\one]>0$ and so
$xyy^*\cge x$, and hence $x\in\langle\fX_{\max}\rangle$.
\end{proof}

The statement of Corollary~\ref{co:zero-divisor} is not necessarily true of other
representation ideals, but at least we have the following.

\begin{prop}\label{pr:xx*yy*}
Let $\fX$ be a representation ideal in $\fa$, and let $x,y\in
\fa_{\cge 0}$. If $xy$ is not in $\langle \fX\rangle$ then nor is any
product of terms of the form $x$, $x^*$, $y$ and $y^*$.
\end{prop}
\begin{proof}
If for example $xy^*$ or $xyy^*$ is in 
$\langle\fX\rangle$  then so is $xyy^*y$. 
But then by Lemma~\ref{le:xx*x} we have $xyy^*y\cge xy$, 
and hence by Lemma~\ref{le:core}\,(v) we have $xy\in\langle\fX\rangle$.
So an easy inductive argument on the number of terms in the product proves the proposition.
\end{proof}

\section{The gamma invariant}\label{se:npj}%
\index{gamma invariant $\npj_\fX(x)$|textbf}

Let $\fX$ be a representation ideal in a representation ring $\fa$, and
$x\in\fa_{\cge 0}$. We define\index{c@$\cc^\fX_n(x)$}
\[ \cc^\fX_n(x) = \dim\core_\fX(x^n) \]
and we form the generating function\index{generating function}\index{f@$f_{\fX,x}(t)$}
\[ f_{\fX,x}(t) = \sum_{n=0}^\infty \cc^\fX_n(x)t^n. \]
Since $\cc_n^\fX(x) \le \dim(x^n)=\dim(x)^n$, this power series
converges in a disc of radius at least $1/\dim(x)$ around the origin
in the complex plane.

\begin{lemma}[Cauchy,  Hadamard]\label{le:radius}\index{Cauchy--Hadamard formula|textbf}
Let $\phi\colon \bZ_{\ge 0} \to \bC$. Then the radius of
convergence\index{radius of convergence}\index{convergence} $r$ of the power series 
\[ f(t)=\sum_{n=0}^\infty \phi(n)t^n \] 
is given by
\[ 1/r = \limsup_{n \to \infty} \sqrt[n]{|\phi(n)|}. \]
Strictly inside the radius, the convergence is
uniform\index{uniform convergence}
and absolute.\index{absolute convergence}
\end{lemma}
\begin{proof}
See for example 
Conway \cite{Conway:1973a}, Theorem~III.1.3.
\end{proof}

\begin{defn}
  Let $x\in\fa_{\cge 0}$
  and let $\fX$ be a representation ideal of $\fa$.
  We define the gamma invariant of $x$ with respect to $\fX$ to
  be\index{gamma@$\npj_\fX(x)$, $\npj(x)$, $\npj_{\max}(x)$}
  \[ \npj_\fX(x) = \limsup_{n\to\infty}\sqrt[n]{\cc_n^\fX(x)}. \]
  By Lemma~\ref{le:radius}, this is equal to $1/r$,
  where $r$ is the radius of convergence of the power
  series $f_{\fX,x}(t)$. 

If $\fX=\varnothing$ then $\npj_\fX(x)=\dim x$. In the minimal
non-zero case $\fX=\fX_\proj$, we just write $\npj(x)$ 
for $\npj_{\fX_\proj}(x)$. In the maximal
case $\fX=\fX_{\max}$, we write $\npj_{\max}(x)$ for
$\npj_{\fX_{\max}}(x)$.\index{X@$\fX_{\max}$}
\end{defn}

\begin{lemma}\label{le:gamma-bounds}
  We have $0 \le \npj_\fX(x) \le\dim\core_\fX(x)\le \dim x$.
\end{lemma}
\begin{proof}
This follows from the inequalities
\[ 0\le \cc_n^\fX(x) \le (\dim\core_\fX(x))^n\le (\dim x)^n, \] 
see Lemma~\ref{le:core}\,(ii) and (iv).
\end{proof}

\begin{lemma}\label{le:mx}
If $x\in \fa_{\cge 0}$ and $m\in\bN$ then $\npj_\fX(mx)=m\npj_\fX(x)$.
\end{lemma}
\begin{proof}
We have $\core_\fX((mx)^n)=m^n\core_\fX(x^n)$. So
$\cc_n^\fX(mx)=m^n\cc_n(x)$, and hence
\begin{equation*}
\sqrt[n]{\cc_n^\fX(mx)}=m\sqrt[n]{\cc_n^\fX(x)}.
\qedhere
\end{equation*}
\end{proof}

\begin{lemma}\label{le:npjx*}
We have $\npj_\fX(x^*)=\npj_\fX(x)$.
\end{lemma}
\begin{proof}
It follows from Lemma~\ref{le:core}\,(vii) that 
\[ \core_\fX((x^*)^n)=\core_\fX((x^n)^*)=\core_\fX(x^n)^* \]
and so $\dim\core_\fX((x^*)^n)=\dim\core_\fX(x^n)$. Now take $n$th
roots and apply $\displaystyle\limsup_{n\to\infty}$.
\end{proof}

\begin{lemma}\label{le:gamma(x+fX)}
If $x,y\in\fa_{\cge 0}$ and $y\in\langle\fX\rangle$ then
$\npj_\fX(x+y)=\npj_\fX(x)$.
\end{lemma}
\begin{proof}
We have $(x+y)^n=x^n+z$ with $z\in\langle\fX\rangle$, so
$\core_\fX((x+y)^n) =\core_\fX(x^n)$ and $\cc_n^\fX(x+y)=\cc_n^\fX(x)$.
\end{proof}

\begin{lemma}\label{le:gamma(xy)}
  We have $\npj_\fX(xy) \le \npj_\fX(x)\npj_\fX(y)$.
\end{lemma}
\begin{proof}
  It follows from Lemma~\ref{le:core}\,(iv) that
  $\cc_n^\fX(xy)\le \cc_n^\fX(x)\cc_n^\fX(y)$, and so
  \[ \sqrt[n]{\cc_n^\fX(xy)} \le
    \sqrt[n]{\cc_n^\fX(x)}\sqrt[n]{\cc_n^\fX(y)}. \]
  Now apply $\displaystyle\limsup_{n\to\infty}$ to both sides.
\end{proof}

Although equality does not generally hold in Lemma~\ref{le:gamma(xy)},
we have the following.

\begin{lemma}\label{le:gamma(x^m)}
   We have $\npj_\fX(x^m) = \npj_\fX(x)^m$.
\end{lemma}
\begin{proof}
  By Lemma~\ref{le:gamma(xy)} we have $\npj_\fX(x^m) \le
  \npj_\fX(x)^m$. Conversely, if $n=ms+i$ with $0\le i < m$ then
  $x^n=x^i(x^m)^s$ and so
  \[ \cc^\fX_n(x) \le (\dim x)^m \cc^\fX_s(x^m). \]
  Thus
  \[ \sqrt[n]{\cc^\fX_n(x)} \le \sqrt[ms]{\cc^\fX_n(x)} \le
    \sqrt[s]{\dim x}\sqrt[m]{\sqrt[s]{\cc^\fX_s(x^m)}}. \]
  Applying $\displaystyle\limsup_{n\to\infty}$, the factor
  $\sqrt[s]{\dim x}$ tends to $1$. It follows that
  \begin{equation*}
    \npj_\fX(x)\le \sqrt[m]{\npj_\fX(x^m)}.
    \qedhere
  \end{equation*}
\end{proof}
  
\begin{lemma}\label{le:gamma<dim}
  We have $\npj_\fX(x)=\dim x$ if and only if $\core_\fX(x^n)=x^n$ for
  all $n\ge 0$.
\end{lemma}
\begin{proof}
If $\core_\fX(x^n)=x^n$ then $\cc_n^\fX(x)=(\dim x)^n$ and $\npj_\fX(x)=\dim
x$. On the other hand, if $\core_\fX(x^n)\ne x^n$ for some $n\ge 0$
then we have $\core_\fX(x^n)\cl x^n$, and
hence $\dim\core_\fX(x^n)<(\dim x)^n$. So using
Lemma~\ref{le:gamma(x^m)} we have
\[ \npj_\fX(x)^n = \npj_\fX(x^n)\le \dim\core_\fX(x^n)<(\dim x)^n. \]
Taking $n$th roots, we have $\npj_\fX(x) < \dim x$.
\end{proof}

\begin{lemma}\label{le:gamma=0}
  We have $\npj_\fX(x)=0$ if and only if $x\in\langle\fX\rangle$.
  Otherwise $\npj_\fX(x)\ge 1$.
\end{lemma}
\begin{proof}
  If $x\in\langle \fX\rangle$ then $f_{\fX,x}(t)$ is the zero power
  series and so $\npj_\fX(x)=0$. On the other hand, if $x\not
  \in\langle\fX\rangle$ then by Proposition~\ref{pr:nilp}, no positive power of
  $x$ is in $\langle\fX\rangle$. So each coefficient of $f_{\fX,x}(t)$
  is at least $1$, and hence $\npj_\fX(x)\ge 1$.
\end{proof}

\begin{theorem}\label{th:ge-m}
We have the following.
\begin{enumerate}
\item
Let $\fX$ be a representation ideal in $\fa$.
If $y_1,\dots,y_m\in\fa_{\cg 0}$ with product 
$y_1\dots y_n\not\in\langle \fX\rangle$ then 
$\npj_\fX(y_1+\dots+y_m)\ge m$. 
\item
If $y_1,\dots,y_m\in\fa_{\cg 0}$ are not in $\langle \fX_{\max}\rangle$ then 
$\npj_\fX(y_1+\dots+y_m)\ge m$.
\end{enumerate}
\end{theorem}
\begin{proof}
(i) If $y_1\dots y_n\not\in\langle\fX\rangle$ then by
Proposition~\ref{pr:nilp},
neither is any element of the form $y_1^{n_1}\dots y_m^{n_m}$. The
element $(y_1+\dots+y_m)^n$ is therefore a sum of at least $m^n$ 
elements of $\fa_{\cg 0}$ not in $\langle \fX\rangle$, and so 
\[ \cc^\fX_n(y_1+\dots+y_m)\ge m^n \]
and 
\[ \sqrt[n]{\cc^\fX_n(y_1+\dots+y_m)}\ge m. \]
Now take $\displaystyle\limsup_{n\to\infty}$.

(ii) If $y_1,\dots,y_n$ are not in $\langle \fX_{\max}\rangle$ then by 
Corollary~\ref{co:zero-divisor}, nor is $y_1\dots y_n$. So we are in
a position to apply (i).
\end{proof}

\begin{lemma}\label{le:gammaXinY}
  If $\fX\subseteq\fY$ are representation ideals in a representation
  ring $\fa$ and $x\in\fa_{\cge 0}$ then
  \[ \npj_\fY(x) \le \npj_\fX(x). \]
\end{lemma}
\begin{proof}
  By Lemma~\ref{le:core}\,(ii) we have $\cc^\fY_n(x) \le \cc^\fX_n(x)$
  for all $n\ge 0$ and so
  \begin{equation*}
    \limsup_{n\to\infty}\sqrt[n]{\cc_n^\fY(x)}\le
    \limsup_{n\to\infty}\sqrt[n]{\cc_n^{\fX\vphantom\fY}(x)}.
    \qedhere
    \end{equation*}
\end{proof}

\section{Pringsheim's Theorem}

This section is not logically necessary for the development of the
subject, but is closely related to Theorem~\ref{th:npj-in-spec}.

\begin{defn}
Given a function $f(t)$ of a complex variable $t$, we say that it is 
\emph{analytic}\index{analytic function} 
at a point $t=a$ if there is a power series in $t-a$ 
with a positive radius of convergence, and converging to the value
of $f(t)$ in an open neighbourhood of $a$. We say that $a$ is a 
\emph{singular point}\index{singular point} of $f(t)$ if it not
analytic at $a$.
\end{defn}

The following theorem is not so well known.
See also
Statement (7.21) in Chapter VII of Titchmarsh~\cite{Titchmarsh:1939a}.

\begin{theorem}[Pringsheim]\label{th:Pringsheim}\index{Pringsheim's Theorem}
Suppose that $\phi\colon \bZ_{\ge 0} \to \bR_{\ge 0}$, and that the
power series 
\[ f(t)=\displaystyle\sum_{n=0}^\infty \phi(n)t^n \]
has radius of
convergence $r$. Then 
$t=r$ is a singular point of $f(t)$.
\end{theorem}
\begin{proof}
The geometric fact used in the proof of this theorem is that the union
in $\bC$ of a disc of radius $r$ centred at zero and a disc of
positive radius centred at $r$ contains a disc of radius strictly
greater than $r/2$ centred at $r/2$.

Expand $f(t)$ as a Taylor series about $t=\frac{r}{2}$:
\[ f(t) = \sum_{n=0}^\infty
\frac{f^{(n)}(\frac{r}{2})}{n!}(t-\textstyle\frac{r}{2})^n \]
where $f^{(n)}(t)/n! = \sum_{m\ge n}\phi(m)\binom{m}{n}t^{m-n}$. Thus
\[ f(t) = \sum_{n=0}^\infty(t-{\textstyle\frac{r}{2}})^n\sum_{m=
	n}^\infty\phi(m)\textstyle\binom{m}{n}(\frac{r}{2})^{m-n}. \]
If $t=r$ is not a singular point of $f(t)$ then for $\ep$ small enough
this converges at $t=r+\ep$. The terms are all non-negative reals, so
the sum is absolutely convergent, and we may rearrange the terms to
get 
\begin{align*} 
f(r+\ep) &= \sum_{n=0}^\infty({\textstyle\frac{r}{2}}+\ep)^n\sum_{m=
	n}^\infty\phi(m)\textstyle\binom{m}{n}(\frac{r}{2})^{m-n} \\
	&= \sum_{m= 0}^\infty\phi(m)\sum_{n=0}^m\textstyle
	\binom{m}{n}(\frac{r}{2}+\ep)^n(\frac{r}{2})^{m-n} \\
	& = \sum_{m=0}^\infty \phi(m)(r+\ep)^m.
\end{align*}
The convergence of this sum implies that the radius of 
convergence of $f$ is larger than $r$,
contradicting the hypotheses of the theorem.
\end{proof}

\begin{corqed}
  Let $x$ be a positive element of a representation ring $\fa$, and
  let $\fX$ be a representation ideal of $\fa$. If $x\not\in\langle \fX\rangle$,
then the positive real number $1/\npj_\fX(x)$ is a singular point of $f_{\fX,x}(t)$.
\end{corqed}

\section{Submultiplicative sequences}\index{submultiplicative sequence}

\begin{defn}
We say that a sequence $c_0,c_1,c_2,\dots$ of non-negative real
numbers is \emph{submultiplicative} if $c_0=1$, and for all $m,n\ge 0$ we have 
$c_{m+n} \le c_m.c_n$.
\end{defn}

\begin{lemma}\label{le:cc-submult}
  If $\fX$ is a representation ideal in a representation ring $\fa$
  and $x$ is a positive element of $\fa$ that is not in
  $\langle\fX\rangle$ then $\cc_n^\fX(x)$ is a
  submultiplicative sequence.
\end{lemma}
\begin{proof}
This follows from Lemma~\ref{le:core}\,(iv).
\end{proof}

\begin{lemma}[Fekete \cite{Fekete:1923a}]\label{le:Fekete}\index{Fekete's lemma}
If $c_n$ is a submultiplicative sequence then
\[ \limsup_{n\to\infty} \sqrt[n]{c_n} = \lim_{n\to\infty}\sqrt[n]{c_n}= \inf_{n\ge 1}
  \sqrt[n]{c_n}. \]
\end{lemma}
\begin{proof}
It suffices to show that $\limsup_{n\to\infty} \sqrt[n]{c_n} \le \inf_{n\ge 1}
\sqrt[n]{c_n}$. If some $c_n$ is equal to zero, then so are all subsequent ones. So we
assume that all $c_n>0$. Suppose that $L$ is a number such that
\[ \inf_{n\to\infty}\sqrt[n]{c_n} < L. \] 
Then there is an $m\ge 1$ with $\sqrt[m]{c_m}<L$. For $n>m$ we use  division
with remainder to write $n=mq_m+r_m$ with $0\le r_m<m$.
By the definition of submultiplicativity, we have
\[ c_n = c_{mq_m+r_m}\le c_{mq_m}c_{r_m}\le (c_{m})^{q_m}c_{r_m}. \]
Now $q_m\le n/m$, so $q_m/n\le 1/m$. So we have
\[ \sqrt[n]{c_n} \le \sqrt[m]{c_m}\sqrt[n]{c_{r_m}} <
  L.\sqrt[n]{c_{r_m}}. \]
As $n$ tends to infinity, the numbers
$\sqrt[n]{c_0},\dots,\sqrt[n]{c_{m-1}}$ all tend to one, and so
\begin{equation*}
\limsup_{n\to\infty}\sqrt[n]{c_n}\le L.
\qedhere
\end{equation*}
\end{proof}

\begin{theorem}\label{th:inf}
  If $x$ is a positive element of of a representation ring $\fa$
  and $\fX$ is a representation ideal of $\fa$ then
\[ \npj_\fX(x)=\displaystyle\lim_{n\to\infty}\sqrt[n]{\cc_n^\fX(x)}
=\inf_{n\ge 1}\sqrt[n]{\cc_n^\fX(x)}. \]
\end{theorem}
\begin{proof}
This follows from Lemmas \ref{le:cc-submult} and \ref{le:Fekete}.
\end{proof}

\begin{lemma}\label{le:binom}
Let $a_n$, $b_n$ and $c_n$ be sequences of non-negative real numbers, satisfying
\[ c_n \le \sum_{i=0}^n \binom{n}{i} a_i b_{n-i}. \]
Then
\[ \limsup_{n\to\infty}\sqrt[n]{c_n}\le
\limsup_{n\to\infty}\sqrt[n]{a_n}+\limsup_{n\to\infty}\sqrt[n]{b_n}. \]
\end{lemma}
\begin{proof}
The statement that $\displaystyle\limsup_{n\to\infty}\sqrt[n]{a_n}= \alpha$ implies
that for all $\ep>0$, there exists $m$ such that 
for all $n\ge m$ we have $a_n \le (\alpha+\ep)^n$. Introducing
a positive constant $A$, we can assume that $a_n \le A(\alpha+\ep)^n$
for all $n\ge 0$.
Similarly, if
$\displaystyle\limsup_{n\to \infty}\sqrt[n]{b_n}=\beta$ then for all $\ep>0$ there
exists a positive constant $B$ such that for all $n\ge 0$ we have $b_n\le B(\beta+\ep)^n$.
Thus for all $\ep>0$ there is a positive constant $C=AB$ such that for all
 $n\ge 0$ we have
\begin{equation*} 
c_n \le \sum_{i=0}^n \binom{n}{i}A(\alpha+\ep)^iB(\beta+\ep)^{n-i} 
= C(\alpha+\beta+2\ep)^n,
\end{equation*}
and so $\displaystyle\limsup_{n\to\infty}\sqrt[n]c_n \le \alpha+\beta$.
\end{proof}

\begin{theorem}\label{th:gamma-subadditive}
Let $x,y\in\fa_{\cge 0}$ and let $\fX$ be a representation ideal in
$\fa$. Then 
\[ \max\{\npj_\fX(x),\npj_\fX(y)\}\le 
\npj_\fX(x+y)\le \npj_\fX(x)+\npj_\fX(y). \]
\end{theorem}
\begin{proof}
It follows from Lemma~\ref{le:core}\,(iv) that
\[ \max\{\cc_n^\fX(x),\cc_n^\fX(y)\}\le
\cc_n^\fX(x+y) \le
  \sum_{i=0}^n\binom{n}{i}\cc_i^\fX(x)\cc_{n-i}^\fX(y). \]
Applying Lemma~\ref{le:binom}, we deduce that 
\begin{multline*} 
\max\{\limsup_{n\to\infty}\sqrt[n]{\cc_n^\fX(x)},
\limsup_{n\to\infty}\sqrt[n]{\cc_n^\fX(y)}\}\\
\le\limsup_{n\to\infty}\sqrt[n]{\cc_n^\fX(x+y)}
\le\limsup_{n\to\infty}\sqrt[n]{\cc_n^\fX(x)}+
\limsup_{n\to\infty}\sqrt[n]{\cc_n^\fX(y)},
\end{multline*}
which are the inequalities in the statement of the theorem.
\end{proof}

\begin{prop}\label{pr:binom2}
Suppose that $a_n$ and $b_n$ are submultiplicative sequences. Define
a sequence $c_n$ by
\[ c_n = \sum_{i=0}^n\binom{n}{i}a_ib_{n-i}. \]
Then $c_n$ is also a submultiplicative sequence, and we have
\[  \lim_{n\to\infty}\sqrt[n]{c_n}=
\lim_{n\to\infty}\sqrt[n]{a_n} +\lim_{n\to\infty}\sqrt[n]{b_n}. \]
\end{prop}
\begin{proof}
Using the fact that
\[ \binom{m+n}{\ell} =\sum_{i+j=\ell}\binom{m}{i}\binom{n}{j}\]
and the submultiplicativity of the sequences $a_n$ and $b_n$, we have
\[ \sum_{\ell=0}^{m+n}\binom{m+n}{\ell}a_\ell b_{m+n-\ell} \le 
\left(\sum_{i=0}^m\binom{m}{i}a_ib_{m-i}\right).
\left(\sum_{j=0}^n\binom{n}{j}a_jb_{n-j}\right) \]
and so the sequence $c_n$ is submultiplicative.

By Lemma \ref{le:binom} we have
\[ \lim_{n\to\infty}\sqrt[n]{c_n}\le
\lim_{n\to\infty}\sqrt[n]{a_n} +\lim_{n\to\infty}\sqrt[n]{b_n}. \]
The reverse inequality is proved similarly. If 
$\displaystyle\lim_{n\to\infty}\sqrt[n]{a_n}=\alpha$ and
$\displaystyle\lim_{n\to\infty}\sqrt[n]{b_n}=\beta$ then given $\ep>0$ there
exist positive constants $A$ and $B$ such that for all $n\ge 0$ we have
$a_n\ge A(\alpha-\ep)^n$ and $b_n\ge B(\beta-\ep)^n$. So for
all $\ep>0$ there is a positive constant $C=AB$ such that for all
$n\ge 0$ we  have
\begin{equation*} 
c_n \ge \sum_{i=0}^n \binom{n}{i}A(\alpha-\ep)^iB(\beta-\ep)^{n-i} 
= C(\alpha+\beta-2\ep)^n,
\end{equation*}
and so
\begin{equation*} 
\lim_{n\to\infty}\sqrt[n]{c_n}\ge\alpha+\beta-2\ep=
\lim_{n\to\infty}\sqrt[n]{a_n} +\lim_{n\to\infty}\sqrt[n]{b_n}-2\ep. 
\qedhere
\end{equation*}
\end{proof}

\begin{theorem}\label{th:a+bx}
  Let $x\in\fa_{\cge 0}$ and
  let $\fX$ be a representation ideal in $\fa$. Then for non-negative
  integers $a$ and $b$ we have
  \[ \npj_\fX(a + bx) = a + b\npj_\fX(x).\]
\end{theorem}
\begin{proof}
  We begin with the case $a=0$. In this case we have
  $\core_\fX(bx)^n=b^n\core_\fX(x^n)$ and so
  $\cc_n^\fX(bx)=b^n\cc_n^\fX(x)$. Thus
  \[ f_{\fX,bx}(t)=\sum_{n=0}^\infty b^n\cc_n^\fX(x)t^n=f_{\fX,x}(bt) \]
  and $\npj_\fX(bx) = b\npj_\fX(x)$.

  It now suffices to show that $\npj_\fX(1+x)=1+\npj_\fX(x)$.
  We have
  \[ \cc_n^\fX(1+x) = \sum_{i=0}^n \binom{n}{i}\cc_i^\fX(x). \]
  So we can apply Proposition~\ref{pr:binom2} with $a_n=\cc_n^\fX(x)$,
  $b_n=1$ and $c_n=\cc_n^\fX(1+x)$.
\end{proof}

\section{The endotrivial group}\label{se:endotriv}\index{endotrivial!group}

\begin{defn}
  Let $\fX$ be a representation ideal in a representation ring $\fa$.
  We say that an element $x$ of $\fa_{\cge 0}$ is
  \emph{$\fX$-endotrivial}\index{endotrivial!element} if
  $xx^*-\one\in\langle\fX\rangle$. Note that this implies that
  $[xx^*:\one]=1$, and hence $xx^*\cge \one$.
In case $\fX=\fX_\proj$,
  we just say that $x$ is \emph{endotrivial}.
\end{defn}

We begin by collecting some properties of $\fX$-endotrivial elements.

\begin{lemma}
\begin{enumerate}
\item
If $x\in\fa_{\cge 0}$ is $\fX$-endotrivial then $x$ is
indecomposable modulo $\fX$ and 
$x\not\in\langle\fX_{\max}\rangle$.\index{X@$\fX_{\max}$}
\item
If $x\in\fa_{\cge 0}$ and $xx^*$ is $\fX$-endotrivial then so is $x$.
\end{enumerate}
\end{lemma}
\begin{proof}
(i) If $x=y+z$ with $y,z\in \fa_{\cge 0}$ and $y,z
\not\in\langle\fX\rangle$ then $xx^*=yy^*+yz^*+zy^*+zz^*$, and by
Proposition~\ref{pr:nilp}\,(i), neither $yy^*$ nor $zz^*$ is in $\langle
\fX\rangle$. But then $xx^*-\one$ cannot be in $\langle\fX\rangle$,
contradicting the definition of $\fX$-endotrivial. So $x$ is
indecomposable modulo $\fX$.

The fact that $x$ is not in
$\langle\fX_{\max}\rangle$ follows from $[xx^*:\one]>0$.

(ii) If $xx^*$ is $\fX$-endotrivial then $xx^*$ is indecomposable
modulo $\fX$,
and $xx^*$ is not in $\langle\fX_{\max}\rangle$, by (i).
By Proposition~\ref{pr:max-min}, $\fX_{\max}$ is a representation
ideal, and so $x$ is not in $\langle\fX_{\max}\rangle$. Thus by
definition of $\fX_{\max}$, we have $[xx^*:\one]>0$. Since $xx^*$ is
indecomposable modulo $\fX$, this implies that
$xx^*-\one\in \langle\fX\rangle$, and hence $x$
is $\fX$-endotrivial.
\end{proof}

\begin{theorem}\label{th:endotriv}
  Let $\fX$ be a representation ideal in a representation ring $\fa$,
  and let $x$ be a non-negative element of $\fa$. If $\npj_\fX(x)=1$
  then $x$ is $\fX$-endotrivial.
\end{theorem}
\begin{proof}
Suppose that $\npj_\fX(x)=1$. Then $\npj_\fX(x^*)=1$, and by Lemma~\ref{le:gamma(xy)} we
have $\npj_\fX(xx^*)\le 1$ and $\npj_\fX(xx^*x)\le 1$.
If $\npj_\fX(xx^*)=0$ then $[xx^*:\one]=0$
and so by Lemma~\ref{le:xx*x} we have $xx^*x\cge 2x$. But then
using Theorem~\ref{th:a+bx} we have
$1\ge \npj_\fX(xx^*x)\ge \npj_\fX(2x)= 2$, a contradiction. It
follows that $\npj_\fX(xx^*)=1$ and $[xx^*:\one]=1$. Thus $xx^*-\one\cge
0$, and again using Theorem~\ref{th:a+bx}, we have $\npj_\fX(xx^*-1)=0$. 
By Lemma~\ref{le:gamma=0}, we have $xx^*-\one\in\langle\fX\rangle$.
\end{proof}

\begin{defn}
  The \emph{big endotrivial group}\index{big endotrivial group}
  modulo $\fX$ of a representation ring $\fa$,
  denoted $\PPic_\fX(\fa)$, is the set of
  $\fX$-endotrivial elements of $\fa_{\cge 0}$. The \emph{small
    endotrivial group}\index{small endotrivial group} modulo $\fX$, denoted $\Pic_\fX(\fa)$, is the set
  of elements of $\fa_{\cge 0}$ satisfying $\npj_\fX(x)=1$. These are
  both abelian groups under multiplication modulo $\fX$, with the inverse of $x$
  being given by $x^*$. The two versions of the
  endotrivial group do not have to be equal, but in examples
  coming from representation theory, they often are. In
  Example~\ref{eg:repring}\,(ii), the elements $u$ and $v$ are
  endotrivial, but $\npj(u)=\npj(v)=d\ge 2$,
  so in this case the big and small endotrivial groups are not equal.
We write $\PPic_{\max}(\fa)$, $\Pic_{\max}(\fa)$ and 
$\PPic(\fa)$, $\Pic(\fa)$ for the cases $\fX=\fX_{\max}$
and $\fX=\fX_\proj$.
\end{defn}

The following is a generalisation of Theorem~3.5 of
Carlson~\cite{Carlson:1996a}, and shows that the only positive
idempotent modulo a representation ideal $\fX$ is the identity element.

\begin{theorem}
  Let $\fX$ be a representation ideal in a representation ring
  $\fa$. If $x\in\fa_{\cge 0}$,
  $x\not\in\langle\fX\rangle$, but
  $x^2-x\in\langle \fX\rangle_{\cge 0}$, then
  $x-\one\in\langle\fX\rangle_{\cge 0}$.
\end{theorem}
\begin{proof}
By Lemma~\ref{le:gamma(x+fX)}, 
the hypotheses imply that $\npj_\fX(x^2)=\npj_\fX(x)$. Using
Lemmas~\ref{le:gamma(x^m)}, this becomes $\npj_\fX(x)^2=\npj_\fX(x)$.
By Lemma~\ref{le:gamma=0}
we have $\npj_\fX(x)\ne 0$, and hence 
$\npj_\fX(x)= 1$. Using Theorem~\ref{th:endotriv},
we deduce that $x$ is $\fX$-endotrivial. Since $x^2-x\in
\langle\fX\rangle$ we deduce that $(x^2-x)x^*\in\langle\fX\rangle$ and
so $x-\one\in\langle\fX\rangle$. Thus $[x:\one]\ge 1$, and so $x-\one\cge 0$.
\end{proof}

\section{Elements with small gamma invariant}

Let $x\in\fa_{\cge 0}$ and let $\fX$ be a representation ideal in $\fa$.
We saw in Lemma~\ref{le:gamma=0} that if
$\npj_\fX(x)=0$ then $x\in\langle\fX\rangle$, and that otherwise
$\npj_\fX(x)\ge 1$. Furthermore, we saw in Theorem~\ref{th:endotriv}
that if $\npj_\fX(x)=1$ then $x$ is $\fX$-endotrivial. We strengthen this in
the following theorem, to
show that if $1\le \npj_\fX(x)<\sqrt{2}$ then $x$ is $\fX$-endotrivial.

\begin{theorem}\label{th:sqrt2}
If $x\in\fa_{\cg 0}$ is not $\fX$-endotrivial then $\npj_\fX(xx^*)\ge 2$, 
and $\npj_\fX(x)\ge \sqrt{2}$. In the case where
$\npj_\fX(x)=\sqrt{2}$, we have $\npj_\fX(xx^*)=2$, and 
$xx^*x-2x\in\langle\fX\rangle$. 

If, furthermore, 
$x\not\in\langle\fX_{\max}\rangle$\index{X@$\fX_{\max}$} then $xx^*=1+y$ with $\npj_\fX(y)=1$.
In particular, $y$ is $\fX$-endotrivial.
\end{theorem}
\begin{proof}
We divide into two cases, according as $[xx^*:\one]=0$ or
$[xx^*:\one]>0$.

If $[xx^*:\one]=0$ then $x\in\langle\fX_{\max}\rangle$, and 
by Lemma~\ref{le:xx*x} we have $xx^*x\cge 2x$
and so $(xx^*)^2\cge 2xx^*$. It follows using Lemma~\ref{le:core}\,(iii) that
$\npj_\fX((xx^*)^2)\ge 2\npj_\fX(xx^*)$. Since $\npj_\fX(xx^*)>0$ it
follows that $\npj_\fX(xx^*)\ge 2$. 
Using Lemma~\ref{le:gamma(xy)} we have
\[ 2\le \npj_\fX(xx^*)\le \npj_\fX(x)\npj_\fX(x^*)=\npj_\fX(x)^2 \]
and so $\npj_\fX(x)\ge \sqrt{2}$. If $\npj_\fX(x)=\sqrt{2}$ this shows
that $\npj_\fX(xx^*)=2$. Set $v=xx^*x-2x$. If $v\not\in\langle\fX\rangle$, then
by Proposition~\ref{pr:nilp}\,(i), $vv^*\not\in
\langle\fX\rangle$. This implies that
$vx^*xx^*=vv^*+2vx^*\not\in\langle\fX\rangle$ and so
$vx\not\in\langle\fX\rangle$. By Theorem~\ref{th:ge-m}\,(i), we then have
$2\sqrt{2}\ge\npj_\fX(xx^*x)=\npj_\fX(v+2x)\ge 3$, a contradiction. Thus
$v\in\langle\fX\rangle$. 

On the other hand, if $[xx^*:\one]>0$ then
$x\not\in\langle\fX_{\max}\rangle$ and $xx^*=\one+y$ with $y\cge 0$. Since 
$x$ is not $\fX$-endotrivial we have $y\not\in\langle\fX\rangle$ and hence
$\npj_\fX(y)\ge 1$. Thus by Theorem~\ref{th:ge-m}, $\npj_\fX(xx^*)\ge 2$.
Again this shows that $\npj_\fX(x)\ge \sqrt{2}$, and if
$\npj_\fX(x)=\sqrt{2}$ then $\npj_\fX(xx^*)=2$ and $xx^*=\one+y$ with
$\npj_\fX(y)=1$. So $y$ is $\fX$-endotrivial, and in particular
indecomposable modulo $\fX$. Furthermore, $y=y^*$, 
so $y^2-\one\in\langle\fX\rangle$.
Thus $(xy)x^*=y+y^2$, which is $\one+y$ plus an element of $\langle\fX\rangle$.
Hence $(xy-x)x^*\in\langle\fX\rangle$, and so
$(xy-x)(xy-x)^*\in\langle\fX\rangle$,
and by Proposition~\ref{pr:nilp}\,(i), $xy-x\in\langle\fX\rangle$. Finally, we have
$xx^*x=(\one+y)x$ and so $xx^*x-2x\in\langle\fX\rangle$.
\end{proof}

\begin{eg}
Let $d\ge 2$ be an integer, and let  
$\fa$ be the representation ring with basis $x_0=1$, $x_1=x$, $x_2=y$ and $x_3=\rho$
with multiplication table
\[ \begin{array}{|cccc} \hline 
1&x&y&\rho \\ 
x&1+y+\rho&x&d\rho\\
y&x&1&\rho\\
\rho&d\rho&\rho&(d^2-2)\rho.
\end{array} \]
with $x^*=x$, $y^*=y$, $\rho^*=\rho$, $\dim x = d$, $\dim y = 1$ and
$\dim \rho= d^2-2$. This is an example of 
Theorem~\ref{th:sqrt2}, with $\fX=\fX_\proj=\fX_{\max}$ and
$x\not\in\fX_{\max}$.
We have $\npj_\fX(x)=\sqrt{2}$, $\npj_\fX(y)=1$ and $\npj_\fX(\rho)=0$.
\end{eg}

The next proposition involves the algebraic integer $\alpha\approx 2.839286755\dots$,
which is the real root of the polynomial $X^3-4X^2+4X-2$. The other
two roots are complex conjugate.

\begin{prop}\label{pr:alpha}
If $x\in\langle\fX_{\max}\rangle_{\cg 0}$ and $xx^*x-2x\not\in\langle\fX\rangle$ then
$\npj_\fX(xx^*)\ge \alpha$.
\end{prop}
\begin{proof}
Suppose that $\npj_\fX(xx^*)<\alpha$. 
Since $x\in\langle\fX_{\max}\rangle$ we have $xx^*x-2x\cge 0$.  If
$xx^*x-2x\not\in\langle\fX\rangle$ then we may write $xx^*x=2x+v$ with
$v\in\langle\fX_{\max}\rangle_{\cg 0}$ and $v\not\in\langle\fX\rangle$.
By Proposition~\ref{pr:nilp}\,(i) we have $vv^*\not\in\langle\fX\rangle$, so 
$vx^*xx^*=vv^*+2vx^*\not\in\fX$ and hence $vx\not\in\fX$.
By Proposition~\ref{pr:xx*yy*}, the elements $xx^*$, $vx^*$, $vv^*$
and $vv^*xx^*$ are not in $\langle\fX\rangle$. 
Since $v\in\langle\fX_{\max}\rangle$, we have $vv^*v=2v+w$ with
$w\cge 0$. Thus we have
\begin{align*}
(xx^*)(xx^*)&=2xx^*+vx^* \\
(xx^*)(vx^*)&=2vx^*+vv^* \\
(xx^*)(vv^*)&=xx^*vv^* \\
(xx^*)(xx^*vv^*)&=2vx^*+2xx^*vv^*+wx^*.
\end{align*}

Ignoring the term $wx^*$, multiplication by $xx^*$ on the linear span of 
the four elements 
$xx^*$, $vx^*$, $vv^*$ and $xx^*vv^*$ is given by the following matrix:
\[ A=\begin{pmatrix} 2&0&0&0\\1& 2 & 0 & 2 \\0& 1 & 0 & 0 \\0& 0 & 1 &
      2 \end{pmatrix} \]
The element $(xx^*)^n$ is therefore a sum of at least 
\[ \begin{pmatrix} 1 & 1 & 1 & 1 \end{pmatrix} A^{n-1}
\begin{pmatrix} 1 \\ 0 \\ 0 \\ 0 \end{pmatrix} \]
positive terms not in $\langle\fX\rangle$. Note that this argument
does not depend on these four elements being linearly independent.
It follows from Perron--Frobenius theory that
$\npj_\fX(xx^*)$ is at least as large as the largest positive real eigenvalue of
$A$. The characteristic equation of $A$ is $t^3-4t^2+4t-2=0$, and so
the largest positive real eigenvalue is the algebraic integer $\alpha$.
\end{proof}

\begin{theorem}\label{th:1+sqrt2}
If $x\in\fa_{\cge 0}$ with 
$2\le\npj_\fX(xx^*)<1+\sqrt{2}$ then $xx^*=1+y$ and $y$ 
is $\fX$-endotrivial.
\end{theorem}
\begin{proof}
If  $x\not\in\fX_{\max}$ then
$xx^*=1+y$ with $1\le\npj_\fX(y)<\sqrt{2}$. By
Theorem~\ref{th:sqrt2}, $y$ is $\fX$-endotrivial.

If $x\in\fX_{\max}$ then
$xx^*x\cge 2x$. If $xx^*x-2x\in\langle\fX\rangle$ then
$(xx^*)^2-2(xx^*) \in \langle\fX\rangle$ and then by Lemma~\ref{le:gamma(x+fX)},
we have $\npj_\fX(xx^*)=2$. So we may suppose that 
$xx^*x-2x\not\in\langle\fX\rangle$. 
We are now in the situation of Proposition~\ref{pr:alpha}, and since
$\alpha>1+\sqrt{2}$, we are done.
\end{proof}

\section{Algebraic elements}

Algebraic modules for finite groups
were first studied by Alperin~\cite{Alperin:1976b}; see 
also \S II.5 of Feit~\cite{Feit:1982a}, as well as
Berger~\cite{Berger:1976a,Berger:1979a}, 
Craven~\cite{Craven:2011a,Craven:2013a},
Feit~\cite{Feit:1980a}, 
Gill~\cite{Gill:2016a}. The following definition 
generalises this.

\begin{defn}
Let $\fX$ be a representation ideal in a representation ring $\fa$.
An element $x\in\fa_{\cge 0}$ is said to be
\emph{algebraic}\index{algebraic element} modulo $\fX$ 
if $x$ satisfies some monic equation with integer coefficients:
\[ x^n + a_{n-1}x^{n-1}+\dots+a_0 \one=0 \]
with $n\ge 1$, in the quotient ring $\fa_\fX=\fa/\langle\fX\rangle$.
If $\fX=\varnothing$, we just say that $x$ is \emph{algebraic}. The
\emph{minimal equation}\index{minimal!equation} of $x$ is the monic
equation of least degree satisfied by $x$.
\end{defn}

\begin{rk}
The element $x\in\fa_{\cge 0}$ determines a homomorphism $\bZ[X] \to
\fa_\fX$ taking $X$ to $x$, and hence also a homomorphism $\bQ[X] \to
\fa_{\bQ,\fX}=\bQ \otimes_\bZ \fa_\fX$. If $x$ is algebraic modulo $\fX$, this map
has a kernel. Since $\bQ[X]$ is a principal
ideal domain, the kernel is generated by some monic polynomial, say
$f(X)$. Every polynomial satisfied by the image of $x$ in
$\fa_{\bQ,\fX}$ is a multiple of $f(X)$. If $x$ also satisfies a monic
polynomial $g(X)$ in $\bZ[X]$ then $g(X)$ is a multiple of $f(X)$ in
$\bQ[X]$ , and so by Gauss' lemma\index{Gauss' lemma} we have $f(X)\in
\bZ[X]$.  
\end{rk}

\begin{lemma}\label{le:algebraic}
If $x$ and $y$ are algebraic modulo $\fX$ then so are $x+y$ and $xy$.
For a non-negative element $x\in\fa$, the following are equivalent:
\begin{enumerate}
\item $x$ is algebraic modulo $\fX$.
\item The additive group of the subring of $\fa_\fX$ generated by $x$ 
is free abelian of finite rank.
\item The additive group of the subring of $\fa_\fX$ generated by $x$ and
$x^*$ is free abelian of finite rank.
\item There are only finitely many basis elements 
$x_i$, $i\in \fI\setminus\fX$ such that for
some $m,n\ge 0$ we have $[x^mx^{*\,n}:x_i]>0$.
\item $x$ is contained in a representation subring $\fa'$ of $\fa$,
containing $\langle \fX\rangle$,
such that the
additive group of $\fa'_\fX=\fa'/\langle\fX\rangle$ is free abelian of finite
rank, containing the $x_i$, $i\in\fI\setminus\fX$
such that for some $m,n\ge 0$ we have $[x^mx^{*\,n}:x_i]>0$.
\end{enumerate}
\end{lemma}
\begin{proof}
(i) $\Leftrightarrow$ (ii): If $x$ is algebraic modulo $\fX$ then the
additive group of the 
subring of $\fa_\fX$ generated by $x$ is spanned by
$\one,x,x^2,\dots,x^{n-1}$. Conversely, suppose that the additive 
group of the subring generated by
$x$ has finite rank. Look at the ascending chain of additive subgroups
whose $i$th term is generated by $\one,x,\dots,x^i$. This 
ascending chain has to terminate, so for some value of $n$, the 
element $x^n$ is in the additive subgroup generated by
$\one,x,\dots,x^{n-1}$. This gives us a degree $n$ monic equation with integer
coefficients satisfied by $x$.

If $x$ and $y$ are algebraic modulo $\fX$ then the $y$ is algebraic
over the subring of $\fa_\fX$ generated by $x$, and so as a module
over that ring, the subring generated by $x$ and $y$ is finitely
generated. It follows that it is finitely generated as an abelian
group, i.e., has finite rank, 
so every element of the subring generated by $x$ and $y$ is
algebraic. In particular, $x+y$ and $xy$ are algebraic.

(ii) $\Leftrightarrow$ (iii): If $x$ is algebraic then so is $x^*$,
with the same minimal equation. Therefore the subring generated by $x$
and $x^*$ has finite rank. Conversely, if the subring generated by $x$
and $x^*$ has finite rank, so does the subring generated by $x$.

The equivalence of (iii) and (iv) is obvious. It is also obvious that
(v) implies (ii). Finally, to see that (iv) implies (v), we look at
the subring $\fa'$ of $\fa$ spanned by the $x_i$ such that $i\in\fX$, or $x_i$ is
projective, or
for some $m,n\ge 0$ we have $[x^mx^{*\,n}:x_i]>0$. This is a
representation subring $\fa$ with the required properties.
\end{proof}

\begin{lemma}\label{le:alg-mod-proj}
An element $x\in\fa_{\cge 0}$ is algebraic if and only if it is
algebraic modulo $\fX_\proj$.
\end{lemma}
\begin{proof}
This follows from Lemma~\ref{le:algebraic}, since there are only
finitely many projective indecomposables.
\end{proof}

\section{The maximal quotient}\index{maximal!quotient}

Let $\fa$ be a representation ring.
In this section we examine some properties of 
the quotient $\fa_{\max}=\fa/\langle\fX_{\max}\rangle$,\index{X@$\fX_{\max}$}\index{a@$\fa_{\max}$}
which has a basis consisting of the $x_i$ with
$i\in\fI\setminus\fX_{\max}$.

\begin{defn}
We define $n_i=[x_ix_{i^*}:\one]$.\index{n@$n_i=[x_ix_{i^*}:\one]$|textbf} Then $n_i>0$ if and only if
$i\in\fI\setminus\fX_{\max}$. If $\fa$ is a closed representation
ring, then $n_i=1$ for $i\in\fI\setminus\fX_{\max}$.
\end{defn}

\begin{lemma}\label{le:xixjxk}
For $i,j,k\in\fI$ we have
\[ [x_ix_jx_k:\one]=n_k[x_ix_j:x_{k^*}]. \]
In particular, $[x_ix_jx_k:\one]=0$ unless $i,j,k\not\in\fX_{\max}$.
\end{lemma}
\begin{proof}
This follows from
\[ [x_ix_jx_k:\one]=\sum_\ell [x_ix_j:x_\ell][x_\ell x_k:\one] \]
and property (ii) in Definition~\ref{def:repring}.
\end{proof}

The following is the analogue of Problem~(4.12), at the end of
Chapter~4 of Isaacs~\cite{Isaacs:1976a}.

\begin{lemma}
For
$i,j,k\in\fI\setminus\fX_{\max}$ we have
\[ [x_ix_j:x_k]\le(\max\{n_i,n_j\}/n_k)\dim x_k. \]
In particular, if $\fa$ is a closed representation ring then
\[ [x_ix_j:x_k]\le\dim x_k. \]
\end{lemma}
\begin{proof}
Swapping the roles of $i$ and $j$ if necessary, we may assume that
\[ \frac{\dim x_i}{n_i}\le \frac{\dim x_j}{n_j}. \]
Then using Lemma~\ref{le:xixjxk} we have
\[ [x_ix_j:x_k] =
  [x_ix_jx_{k^*}:\one]/n_k=n_j[x_ix_{k^*}:x_j]/n_k. \]
Now 
\[ \dim x_i\dim x_k \ge [x_ix_{k^*}:x_j]\dim x_j \]
and so
\begin{equation*} 
[x_ix_j:x_k] \le \frac{n_j}{n_k}\frac{\dim x_i\dim x_k}{\dim x_j}
=\frac{n_i(\dim x_i/n_i)(\dim x_k/n_k)}{(\dim x_j/n_j)}
\le n_i\dim x_k/n_k. 
\qedhere
\end{equation*}
\end{proof}

The following result will be the model for the 
proof of the sharper Theorem~\ref{th:|xy|}.

\begin{lemma}\label{le:l2}
For $i,j\in\fI\setminus\fX_{\max}$ we have
\[ \sum_{k\in\fI\setminus\fX_{\max}} \frac{n_k}{n_j}[x_ix_j:x_k]^2 
=[x_{i^*}x_ix_j:x_j]
\le (\dim  x_i)^2. \]
\end{lemma}
\begin{proof}
Using Lemma~\ref{le:xixjxk}, we have 
\[ [x_{i^*}x_ix_j:x_j]=
\sum_k[x_{i^*}x_j:x_k][x_ix_k:x_j] 
= \sum_k  \frac{n_k}{n_j}[x_ix_j:x_k]^2. \]
We also have
$[x_{i^*}x_ix_j:x_j]\dim x_j\le (\dim x_i)^2\dim x_j$.
Now divide by $\dim x_j$.
\end{proof}

\chapter{Commutative Banach algebras}\label{ch:Banach}

Most of the usual examples of Banach spaces and Banach algebras are not
relevant to our purpose. We therefore eschew their description in favour of
brevity. The only examples of interest to us in this work are completions of
representation rings, and their quotients. We examine the group ring
of $\bZ$ via the theory of Fourier series, because of the similarity 
to our context.

\section{Banach spaces}\index{Banach!norm}

Throughout, we work over the complex numbers. Most of
the background we need concerning commutative Banach
algebras was developed in the paper of Gelfand~\cite{Gelfand:1941a}.
As  references, we use 
Berberian~\cite{Berberian:1974a}, 
Folland~\cite{Folland:2015a},
Gelfand, Raikov and Shilov~\cite{Gelfand/Raikov/Shilov:1964a},
Kaniuth~\cite{Kaniuth:2009a}, 
Lax \cite{Lax:2002a}, and
Rickart~\cite{Rickart:1974a}.
Beware that terminology varies among these sources. In particular, 
in~\cite{Gelfand/Raikov/Shilov:1964a} a normed vector space is assumed
to be complete; we shall follow the other references for our
definition of normed vector space. 

\begin{defn}\label{def:nvs}
A \emph{normed vector space}\index{normed!vector space}
is a complex vector space $V$
together with a \emph{norm}\index{norm} 
$V\to \bR$, $v\mapsto \|v\|$, satisfying
\begin{enumerate}
\item positivity:\index{positivity} $\|x\|\ge 0$, with $\|x\|=0$ if
and only if $x=0$,
\item  subadditivity:\index{subadditivity|textbf} $\|x+y\|\le\|x\|+\|y\|$, and
\item homogeneity:\index{homogeneity} $\|cx\|=|c|\|x\|$. 
\end{enumerate}

A norm gives rise to a metric\index{metric} on $V$ via $d(x,y)=\|x-y\|$, and
we say that $V$ is \emph{complete} with respect to the norm if
this metric space is complete. A complete normed vector space
is called a \emph{Banach space}.\index{Banach!space}
\end{defn}

\begin{defn}
If $V$ and $W$ are normed spaces, and
$f\colon V \to W$ is a linear map, we define the 
\emph{sup norm}\index{sup!norm|textbf} of $f$,
$\|f\|_{\sup}$ to be the supremum of $\|f(v)\|$ as $v$ runs over elements of
$V$ with $\|v\|\le 1$. If $\|f\|_{\sup}<\infty$, we say that $f$ is 
\emph{bounded}.\index{bounded!linear map}
\end{defn}

\begin{lemma}\label{le:bd=cts}
For a linear map of normed spaces $V \to W$ the following are
equivalent:
\begin{enumerate}
\item $f$ is bounded,
\item $f$ is continuous,
\item $f$ is continuous at some point $v\in V$.
\end{enumerate}
\end{lemma}
\begin{proof}
We have $\|f(x-y)\|\le \|f\|_{\sup}\|x-y\|$, and so (i) implies (ii).
Clearly (ii) implies (iii). Finally, to prove (iii) implies (i),
if $f$ is continuous at $v$ then given $\ep>0$ there
exists $\delta>0$ such that $\|x\|<\delta$ implies
$\|f(v+x)-f(v)\|<\ep$. By linearity, $\|f(x)\|<\ep$,
and so $\|f\|_{\sup}<\ep/\delta$. Thus $f$ is bounded.
\end{proof}

\begin{eg}\label{eg:l1Z}
We write $\ell^1(\bZ)$\index{l1Z@$\ell^1(\bZ)$} 
for the space of functions $x\colon \bZ\to \bC$
with the property that $\|x\|=\sum_{n\in\bZ}|x(n)|<\infty$. This norm
makes $\ell^1(\bZ)$ into a Banach space. We can identify this with the
space of absolutely convergent Fourier series\index{Fourier series}
$f\colon S^1\to \bC$, via
$f(e^{\bi\theta})=\sum_{n\in\bZ}x(n)e^{\bi n\theta}$. Here, $S^1$ is the
unit circle in $\bC$, parametrized by angle $\theta$. 
The coefficients $x(n)$ can be recovered from
$f$ by the Fourier inversion formula:
\[ x(n) = \frac{1}{2\pi}\int_0^{2\pi}
  f(e^{\bi\theta})e^{-\bi n\theta}d\theta. \]

We shall continue with this as a running example throughout this
chapter, and we shall make use of it in the context of representation
rings in a later chapter.
\end{eg}

\begin{lemma}\label{le:Vhat}
The metric completion\index{completion}\index{metric!completion} $\hat V$ 
of a normed vector space $V$ is a Banach space in which $V$
is a dense subspace.\index{dense subspace}
\end{lemma}
\begin{proof}
  The termwise sum of two Cauchy sequences\index{Cauchy sequence} in $V$, and a scalar
  multiple of a Cauchy sequence, are again Cauchy
  sequences. These operations are compatible with the equivalence
  relation on Cauchy sequences. 
\end{proof}

\begin{defn}\label{def:quotient-norm}\index{quotient norm}\index{norm!quotient}
If $W$ is a closed subspace of a normed vector space $V$ then
the \emph{quotient norm} on the space $V/W$ is defined via
\[ \|x+W\| = \inf_{w\in W}\|x+w\|. \]
\end{defn}

\begin{lemma}\label{le:quotient-norm}
Let $W$ be a closed subspace of a normed vector space $V$.
  \begin{enumerate}
\item 
Definition~\ref{def:quotient-norm} defines a norm on the
quotient space $V/W$.
\item If $V$ is complete then so are $W$ and $V/W$.
\item
The natural map from the completion $\widehat{V/W}$ of $V/W$ to 
the quotient of the completions $\hat V / \hat W$ is an
isometric isomorphism.
\end{enumerate}
\end{lemma}
\begin{proof}
(i) To prove positivity of the quotient norm, let $u$ be an
element of $V/W$. Clearly $\|u\|\ge 0$. If $\|u\|=0$ then
there is a sequence of elements
$v_1,v_2,\dots$ of $V$ lying in the coset $u\in V/W$
such that $\displaystyle\lim_{n\to\infty}\|v_n\|=0$. Let $y_n=v_1-v_n\in W$.
Then $\displaystyle\lim_{n\to\infty}\|v_1-y_n\|=0$ and so
$\displaystyle\lim_{n\to\infty}y_n=v_1$.
It follows that $v_1$ is in the closure of $W$, and hence in $W$.
Since it was chosen to be in the coset $u\in V/W$, we have $u=0$.

To prove subadditivity,\index{subadditivity} 
let $u_1,u_2\in V/W$. Then by definition of
quotient norm, given $\ep>0$ we
may choose representatives $v_1,v_2\in V$ of $u_1,u_2$ such that
$\|v_1\|<\|u_1\|+\ep$ and $\|v_2\|<\|u_2\|+\ep$. Then
\[ \|u_1+u_2\|\le\|v_1+v_2\|\le\|v_1\|+\|v_2\|<\|u_1\|+\|u_2\|+2\ep. \]
Since this is true for all $\ep>0$, we have $\|u_1+u_2\|\le\|u_1\|+\|u_2\|$.

Finally, homogeneity of the quotient norm is trivial to verify.
  
(ii) Completeness of $W$ is clear since a closed subspace of a
complete metric space is complete. For the quotient $V/W$, we argue as
follows. 
Given a Cauchy sequence $x_n+ W$ in $V/W$, we may 
replace with a subsequence satisfying $d(x_n+W,x_{n+1}+W)<2^{-(n+1)}$,
or equivalently $\|(x_n-x_{n+1})+W\|<2^{-(n+1)}$. By definition of the norm on
the quotient, we may inductively choose a sequence of elements
$w_n\in W$ so that $w_1=0$ and for $n\ge 1$ we have 
\[ \|(x_n+w_n)-(x_{n+1}+w_{n+1})\|=\|(x_n-x_{n+1})+(w_n-w_{n+1})\|< 2^{-n}. \]
Then the sequence $x_n+w_n$ is a Cauchy sequence in $V$.
If $x$ is its limit, then $x+W$ is the limit of the $x_n+W$ in $V/W$.

(iii) The canonical map $V \to \hat V \to \hat V/\hat W$ factors
through $V/W$, and the induced map $V/W \to \hat V/\hat W$
is an isometric embedding. Since $V$ is dense in $\hat V$, 
$V/W$ is dense in $\hat V/\hat W$. So $\widehat{V/W}$ is a
complete, dense subspace of $\hat V/\hat W$, and hence the
embedding is an isomorphism. 
\end{proof}

\section{Banach algebras}

We shall only be interested in Banach algebras with a multiplicative identity,
so we make this part of the definition. If we wish to talk of 
\emph{Banach algebras without a unit},\index{Banach!algebra!without unit} 
then that is what we shall call them.

\begin{defn}\label{def:na}
  A (unital) \emph{normed algebra}\index{normed!algebra}
  is an associative algebra $A$ over
$\bC$ with identity~$\one$, together with a norm $A \to \bR$, $x
\mapsto \|x\|$, satisfying the conditions (i)--(iii) for a normed vector space
given in Definition~\ref{def:nvs}, together with

\begin{enumerate}
  \item[(iv)]
    submultiplicativity:\index{submultiplicativity|textbf}
    for all $x,y\in A$, we have
$\|xy\|\le\|x\|\|y\|$.
\item[(v)] normalisation:\index{normalisation}
$\|\one\|=1$.
\end{enumerate}
A \emph{Banach algebra}\index{Banach!algebra} is a normed algebra
that is complete with respect to the norm. A
\emph{commutative Banach algebra}\index{commutative Banach algebra}%
\index{Banach!algebra!commutative}
is a Banach algebra which is commutative as an abstract ring.
\end{defn}

\begin{lemma}\label{le:Ahat}
  The metric completion $\hat A$ of a normed algebra $A$ is a Banach
  algebra in which $A$ is a dense subalgebra.
\end{lemma}
\begin{proof}
  By Lemma~\ref{le:Vhat}, $\hat A$ is a Banach space.
  The termwise product of two Cauchy sequences in $A$ is again a Cauchy
  sequence in $A$, and this is compatible with the equivalence
  relation on Cauchy sequences.
\end{proof}

\begin{eg}\label{eg:l1Zalg}
The Banach space $\ell^1(\bZ)$ discussed in Example~\ref{eg:l1Z} can
be made into a
Banach algebra by putting on it the \emph{convolution product}\index{convolution product}
\[ (x*y)(n)=\sum_{i+j=n}x(i)y(j). \] 
This corresponds to pointwise
multiplication of Fourier series, $fg(e^{\bi\theta})=f(e^{\bi\theta})g(e^{\bi\theta})$ where
$f(e^{\bi\theta})=\sum_{n\in\bZ}x(n)e^{\bi n\theta}$ and
$g(e^{\bi\theta})=\sum_{n\in\bZ}y(n)e^{\bi n\theta}$.

The identity element of this Banach algebra is the function $\one$
given by $\one(0)=1$, $\one(n)=0$ if $n\ne 0$.
Let $u\colon \bZ\to \bC$ be the function $u(1)=1$, 
$u(n)=0$ if $n\ne 1$. It is easy to check that for 
$j\in\bZ$, $u^j\colon\bZ\to\bC$ is the function
$u^j(j)=1$, $u^j(n)=0$ if $n\ne j$. The algebra
$\ell^1(\bZ)$ is the completion of 
$\bC[u,u^{-1}]$ with respect to the norm, and has $\bC[u,u^{-1}]$ as a
dense subalgebra. The algebra $\ell^1(\bZ)$ is sometimes called the
\emph{Wiener algebra}.\index{Wiener algebra}
\end{eg}

\begin{lemma}\label{le:quotient}
We have the following.
\begin{enumerate}
\item
If $I$ is a closed ideal in a Banach algebra $A$, then $A/I$ is
a Banach algebra with the quotient norm
\[ \|x+I\|=\inf_{y\in I}\|x+y\|. \]
The canonical map $\pi\colon A \to A/I$ is continuous, with $\|\pi\|=1$.
\item
If $J\le I \le A$ are closed ideals then the natural map 
$(A/J)/(I/J)\to A/I$ is an isometric isomorphism.
\item
If $A$ is a normed algebra and $I$ is a closed ideal then
the natural map of completions $\widehat{A/I} \to \hat A/\hat I$ is
an isometric isomorphism.
\end{enumerate}
\end{lemma}
\begin{proof}
(i) By Lemma~\ref{le:quotient-norm}\,(ii), $A/I$ is a Banach space.

To prove submultiplicativity of the norm,
let $x_1,x_2$ be elements of $A/I$. Given $\ep>0$ we may choose
representatives $y_1,y_2\in A$ of $x_1,x_2$ such that $\|y_1\|<
\|x_1\|+\ep$ and $\|y_2\|<\|x_2\|+\ep$. Then
\[ \|x_1x_2\|\le\|y_1y_2\|\le\|y_1\|\|y_2\|
  <\|x_1\|\|x_2\|+\ep(\|x_1\|+\|x_2\|+\ep). \]
Since this is true for all $\ep>0$, we have $\|x_1x_2\|\le\|x_1\|\|x_2\|$.  

The identity element of $A/I$ is the coset of $\one\in A$.
To prove normalisation, the unit element
clearly has norm at most one, and by submultiplicativity it has
norm at least one or norm zero. Since $I$ is a proper subspace of $A$,
the norm is not zero, therefore it is one.
  
Part (ii) is an elementary verification using the definition of the norm,
and part (iii) follows from Lemma~\ref{le:quotient-norm}\,(iii).
\end{proof}

\section{Spectrum}

If an element $x$ of a Banach algebra has both a left inverse $u$ and
a right inverse $v$,
then $u=uxv=v$. We then say that $x$ is
\emph{invertible}\index{invertible element} with
\emph{inverse}\index{inverse} $u=v$.

\begin{prop}\label{pr:open}
  Invertible elements of a Banach algebra $A$ form an open subset.
  More precisely, if $x$ is invertible and $\|x-y\|<1/\|x^{-1}\|$ then
  $y$ is invertible, and the inverse is given by an absolutely
  convergent power series.
\end{prop}
\begin{proof}
    Set $z=x^{-1}(x-y)$, so
  that $y=x(1-z)$ and $\|z\|\le \|x^{-1}\|\|x-y\|<1$. Then the series
  $1+z+z^2 + \cdots$
  converges to an inverse of $1-z$, and hence
  $x^{-1}(1+z+z^{-1}+\cdots)$ converges to an inverse of $y$.
\end{proof}

\begin{eg}
  Taking $x=\one$ in Proposition~\ref{pr:open}, we see that
  the open ball of radius one centred at $\one$ consists of invertible
  elements.
\end{eg}

\begin{cor}\label{co:maxideal}
  Every maximal ideal\index{maximal!ideal} of a Banach algebra $A$ is closed.
\end{cor}
\begin{proof}
By Proposition~\ref{pr:open}, the closure of an ideal does not contain
$\one$. A maximal ideal is therefore equal to its closure.
\end{proof}

\begin{defn}\label{def:spectrum}
If $x$ is an element of a Banach algebra $A$, we write 
$\Spec(x)$\index{Spec@$\Spec(x)$} for the 
\emph{spectrum}\index{spectrum|textbf} of $x$, namely the 
set of $\lambda\in\bC$ such that $x-\lambda\one$ is
not invertible in $A$. If we wish to emphasise the ambient Banach
algebra $A$, we write $\Spec_A(x)$.
\end{defn}

\begin{rk}\label{rk:spec-alg}
From its definition, the spectrum of an element $x$ of a Banach
algebra $A$ only depends on the algebraic structure of $A$ and not on
the metric or topological structure.  
\end{rk}

\begin{eg}\label{eg:l1Zspec}
Continuing Example~\ref{eg:l1Zalg}, let us determine the spectrum of
$u\in \ell^1(\bZ)$. We claim that $u-\lambda\one$ is invertible if and
only if $|\lambda|\ne 1$. To see this, we note that for $x\in\ell^1(\bZ)$,
\[ ((u-\lambda\one)*x)(n)=x(n-1)-\lambda x(n). \]
So for $x$ to be inverse to $u-\lambda\one$, we need
\[ x(n-1)-\lambda x(n) = \begin{cases} 1 & n=0 \\ 0 & n\ne 0. \end{cases} \]
Solving these equations inductively, we get
\[ x(n) = \begin{cases} \lambda^{-n} x(0) & n\ge 0 \\
    \lambda^{n-1}x(-1) & n<0, \end{cases} \]
and $x(-1)-\lambda x(0)=1$. For this to satisfy
$\sum_{n\in\bZ}|x(n)|<\infty$, if $|\lambda|>1$ there is a unique
solution, with $x(-1)=0$ and $x(0)=-\lambda^{-1}$. If $|\lambda|<1$
there is also a unique solution,  with $x(-1)=1$ and
$x(0)=0$. Finally, if $|\lambda|=1$ there is no solution.

The conclusion is that $\Spec(u)=S^1$, the unit circle in $\bC$.
\end{eg}

\begin{theorem}[Spectral Theorem, Gelfand~\cite{Gelfand:1941a}]%
\label{th:Gelfand}\index{spectral!theorem}
If $x\in A$ then $\Spec(x)$ is a non-empty closed bounded subset of
$\bC$. Outside of $\Spec(x)$, the inverse $(x-\lambda\one)^{-1}$ is an
analytic function of $\lambda$.
\end{theorem}
\begin{proof}
By Proposition~\ref{pr:open}, $\bC\setminus\Spec(x)$ is open, and so
$\Spec(x)$ is closed.
  
We have
\begin{equation}\label{eq:ps} 
(x-\lambda\one)^{-1}=-\lambda^{-1}(\one + \lambda^{-1}x +
\lambda^{-2}x^2 + \dots),
\end{equation}
and the right hand side converges if $\|\lambda^{-1}x\|<1$, namely $|\lambda|>\|x\|$,
so the spectrum is contained in a closed circle of radius $\|x\|$.

In a neighbourhood of any particular
$\lambda_0\in\bC\setminus\Spec(x)$, the inverse is given by an absolutely convergent power
series, which is hence an analytic function of $\lambda$.

Suppose that $\Spec(x)$ is empty. Then $(x-\lambda\one)^{-1}$ is
defined and analytic for all $\lambda\in\bC$. Then by
Cauchy's theorem, for any closed contour we have 
\[ \frac{1}{2\pi i}\oint
(x-\lambda\one)^{-1}d\lambda=0. \] 
However, taking the integral around a
circular contour of radius bigger than $\|x\|$, integrating the
power series~\eqref{eq:ps} term by term, and using Cauchy's integral
formula, we get $-\one$. This contradiction shows
that $\Spec(x)$ is non-empty.
\end{proof}

\begin{cor}[Gelfand--Mazur Theorem]\label{co:divisionalgebras}\index{Gelfand!--Mazur theorem}
If $A$ is a Banach algebra in which every non-zero element is
invertible then $A$ is isomorphic to $\bC$ as a $\bC$-algebra.
\end{cor}
\begin{proof}
If $x\in A$, then by Theorem~\ref{th:Gelfand} there
exists $\lambda\in\bC$ such that $x-\lambda\one$ is not
invertible. Therefore $x-\lambda\one=0$ and so $x=\lambda\one$.
So every element is a multiple of $\one$. It is now easy to check that
the map sending $x$ to $\lambda$ is a Banach algebra isomorphism
$A\to\bC$.
\end{proof}

\begin{cor}\label{co:maxkernel}
Every maximal ideal\index{maximal!ideal} 
of a commutative Banach algebra $A$ is the kernel of an
algebra homomorphism $A\to \bC$.
\end{cor}
\begin{proof}
By Corollary~\ref{co:maxideal}, every maximal ideal $I$ is
closed. By Lemma~\ref{le:quotient}, the quotient $A/I$ is a Banach
algebra in which every non-zero element is invertible. So by
Corollary~\ref{co:divisionalgebras} the quotient is isomorphic to $\bC$.
\end{proof}

Corollary~\ref{co:maxkernel} is closely related to the following.

\begin{prop}[Automatic continuity]%
\label{pr:continuous}\index{automatic continuity}
If $A$ is a commutative Banach algebra and
$\phi\colon A \to \bC$ is an
algebra homomorphism then for all $x\in A$ we have
$|\phi(x)|\le\|x\|$. In particular, $\phi$ is continuous
with respect to the norm.
\end{prop}
\begin{proof}
  Suppose that $|\phi(x)|>\|x\|$. 
  Setting $y=x/\phi(x)$, we have $\|y\|<1$ and $\phi(y)=1$.
  So by Proposition~\ref{pr:open}, $\one-y$ is invertible with inverse the sum of the convergent power
  series $\one+y+y^2+\cdots$. On the other hand,
 \[ \phi(\one-y)=\phi(\one)-\phi(y)=1-1=0. \] 
This contradiction proves
  that $|\phi(x)|\le\|x\|$.
\end{proof}

\section{Spectral radius}

\begin{defn}
Let $A$ be a commutative Banach algebra.
The \emph{spectral radius}\index{spectral!radius|textbf} of $x\in A$ is defined to be
$\rho(x)=\displaystyle\sup_{\lambda\in\Spec(x)}|\lambda|$.\index{rho@$\rho(x)$, spectral radius}
\end{defn}

\begin{prop}[Spectral radius formula, Gelfand \cite{Gelfand:1941a}]\label{pr:Gelfand}
If $A$ is a Banach algebra and $x\in A$ then the spectral radius 
$\rho(x)$ is related to the norm by the formula
\[ \rho(x)=\limsup_{n\to\infty}\sqrt[n]{\|x^n\|}=
  \lim_{n\to\infty}\sqrt[n]{\|x^n\|}=\inf_{n\to\infty}\sqrt[n]{\|x^n\|}. \]
\end{prop}
\begin{proof}
By Theorem~\ref{th:Gelfand}, $(x-\lambda\one)^{-1}$ is an analytic
function of $\lambda$ outside of $\Spec(x)$.
So to find the spectral radius of $x$, we must find the exact radius of
convergence of the power series~\eqref{eq:ps}. 
By the Cauchy--Hadamard formula\index{Cauchy--Hadamard formula}
(Lemma~\ref{le:radius}) the radius of
convergence in the variable $\lambda^{-1}$ is given by
\[ 1/r = \limsup_{n\to\infty}\sqrt[n]{\|x^n\|}. \]
Now by axiom (iv) in Definition~\ref{def:na}, the sequence
$\|x^n\|$ is a submultiplicative sequence, and so we can
apply Fekete's lemma~\ref{le:Fekete} to deduce that
\[ \limsup_{n\to\infty}\sqrt[n]{\|x^n\|}=
  \lim_{n\to\infty}\sqrt[n]{\|x^n\|}=\inf_{n\to\infty}\sqrt[n]{\|x^n\|}. \]
The series therefore converges for $|\lambda^{-1}|<r$ and diverges for
$|\lambda^{-1}|>r$. Inverting, it converges for $|\lambda|>1/r$ and
diverges for $|\lambda|<1/r$. It follows that the spectral radius is equal to $1/r$. 
\end{proof}

\begin{rk} 
If $B$ is a closed subalgebra of a Banach algebra $A$ and $x\in B$
then $\Spec_A(x)\subseteq\Spec_B(x)$,
since if $x-\lambda\one$ is not invertible in $A$ then it is
not invertible in $B$. Although it can happen that $\Spec_A(x)\ne\Spec_B(x)$,
Proposition~\ref{pr:Gelfand} shows that the spectral radius is the
same in $A$ as in $B$.
\end{rk}

\begin{lemma}\label{le:rho}
If $x$ and $y$ are elements of $A$ then
\begin{enumerate}
\item $\rho(xy) \le \rho(x)\rho(y)$,
\item $\rho(x+y) \le \rho(x) + \rho(y)$.
\end{enumerate}
\end{lemma}
\begin{proof}
Part (i) is straightforward, so we prove part (ii).
Using Lemma~\ref{le:binom} and Proposition~\ref{pr:Gelfand}, we
have
\begin{align*}
\rho(x+y)&=\limsup_{n\to\infty}\sqrt[n]{\|(x+y)^n\|} \\
&=\limsup_{n\to\infty}\sqrt[n]{\|\textstyle\sum_{i=0}^n\binom{n}{i}x^iy^{n-i}\|} \\
&\le\limsup_{n\to\infty}\sqrt[n]
{\textstyle\sum_{i=0}^n\binom{n}{i}\|x^i\|\|y^{n-i}\|} \\
&\le\limsup_{n\to\infty}\sqrt[n]{\|x^n\|} + 
\limsup_{n\to\infty}\sqrt[n]{\|y^n\|}.
\qedhere
\end{align*}
\end{proof}

\begin{theorem}\label{th:invertible}
An element $x$ of a commutative Banach algebra $A$ is invertible if and
only if $\phi(x)\ne 0$ for all algebra homomorphisms 
$\phi\colon A\to\bC$.
\end{theorem}
\begin{proof}
If $x$ is invertible then for every algebra homomorphism $\phi\colon
A\to\bC$ we have $\phi(x)\phi(x^{-1})=\phi(xx^{-1})=\phi(\one)=1$ and
so $\phi(x)\ne 0$. On the other hand, if $x$ is not invertible then
$x$ generates an ideal in $A$. By Zorn's lemma, this ideal is
contained in some maximal ideal $I$ of $A$. By
Corollary~\ref{co:maxideal},
$I$ is closed in $A$, and therefore using Lemma~\ref{le:quotient},
the quotient $A/I$ is a Banach
algebra in which every non-zero element is invertible. By
Corollary~\ref{co:divisionalgebras}, we have $A/I\cong\bC$. There is
therefore an algebra homomorphism $\phi\colon A\to\bC$ with kernel
$I$, and then we have $\phi(x)=0$. 
\end{proof}

\begin{cor}\label{co:radius}
If $x$ is an element of a commutative Banach algebra $A$ then
$\Spec(x)$ is the set of values of $\phi(x)$ as $\phi$ runs over
the algebra homomorphisms $A \to \bC$.
The spectral radius $\rho(x)$ 
is equal to $\displaystyle\sup_{\phi\colon A \to \bC} |\phi(x)|$.
\end{cor}
\begin{proof}
It follows from Theorem~\ref{th:invertible} that  
$x-\lambda\one$ is not invertible if and only if there
exists an algebra homomorphism 
$\phi\colon A \to \bC$ such that $\phi(x)=\lambda$.
\end{proof}

\begin{cor}\label{co:Specx^-1}
If $x$ is an invertible element of a commutative Banach algebra $A$ 
then $\Spec(x^{-1})=\{\lambda^{-1}\mid \lambda\in\Spec(x)\}$.
\end{cor}
\begin{proof}
This follows from Corollary~\ref{co:radius}, since if $\phi\colon A
\to \bC$ is an algebra homomorphism then $\phi(x^{-1})=\phi(x)^{-1}$.

Alternatively, we have $x^{-1}-\lambda^{-1}\one =
-x^{-1}\lambda^{-1}(x-\lambda\one)$, so $x-\lambda\one$ is
invertible if and only if $x^{-1}-\lambda^{-1}\one$ is invertible.
\end{proof}

\begin{eg}\label{eg:l1Zspec2}
Continuing Example~\ref{eg:l1Zspec}, recall that $\ell^1(\bZ)$ is the
completion of the algebra $\bC[u,u^{-1}]$ with respect to the $\ell^1$ norm.
A $\bC$-algebra homomorphism $\phi\colon\bC[u,u^{-1}]\to\bC$
is determined by $\phi(u)$, and this may be any complex number except
zero.  The ones that extend to $\ell^1(\bZ)$ are
those with $|\phi(u)|=1$. 
\end{eg}

\begin{prop}\label{pr:Spec(1+x)}
Let $x$ be an element of $A$ with spectral radius $r$. Then $r$ is an
element of $\Spec(x)$ if and only if $\one+x$ has spectral radius
$1+r$.
\end{prop}
\begin{proof}
The set $\Spec(\one+x)$ is the set of $1+\lambda$ with
$\lambda\in\Spec(x)$. This is contained in a disc of radius $r$
centred at $1\in\bC$. The only point in this disc at distance $1+r$
from the origin is the real number $1+r$. Using the fact that
$\Spec(x)$ is closed (Theorem~\ref{th:Gelfand}), we see that the
spectral radius of $\one+x$ is $1+r$ if and only if $r\in\Spec(x)$.
\end{proof}

We end this section by noting that there are various interesting
subsets of the spectrum that are important in the subject.

\begin{defn}
The \emph{point spectrum}\index{point spectrum}\index{spectrum!point} 
of an element $x$ of a
Banach algebra $A$ is the set of $\lambda\in\bC$ such
that multiplication $x-\lambda\one$ is not injective; in other words,
$x-\lambda\one$ is a divisor of zero.

The \emph{peripheral spectrum}\index{peripheral spectrum}%
\index{spectrum!peripheral} of $x$ is the set of $\lambda\in\Spec(x)$
such that $|\lambda|$ is equal to the spectral radius of $x$. This is
a non-empty closed subset of the circle whose radius is the spectral
radius. 
\end{defn}

\section{The structure space}\label{se:struct-space}

\begin{defn}
If $A$ is a commutative Banach algebra, the
\emph{structure space}\index{structure!space|textbf}%
\footnote{Other names for this in the literature are
the \emph{carrier space},\index{carrier space} 
the \emph{spectrum},\index{spectrum}
the \emph{Gelfand space},\index{Gelfand!space}
and the \emph{maximal ideal space}\index{maximal!ideal!space}
of $A$.}
$\Struct(A)$ is the set of algebra homomorphisms
$\phi\colon A \to \bC$, endowed with the
\emph{weak* topology}.\index{weak* topology|textbf}%
\index{topology, weak*|textbf} This
is the topology defined by the following basic open neighbourhoods of
an element $\phi\in\Struct(A)$. For each
finite list of elements $x_1,\dots,x_n\in A$ and for each 
$\ep>0$,  we have a basic open neighbourhood\index{basic open neighbourhood}
$[\phi;x_1,\dots,x_n;\ep]$
of $\phi$ in $\Struct(A)$ consisting of those $\phi'\colon A \to \bC$ such that 
$|\phi(x_i)-\phi'(x_i)|<\ep$ for  $i=1,\dots,n$.
\end{defn}

\begin{lemma}\label{le:weak*}
If we put the product topology\index{product topology} on
$Q=\prod_{x\in A}\bC$ then the natural map $\Struct(A)\to Q$ sending
$\phi$ to $\prod_{x\in A}\phi(x)$ is injective. 
The weak* topology on $\Struct(A)$ is
the subspace topology for this embedding. 
It is the coarsest topology with the property that for all $x\in A$
the evaluation map 
\[  \hat x\colon \Struct(A)\to \bC \] 
sending $\phi$ to $\phi(x)$ is continuous.
\end{lemma}
\begin{proof}
The map $\Struct(A)\to Q$  is injective because a
homomorphisms $\phi$ is determined by its values at elements of $A$.
The basic open neighbourhoods 
\[ [\phi;x_1,\dots,x_n;\ep] \] 
defining
the weak* topology are the inverse images of the basic open
neighbourhoods of the image of $\phi$ in $Q$ in the product topology.
\end{proof}

\begin{prop}\label{pr:weak*}
With the weak* topology, $\Struct(A)$ is a
compact Hausdorff topological space.%
\index{compact Hausdorff space}\index{Hausdorff space}
\end{prop}
\begin{proof}
Let $Q$ be as in Lemma~\ref{le:weak*}, and let $Q'$ be
the subset of $Q$ given by the product over $x\in A$ of
the closed disc of radius $\|x\|$ centred at the origin in $\bC$. Then
by Proposition~\ref{pr:continuous}, the image of $\Struct(A)$ in $Q$ lies in
$Q'$. By Tychonoff's theorem\index{Tychonoff's theorem} $Q'$ is
compact, so it remains to show that the image of $\Struct(A)$ is
closed in $Q'$.

If $\phi=\prod_{x\in A}\phi(x)$ is in the closure of the image of $\Struct(A)$ in $Q'$,
we must show that the map $x\mapsto \phi(x)$ is an algebra homomorphism.
Given $\ep>0$ and $x,y\in A$
there exists $\phi'\in\Struct(A)$ such that
\begin{align*}
  |\phi(\one)-\phi'(\one)|&=|\phi(\one)-\one|<\ep, \\
|\phi(x)-\phi'(x)|&<\ep, \\
|\phi(y)-\phi'(y)|&<\ep, \\
  |\phi(xy)-\phi'(xy)|&=|\phi(xy)-\phi'(x)\phi'(y)|<\ep.
\end{align*}
Then
\begin{align*}
|\phi(xy)-\phi(x)\phi(y)|&\le |\phi(xy)-\phi'(x)\phi'(y)| \\
&\qquad+|\phi'(x)\phi'(y)-
\phi'(x)\phi(y)|+|\phi'(x)\phi(y)-\phi(x)\phi(y)| \\
&< \ep + |\phi'(x)|\ep + \ep|\phi(y)| \\
&\le\ep(1+\|x\|+|\phi(y)|).
\end{align*}
This is true for all $\ep>0$, and so $\phi(xy)=\phi(x)\phi(y)$.
Similar arguments show that $\phi(\one)=1$,
$\phi(x+y)=\phi(x)+\phi(y)$ and $\phi(\lambda x)=\lambda\phi(x)$.
\end{proof}

\begin{defn}
Let $C(\Struct(A),\bC)$ be the algebra of continuous maps from the
compact Hausdorff space $\Struct(A)$ to $\bC$.
The \emph{Gelfand representation}\index{Gelfand!representation},
also called the
\emph{canonical representation}\index{canonical representation}, of a
commutative Banach algebra $A$ is the map 
$\Gamma_A\colon A\to C(\Struct(A),\bC)$\index{Gamma@$\Gamma_A$}
sending $x$ to the continuous function $\hat x\colon \Struct(A)\to\bC$ given by
\begin{equation*} 
\hat x \colon \phi \mapsto \phi(x).
\end{equation*}
The map $\hat x$ is called the 
\emph{Gelfand transform} of $x$.\index{Gelfand!transform}
\end{defn}

\begin{rk}
By Corollary~\ref{co:radius}, for $x\in A$ the image of $\hat x$ is
\[ \hat x(\Struct(A))=\Spec(x)\subseteq \bC. \]
\end{rk}

\begin{eg}\label{eg:l1Zspec3}
Continuing Example~\ref{eg:l1Zspec2}, we have $\Delta(\ell^1(\bZ))= S^1$.
The map 
\[ \Gamma_{\ell^1(\bZ)}\colon\ell^1(\bZ) \to C(S^1,\bC) \] 
is given by
$\hat{x}(e^{\bi\theta})=\sum_{n\in\bZ}  x(n)e^{\bi n\theta}$. 
\end{eg}

As an application of our running example we can prove the following
theorem about Fourier series.

\begin{theorem}[Wiener]\index{Wiener's theorem}
If $f(e^{\bi\theta})=\sum_{n\in\bZ}x(n)e^{\bi n\theta}$ with
$\sum_{n\in\bZ}|x(n)|<\infty$, and if $f(e^{\bi\theta})\ne 0$ for all
$e^{\bi\theta}\in S^1$, then
$1/f(e^{\bi\theta})=\sum_{n\in\bZ}y(n)e^{\bi n\theta}$ with
$\sum_{n\in\bZ}|y(n)|<\infty$. 
\end{theorem}
\begin{proof}
By Example~\ref{eg:l1Zspec3}, 
we have $f = \hat x$ with $x\in \ell^1(\bZ)$. If $f(e^{\bi\theta})\ne 0$
for all $e^{\bi\theta}\in S^1$ then by Theorem~\ref{th:invertible}, $x$
is invertible in $\ell^1(\bZ)$. Let $y$ be its inverse. Then $1/f=\hat y$.
\end{proof}

\section{Closed ideals and the structure space}

Let $A$ be a commutative Banach algebra and let $I$ be a closed ideal
in $A$. Write $i\colon I \to A$ for the inclusion and 
$q\colon A\to A/I$ for the quotient homomorphism.

We define $\Struct(I)$ to be the set of \emph{non-zero}
homomorphisms from $I$ to $\bC$, where $I$ is regarded as a
Banach algebra without a unit.\index{Banach!algebra!without unit}

\begin{lemma}\label{le:extend}
If $\phi\colon I \to \bC$ is a non-zero homomorphism then $\phi$
extends uniquely to a homomorphism $A\to \bC$.
\end{lemma}
\begin{proof}
Since $\phi$ is non-zero and $\bC$ is a field, there exists an element
$y\in I$ such that $\phi(y)=1$. If $\psi\colon A\to \bC$ is an
extension of $\phi$ to $A$ then for $x\in A$ we have $xy\in I$ and 
\[ \psi(x)=\psi(x)\phi(y)=\psi(x)\psi(y)=\psi(xy)=\phi(xy). \] 
So $\psi$ is determined by $\phi$. It remains to check that the $\psi$
defined this way is an algebra homomorphism. We have 
\begin{align*} 
\psi(\lambda x)&= \phi(\lambda xy)=\lambda\phi(xy)=\lambda\psi(x) \\
\psi(x_1+x_2)&=\phi((x_1+x_2)y)=\phi(x_1y+x_2y) \\
&=\phi(x_1y)+\phi(x_2y)=\psi(x_1)+\psi(x_2) \\
\psi(x_1x_2)&=\phi(x_1x_2y)=\phi(x_1x_2y)\phi(y)\\
&=\phi(x_1yx_2y)=\phi(x_1y)\phi(x_2y)
= \psi(x_1)\psi(x_2).
\end{align*}
Thus $\psi$ is indeed an algebra homomorphism.
\end{proof}

We can also look at the subalgebra $I_+= \bC\oplus I$ of $A$ generated
by $\one$ and $I$. Every non-zero homomorphism $I\to \bC$ extends
uniquely to an algebra homomorphism $I_+\to \bC$, since $\one$ has
to go to $1\in\bC$. But there is one more algebra homomorphism denoted
$\zero\colon I_+\to \bC$, which is identically zero on $I$. Thus
$\Struct(I_+)=\Struct(I)\cup\{\zero\}$.

There are obvious maps $q^*\colon \Struct(A/I)\to\Struct(A)$ and 
$i^*\colon\Struct(A)\to \Struct(I_+)$. The following lemma shows us
that these maps allow us to identify $\Struct(I_+)$ with the quotient
space $\Struct(A)/\Struct(A/I)$ with basepoint
$\zero$.\index{basepoint of $\Struct(A)$}

\begin{lemma}\label{le:structure-subset}
We have the following.
\begin{enumerate}
\item
The map $q^*$ is injective, open and continuous with image 
\[ q^*(\Struct(A/I))=\{\phi\in\Struct(A)\mid i^*(\phi)=\zero\}. \]
\item
$\Struct(I_+)$ is the one point compactification of $\Struct(I)$ with
basepoint $\zero$.
The map $i^*$ is surjective, and identifies $\Struct(I_+)$ with the
quotient of $\Struct(A)$ by the image of $q^*$.
\end{enumerate}
\end{lemma}
\begin{proof}
(i) Clearly $q^*$ is injective with image
$(i^*)^{-1}(\zero)\subseteq\Struct(A)$.
To check that it is a homeomorphic onto its image, we note that
\begin{align*}
q^*([\phi;x_1+I,\dots,x_n+I;\ep]) &= \{q^*(\phi')\mid
|\phi'(x_i+I)-\phi(x_i+I)|<\ep \text{ for } 1\le i \le n\} \\
&=\{\psi\mid |\psi(x_i)-\phi(q(x_i))|<\ep \text{ for } 1\le i \le n\} \\
&=[q^*(\phi);x_1,\dots,x_n;\ep].
\end{align*}

(ii) Denote the basic open neighbourhoods of $\phi$ in $\Struct(I_+)$
by $[\phi;x_1,\dots,x_n;\ep]_+$. Then for $\phi\in\Struct(I)$ we have
\[ [\phi;x_1,\dots,x_n;\ep]_+ = \begin{cases} [\phi;x_1,\dots,x_n;\ep]
    & |\phi(x_i)|<\ep \text{ for } 1\le i \le n \\
[\phi;x_1,\dots,x_n;\ep] \cup \{\zero\} & \text{ otherwise.} 
\end{cases} \]
For $\zero$ we have
\begin{align*} 
[\zero;x_1,\dots,x_n;\ep]_+& =\{\zero\} \cup 
\{\phi\in\Struct(I)\mid |\phi(x_i)|<\ep \text{ for } 1\le i \le n\} \\
&=\bigcap_{i=1}^n \,[\zero;x_i;\ep]_+. 
\end{align*}
For each $x_i$, the complement of $[\zero;x_i;\ep]_+$ is closed in
$\Struct(I_+)$ and hence compact by Proposition~\ref{pr:weak*}.
So the complement of $[\zero;x_1,\dots,x_n;\ep]$ in $\Struct(I_+)$ is a 
finite union of compact sets, and hence compact. It follows that 
$\Struct(I_+)$ is the one point compactification of $\Struct(I)$.

Lemma~\ref{le:extend} shows that every homomorphism $\phi\colon I_+ \to \bC$
apart from $\zero$
extends uniquely to an algebra homomorphism $\hat\phi\colon A \to \bC$
which is not in the image of $q^*$.
On the other hand, $\zero\colon I_+ \to \bC$ is the image in
$\Struct(I_+)$ of every element of $q^*(\Struct(A/I))$.

Putting these statements together with the description above of the topology on
$\Struct(I_+)$, we see that $\Struct(I_+)$ is
homeomorphic with the quotient of $\Struct(A)$ by the image of $q^*$.
\end{proof}

\section{Systems of generators}

\begin{defn}
We say that a set $K$ of elements of $A$ is a
\emph{system of generators}\index{system of generators}
if the smallest closed subalgebra of $A$ containing $K$ (and the
identity element) is the entire algebra $A$. In other words, 
the subalgebra generated by $K$ is dense in $A$.
\end{defn}

\begin{lemma}
If $K$ is a system of generators of $A$ then
the weak* topology on $\Struct(A)$ is generated by the open sets
$[\phi;x_1,\dots,x_n;\ep]$ with $x_1,\dots,x_n\in K$.
\end{lemma}
\begin{proof}
We must show that every open neighbourhood $[\phi;y_1,\dots,y_m;\ep]$
of $\phi$ contains one of these open sets. Since polynomials in
elements of $K$
are dense in $A$, there is a finite list $x_1,\dots,x_n$  of elements
of $K$ and polynomials $f_1,\dots,f_m$ in $n$ variables such that
for $1\le i\le m$ we have $\|y_i-f_i(x_1,\dots,x_n)\|<\ep/3$. Then by
Proposition~\ref{pr:continuous}, for any $\phi'$ we have
\[ |\phi'(y_i)-f_i(\phi'(x_1),\dots,\phi'(x_n))|<\ep/3. \]
Choose $\delta>0$ such that
\[ [\phi;x_1,\dots,x_n;\delta]\subseteq
  [\phi;f_1(x_1,\dots,x_n),\dots,f_m(x_1,\dots,x_n);\ep/3]. \]
Then for all $\phi'\in[\phi;x_1,\dots,x_n;\delta]$ and $1\le i \le m$
we have
\[ |f_i(\phi'(x_1),\dots,\phi'(x_n))-f_i(\phi(x_1),\dots,\phi(x_n)|<\ep/3 \]
and hence
\begin{align*}
  |\phi'(y_i)-\phi(y_i)|&\le |\phi'(y_i)-f_i(\phi'(x_1),\dots,\phi'(x_n))|\\
  &\qquad +|f_i(\phi'(x_1),\dots,\phi'(x_n))-f_i(\phi(x_1),\dots,\phi(x_n))|\\
  &\qquad + |f_i(\phi(x_1),\dots,\phi(x_n))-\phi(y_i)| \\
  &< \ep/3 + \ep/3 + \ep/3 = \ep.
\end{align*}
It follows that $[\phi;x_1,\dots,x_n;\delta]\subseteq[\phi;y_1,\dots,y_m;\ep]$.
\end{proof}

\begin{defn}
If $K$ is a system of generators for $A$ and $K=\{y_1,\dots,y_n\}$ is a finite set, we
say that $A$ is \emph{finitely generated}.\index{finitely!generated}
\end{defn}

\begin{rk}
If $A$ is finitely generated by $K=\{y_1,\dots,y_n\}$, 
then we have a map $\Delta(A)\to\bC^n$ given by 
$\hat y_1,\dots, \hat y_n$. This is a homeomorphism from $\Delta(A)$
to a compact subset of $\bC^n$. It turns out that 
the image can be characterised by the property of being ``polynomially
convex,'' a notion weaker than convexity, see Stout~\cite{Stout:2007a}. The complement of a
polynomially convex subset of $\bC^n$  is always $(n-1)$-connected, by a theorem of
Forsterni\v{c}~\cite{Forsternic:1994a}.\index{Forsterni\v{c}'s theorem}
\end{rk}

\section{The Jacobson radical}

\begin{defn}
An element $x$ of a Banach algebra $A$ is said to be
\emph{quasi-nilpotent}\index{quasi-nilpotent element|textbf}
if $\lim_{n\to\infty}\sqrt[n]{\|x^n\|}=0$. In the literature, this is
sometimes also called
\emph{topologically nilpotent}.\index{topologically nilpotent element}
It follows from Lemma~\ref{le:binom} that the sum of two
quasi-nilpotent elements is quasi-nilpotent. Since a linear multiple
of a quasi-nilpotent element is also nilpotent, it follows that the 
quasi-nilpotent elements form a linear subspace. In fact, we shall
see that they form a closed ideal.
\end{defn}

\begin{defn}
The \emph{Jacobson radical}\index{Jacobson radical} of a ring $A$,
denoted $J(A)$, is
the intersection of its maximal right ideals, or equivalently the
intersection of its maximal left ideals.\index{maximal!ideal}
The ring $A$ is said to be
\emph{semisimple}\index{semisimple|textbf} if $J(A)=0$.
\end{defn}

\begin{prop}\label{pr:J}
In the case where $A$ is a commutative
Banach algebra, 
the maximal ideals are closed, and are the kernels of
(automatically continuous) algebra homomorphisms 
$A\to \bC$.
The Jacobson radical $J(A)$ is the 
intersection of the kernels of these algebra homomorphisms.
In particular, $A$ is semisimple if and only if its elements are separated by
algebra homomorphisms $A\to \bC$.
\end{prop}
\begin{proof}
This follows from  Corollaries~\ref{co:maxideal}
and~\ref{co:maxkernel} and Proposition~\ref{pr:continuous}.
\end{proof}

\begin{theorem}\label{th:qn}
For an element $x$ in a commutative Banach algebra $A$, the following are
equivalent:
\begin{enumerate}
\item $x$ is quasi-nilpotent.
\item The spectral radius of $x$ is zero.
\item For every algebra homomorphism $\phi\colon A\to \bC$ 
we have $\phi(x)=0$.
\item The image of $\hat x$ is $\{0\}$.
\item $x$ is in the Jacobson radical $J(A)$.
\end{enumerate}
\end{theorem}
\begin{proof}
The equivalence of (i) and (ii) follows from the spectral radius
formula, Proposition~\ref{pr:Gelfand}. The equivalence of (ii),
(iii) and (iv) follows from Corollary~\ref{co:radius}. The equivalence of (iii)
and (v) follows from Proposition~\ref{pr:J}.
\end{proof}

\section{\texorpdfstring{Banach $*$-algebras}
{Banach *-algebras}}

\begin{defn}\label{def:B*alg}
A \emph{$*$-algebra}\index{star!-algebra} $A$ is an
algebra
with a star operation\index{star!operation} $x \mapsto x^*$ satisfying
\begin{enumerate}
\item involutory:\index{involutory} for all $x\in A$ we have
  $x^{**}=x$.
\item antilinear:\index{antilinear}
for all $\lambda\in \bC$ and $x,y\in A$ we have $(x+y)^*=x^*+y^*$
and $(\lambda x)^*=\bar\lambda x^*$.
\item multiplicative antiautomorphism:\index{antiautomorphism} 
for all $x,y\in A$ we have $(xy)^*=y^*x^*$.
\end{enumerate}
A \emph{normed $*$-algebra}\index{normed!star@$*$-algebra|textbf} is a
normed algebra which is simultaneously a $*$-algebra, in such a way
that the star operation is continuous with respect to the norm.
A \emph{Banach $*$-algebra}\index{Banach!star@$*$-algebra} $A$ is a 
normed $*$-algebra that is complete with respect to the norm.
\end{defn}

\begin{lemma}\label{le:A*hat}
The metric completion $\hat A$ of a normed $*$-algebra $A$ 
is a Banach $*$-algebra in which $A$ is a dense subalgebra.
\end{lemma}
\begin{proof}
By Lemma~\ref{le:Ahat}, $\hat A$ is a Banach algebra. Applying the
star operation to a Cauchy sequence yields another Cauchy sequence.
This preserves the equivalence relation, and hence defines a star
operation on $\hat A$. The properties in Definition~\ref{def:B*alg}
for this star operation on $\hat A$ follow from those properties on
$A$.
\end{proof}

\begin{lemma}\label{le:Specx*}
If $x$ is an element of a Banach $*$-algebra $A$ then
$\Spec(x^*)=\overline{\Spec(x)}$.
\end{lemma}
\begin{proof}
This follows from the fact that $x-\lambda\one$ is invertible if and
only if 
\[ (x-\lambda\one)^*=x^*-\bar\lambda\one \] 
is invertible.
\end{proof}

\begin{eg}\label{eg:l1Z*alg}
The Banach algebra $\ell^1(\bZ)$ of Example~\ref{eg:l1Zalg}
is a Banach $*$-algebra with star operation
\[ x^*(n)=\overline{x(-n)}. \] 
This star operation swaps $u$ and $u^{-1}$.
In terms of Fourier series, this star operation is given by
$f^*(e^{\bi\theta})=\overline{f(e^{-\bi\theta})}$.
We shall see later in this section that this is an example of a 
\emph{symmetric} Banach $*$-algebra.
\end{eg}

\begin{defn}
A $*$-\emph{ideal}\index{star!ideal} in a Banach $*$-algebra is an
ideal $I$ such that for all $x\in I$ we have $x^*\in I$.
\end{defn}

\begin{lemma}
\begin{enumerate}
\item If $I$ is a closed $*$-ideal in a Banach $*$-algebra $A$, then $A/I$
is a Banach $*$-algebra with the quotient norm, and with the star
operation $(x+I)^*=x^*+I$.
\item If $J\le I \le A$ are closed $*$-ideals then then natural map
$(A/J)/(I/J)\to A/I$ is an isometric isomorphism of Banach
$*$-algebras.
\item If $A$ is a normed $*$-algebra and $I$ is a closed ideal then
  the natural map of completions $\widehat{A/I}\to \hat A/\hat I$ is
  an isometric isomorphism of Banach $*$-algebras.
\end{enumerate}
\end{lemma}
\begin{proof}
The proof is the same as the proof of
Lemma~\ref{le:quotient}, but keeping track of the star operation.
\end{proof}

\begin{defn}
An element $x$  is \emph{self-conjugate}\index{self-conjugate element|textbf}
if $x=x^*$.
\end{defn}

\begin{rk}\label{rk:y+iz}
Given any element $x$, we can write $x$ uniquely as 
$y+\bi z$ where $y$ and $z$ are
self-conjugate elements. Namely, we have $y=(x+x^*)/2$ and
$z=(x-x^*)/2\bi$. Then we have $x^*=y-\bi z$.
\end{rk}

\begin{defn}
If $A$ is a commutative Banach $*$-algebra and 
$\phi\colon A \to \bC$ is an algebra homomorphism, we define
the \emph{conjugate} of $\phi$ to be the algebra homomorphism
$\phi^*\colon A \to \bC$\index{phi@$\phi^*$, conjugate of $\phi$} 
defined by
\[ \phi^*(x)=\overline{\phi(x^*)}. \]
This defines a continuous involutary automorphism on the structure space.
\end{defn}

\begin{prop}\label{pr:symmetric}
The following conditions on a commutative Banach $*$-algebra $A$ are equivalent:
\begin{enumerate}
\item Every algebra homomorphism $\phi\colon A\to\bC$ is self
  conjugate. In other words, for every $x\in A$ we have 
\[ \phi(x^*)=\overline{\phi(x)}. \]
\item For every algebra homomorphism $\phi\colon A \to \bC$, and every
  self-conjugate $x\in A$, $\phi(x)$ is real.
\item Every element of the form $x^*x+\one$ is invertible in $A$.
\item If $x\in A$ is self-conjugate then $x-\bi\one$ is invertible.
\item For every $x\in A$, the spectral radius satisfies $\rho(x^*x)=\rho(x)^2$.
\end{enumerate}
\end{prop}
\begin{proof}
Remark~\ref{rk:y+iz} shows that (i) and (ii) are equivalent. 

If (i)
holds, then for $x\in A$ and $\phi\colon A\to\bC$ we have 
\[  \phi(x^*x+\one)=\phi(x^*)\phi(x)+1=\overline{\phi(x)}\phi(x)+1=|\phi(x)|^2+1>0. \]
By Theorem~\ref{th:invertible} it follows that for every $x\in A$ the element $x^*x+\one$ is invertible,
and so (iii) holds. 

Next, we show that (iii) implies (iv). If $x$ is self-conjugate then
\[ (x-\bi\one)^*(x-\bi\one)=(x+\bi\one)(x-\bi\one)=x^2+\one = x^*x+\one \]
is invertible, and hence so is $x-\bi\one$.

Next, we prove that (iv) implies (ii).
If $x$ is self-conjugate, $\phi\colon A\to \bC$, 
and $a+\bi b\in\bC$ with $b\ne 0$, we must show
that $\phi(x)\ne a+\bi b$. So it is enough to show that $x-(a+\bi b)\one$ is
invertible. This follows from (iv) by writing 
\[ x-(a+\bi b)\one=b((x-a\one)/b-\bi\one). \]

To prove that (i) implies (v), if $\phi\colon A \to \bC$ is an algebra
homomorphism and (i) holds then
\[ \phi(x^*x)=\phi(x^*)\phi(x) = \overline{\phi(x)}\phi(x) =
  |\phi(x)|^2. \]
So by Corollary~\ref{co:radius},
\[ \rho(x^*x) = \sup_{\phi\colon A\to \bC}|\phi(x^*x)|
=\sup_{\phi\colon A\to\bC}|\phi(x)|^2 = \rho(x)^2. \]

Finally, to prove that (v) implies (ii), we suppose that (v) holds 
but (ii) does not hold, and deduce a contradiction. 
So there exists a self-conjugate element $x$ such
that $\phi(x)=\alpha+\beta\bi$ with $\alpha$ and $\beta$ real and 
$\beta$ non-zero. Replacing $x$ by $\beta^{-1}(x-\alpha\one)$,
we may assume that $\phi(x)=\bi$. 
If $a$ is a positive real number then $(a\one-\bi x)^*=a\one+\bi x$.
We have $\phi(a\one -\bi x)= a+1$, and so by
Corollary~\ref{co:radius}, it follows that $a+1\le \rho(a\one-\bi x)$.
If (v) holds then using Lemma~\ref{le:rho} we have
\[ (a+1)^2\le \rho(a\one-\bi x)^2=\rho((a\one+\bi x)(a\one-\bi x))
=\rho(a^2\one + x^2) \le a^2+\rho(x)^2, \]
and hence $2a+1\le \rho(x)^2$. This holds for all $a>0$, which is
the desired contradiction.
\end{proof}

\begin{rk}
Condition (ii) of Proposition~\ref{pr:symmetric} may be interpreted as saying that
the spectrum of a self-conjugate element of a 
Banach $*$-algebra satisfying these equivalent conditions 
is real (cf.\ Corollary~\ref{co:radius}), 
whereas for a more general Banach $*$-algebra 
the spectrum of a self-conjugate element is
merely symmetric about the real axis (Lemma~\ref{le:Specx*}).
\end{rk}

\begin{defn}\label{def:symmetric}
A Banach $*$-algebra is said to be
\emph{symmetric}\index{symmetric!Banach $*$-algebra|textbf}%
\index{Banach!star@$*$-algebra!symmetric|textbf}
if condition (iii) of Proposition~\ref{pr:symmetric} is 
satisfied, and \emph{Hermitian}\index{Hermitian Banach $*$-algebra}
if every self-conjugate element has a real spectrum. These conditions
are equivalent for commutative Banach $*$-algebras by the Proposition,
but are
not equivalent without the commutativity assumption. Nonetheless, a
theorem of Shirali~\cite{Shirali:1967a,Shirali/Ford:1970a} shows that
even for a non-commutative Banach $*$-algebra, Hermitian implies symmetric. 
\end{defn}

\begin{eg}
For an example of a symmetric Banach $*$-algebra, 
see the algebra $\ell^1(\bZ)$ of Example~\ref{eg:l1Z*alg}.
It is an easy exercise to check that
this satisfies condition (i) of Proposition~\ref{pr:symmetric}.

For an example of a non-symmetric Banach $*$-algebra, 
consider the
algebra of continuous functions from the closed disc $D= \{z\in\bC\mid
|z|\le 1\}$ to $\bC$ which are holomorphic on the interior of $D$,
with pointwise addition and multiplication, and with
$f^*(z)=\overline{f(\bar z)}$. The spectrum of $f$ is its image, so
for example the element $f(z)=z$ is self-conjugate; its spectrum is
symmetric about the real axis, but not real.
\end{eg}

\begin{cor}\label{co:symm-quot}
If $A$ is a commutative symmetric Banach $*$-algebra an $I$ is a
closed $*$-ideal then $A/I$ is a commutative symmetric Banach
$*$-algebra.
\end{cor}
\begin{proof}
Condition (iv) of Proposition~\ref{pr:symmetric} is inherited by $A/I$.
\end{proof}

\begin{theorem}\label{th:e=e*}
If $e$ is an idempotent in a commutative symmetric Banach $*$-algebra then
$e=e^*$.
\end{theorem}
\begin{proof}
Let 
\[ z=1+(e^*-e)^*(e^*-e) = 1-e-e^*+2ee^*. \]
Then we have $z=z^*$ and
\[ ez= e-e-ee^*+2ee^*=ee^*. \]
Similarly, $e^*z=ee^*$, so $ez=e^*z$. But $z$ is invertible by
Proposition~\ref{pr:symmetric}, and so $e=e^*$.
\end{proof}

\section{\texorpdfstring{$C^*$-algebras}
{C*-algebras}}

\begin{defn}\label{def:Cstar-algebra}
A $C^*$-\emph{algebra}\index{C@$C^*$-algebra} 
is a Banach $*$-algebra $A$ in which
 for all $x\in A$ we have $\|x^*x\|=\|x\|^2$.
\end{defn}

\begin{theorem}\label{th:comm-is-symm}
If $A$ is a commutative $C^*$-algebra then $A$ is a 
symmetric Banach\linebreak[3]
$*$-algebra.
\end{theorem}
\begin{proof}
We shall verify that condition (iv) of Proposition~\ref{pr:symmetric}
holds. If $x$ is self-adjoint and $x-\bi\one$ is not invertible then 
$\bi$ is in the spectrum of $x$, 
and hence for every real $a>0$, $a+1=a-\bi^2$ is in
the spectrum of $a\one - \bi x$.
So by Corollary~\ref{co:radius} we have
$a+1\le\|a\one-\bi x\|$. Hence
\[ (a+1)^2 \le \|a\one-\bi x\|^2=\|(a\one+\bi x)(a\one-\bi x)\|
= \|a^2\one+x^2\|\le a^2+\|x\|^2. \]
So we have $\|x\|^2\ge 2a+1$ for every $a>0$, which is
absurd. Hence $x-\bi\one$ is invertible.
\end{proof}

\begin{eg}
The symmetric Banach $*$-algebra $\ell^1(\bZ)$ 
(see Example~\ref{eg:l1Z*alg}) is the completion
of $\bC[u,u^{-1}]$ with respect to the $\ell^1$ norm,
with star operation sending $u$ to $u^{-1}$. 
To see that this
is not a $C^*$-algebra,
 we let $x=u^2+u-1$.
Then  $\|x^2\|=7$ while $\|xx^*\|=5$.

On the other hand, if $Z$ is any compact Hausdorff topological space
then the algebra $C(Z)$ of continuous functions on $Z$ is a
commutative 
$C^*$-algebra. For the norm we take 
$\displaystyle\|x\|=\sup_{z\in Z}|x(z)|$, and the $*$
operation is defined by $x^*(z)=\overline{x(z)}$. We have
\[ \|x^*x\|
=\sup_{z\in Z}(\overline{x(z)}x(z))
=\sup_{z\in Z}|x(z)|^2
= \bigl(\,\sup_{z\in Z}|x(z)|\,\bigr)^2
=\|x\|^2. \]
\end{eg}

\begin{rk}
The reader will notice that the proof of Theorem~\ref{th:comm-is-symm}
is similar to the proof of (v) implies (ii) in
Proposition~\ref{pr:symmetric}. Indeed, in the light of the following
theorem, we could have just invoked that proposition.
\end{rk}

\begin{theorem}\label{th:C*radius}
If $x$ is an element of a commutative $C^*$-algebra then its spectral
radius is equal to $\|x\|$.
\end{theorem}
\begin{proof}
We use the spectral radius formula, Proposition~\ref{pr:Gelfand}.
Let $y=x^*x$, so that $y^*=y$. Then the spectral radius of $y$ is
\begin{align*} 
\rho(y)&=\lim_{n\to\infty}\sqrt[n]{\|(x^*x)^n\|}=\lim_{n\to\infty}
  \sqrt[n]{\|(x^n)^*(x^n)\|}\\
&=\lim_{n\to\infty}\sqrt[n]{\|x^n\|^2}
=\left(\lim_{n\to\infty}\sqrt[n]{\|x^n\|}\right)^2=\rho(x)^2. 
\end{align*}
So it suffices to prove
that $\rho(y)=\|y\|=\|x\|^2$.

We have $\|y^2\|=\|y^*y\|=\|y\|^2$, and by induction, for all $k\ge 0$ we have
$\|y^{2^k}\|=\|y\|^{2^k}$. Thus 
\begin{equation*}  
\rho(y)=
\lim_{k\to\infty}\sqrt[2^k]{\|y^{2^k}\|}=\lim_{k\to\infty}\sqrt[2^k]{\|y\|^{2^k}}=\|y\|. 
\qedhere
\end{equation*}
\end{proof}

\begin{cor}\label{co:J(C)}
The Jacobson radical of a commutative $C^*$-algebra is zero.
Thus the only quasi-nilpotent element is zero.
\end{cor}
\begin{proof}
By Theorems~\ref{th:qn}, 
if $x$ is in the Jacobson radical then the spectral radius is
zero. By Theorem~\ref{th:C*radius} the spectral radius of $x$ is 
$\|x\|$ so $\|x\|=0$. This implies that $x=0$. By Theorem~\ref{th:qn},
quasi-nilpotent elements lie in the Jacobson radical, and 
therefore the only quasi-nilpotent element is zero.
\end{proof}

\section{Hilbert space}

\begin{defn}\index{Hilbert space}
A \emph{Hilbert space} is a complex vector space 
$H$ with an inner product 
$\langle -,-\rangle\colon H \times H \to \bC$,
satisfying 
\begin{enumerate}
\item Linearity in the first variable: $\langle ax,y\rangle=a\langle
  x,y\rangle$ and $\langle x_1+x_2,y\rangle = \langle x_1,y\rangle +
  \langle x_2,y\rangle$,
\item Conjugate symmetry: $\langle y,x\rangle = \overline{\langle
    x,y\rangle}$,
\item Positive definiteness: if $x\ne 0$ then $\langle x,x\rangle >
  0$,
\item Completeness: With respect to the norm coming from the inner
  product   $|x|=\sqrt{\langle x,x\rangle}$, $H$ is complete.
\end{enumerate}
Thus a Hilbert space is a particularly rigid kind of Banach space. 
\end{defn}

\begin{lemma}\label{le:|x|}
If $x\in H$ then $|x|=\displaystyle
\sup_{|w|=1}|\langle x,w\rangle|=\sup_{|w|=1}|\langle w,x\rangle|$.
\end{lemma}
\begin{proof}
We prove the first equality, as the second follows from the fact that
by conjugate symmetry we have
$|\langle w,x\rangle|=|\langle x,w\rangle|$.

If $w=x/|x|$ then $|w|=1$ and $\langle x,w\rangle=|x|$. 
Conversely, for any $w\in H$ with $|w|=1$ we set $w'=w-\frac{\langle
  x,w\rangle}{\langle x,x\rangle} x$. Then 
\[ \langle x,w'\rangle =
\langle x,w\rangle - \frac{\langle x,w\rangle}{\langle
  x,x\rangle}\langle x,x\rangle= 0. \]
So 
\[ 1 = |w|^2=|w'|^2+\left|\frac{\langle x,w\rangle}{\langle
      x,x\rangle}x\right|^2, \]
and hence $\left|\frac{\langle x,w\rangle}{\langle
      x,x\rangle}x\right|\le 1$. This implies that
$|\langle x,w\rangle| \le \frac{|\langle x,x\rangle|}{|x|}=|x|$.
\end{proof}

\begin{lemma}[Parallelogram identity]\index{parallelogram identity}
If $u$ and $v$ are elements of a Hilbert space $H$ then
\[ |u+v|^2+|u-v|^2=2|u|^2+2|v|^2. \]
\end{lemma}
\begin{proof}
Expand $\langle u+v,u+v\rangle$ and $\langle u-v,u-v\rangle$ using
bilinearity.
\end{proof}

\begin{lemma}\label{le:orthog-vec}
If $V$ is a closed subspace of a Hilbert space $H$ then every coset of
$V$ in $H$ contains a unique element orthogonal to $V$.
\end{lemma}
\begin{proof}
Let $w$ be an element of the coset. We  let $d$ be the infimum of
the values of $|w|$ as $w$ runs over elements the coset. Then there is a
sequence of elements $w_1,w_2,\dots$ of such elements with
$\lim_{i\to\infty}|w_i|=d$. We claim that $w_1,w_2,\dots$ is a
Cauchy sequence in $H$. To see this, 
we apply the parallelogram identity
to $w_i$ and $w_j$:
\[ |w_i+w_j|^2 +|w_i-w_j|^2 =2|w_i|^2+2|w_j|^2, \]
which gives
\[ |w_i-w_j|^2 = 2|w_i|^2+2|w_j|^2-4|\half(w_i+w_j) |^2. \]
Since $\half(w_i+w_j)$ is an element of the same coset, we have
$|\half(w_i+w_j)|\ge d$. If $|w_i|<d+\ep$ and $|w_j|<d+\ep$, 
we obtain
\[ |w_i-w_j|^2< 2(d+\ep)^2+2(d+\ep)^2-4d^2 = 4\ep(2d+\ep). \]
So as $\ep\to 0$ we have $|w_i-w_j|\to 0$, and $w_1,w_2,\dots$ is a
Cauchy sequence. Now $H$ is
complete and $V$ is closed, so $V$ is complete and hence the
coset is complete. So this Cauchy sequence has a limit $w$, 
and we have $|w|=d$. 

If $0\ne v\in V$ let $w'=w-
\frac{\langle  v,w\rangle}{\langle v,v\rangle}v$. 
Then $\langle v,w'\rangle = 0$
and so 
\[ |w|^2=|w'|^2 + \left|\frac{\langle  v,w\rangle}{\langle
    v,v\rangle}v\right|^2. \]
Since $|w'|^2\ge |w|^2$ it follows that $\langle v,w\rangle=0$.
Thus $w$ is orthogonal to $V$. If $w_0$ is another element of the
coset orthogonal to $V$ then $w-w_0$ is in $V$ and orthogonal to $V$,
and hence equal to zero.
\end{proof}

\begin{theorem}[Fr\'echet--Riesz]\label{th:FR}\index{Fr\'echet--Riesz Theorem}
Given a continuous $\bC$-algebra homomorphism $\phi\colon H\to \bC$
there exists a unique element $y\in H$ such that for all $x\in H$ 
$\phi(x)=\langle x,y\rangle$.
\end{theorem}
\begin{proof}
Let $V$ be the kernel of $\phi$. Since $\phi$ is continuous, $V$ is a
closed subspace of $H$ of codimension one. The set of elements $x\in H$ such
that $\phi(x)=1$ is a coset of $V$. Using Lemma~\ref{le:orthog-vec}
there is an element $y$ in this coset which is orthogonal to all $v\in
V$. Every element of $H$ can be written as $v+\lambda y$ with $v\in V$
and $\lambda \in \bC$, and we have 
\begin{equation*} 
\phi(v+\lambda y) = \phi(v)+\lambda\phi(y) = \lambda = \langle
  v+\lambda y, y\rangle. 
\end{equation*}
Finally, for uniqueness, if $y'$ is another then $y-y'$ is both in $V$
and orthogonal to $V$, and hence equal to zero.
\end{proof}

\begin{defn}
If $f\colon H \to H$ is a linear transformation, we set
\[ \|f\|_{\sup}=\sup_{|x|=1} |f(x)|. \]
If $\|f\|_{\sup}<\infty$, we say that $f$ is a 
\emph{bounded operator}\index{bounded!operator|textbf} on $H$.
By Lemma~\ref{le:bd=cts}, bounded operators are continuous.
We write $\Lin(H)$\index{L@$\Lin(H)$}  
for the set of bounded operators on $H$. We shall see in
Theorem~\ref{th:LinC*} that $\Lin(H)$ has the structure of a $C^*$-algebra.
\end{defn}

\begin{theorem}[Adjoints]\label{th:adjoints}\index{adjoint map}
Let $f\colon H\to H$ be a bounded operator. Then there is a unique
bounded operator $f^*$, called the \emph{adjoint} of $f$, such that
for all $x$ and $y$ in $H$ we have 
\[ \langle f(x),y\rangle = \langle x,f^*(y)\rangle. \]
We have $f^{**}=f$, $(\lambda f+ \mu f')^*=\bar\lambda f^*+ \bar\mu
f'^{\,*}$,  $(f\circ f')^*=f'^{\,*}\circ f^*$,
$\|f^*\circ f\|_{\sup}=\|f\|^2_{\sup}$, and
 $\|f^*\|_{\sup}=\|f\|_{\sup}$. 
\end{theorem}
\begin{proof}
For $x\in H$, the map $\phi\colon H \to \bC$ given by $\phi(y)=\langle
f(x),y\rangle$ is a continuous $\bC$-algebra homomorphism. By
Theorem~\ref{th:FR}, there exists a unique element $z\in H$ such that
for all $x\in H$ we have $\langle f(x),y\rangle = \langle x,z\rangle$.
We define $f^*(x)=z$. To see that $f^*$ is linear, we have
\begin{align*} 
\langle x, f^*(\lambda y+\mu z) \rangle 
&= \langle f(x),\lambda  y+\mu z\rangle\\ 
&= \bar\lambda\langle f(x),y\rangle + \bar\mu \langle f(x),z\rangle\\
& = \langle x,\lambda f^*(y)+\mu f^*(z)\rangle, 
\end{align*}
so that by uniqueness, $f^*(\lambda y+\mu z)=\lambda f^*(y)+\mu
f^*(z)$. It is now easy to check that $f^{**}=f$ and 
$(\lambda f+ \mu f')^*=\bar\lambda f^*+ \bar\mu f'^{\,*}$.

For all $x$ and $y$ in $H$ we have
\begin{align*} 
\langle (f\circ f')(x),y\rangle& = \langle f(f'(x),y\rangle =
  \langle f'(x),f^*(y)\rangle \\
&= \langle x,f'^{\,*}(f^*(y))\rangle
= \langle x,(f'^{\,*}\circ f^*)(y)\rangle, 
\end{align*}
and so $(f\circ f')^*=f'^{\,*}\circ f^*$.

Using Lemma~\ref{le:|x|},
we have
\begin{align*} 
\|f^*\circ f\|_{\sup}&=\sup_{|y|=1}|f^*\circ f(y)|\\
&=\sup_{|x|=1,|y|=1}\langle x,  f^*(f(y))\rangle\\
& = \sup_{|x|=1,|y|=1}\langle f(x),f(y)\rangle\\
&=\sup_{|x|=1}|f(x)|^2=\|f\|_{\sup}^2. 
\end{align*}
Thus $\|f^*\circ f\|_{\sup}=\|f\|^2_{\sup}$.
Similarly,
using Lemma~\ref{le:|x|} we have
\begin{align*} 
\|f^*\|_{\sup} &= \sup_{|y|=1}|f^*(y)|\\
& = \sup_{|x|=1,|y|=1}|\langle  x,f^*(y)\rangle|\\
& = \sup_{|x|=1,|y|=1}|\langle f(x),y\rangle|\\
&=\sup_{|x|=1}|f(x)|=\|f\|_{\sup}. 
\end{align*}
Thus $\|f^*\|_{\sup}=\|f\|_{\sup}$. It follows that $f^*$ is bounded.
\end{proof}

\begin{theorem}\label{th:LinC*}
With the operation of composition, 
the sup norm,\index{sup!norm} 
and the star operation\index{star!operation} of taking adjoints, $\Lin(H)$ is a
$C^*$-algebra.
\end{theorem}
\begin{proof}
The sum of two bounded operators is a bounded operator, with
\[ \|f+f'\|_{\sup}=\sup_{|x|=1}|f(x)+f'(x)|\le
  \sup_{|x|=1}|f(x)|+\sup_{|x|=1}|f'(x)|=\|f\|_{\sup}+\|f'\|_{\sup}. \]
Positivity and homogeneity is easily checked, so $\Lin(H)$ is a normed
space. 
The limit of a Cauchy sequence of bounded operators is a bounded
operator, so $\Lin(H)$ is a Banach space.

The composite of two bounded operators is a bounded operator, with
\[ \|f\circ f'\|_{\sup}=\sup_{|x|=1}|f(f'(x))|\le \sup_{|y|=\|f'\|_{\sup}}|f(y)|=
  \|f\|_{\sup}\|f'\|_{\sup}. \]
The identity linear transformation is bounded with sup norm one. So
$\Lin(H)$ is a Banach algebra. 
Using Theorem~\ref{th:adjoints},
the star operation is involutory, antilinear, and
an antiautomorphism, so it is a Banach $*$-algebra.
Finally, by the same theorem it also satisfies 
$\|f^*\circ f\|_{\sup}=\|f\|_{\sup}^2$, so it is a $C^*$-algebra.
\end{proof}

\begin{defn}
An \emph{action}\index{action on Hilbert space} of a Banach $*$-algebra $A$ 
on a Hilbert space $H$ is a continuous $*$-algebra homomorphism
$A\to \Lin(H)$. The action is \emph{faithful}\index{faithful action}
if this map is injective.
\end{defn}

\begin{theorem}
If $A$ is a Banach $*$-algebra and $A \to \Lin(H)$ is an action on a
Hilbert space $H$ then the closure of the image of $A$ in $\Lin(H)$ is
a $C^*$-algebra. If $A$ is commutative then every quasi-nilpotent
element of $A$ is in the kernel of the action on $H$. 
\end{theorem}

\chapter{Completions of representation rings}%
\index{completion!of representation ring}\label{ch:complete-repring}

\section{The norm on a representation ring}%
\index{norm!on representation ring}

\begin{defn}\label{def:norm}
Let $\fa$ be a representation ring (see Definition~\ref{def:repring}),
and let $\fa_\bC=\bC\otimes_\bZ\fa$\index{a@$\fa_{\bC}$}
(see Definition~\ref{def:a_R}).
If $x=\sum_{i\in\fI}a_ix_i\in\fa_\bC$ then we define the 
\emph{weighted $\ell^1$ norm}%
\index{weighted!$\ell^1$ norm}\index{norm!weighted $\ell^1$} 
of $x$ to be
\[ \|x\|=\sum_{i\in\fI}|a_i|\dim x_i. \]
\end{defn}

\begin{lemma}\label{le:Anormed}
The map $\fa_\bC\to\bR$ given by $x \mapsto \|x\|$ satisfies 
Definition~\ref{def:na}\,{\rm(i)--(v)}, and hence makes $\fa_\bC$ a normed
algebra. Furthermore, the star operation\index{star!operation} taking 
$x=\sum_{i\in\fI} a_ix_i$ to $x^*=\sum_{i\in\fI}\bar a_ix_{i^*}$ makes
$\fa_\bC$ into a commutative normed $*$-algebra,\index{normed!star@$*$-algebra} 
see Definition~\ref{def:B*alg}.
\end{lemma}
\begin{proof}
To verify that submultiplicativity\index{submultiplicativity} holds,  let
$x=\sum_{i\in\fI}a_ix_i$ and $y=\sum_{j\in\fI}b_jx_j$. We have
$x_ix_j=\sum_{k\in\fI}c_{i,j,k}x_k$ with $c_{i,j,k}$ non-negative integers. So
\[ xy = \sum_{k\in\fI}\bigl(\sum_{i,j\in\fI}c_{i,j,k}a_ib_j\bigr)x_k \]
and hence
\begin{align*}
\|xy\|&=\sum_{k\in\fI} \bigl|\bigl(\sum_{i,j\in\fI}c_{i,j,k}a_ib_j\bigr)\bigr|\dim x_k \\
&\le \sum_{k\in\fI} \bigl(\sum_{i,j\in\fI}c_{i,j,k}|a_ib_j|\bigr)\dim x_k \\
      &= \sum_{i,j\in\fI}   |a_ib_j|\dim x_ix_j \\
      &=\sum_{i\in\fI} |a_i|\dim x_i\,\sum_{j\in\fI}|b_j|\dim x_j \\
  &=\|x\|\|y\|.
\end{align*}
The remaining axioms for the norm and star operation are easy to verify.
\end{proof}

\begin{defn}
We define $\hat\fa$\index{a@$\hat\fa$} to be the 
completion of $\fa_\bC$ with respect to the
norm defined above. By Lemmas~\ref{le:A*hat} and~\ref{le:Anormed},
$\hat \fa$ is a commutative Banach $*$-algebra associated to the
representation ring $\fa$, and $\fa_\bC$ is a dense subalgebra of $\hat\fa$.
\end{defn}

We can think of elements of $\hat\fa$ concretely as possibly
infinite linear combinations $\sum_{i\in\fI} a_ix_i$ where
$\sum_{i\in\fI} |a_i|\dim x_i <\infty$.
Note that any such sum automatically has countable support, by the following lemma.

\begin{lemma}\label{le:countable}\index{countable sum}
Let $\{r_\alpha\}_{\alpha\in J}$ be a collection of positive real
numbers indexed by a set $J$. If the
sums $\sum_{\alpha\in I}r_\alpha$ over finite subsets $I\subseteq J$
are bounded above then $J$ is countable.
\end{lemma}
\begin{proof}
For $n>0$, the subset $J_n\subseteq J$ consisting of those 
$\alpha\in J$ for which $r_\alpha>1/n$ is finite,
and $J=\bigcup_{n>0}J_n$ is a countable union of finite sets.
\end{proof}

However, $\hat\fa$ is usually not separable:

\begin{defn}
A Banach algebra $A$ is \emph{separable}\index{separable Banach algebra} if there is a countable
subset $K$ of $A$ such that the closure of the subalgebra generated
by $K$ is the whole of $A$.
\end{defn}

\begin{propqed}
The Banach algebra $\hat\fa$ is separable if and only if the index set $\fI$ is
countable.
\end{propqed}

\begin{rk}
Another basis for $\fa_\bC$ consists of the elements
$x_j$ for $j=j^*\in\fI$, together with the elements $(x_j+x_j^*)/2$
and $(x_j-x_j^*)/2\bi$ for $j\ne j^*\in\fI$. The elements in this basis are
self-conjugate.\index{self-conjugate element} 
Their linear span $\fa_\bR$\index{a@$\fa_\bR$} 
is a real normed algebra, whose completion
$\hat\fa_\bR$\index{a@$\hat\fa_\bR$} is a real Banach algebra 
with the property that $\bC\otimes_\bR
\hat\fa_\bR\cong\hat\fa$, with star operation coming from complex
conjugation on the first tensor factor.
\end{rk}

\section{Norms and cores}

\begin{defn}
Let $\fX\subset\fI$ be a representation ideal in a representation ring
$\fa$. Then $\langle\fX\rangle_\bC=\bC\otimes_\bZ\langle\fX\rangle$ is an ideal in
$\fa_\bC=\bC\otimes_\bZ\fa$, and its closure
$\widehat{\langle\fX\rangle}_\bC$ in $\hat\fa$ is a closed ideal.
We write $\hat\fa_\fX$\index{a@$\hat\fa_{\fX}$, $\hat\fa_{\max}$, $\hat\fa_{\proj}$}
for the quotient $\hat\fa/\widehat{\langle\fX\rangle}_\bC$.
We write $\hat\fa_{\max}$ for $\hat\fa_{\fX_{\max}}$\index{X@$\fX_{\max}$} and
$\hat\fa_\proj$ for $\hat\fa_{\fX_\proj}$ (see Definition~\ref{def:max-proj}).
\end{defn}

By Lemma~\ref{le:quotient}\,(iii), $\hat\fa_\fX$ is isometrically isomorphic to
the completion $\widehat{\fa_{\bC,\fX}}$ of $\fa_{\bC,\fX}$ with
respect to the quotient norm.
It is easy to check that the quotient norm on $\hat\fa_\fX$ is given by
\begin{equation}\label{eq:normA/X}
  \left\|\sum_{i\in\fI}a_ix_i\right\|_\fX=\sum_{i\in\fI}|a_i|\dim\core_\fX(x_i)
  =\sum_{i\in\fI\setminus\fX}|a_i|\dim x_i. 
\end{equation}

We can think of elements of $\hat\fa_\fX$ concretely as possibly
infinite (but necessarily countably supported, see Lemma~\ref{le:countable})
linear combinations $\sum_{i\in\fI\setminus\fX}a_ix_i$ where  
\[ \sum_{i\in\fI\setminus\fX}|a_i|\dim x_i<\infty. \] 

\begin{lemma}
If $x,y\in\fa_{\cge 0}$ then $\|x+y\|_\fX=\|x\|_\fX + \|y\|_\fX$.
\end{lemma}
\begin{proof}
This is clear from the definition.
\end{proof}

\begin{lemma}\label{le:norm}
If $x\in\fa_{\cge 0}$ then the quotient norm on the
image of $x$ in $\hat\fa_\fX$ is equal to $\dim\core_\fX(x)$.
Thus $\npj_\fX(x)=\displaystyle\lim_{n\to\infty}\sqrt[n]{\|x^n\|_\fX}$.
\end{lemma}
\begin{proof}
  This follows from \eqref{eq:normA/X}.
\end{proof}

\begin{theorem}\label{th:spec-radius}
If $x\in\fa_{\cge 0}$ then $\npj_\fX(x)$ is equal to the spectral
radius of the image of $x$ in $\hat\fa_\fX$.
\end{theorem}
\begin{proof}
This follows from Lemma~\ref{le:norm} and Proposition~\ref{pr:Gelfand}.
\end{proof}

\begin{lemma}
If $\fY\le \fX$ are representation ideals in $\fa$ then
the closure $I$ of $\langle\fX\rangle_\bC/\langle\fY\rangle_\bC$ in
$\hat\fa_\fY$ is a closed ideal, and
the quotient $(\hat\fa_\fY)/I$, with the quotient norm,
is a commutative Banach algebra
isometrically isomorphic to $\hat\fa_\fX$.
\end{lemma}
\begin{proof}
This follows from Lemma~\ref{le:quotient}.
\end{proof}

\section{Spectrum, species, structure space}

\begin{defn}
Let $\fX\subset\fI$ be a representation ideal in a representation ring $\fa$.
If $x\in\hat\fa_\fX$, we write 
$\Spec_\fX(x)$\index{Spec@$\Spec_\fX(x)$} for it spectrum\index{spectrum} as in 
Definition~\ref{def:spectrum}. This is a closed, bounded subset of
$\bC$ by Theorem~\ref{th:Gelfand}. For $x\in\fa$, we write
$\Spec_\fX(x)$ for the spectrum of the image of $x$ in $\hat\fa_\fX$.
\end{defn}

\begin{theorem}\label{th:npj-in-spec}
If $x\in\fa_{\cge 0}$ then the spectral radius $\npj_\fX(x)$ is
an element of $\Spec_\fX(x)$.
\end{theorem}
\begin{proof}
By Theorem~\ref{th:a+bx}, the spectral radius $\npj_\fX(\one+x)$ is 
equal to $1+\npj_\fX(x)$. The theorem now follows from 
Proposition~\ref{pr:Spec(1+x)}.
\end{proof}

\begin{rk}
This theorem is a special case of a theorem about the
spectrum of a positive operator on a Banach lattice.\index{Banach!lattice}
For a more direct proof in that context, see for example Theorem~7.9 
of Abramovich and Aliprantis~\cite{Abramovich/Aliprantis:2002a};
but in fact our approach through Proposition~\ref{pr:Spec(1+x)} also works in
this generality.
\end{rk}

Recall the definition of species from Definition~\ref{def:species}.

\begin{defn}
Let $\fa$ be a representation ring and $\fa_\bC=\bC\otimes_\bZ\fa$.
A \emph{species}\index{species|textbf} of
$\fa$ is a ring homomorphism $s\colon\fa\to\bC$. A species of $\fa$
extends uniquely to give a $\bC$-algebra homomorphism $s\colon
\fa_\bC\to\bC$, which we call a species of $\fa_\bC$.

Let $\fX$ be a representation ideal in $\fa$.
We say that a species $s$ of $\fa$ is \emph{$\fX$-core bounded}%
\index{bounded!species}%
\index{core!bounded species}%
\index{species!x@$\fX$-core bounded ---}%
\index{X@$\fX$-core!bounded species} 
if for all $x\in\fa_{\cge 0}$ we have
\[ |s(x)| \le \dim\core_\fX(x). \]
Of course, this only needs checking on the basis elements $x_i$, $i\in
\fI$.

If $\fX=\varnothing$ is the empty representation ideal, we have
$\core_\varnothing(x)=x$, and we call
a $\varnothing$-core bounded species a
\emph{dimension bounded species}.%
\index{dimension!bounded species|textbf}%
\index{species!dimension bounded|textbf}
Thus a species $s$ is dimension bounded if and only if for all
$i\in\fI$ we have
\[ |s(x_i)|\le \dim(x_i). \]
So $s$ is $\fX$-core bounded if and only if it is dimension bounded
and vanishes on $x_i$ for $i\in \fX$.

If $\fX=\fX_\proj$, we call an $\fX_\proj$-core bounded species a
\emph{core bounded species}.\index{species!core bounded|textbf}
\end{defn}

\begin{theorem}\label{th:core-bounded}
For a species $s\colon \fa_\bC\to\bC$, the following are equivalent:
\begin{enumerate}
\item $s$ is $\fX$-core bounded.
\item For all $x\in \fa_\bC$ we have $|s(x)|\le \|x\|_\fX$.
\item $s$ vanishes on every $x_i$ with $i\in\fX$ and 
is continuous with respect to the norm on $\fa_\bC/\langle\fX\rangle_\bC$.
\item $s$ vanishes on $\langle\fX\rangle_\bC$ and extends to an
  algebra homomorphism
  $\hat\fa_\fX\to\bC$.
\end{enumerate}
\end{theorem}
\begin{proof}
The implications (ii) $\Rightarrow$ (iii) $\Rightarrow$ (iv) are
clear, and the implication (iv) $\Rightarrow$ (i) follows from
Proposition~\ref{pr:continuous}. So it remains to prove that
(i) $\Rightarrow$ (ii). Suppose that $s$ is core-bounded, and 
write $x=\sum_{i\in\fI}a_ix_i$. Then $s(x)=\sum_{i\in\fI}a_is(x_i)$ and so
\begin{equation*}
|s(x)|\le \sum_{i\in\fI}|a_i| |s(x_i)| \le \sum_i|a_i|\dim\core_\fX(x_i) =\|x\|_\fX.
\qedhere
\end{equation*}
\end{proof}

\begin{cor}
A species $s\colon \fa_\bC\to\bC$ is dimension bounded if and only if
it is continuous with respect to the norm on $\fa_\bC$.\qed
\end{cor}

\begin{theorem}\label{th:spec-radius-fa}
For $x\in \fa$, the spectrum $\Spec_\fX(x)$ is the set of
values of $s(x)$ as $x$ runs over the $\fX$-core bounded species of
$\fa$. The spectral radius is
\[ \npj_\fX(x)=\max_{\substack{s\colon\fa\to\bC\\ \fX\text{\rm -core bounded}}}|s(x)|. \]
There is an $\fX$-core bounded species $s$ with $\npj_\fX(x)=s(x)$.
\end{theorem}
\begin{proof}
This follows from Corollary~\ref{co:radius} and
Theorems~\ref{th:spec-radius}, \ref{th:npj-in-spec}
and~\ref{th:core-bounded}.
\end{proof}

\begin{defn}
A species $s$ of $\fa$ is said to be a \emph{Brauer species}%
\index{Brauer species|textbf}\index{species!Brauer|textbf} if
$s(x_i)\ne 0$ for some projective basis element $x_i$ (see Definition~\ref{def:proj}).
\end{defn}

\begin{eg}
The dimension function is always a Brauer species. This is because it
is a ring homomorphism, and is non-zero on $\rho$ and therefore on
some projective basis element.
\end{eg}

\begin{prop}\label{pr:Brauer-species}
A Brauer species is determined by its value on the projective basis
elements. The Brauer species form a finite set, whose cardinality is
at most the number of projective basis elements.
Every dimension bounded species\index{dimension!bounded species}%
\index{species!dimension bounded}
is either a Brauer species or a core
bounded species,\index{core!bounded species}%
\index{species!core bounded} but not both.
\end{prop}
\begin{proof}
By Lemma~\ref{le:extend}, a Brauer species is determined by its value
on the projective basis elements.
It therefore cannot be core bounded, because such species vanish on
projective basis elements. 
Since distinct
species are linearly independent, it follows that the number of Brauer
species is at most the number of projective basis elements.

On the other hand, if $s$ is not a Brauer species then it vanishes on
the ideal $\langle\fX_\proj\rangle$ of projectives, and is therefore
core bounded by Theorem~\ref{th:core-bounded}.
\end{proof}

\begin{rk}
Example~\ref{eg:repring}\,(iii) shows that the number of Brauer
species can be strictly less than the number of projective basis
elements. In this example, there are two projective basis elements but
only one Brauer species, namely the dimension function.
\end{rk}

Similarly, we have the following.

\begin{prop}
Let $\fX$ be a representation ideal in $\fa$. Then the dimension
bounded species of $\fa$ fall into two disjoint subsets, the
$\fX$-core bounded species, and the species which take non-zero value
on some $x_i$ with $i\in\fX$. The latter are determined by their values on
the elements $x_i$, $i\in\fX$.
\end{prop}
\begin{proof}
The proof is essentially the same as the proof of
Proposition~\ref{pr:Brauer-species}.
\end{proof}

\begin{defn}
We write $\Struct_\fX(\fa)$ for the structure
space\index{structure!space} $\Struct(\hat\fa_\fX)$ of the commutative 
Banach $*$-algebra $\hat\fa_\fX$. In the case $\fX=\fX_{\max}$, we
write $\Struct_{\max}(\fa)$, and in the case $\fX=\fX_\proj$ we write $\Struct_\proj(\fa)$.
\end{defn}

\begin{theorem}
Let $x\in\fa_{\cge0}$ and $\fX$ be a representation ideal in $\fa$.
The structure space\index{structure!space} $\Struct_\fX(\fa)$ may be
identified with the set of $\fX$-core bounded species of $\fa$, or
equivalently of $\fa_\bC$, with the 
weak* topology. It is a compact Hausdorff space.
\end{theorem}
\begin{proof}
Using Theorem~\ref{th:core-bounded}, we identify the species of $\hat
\fa_\fX$ with the $\fX$-core bounded species of $\fa$, or equivalently
of $\fa_\bC$. Then the theorem follows from Proposition~\ref{pr:weak*}.
\end{proof}

\section{Symmetric representation rings}\label{se:symm-rep-ring}%
\index{symmetric!representation ring|textbf}%
\index{representation!ring!symmetric|textbf}

\begin{defn}
We say that a representation ring $\fa$ is \emph{symmetric} if the
completion $\hat\fa$ of $\fa_\bC$ is a 
symmetric Banach $*$-algebra,\index{symmetric!Banach $*$-algebra}%
\index{Banach!star@$*$-algebra!symmetric} 
see Definition~\ref{def:symmetric}. 
\end{defn}

\begin{prop}\label{pr:symm-rep-ring}
For a representation ring $\fa$, the following are equivalent.
\begin{enumerate}
\item $\fa$ is symmetric.
\item Every dimension bounded 
species\index{dimension!bounded species}%
\index{species!dimension bounded}
of $\fa$ is self conjugate. In
  other words, for every dimension bounded species $s\colon \fa\to
  \bC$ and every basis element $x_i$, we have $s(x_i^*) =
  \overline{s(x_i)}$.
\item For every $x\in\hat\fa$, the spectral radius $\rho$ satisfies $\rho(x^*x)=\rho(x)^2$.
\end{enumerate}
If these hold then for every $x\in\fa_{\cge 0}$ and every
representation ideal $\fX$ of $\fa$ we have 
\[ \npj_\fX(xx^*)=\npj_\fX(x)^2. \]
\end{prop}
\begin{proof}
The equivalence of (i), (ii) and (iii) follows from the equivalence of (i), (iii) and (v) in 
Proposition~\ref{pr:symmetric}. By Theorem~\ref{th:spec-radius},
$\npj_\fX(x)$ is the spectral radius of the image of $x$ in
$\hat\fa_\fX$, so using Corollary~\ref{co:symm-quot}, 
condition (iii) implies the last statement.
\end{proof}

It would be advantageous to have a condition for symmetry that is
easier to check. For example, we do not know whether the modular
representation ring of a finite group is symmetric. At least we have
the following.

\begin{eg}\label{eg:ordinary=>symmetric}
Theorem~\ref{th:ordinary=>symmetric} implies that an ordinary
representation ring\index{ordinary representation!ring}%
\index{representation!ring!ordinary} is symmetric.
\end{eg}

\section{Algebraic elements revisited}

\begin{theorem}\label{th:PF}
Let $\fa$ be a representation ring,
let $x\in\fa_{\cge 0}$ be algebraic modulo $\fX_{\max}$, and let
$\fa'$ be a representation subring of $\fa$ as described in 
Lemma~\ref{le:algebraic}\,(v). Then there is a unique species
\[ s\colon \fa'_{\max}=\fa'/\langle\fX_{\max}\rangle\to\bC \] 
such that for $x\in\fa'_{\max,\cge 0}$, we have $\npj_{\max}(x)=s(x)$.
\end{theorem}
\begin{proof}
Let $\fJ\subseteq\fI\setminus\fX_{\max}$ be the set of basis elements
of $\fa'$ not in $\fX_{\max}$, and set $y=\sum_{j\in\fJ}x_j$.
Let $B$ be the matrix $(b_{i,j})$ where
$yx_i\equiv\sum_{j\in\fJ}b_{i,j}x_j\pmod{\langle\fX_{\max}\rangle}$
for $i\in\fJ$. We claim that the entries of $B$ are strictly
positive. To see this, for $i,j\in\fJ$ we have 
\[ [x_{i^*}x_jx_i:x_j]\ge [x_{i^*}x_i:\one]>0. \] 
So there exists $k\in\fJ$ with
$[x_{i^*}x_j:x_k]>0$ and $[x_kx_i:x_j]>0$, and hence 
$b_{i,j}=[yx_i:x_j]>0$. 

It follows from the fact that the entries of $B$ are strictly
positive, that we may apply the 
Perron--Frobenius theorem.\index{Perron--Frobenius theorem} This says
that there is an eigenvalue $\lambda$ of $B$ which is positive real,
and larger in absolute value than all the other eigenvalues, and that
the eigenspace is one dimensional, spanned by a vector 
$u=\sum_{i\in\fJ}a_ix_i$ with all the $a_i$ positive real. We thus
have $yu=\lambda u$. 

Given $x\in\fa'_{\max}$, we have $yxu=xyu=\lambda xu$, and so
$xu$ is another non-zero eigenvector of multiplication by $y$. Since
the eigenspace is one dimensional, we have $xu=s(x)u$ for some
$s(x)\in\bC$. It is now easy to check that $s$ is a ring homomorphism
from $\fa'_{\max}$ to $\bC$. 

For $x\in\fa'_{\max,\cge 0}$, $s(x)$ is the spectral radius
of $x$ as an element of $\fa'_{\max}$. By the spectral radius formula,
Proposition~\ref{pr:Gelfand}, the spectral radius of $x$ regarded as an element of
$\fa_{\max}$ is the same as its spectral radius regarded as an element
of $\fa'_{\max}$. It now follows from Theorem~\ref{th:spec-radius} that
$s(x)=\npj_{\max}(x)$. 
\end{proof}

\section{Quasi-nilpotent elements}\index{quasi-nilpotent element}

We suppose, for the purposes of the next theorem,  
that we are in the following situation. We are given a
commutative, associative $\bC$-algebra $A$ with a 
vector space basis $\{x_i,\ i \in \fJ\}$ satisfying $x_ix_j=\sum_k
c_{i,j,k}x_k$ where the structure constants $c_{i,j,k}$ are
non-negative 
integers. We are also given an algebra homomorphism 
$d\colon A \to \bC$ such that each $d(x_i)$ is a positive integer.
We put a norm on $A$ by setting
\[ \|\sum_ia_ix_i\| = \sum_i |a_i|d(x_i). \]
As in Lemma~\ref{le:Anormed}, this does indeed define a norm.
Under these circumstances, the following theorem shows that
quasi-nilpotent elements are nilpotent. I would like to thank Pavel
Etingof for suggesting this method of proof.

\begin{theorem}\label{th:qn=>n}
If $a=\sum_i a_ix_i\in A$ satisfies 
$\sqrt[n]{\|a^n\|}\to 0$ as $n\to\infty$ then $a$ is nilpotent.
\end{theorem}
\begin{proof}
We suppose that $a$ is not nilpotent, and obtain a contradiction.
Write $a=a'+\bi a''$ in such a way that the coefficients of $x_i$ in $a'$, $a''$ are
real. Then the element
$\bar a=a'-\bi a''$ also satisfies the hypothesis. By
Lemma~\ref{le:binom}, elements satisfying the hypothesis form a linear
subspace of $A$. 
It follows that $a'$ and $a''$ also
satisfy the hypothesis. 
Furthermore, if $a'$ and $a''$ are both nilpotent, then so is $a$.
Therefore we may assume without loss of generality that 
the coefficients $a_i$ of $a$ are real.

Since $a=\sum_i a_ix_i$ is a finite sum, we let $V$ be the real linear span
in $A$ of those $x_i$ with $a_i\ne 0$. Then $V$ is a finite
dimensional $\bR$-vector subspace of $A$. Consider the elements of $V$ of the
form $b=\sum_i b_i x_i$ with $b_i\in \bZ$. These form a lattice $\Lambda$ in $V$.
It follows that there is a constant $C$ such that 
given any element $v$ of $V$ there is an element of
$\Lambda$ at distance at most $C$ from $v$. In particular, given
$q\in\bZ$ there is an element $b$ of $\Lambda$ within distance at most
$C$ from $qa$.
So looking at the elements
of the form $b-qa$ with $q\in\bZ$, $b\in\Lambda$, there have to be two such,
at an arbitrarily small distance from each other. Taking the
difference, we see that given $\ep>0$ we can choose $q>0$ and $b$ with $\|b-qa\|<\ep$,
and hence $\|\frac{b}{q}-a\|<\frac{\ep}{q}$.

The nilpotent elements of $V$ form a linear subspace, which is
therefore a closed subset in the norm topology. Since $a$ is not
nilpotent, we may choose $b$ as above such that
$\frac{b}{q}$ is also not nilpotent. So for all $n>0$ we have
$\|b^n\|\ge 1$ (because $b$ has integer coefficients) and hence
$\|(\frac{b}{q})^n\|\ge \frac{1}{q^n}$. 

Now we have
\begin{align*}
\textstyle\frac{1}{q^n}&\textstyle\le \|\left(\frac{b}{q}\right)^n\| 
=\textstyle\|(a+(\frac{b}{q}-a))^n\| 
=\textstyle\left\|\sum_{i=0}^n\binom{n}{i}a^i(\frac{b}{q}-a)^{n-i}\right\| \\
&\le\textstyle\sum_{i=0}^n\binom{n}{i}\|a^i\|\ \|(\frac{b}{q}-a)^{n-i}\|.
\end{align*}
Applying Lemma~\ref{le:binom}, we deduce that
\[ {\textstyle\frac{1}{q}}=\limsup_{n\to\infty}\sqrt[n]{\textstyle\frac{1}{q^n}} \le
\limsup_{n\to \infty}\sqrt[n]{\|a^n\|} + 
\limsup_{n\to\infty}\sqrt[n]{\textstyle\|(\frac{b}{q}-a)^n\|}<0+\textstyle\frac{\ep}{q}=\frac{\ep}{q}. \]
So choosing $\ep\le 1$, we obtain a contradiction. Hence $a$ is nilpotent.
\end{proof}

We apply this theorem to characterise the nilpotent elements in a representation
ring in terms of species. Recall that the 
\emph{nil radical}\index{nil radical} of a commutative
ring is the ideal of nilpotent elements.

\begin{theorem}\label{th:J=nil}
  Let $\fa$ be a representation ring.
An element $x\in \fa_\bC= \bC\otimes_\bZ\fa$\index{a@$\fa_{\bC}$} 
is nilpotent if and only if for
every dimension bounded species $s\colon \fa_\bC\to \bC$ we have
$s(x)=0$.
Thus the Jacobson radical of $\fa_\bC$ is equal to the nil radical.
\end{theorem}
\begin{proof}
  We embed  $\fa_\bC$ in $\hat \fa$. The 
  species of $\hat\fa$ are the dimension bounded species, so by
  Theorem~\ref{th:qn}
  the intersection of the kernels of the dimension bounded species
  is the set of quasi-nilpotent elements of $\hat\fa$. It follows from
Theorem~\ref{th:qn=>n} that a
  quasi-nilpotent element $x\in\fa_\bC$ is nilpotent.
\end{proof}

\begin{rk}
  Since $\fa_\bC$ is not necessarily Noetherian, we cannot conclude from the theorem
  that if $\fm$ is a maximal ideal of $\fa_\bC$ then $\fa_\bC/\fm\cong \bC$.
  For example, letting $R$ be the polynomial algebra
  $\bC[v,\{u_\lambda\}_{\lambda\in\bC}]$,
given any non-zero element $x\in R$, there is an algebra homomorphism
$s\colon R \to \bC$ such that $s(x)\ne 0$. On the other hand, there is
a surjective algebra homomorphism $\phi\colon R \to \bC(t)$ sending $v$ to $t$
and $u_\lambda$ to $(t-\lambda)^{-1}$, and $R/\Ker(\phi)\cong \bC(t)$.
\end{rk}

The following generalises Theorem~2.7 of \cite{Benson/Carlson:1986a},
with essentially the same proof.

\begin{theorem}\label{th:no-n}
  There are no non-zero nilpotent elements in $\hat\fa_{\max}$.
\end{theorem}
\begin{proof}
  Suppose that $x\in\hat\fa_{\max}$ is nilpotent, and write
  $x=\sum_{i\in\fX\setminus\fX_{\max}}a_ix_i$. Let $n_i = [x_ix_{i^*}:\one]>0$.
  Then $x^* = \sum_i \bar a_i x_i^*$. We have
\[ xx^* = \sum_i |a_i|^2x_i x_i^* + \sum_{i\ne j}a_i\bar a_jx_i x_j^*, \]
and the coefficient of $\one$ in this is $\sum_i n_i|a_i|^2$. This is zero if and only if
$x=0$, so $xx^*=0$ implies $x=0$. 
If $x^2=0$ then $(xx^*)(xx^*)^* = x^2x^{*\,2}=0$, so $xx^*=0$ and hence $x=0$.
\end{proof}

\begin{rk}
We shall use the proof of Theorem~\ref{th:no-n} as motivation for the
introduction of the trace map. This will eventually enable us to prove
that there are no quasi-nilpotent elements in $\hat\fa_{\max}$, see
Theorem~\ref{th:no-qn}.
\end{rk}

\section{Action on Hilbert space}

Recall that for a representation ring $\fa$, we have the completion
$\hat\fa_{\max}$ with respect to the norm $\|\ \|_{\max}$.
In this section we investigate an action of $\hat\fa_{\max}$ on a
Hilbert space $H(\fa)$. The crucial inequality allowing us to do this is given
in Theorem~\ref{th:|xy|}, which turns out to be quite tricky to prove.
We begin with the following definitions, which are suggested by
the proof of Theorem~\ref{th:no-n}.

\begin{defn}
The \emph{trace map}\index{trace map} $\Tr\colon \fa \to \bZ$ is defined by
\[ \Tr\left(\sum_{i\in\fI} a_ix_i\right) = a_0 \] 
(recall that $x_0$ is the basis element $\one$ of $\fa$). This extends
to a trace map $\Tr\colon \fa_\bC\to\bC$ given by the same formula.
\end{defn}

Recall that $n_i$ is defined to be $[x_ix_{i^*}:\one]$, and that
$n_i=0$ if and only if $i\in\fX_{\max}$.

\begin{lemma}\label{le:Tr}
We have $\Tr(x_ix_{i^*})=n_i$, and $\Tr(x_ix_{j^*})=0$ for $i\ne j$.
\end{lemma}
\begin{proof}
This follows from Definition~\ref{def:repring}\,(ii).
\end{proof}

\begin{defn}\label{def:weighted-ell2}
We define the \emph{weighted $\ell^2$ norm}%
\index{weighted!$\ell^2$ norm}\index{norm!weighted $\ell^2$} 
on $\fa_{\bC,\max}=\fa_\bC/\langle\fX_{\max}\rangle_\bC$%
\index{X@$\fX_{\max}$} to be
\[  \left|\sum_{i\in\fI}a_ix_i\right|=\sqrt{
\sum_{i\in\fI}n_i|a_i|^2 }
= \sqrt{\sum_{i\in\fI\setminus\fX_{\max}}n_i|a_i|^2}. \]
This is associated to the inner product
\[ \left\langle
    \sum_{i\in\fI}a_ix_i,\,\sum_{i\in\fI}b_ix_i\right\rangle =
\sum_{i\in\fI}n_ia_i\bar b_i =
  \sum_{i\in\fI\setminus\fX_{\max}} n_i a_i\bar b_i. \]
The completion of $\fa_{\bC,\max}$ with respect to the
weighted $\ell^2$ norm is a Hilbert space which
is denoted $H(\fa)$.\index{H@$H(\fa)$}
We can think of elements of $H(\fa)$ as countably supported infinite
sums $\sum_{i\in\fI\setminus\fX_{\max}}a_ix_i$, with
$\sum_{i\in\fI\setminus\fX_{\max}}n_i|a_i|^2<\infty$, and with trace and inner product
given by the same formulas as above.
\end{defn}\pagebreak[3]

\begin{lemma}\label{le:H}
For $x,y\in H(\fa)$ we have
\begin{enumerate}
\item $\langle x,y\rangle = \Tr(xy^*)$
\item $\langle xy,z\rangle = \langle y,x^*z\rangle$
\end{enumerate}
\end{lemma}
\begin{proof}
(i) If $x=\sum_{i\in\fI\setminus\fX_{\max}}a_ix_i$ and 
$y=\sum_{i\in\fI\setminus\fX_{\max}}b_ix_i$ then
\[ xy^*=\sum_{i\in\fI\setminus\fX_{\max}}a_i\bar b_i x_ix_{i^*} + \sum_{i\ne
    j\in\fI\setminus\fX_{\max}}a_i\bar b_j x_ix_{j^*}. \]
By Lemma~\ref{le:Tr}, the trace of this is equal to
$\sum_{i\in\fI\setminus\fX_{\max}}n_ia_i\bar b_i$, which is $\langle
x,y\rangle$.

(ii) Using (i) we have 
\begin{equation*} 
\langle xy,z\rangle = \Tr(xyz^*)=\Tr(y(x^*z)^*)=\langle
  y,x^*z\rangle. 
\qedhere
\end{equation*}
\end{proof}

\begin{theorem}\label{th:|xy|}
For $x\in\fa_{\bC}$, $y\in H(\fa)$ we have $|xy|\le \|x\|_{\max}|y|$.
\end{theorem}
\begin{proof}
The model for this proof is the inequality in Lemma~\ref{le:l2}, but
the details are much more complicated.

First we treat the case $x=x_i$ and $y=\sum_{j\in\fI}b_jx_j$.
We have
\begin{align*} 
|x_iy|^2&=\langle x_iy,x_iy\rangle \\
&=\langle x_{i^*}x_iy,y\rangle \\
&=\sum_{j,k\in\fI\setminus\fX_{\max}}b_j\bar b_k \langle x_{i^*}x_ix_j,x_k\rangle\\
&\le \sum_{j,k\in\fI\setminus\fX_{\max}}|b_j||b_k|\langle x_{i^*}x_ix_j,x_k\rangle\\
&=\sum_{j,k\in\fI\setminus\fX_{\max}}
\frac{|b_j|n_j}{\dim x_j}\,\frac{|b_k|n_k}{\dim x_k}\,
\langle x_{i^*}x_ix_j,x_k\rangle\,\frac{\dim x_j}{n_j}\,\frac{\dim x_k}{n_k}
\end{align*}
For $j\ne k$, the $(j,k)$ term and the $(k,j)$ term in this sum are equal. So
using the inequality
\[ 2\,\frac{|b_j|n_j}{\dim x_j}\,\frac{|b_k|n_k}{\dim x_k} \le
\left(\frac{|b_j|n_j}{\dim x_j}\right)^2+
\left(\frac{|b_k|n_k}{\dim x_k}\right)^2   \]
we have
\begin{align*}
|x_iy|^2&\le \sum_{j,k\in\fI\setminus\fX_{\max}}\left(\frac{|b_j|n_j}{\dim x_j}\right)^2\,
\langle x_{i^*}x_ix_j,x_k\rangle\,\frac{\dim x_j}{n_j}\,\frac{\dim x_k}{n_k} \\
&=\sum_{j,k\in\fI\setminus\fX_{\max}}\frac{|b_j|^2n_j}{\dim x_j}\,
\langle x_{i^*}x_ix_j,x_k\rangle\,\frac{\dim x_k}{n_k}.
\end{align*}
Now we also have
\begin{align*} 
\sum_{k\in\fI\setminus\fX_{\max}}\langle x_{i^*}x_ix_j,x_k\rangle\,\frac{\dim x_k}{n_k} 
&= \sum_{k\in\fI\setminus\fX_{\max}}[x_{i^*}x_ix_j:x_k]n_k\,\frac{\dim x_k}{n_k}\\
&= \sum_{k\in\fI\setminus\fX_{\max}}[x_{i^*}x_ix_j:x_k]\dim x_k\\
&=\dim\core_{\max}(x_{i^*}x_ix_j) \\
&\le(\dim\core_{\max}(x_i))^2\dim x_j
\end{align*}
and so we get
\begin{align*}
|x_iy|^2&\le \sum_{j\in\fI\setminus\fX_{\max}}
\frac{|b_j|^2n_j}{\dim x_j}\,(\dim\core_{\max}(x_i))^2\dim x_j\\
&=\sum_{j\in\fI\setminus\fX_{\max}}|b_j|^2n_j(\dim\core_{\max}(x_i))^2\\
&=(\dim\core_{\max}(x_i))^2\langle y,y\rangle\\
&=\|x_i\|_{\max}^2|y|^2.
\end{align*}
Taking square roots of both sides, we obtain
$|x_iy|\le\|x_i\|_{\max}|y|$.

Finally, in general if $x=\sum_{i\in\fI} a_ix_i$ then using the case proved
above, we have
\begin{align*} 
|xy|&=\left|\sum_{i\in\fI}a_ix_iy\right|
\le \sum_{i\in\fI}|a_i||x_iy|
\le \sum_{i\in\fI}|a_i|\|x_i\|_{\max}|y|\\
&\quad=\left\|\sum_{i\in\fI}a_ix_i\right\|_{\max}|y|
=\|x\|_{\max}|y|. 
\qedhere
\end{align*}
\end{proof}

\begin{rk}\label{rk:|x|}
Setting $y=\one$ in the theorem, we have $|x|\le \|x\|_{\max}$.
In particular, every infinite sum
$\sum_{i\in\fI\setminus\fX_{\max}}a_ix_i$ in $\hat\fa_{\max}$ is in
$H(\fa)$. So we have a norm decreasing (continuous) injective map of 
Banach $*$-algebras $\hat\fa_{\max}\hookrightarrow H(\fa)$,
with dense image.
\end{rk}

\begin{prop}
For $x\in\fa_\bC$
the map $y\mapsto xy$ of $H(\fa)$ is
bounded. Elements of $\langle\fX_{\max}\rangle$ act as zero, 
and so we have a map $\fa_{\bC,\max} \to
\Lin(H(\fa))$.
\end{prop}
\begin{proof}
This follows from the inequality in Theorem~\ref{th:|xy|}.
\end{proof}

\begin{defn}
We write $\|x\|_{\sup}$ for the sup norm of $x$ under the map
\[ \fa_{\bC,\max} \to \Lin(H(\fa))\index{L@$\Lin(H(\fa))$} \]
sending $x\in\fa_\bC$ to the map $y \mapsto xy$ of
$H(\fa)$. This is given by
\[ \|x\|_{\sup} = \sup_{|y|=1}|xy|. \]
\end{defn}

\begin{theorem}\label{th:fa-in-Lin}
For $x\in\fa_{\bC,\max}$ we have $|x|\le\|x\|_{\sup}\le \|x\|_{\max}$. The map
$\fa_{\bC,\max}\to \Lin(H(\fa))$ is a
continuous $*$-homomorphism of normed $*$-algebras, 
and extends to an injective continuous 
$*$-homomorphism of Banach $*$-algebras $\hat\fa_{\max}\to \Lin(H(\fa))$.
\end{theorem}
\begin{proof}
Theorem~\ref{th:|xy|} shows
that $\|x\|_{\sup}\le \|x\|_{\max}$,
so by Lemma~\ref{le:bd=cts} this map
is continuous, and hence extends to a continuous map $\hat\fa_{\max}\to
\Lin(H(\fa))$. 
By Lemma~\ref{le:H}\,(ii), this map preserves the star operation.
The action of $x$ on $\one\in H(\fa)$ shows that $|x|\le \|x\|_{\sup}$,
and that this map is injective, see Remark~\ref{rk:|x|}.
\end{proof}

\section{\texorpdfstring{The $C^*$-algebra $C^*_{\max}(\fa)$}
{The C*-algebra C*max(𝔞)}}

\begin{defn}\label{def:Cmaxfa}
We saw in Theorem~\ref{th:fa-in-Lin} that the map
$\fa_{\bC,\max}\to \Lin(H(\fa))$ sending an element $x$ to left
multiplication by $x$ is a continuous map
of normed $*$-algebras, and extends to an injective map of
Banach $*$-algebras $\hat\fa_{\max}\to\Lin(H(\fa))$.
We let $C^*_{\max}(\fa)$\index{C@$C^*_{\max}(\fa)$} denote
the closure of the image of this map. This is
a commutative $C^*$-algebra.
\end{defn}

\begin{defn}
A species $s\colon \fa\to \bC$ is 
\emph{sup bounded}\index{sup!bounded species|textbf}%
\index{species!sup bounded|textbf} if for
all $x\in\fa_{\cge 0}$ we have 
\[ |s(x)|\le\|x\|_{\sup}=\sup_{|y|=1}|xy|. \]
\end{defn}

\begin{prop}\label{pr:sup-bounded}
For an $\fX_{\max}$-core bounded species of $\fa$, the following are
equivalent:
\begin{enumerate}
\item $s$ is sup bounded,
\item $s$ is continuous with respect to the sup norm,
\item $s$ extends to a $\bC$-algebra homomorphism
  $C^*_{\max}(\fa)\to \bC$.
\end{enumerate}
\end{prop}

This leads us to another invariant of an element $x\in\fa_{\cge 0}$,
namely the spectral radius in the sup norm. By
Theorem~\ref{th:C*radius}, this is equal to $\|x\|_{\sup}$:\index{gamma@$\npj_{\sup}(x)$}
\[ \npj_{\sup}(x) = \|x\|_{\sup}=\sup_{|y|=1}|xy|. \]
This is really an invariant of the image of $x$ in 
$\fa_{\max}$, and we have
\[ \npj_{\sup}(x) \le \npj_{\max}(x). \]
This invariant has one advantage over the others we have been
examining. Namely, 

\begin{theorem}
For $x\in \fa_{\cge 0}$ we have
\[ \npj_{\sup}(xx^*)=\npj_{\sup}(x)^2. \]
\end{theorem}
\begin{proof}
By Theorem~\ref{th:LinC*}, $\Lin(H(\fa))$ is a $C^*$-algebra with 
respect to the sup norm. In particular, in accordance with 
Definition~\ref{def:Cstar-algebra}, we have $\|xx^*\|_{\sup}=\|x\|^2_{\sup}$.
\end{proof}

\begin{rk}
The $C^*$-algebra $C^*_{\max}(\fa)$ is analogous to the $C^*$-algebra $C^*(\Gamma)$
of a discrete abelian group $\Gamma$. 

Recall that in general, if $\Gamma$ is a
discrete group and $x\in\ell^1(\Gamma)$ then we look at all actions of
$\ell^1(\Gamma)$ on a Hilbert space, and take the \emph{$C^*$-norm}
$\|x\|_{C^*}$ to be
the supremum of the sup norms in these actions. We have
$\|x\|_{C^*}\le \|x\|_{\ell^1}$, and we define the $C^*$-algebra
 $C^*(\Gamma)$ of $\Gamma$ to be the
completion of $\ell^1(\Gamma)$ with respect to this norm. 

If we restrict our attention to the action of $\ell^1(\Gamma)$ on the
Hilbert space $\ell^2(\Gamma)$ by convolution, rather than
using all actions on Hilbert spaces, then we obtain a possibly smaller
norm called the \emph{reduced $C^*$-norm},\index{reduced!$C^*$-norm}
\[ \|x\|_{C^*_r} = \sup_{|y|=1}|x*y|, \] 
and the
\emph{reduced $C^*$-algebra}\index{reduced!$C^*$-algebra} 
$C^*_r(\Gamma)$ is the completion of
$\ell^1(\Gamma)$ with respect to this norm.

The group $\Gamma$ is said to be 
\emph{amenable}\index{amenable group}\index{group!amenable} 
if there exists a 
finitely additive measure\index{finitely!additive measure} 
on $\Gamma$ which is invariant under left
multiplications. If $\Gamma$ is amenable then the $C^*$-norm and
the reduced $C^*$-norm are equal, and we have
$C^*(\Gamma)=C^*_r(\Gamma)$. In particular, all abelian groups are
amenable, and so if $\Gamma$ is abelian then $C^*(\Gamma)$ is the
closure of the image of the action of $\ell^1(\Gamma)$ on $\ell^2(\Gamma)$.
\end{rk}

\section{Quasi-nilpotent elements revisited}\label{se:qn}%
\index{quasi-nilpotent element}

\begin{theorem}\label{th:no-qn}
There are no non-zero quasi-nilpotent elements in $\hat\fa_{\max}$.
\end{theorem}
\begin{proof}
By Theorem~\ref{th:fa-in-Lin} and Definition~\ref{def:Cmaxfa}, 
we have a continuous injective map
from $\hat\fa_{\max}$ to $\Lin(H(\fa))$, and the closure of its image
is the commutative $C^*$-algebra $C^*_{\max}(\fa)$. If $x$ is a quasi-nilpotent
element of $\hat\fa_{\max}$ then its image in $C^*_{\max}(\fa)$ is also
quasi-nilpotent and hence, by Corollary~\ref{co:J(C)}, equal to zero.
\end{proof}

\begin{cor}
If $x$ is a non-zero element of $\hat\fa_{\max}$ then there is a sup
bounded\index{sup!bounded species} (and hence 
$\fX_{\max}$-core\index{X@$\fX_{\max}$} bounded)
species $s$ with $s(x)\ne 0$.
\end{cor}
\begin{proof}
This follows from Theorem~\ref{th:qn},
Proposition~\ref{pr:sup-bounded}, and Theorem \ref{th:no-qn}.
\end{proof}

\begin{cor}
\begin{enumerate}
\item The Jacobson radical of $\hat\fa_{\max}$ is zero.
\item The Jacobson radical of $\fa_\bC/\langle\fX_{\max}\rangle$ is zero.
\end{enumerate}   
\end{cor}
\begin{proof}
(i) By the previous corollary, every non-zero element $x$ lies outside some
maximal ideal $I$ of $\hat\fa_{\max}$ with $\hat\fa_{\max}/I$
isomorphic to $\bC$
via an $\fX_{\max}$-core bounded species $s$.

(ii) If $x$ is a non-zero element of 
$\fa_\bC/\langle\fX_{\max}\rangle\subseteq\hat\fa_{\max}$ then
the species of part (i), $s\colon \fa_\bC/\langle\fX_{\max}\rangle \to \bC$ 
is surjective, and $x$ is
not in its kernel. So $x\not\in J(\fa_\bC/\langle\fX_{\max}\rangle)$.
\end{proof}

\section{Idempotents}\label{se:idempotents}

In this section we examine idempotents in $\hat\fa_{\max}$ and
in $\fa_{K,\max}=K\otimes_\bZ \fa_{\max}$%
\index{a@$\fa_{K,\max}=K\otimes_\bZ \fa_{\max}$}
with $K$ a field. We show that if
$e\in\hat\fa_{\max}$ is idempotent and not equal to zero or one 
then $0<\Tr(e)<1$; and if $e\in\fa_{K,\max}$ then $\Tr(e)$ is
in the ground field of $K$.

The following is the analogue for representation rings 
of a theorem of Kaplansky on group rings (unpublished,
but see the end of \S II.3 of Kaplansky~\cite{Kaplansky:1972a},
Lemma~2 of Montgomery~\cite{Montgomery:1969a}, or \S 2.1 of
Passman~\cite{Passman:1977a}).

\begin{theorem}\label{th:Tr(e)}
If $e\in\hat\fa_{\max}$ is idempotent,\index{idempotent} 
$e\ne 0,1$ then $0<\Tr(e)<1$.
\end{theorem}
\begin{proof}
We have $e\in \hat\fa_{\max}\subseteq C^*_{\max}(\fa)$.
Now $C^*_{\max}(\fa)$ is a $C^*$-algebra, 
and hence a symmetric Banach $*$-algebra. 
By Theorem~\ref{th:e=e*} we have $e=e^*$ and so $e=e^*e$.
Now using Lemma~\ref{le:H}, we have
\[ \Tr(e)=\Tr(e^*e) = \langle e,e\rangle >0 .\]
Since $1-e$ is also an idempotent, we have $0<\Tr(e)<1$.
\end{proof}

\begin{cor}
There are no non-trivial idempotents in
$\fa_{\max}$.
\end{cor}
\begin{proof}
If $e\in \fa_{\max}$ then $\Tr(e)$ is a rational
integer. By Theorem~\ref{th:Tr(e)} it follows that if $e$ is
idempotent then $e$ is equal to zero or one.
\end{proof}

\begin{cor}\label{co:Tr-tot-real}
Let $K$ be a field of characteristic zero.
If $e$ is an idempotent in $\fa_{K,\max}$ then $\Tr(e)$ is 
a totally real algebraic element of $K$. Every element 
of $\bC$ satisfying its minimal equation is a real number 
between $0$ and $1$.
\end{cor}
\begin{proof}
Let $e=\sum_{i=1}^n a_i x_i$ with $i\in\fI\setminus\fX$, and
let $K_0=\bQ(a_1,\dots,a_n)$. 
As an abstract field, $\bC$ is an algebraic closure of an infinite
transcendental extension of $\bQ$. For every field
embedding $K_0\to\bC$,
Theorem~\ref{th:Tr(e)} shows that the image of $\Tr(e)$ is a real
number lying between $0$ and $1$. If $\Tr(e)$ were transcendental,
there would exist an embedding $K_0\to\bC$ taking $\Tr(e)$ to a
non-real number, and therefore $\Tr(e)$ is algebraic. Moreover, 
given any complex number satisfying its minimal equation, again 
there exists a field embedding $K_0\to\bC$ taking $\Tr(e)$ there.
\end{proof}

In the case of group rings, 
Zalesskii~\cite{Zalesskii:1972a} 
has shown that $\Tr(e)$ has to be rational. 
The proof does not appear to extend to our situation.

\begin{cor}
Let $K$ be a field of characteristic zero whose only totally real
subfield is $\bQ$, and let $\cO_K$ be its ring of integers. 
Then there are no idempotents in $\fa_{\cO_K,\max}$
other than $0$ and $1$.
\end{cor}
\begin{proof}
Let $e$ be an idempotent in $\fa_{\cO_K,\max}$. Then by
Corollary~\ref{co:Tr-tot-real}, $\Tr(e)$ is a totally real element of
$\cO_K$. By the hypothesis on $K$, it follows that $\Tr(e)$ is in
$\bQ\cap\cO_K=\bZ$. By Theorem~\ref{th:Tr(e)} it follows that $e=0$ or $e=1$.
\end{proof}

\chapter{Representation rings of finite groups}\label{ch:fingrp}

\section{\texorpdfstring{Preliminaries on $kG$-modules}
{Preliminaries on kG-modules}}

Let $G$ be a finite group\index{finite!group} and $k$ a field of characteristic $p$.
Throughout this text, we only consider finitely
generated\index{finitely!generated} $kG$-modules.\index{module!for finite group}
When we write the tensor product\index{tensor product} $M\otimes N$ of two
$kG$-modules $M$ and $N$ we mean the tensor product over the field
$k$, $M\otimes_k N$, with diagonal $G$-action. Thus
$g(m\otimes n)=gm\otimes gn$ for $g\in G$, $m\in M$ and $n\in N$.
We write $M^*$ for the linear dual\index{dual module} of $M$, with $G$-action given by 
$(g(f))(m)=f(g^{-1}(m))$. We write $k$ for the trivial
$kG$-module,\index{trivial!module}\index{module!trivial}
namely a copy of the field $k$ on which all elements of $G$ act as the
identity.

The various parts of the following proposition are due to
Benson and Carlson~\cite{Benson/Carlson:1986a}
and Auslander and Carlson~\cite{Auslander/Carlson:1986a}.

\begin{prop}\label{pr:AC}
Let $M$ be a $kG$-module. 
\begin{enumerate}
\item
$M$ is isomorphic to a direct summand of 
$M\otimes M^*\otimes M$. 
\item
If $p$ divides the dimension of
$M$ then $M \otimes M^* \otimes M$
has a direct summand isomorphic to $M \oplus M$. 
\item
  If $p$ does not divide the dimension of $M$ then $M \otimes M^*$
  has a direct summand isomorphic to $k$. 
\item
If $M$ and $N$ are indecomposable and $M\otimes N$ has a direct
summand isomorphic to $k$ then $N\cong M^*$.
\end{enumerate}
\end{prop}
\begin{proof}
Let $\{m_1,\dots,m_n\}$ be a basis for $M$ and $\{f_1,\dots,f_n\}$ the dual basis of $M^*$.
Thus $\sum_i f_i(m_i)=n=\dim(M)$, and 
for $m\in M$ we have $m=\sum_i f_i(m)m_i$.

For (i) we have maps $M \to M \otimes M^* \otimes M$ given
by $m \mapsto \sum_i m \otimes f_i \otimes m_i$ and
$M \otimes M^* \otimes M \to M$ given by $m \otimes f \otimes m'
\mapsto f(m)m'$, with composite the identity on $M$.

For (ii)
(cf.~Proposition 4.9 in Auslander and Carlson 
\cite{Auslander/Carlson:1986a}, where this is proved
with the further hypothesis that $M$ is indecomposable, but
this hypothesis is not used in the proof), we have maps $M \oplus M \to M \otimes M^* \otimes M$
given by 
\[ (m,m') \mapsto \sum_i(m \otimes f_i \otimes m_i +
m_i \otimes f_i \otimes m') \] 
and $M \otimes M^* \otimes M
\to M\oplus M$ given by $m \otimes f \otimes m' \mapsto (f(m)m',f(m')m)$.
If $M$ has dimension divisible by $p$ then the composite is 
the identity on $M \oplus M$.

For (iii), we have maps $k\to M \otimes M^*$ given by $1 \mapsto
\sum_i m_i \otimes f_i$ and $M \otimes M^* \to k$ given by $m \otimes
f \mapsto f(m)$. If $p$ does not divide the dimension of $M$ then the
composite of these maps is non-zero.

For (iv), we have maps $k\to M\otimes N \to k$ whose composite is
non-zero. Associated to these is the map
\[ \Hom_{kG}(M \otimes N,k) \times \Hom_{kG}(k,M \otimes N) \to k \]
given by composition. Adjointly,
\[ \Hom_{kG}(N,M^*)\times \Hom_{kG}(M^*,N) \to k \]
is given by composition followed by trace. Thus there are maps $M^*\to
N\to M^*$ whose composite has non-zero trace. Since $M^*$ is
indecomposable, endomorphisms which are not isomorphisms have zero
trace, so the composite is an isomorphism. Finally, since $N$ is
indecomposable, both maps must be isomorphisms.
\end{proof}

\section{\texorpdfstring{The representation ring $a(G)$}
{The representation ring a(G)}}

Let $G$ be a finite group and $k$ a field of characteristic $p$.
Let $a(G)$\index{a@$a(G)$, $a_\bC(G)$|textbf} be the representation
ring,\index{representation!ring} or
\emph{Green ring}\index{Green ring} of $kG$. This has as a basis the symbols
$[M_i]$, where $M_i$ is a indecomposable
$kG$-module, and $i$ is in a suitable indexing set $\fI$.
The symbol $[M_i]$ only depends on the isomorphism
class of $M_i$. If $M=\bigoplus_i n_iM_i$ then we write $[M]$ for
$\sum_i n_i[M_i]\in a(G)$.
By the Krull--Schmidt theorem, this is well defined, and identifies the non-negative
elements of $a(G)$ with the isomorphism classes
$[M]$ of $kG$-modules. The multiplication in
$a(G)$ is then given on non-negative elements by the tensor product, $[M][N]=[M\otimes_k N]$,
and extended bilinearly to all elements.
  
We shall also be interested in the complexification of the
representation ring, $a_\bC(G) = \bC \otimes_\bZ a(G)$.

\begin{prop}\label{pr:a(G)repring}
The representation ring $a(G)$ of a finite group over a field $k$ 
is an example of a representation ring in the
sense of Definition~\ref{def:repring}. It is an ordinary
representation ring in the sense of Definition~\ref{def:modular} 
if $k$ has characteristic zero, or prime
characteristic $p$ not dividing $|G|$, and a modular representation
ring if $k$ has characteristic $p$ dividing $|G|$.
\end{prop}
\begin{proof}
For property (i), we set $x_i=[M_i]$, and define the star operation by letting $M_{i^*}$
be the dual module $M_i^*$ of $M_i$.
Property (ii) follows from Proposition~\ref{pr:AC}\,(iv).
Proposition~\ref{pr:AC}\,(iii) shows that if $M_i\otimes M_{i^*}$ does
not have a direct summand isomorphic to $k$ then the $\dim M_i$ is not
divisible by $p$, and then Proposition~\ref{pr:AC}\,(ii) shows that
$M_i\otimes M_{i^*} \otimes M_i$ has a direct summand isomorphic to 
$M_i\oplus M_i$; this proves that property (iii) holds.
The dimension fuction for property (iv) is defined so that the
dimension of $[M_i]$ is its dimension as a $k$-vector space.
For Property (v), the role of the element
$\rho$ is played by the regular representation\index{regular representation} $[kG]$. For any
$kG$-module $M$, the tensor product $M\otimes kG$ is isomorphic to a
direct sum of $\dim M$ copies of $kG$, so we have $[M][kG]=(\dim
[M]).[kG]$ in $a(G)$.

If $k$ has characteristic zero, or prime characteristic $p$ not dividing
$|G|$, then the surjective augmentation map\index{augmentation map}
$kG\to k$ is split by the map $k\to kG$ sending $1$ to
$\frac{1}{|G|}\sum_{g\in G} g$. So the trivial module $k$ is
projective, and hence $\one$ is a projective basis element of $a(G)$.
On the other hand, if $k$ has characteristic dividing $|G|$ then the
augmentation map $kG\to k$ does not split, so the
trivial module is not projective, 
and therefore $\one$ is not a projective
basis element of $a(G)$. 
\end{proof}

\begin{rk}
Proposition~\ref{pr:a(G)repring} generalises in an obvious way to 
finite group schemes\index{finite!group!scheme}.
For finite supergroup schemes,\index{finite!supergroup scheme}
it is usual to take the super dimension function\index{super dimension function}
to be the dimension of the even part minus the dimension of
the odd part. However, this
fails axiom~(iv), as the dimension function is not positive.
Instead, one needs to take the na{\"\i}ve dimension, namely
the dimension of the even part plus the dimension of the odd part.
With this definition, the representation ring of a finite supergroup
scheme satisfies the axioms of Definition~\ref{def:repring}.
However, for the analogue of Proposition~\ref{pr:AC}, the super
dimension function is the relevant notion.

The proposition also generalises 
to any coefficients where the Krull--Schmidt
theorem\index{Krull--Schmidt theorem} holds for
finitely generated modules, such as the $\fp$-localisation
$\cO_{(\fp)}$ or the
$\fp$-adic completion $\cO_\fp$ of an
algebraic number ring $\cO$. The integral representation ring of a
finite group in the presence of the Krull--Schmidt theorem has been
studied by Reiner~\cite{Reiner:1965a,Reiner:1966a,Reiner:1966b},
Hannula~\cite{Hannula:1968a}; see also Jensen~\cite{Jensen:1983a}.

The representation ring of a finite dimensional
quasitriangular Hopf algebra\index{quasitriangular Hopf algebra}%
\index{Hopf algebra!quasitriangular}
(with mild extra conditions)
is a further generalisation which fits into our framework of
representation rings given by Definition~\ref{def:repring}. 
The quasitriangular condition ensures
that for any modules $M$ and $N$, the tensor products 
$M \otimes N$ and $N \otimes M$ are isomorphic. In this case, there
are two natural isomorphisms between these modules, which may be
thought of as moving $M$ over or under $N$. These isomorphisms 
satisfy braid relations\index{braid relations} given by the
Yang--Baxter equations.\index{Yang--Baxter equations} 
Finite quantum groups\index{finite!quantum group} 
are examples of suitable quasitriangular Hopf algebras.
Witherspoon~\cite{Witherspoon:1996a} has investigated species of the
representation ring in the particular case of the quantum double of a
finite group.

It would be 
interesting to see how much goes through in the case of Hopf algebras with
non-commutative tensor products, but that is a task for another day.
The spectral theory of non-commutative Banach *-algebras is not as
clean as the Gelfand theory for commutative ones.
\end{rk}

\begin{defn}\label{def:ideal}
Let $\fX$ be a collection of indecomposable $kG$-modules, 
closed under isomorphism, with the
property that if $M$ is in $\fX$ and $N$ is any $kG$-module then 
$M \otimes N$ is a direct sum of modules in $\fX$.
We also suppose that not every $kG$-module is in $\fX$.
We say that $\fX$ is an
\emph{ideal} of indecomposables.\index{ideal of indecomposables}
In this situation, by abuse of notation we also write $\fX$ for
the representation ideal\index{representation!ideal} of the representation ring
$a(G)$ consisting of the $i\in\fI$ such that
$M_i$ in $\fX$. We write $a(G,\fX)$
for the ideal $\langle\fX\rangle$, namely the linear combinations in $a(G)$ of the 
elements $[M_i]$ with $M_i\in\fX$, and $a_\fX(G)$ for the quotient
$a(G)/a(G,\fX)$. We write $a_\bC(G,\fX)$ for the ideal
$\bC \otimes_\bZ a(G,\fX)$ of 
$a_\bC(G)$,\index{a@$a(G,\fX)$, $a_\bC(G,\fX)$} 
and $a_{\bC,\fX}(G)$\index{a@$a_{\fX}(G)$, $a_{\bC,\fX}(G)$}
for the quotient $a_\bC(G)/a_\bC(G,\fX)$.
\end{defn}

\begin{eg}\label{eg:p}
The following are examples of ideals of indecomposables.
\begin{enumerate}
\item
There is, of course, the empty example $\fX=\varnothing$. In this
case, we have
$a(G,\fX)=\{0\}$, $a_\bC(G,\fX)=\{0\}$, $a_\fX(G)=a(G)$, $a_{\bC,\fX}(G)=a_\bC(G)$.
\item
Let $\fX_\proj$\index{X@$\fX_{\proj}$} be the set of projective indecomposable
modules.\index{projective!indecomposable!module} In this case
we write $a(G,1)$ for $a(G,\fX_\proj)$ and $a_\bC(G,1)$ for
$a_\bC(G,\fX_\proj)$.\index{a@$a(G,1)$, $a_\bC(G,1)$} This is the minimal example, in the sense
that every non-empty ideal $\fX$ of
indecomposables contains this one. This follows from Proposition~\ref{pr:max-min}.
\item
More generally, if $H$ is a subgroup of $G$ then the collection $\fX_H$ of 
indecomposable modules that are
projective relative to $H$ is ideal.\index{relatively!projective module}
We write $a(G,H)$ and $a_\bC(G,H)$\index{a@$a(G,H)$, $a_\bC(G,H)$}
for $a(G,\fX_H)$ and $a_\bC(G,\fX_H)$. The case where $H$ is the trivial
subgroup gives the previous example. 

We can do the same with any
collection $\fH$ of subgroups $H\le G$ with the property that
if $H$ is contained in a conjugate of $H'$ and $H'\in \fH$ then
  $H\in\fH$. The collection of indecomposable modules that are
  projective relative to such a collection of subgroups $\fH$ is
  ideal.\index{projective!relative to $\fH$}
This example is discussed in further detail in
Section~\ref{se:rel-proj}.
\item
If $\cV$ is a specialisation closed subset of
$\Proj H^*(G,k)$ then the indecomposable modules supported in $\cV$ form an ideal $\fX_\cV$.
We write $a(G,\cV)$ and 
$a_\bC(G,\cV)$\index{a@$a(G,\cV)$, $a_\bC(G,\cV)$} 
for $a(G,\fX_\cV)$ and $a_\bC(G,\fX_\cV)$.
For example, if $\cV$ is the collection of all closed points, then $\fX_\cV$ is
the collection of indecomposable projective or periodic modules.
\item
It is proved in Benson and Carlson \cite{Benson/Carlson:1986a}
that if $k$ is algebraically closed, then
the collection $\fX_p$ of indecomposable
modules of dimension divisible by $p$ is ideal.
We write $a(G;p)$ and
$a_\bC(G;p)$\index{a@$a(G;p)$, $a_\bC(G;p)$} 
for $a(G,\fX_p)$ and $a_\bC(G,\fX_p)$.
This is the maximal
example $\fX_{\max}$,\index{X@$\fX_{\max}$} in the
sense that every ideal $\fX$ of indecomposable modules that does not
consist of them all is contained in this one, see
Proposition~\ref{pr:max-min}\,(i).

Over a field which is not algebraically closed, $\fX_{\max}$ consists
of the indecomposable modules $M$ such that $M\otimes M^*$ has no
summand isomorphic to the trivial module $k$.
Every indecomposable module in
$\fX_{\max}$ has dimension divisible by $p$, but there may be indecomposables
of dimension divisible by $p$ not in $\fX_{\max}$. In this case, we
write $a(G;\max)$\index{a@$a(G;\max)$, $a_\bC(G;\max)$} 
for the ideal $\langle \fX_{\max}\rangle$ of $a(G)$,
$a_{\max}(G)$\index{a@$a_{\max}(G)$, $a_{\bC,\max}(G)$} 
for the quotient $a(G)/\langle \fX_{\max}\rangle$,
and $a_\bC(G;\max)$, $a_{\bC,\max}(G)$ for their complexifications.
\end{enumerate}
\end{eg}

\begin{defn}
If $\fX$ is an ideal of indecomposable $kG$-modules, and $M$ is
any $kG$-module, we may write $M=M'\oplus M''$, where $M''$ is
a direct sum of elements of $\fX$ and no summand of $M'$ is
in $\fX$. We define the $\fX$-\emph{core},\index{core}\index{X@$\fX$-core}
$\core_{G,\fX}(M)$\index{core@$\core_{G,\fX}(M)$, $\core_G(M)$} to be 
$M'$. This is well defined up to isomorphism. If $\fX$ is the ideal
of projective indecomposables, we just call this the \emph{core},
denoted $\core_G(M)$.
\end{defn}

\section{The gamma invariant}

By Proposition~\ref{pr:a(G)repring}, we may
regard $a(G)$ as a representation ring in the sense of
Definition~\ref{def:repring}. In this context,
we now investigate the gamma invariant defined in Section~\ref{se:npj}.
Let $\fX$ be an ideal of indecomposable $kG$-modules. 
We consider the $\fX$-cores of tensor powers of a module. We write
$\cc_n^{G,\fX}(M)$ for $\dim\core_{G,\fX}(M^{\otimes n})$. If
$\fX=\fX_\proj$ we write $\cc_n^G(M)$.\index{c@$\cc_n^{G,\fX}(M)$, $\cc_n^G(M)$}
Since
$\cc_n^{G,\fX}(M)\le \dim(M)^n$, 
the generating function
\[ f(t)=\sum_{n=0}^\infty \cc_n^{G,\fX}(M)t^n \]
converges in a disc of radius at least $1/\dim(M)$ around the origin.
The radius of convergence $r$ is given by the formula
\[ 1/r = \limsup_{n\to\infty}\sqrt[n]{\cc_n^{G,\fX}(M)}. \]
We write $\npj_{G,\fX}(M)$ for the value of $1/r$ given by this
formula. In the case where $\fX=\fX_\proj$, we just write
$\npj_G(M)$. This is the invariant that was investigated
in~\cite{Benson/Symonds:2020a}. 

The following properties are consequences of our investigations in
Chapter~\ref{ch:abstract}.

\begin{theorem}\label{th:npj_G}
The invariant $\npj_{G,\fX}(M)$ has the following properties:
\begin{enumerate}
\item We have $\displaystyle\npj_{G,\fX}(M)=\lim_{n\to\infty}\sqrt[n]{\cc_n^{G,\fX}(M)}=
\inf_{n\ge 1}\sqrt[n]{\cc_n^{G,\fX}(M)}$.
\item We have $0\le\npj_{G,\fX}(M)\le\dim M$.
\item Some tensor power of a $kG$-module $M$ has an indecomposable summand in $\fX$
if and only if $\npj_{G,\fX}(M)<\dim M$.
\item A $kG$-module $M$ is a direct sum of modules in $\fX$ if and
  only if $\npj_{G,\fX}(M)=0$; otherwise we have $\npj_{G,\fX}(M)\ge 1$.
\item If $m\in\bN$ then $\npj_{G,\fX}(mM)=m\,\npj_{G,\fX}(M)$, where
  $mM$ denotes a direct sum of $m$ copies of $M$.
\item More generally, if $a$ and $b$ are non-negative integers then
\[ \npj_{G,\fX}(ak\oplus b M) = a+b\npj_{G,\fX}(M), \]
where $k$ denotes the trivial $kG$-module.
\item If $\fX\subseteq\fY$ are ideals of indecomposable $kG$-modules
  then $\npj_{G,\fY}(M) \le \npj_{G,\fX}(M)$.
\item If $1\le\npj_{G,\fX}(M)<\sqrt{2}$ then $M$ is $\fX$-endotrivial.
\item We have $\npj_{G,\fX}(M\otimes N) \le
  \npj_{G,\fX}(M)\npj_{G,\fX}(N)$,
\item If $m\in\bN$ then we have $\npj_{G,\fX}(M^{\otimes m}) =
  \npj_{G,\fX}(M)^m$,
\item We have $\npj_{G,\fX}(M^*)=\npj_{G,\fX}(M)$.
\end{enumerate}
\end{theorem}
\begin{proof}
  By Proposition~\ref{pr:a(G)repring},
  $a(G)$ is a representation ring in the sense of
  Definition~\ref{def:repring}.
  So (i) follows from Theorem~\ref{th:inf},
  (ii) follows from Lemma~\ref{le:gamma-bounds},
  (iii) follows from Lemma~\ref{le:gamma<dim},
 (iv) follows from Lemma~\ref{le:mx},
(v) follows from Theorem~\ref{th:a+bx},
  (vi) follows from Lemma~\ref{le:gamma=0},
  (vii) follows from Lemma~\ref{le:gammaXinY}, 
  (viii) follows from Theorem~\ref{th:sqrt2},  
(ix) follows from Lemma~\ref{le:gamma(xy)}, 
(x) follows from Lemma~\ref{le:gamma(x^m)}, and
(xi) follows from Lemma~\ref{le:npjx*}.
\end{proof}

\begin{defn}
If $\fX$ is an ideal of indecomposable $kG$-modules,
and $H$ is a subgroup of $G$, we write 
$\fX{\downarrow_H}$\index{X@$\fX{\downarrow_H}$}
for the ideal of indecomposable summands of modules of the form
$M{\downarrow_H}$ with $M$ in $\fX$.
\end{defn}

\begin{theorem}
If $H\le G$ and $M$ is a $kG$-module then $\npj_{H,\fX{\downarrow_H}}(M)\le \npj_{G,\fX}(M)$.
\end{theorem}
\begin{proof}
We have $\cc_n^{H,\fX{\downarrow_H}}(M)\le \cc_n^{G,\fX}(M)$ for all $n$.
\end{proof}

\begin{cor}
If $H\le G$ and $M$ is a $kG$-module then $\npj_H(M)\le \npj_G(M)$.
\end{cor}
\begin{proof}
If $\fX=\fX_\proj$ for $kG$ then $\fX{\downarrow_H}=\fX_\proj$ for
$kH$.
\end{proof}

The following definition and theorem generalise 
Lemma~2.11 of~\cite{Benson/Symonds:2020a} (which is
Corollary~\ref{co:K/k} below).

\begin{defn}
If $K$ is an extension field of $k$, and $\fX$ is an ideal
of indecomposable $kG$-modules, we write $K\fX$ for the
ideal of indecomposable $KG$-modules that are summands
of modules of the form $K\otimes_k M$ with $M$ in $\fX$.
\end{defn}

\begin{lemma}\label{le:K/k}
Let $K$ be an extension field\index{field!extension}\index{extension field} 
of $k$, let $\fX$ be an ideal of indecomposable $kG$-modules, 
and let $M$ be a $kG$-module. 
Then $K\otimes_k\core_{G,\fX}(M)\cong \core_{G,\fX}(K\otimes_k M)$.
\end{lemma}
\begin{proof}
We first note that $(K\otimes_k M){\downarrow_{kG}}$ is a direct
sum of (possibly infinitely many) copies of $M$.

We need to show that if $K \otimes_kM$ has an indecomposable 
summand $N$ in $K\fX$ then $M$ has an indecomposable summand
in $\fX$.
Consider the restrictions of $K \otimes_kM$ and $N$
from $KG$ to $kG$. The restriction $N{\downarrow_{kG}}$
is a direct sum of finite dimensional indecomposable
$kG$-modules in $\fX$; let $N'$ be one of them. It is a
summand of $(K \otimes_k M){\downarrow_{kG}}$, which is
a direct sum of copies of $M$. Because it is finite dimensional,
$N'$ is a summand of a finite direct sum of copies of $M$, hence
is a summand of $M$, by the Krull--Schmidt Theorem.
\end{proof}

\begin{theorem}\label{th:K/k}
If $K$ is an extension field of $k$ and $\fX$ is an
ideal of indecomposable $kG$-modules then $\npj_{G,K\fX}(K\otimes_k M)=\npj_{G,\fX}(M)$.
\end{theorem}
\begin{proof}
This follows immediately from Lemma~\ref{le:K/k}.
\end{proof}

\begin{cor}\label{co:K/k}
If $K$ is an extension field of $k$ then $\npj_G(K\otimes_k M) = \npj_G(M)$.
\end{cor}
\begin{proof}
If $P$ is a projective $kG$-module
then every summand of $K\otimes_k P$ is a projective $KG$-module.
The corollary now follows from Theorem~\ref{th:K/k}.
\end{proof}

\section{Elementary abelian subgroups}\index{elementary abelian!subgroup}

In the case where $\fX=\fX_\proj$, the elementary abelian subgroups of
$G$ play a crucial role, because of theorems of
Chouinard and Carlson.
Most of the material in this section is taken from 
Section~7 of~\cite{Benson/Symonds:2020a}.

\begin{theoremqed}[Chouinard~\cite{Chouinard:1976a}]\label{th:Chouinard}
A $kG$-module $M$ is projective if and only if its restriction to
every elementary abelian $p$-subgroup of $G$ is projective.
\end{theoremqed}

A strengthening of Chouinard's theorem is the following theorem of Carlson.

\begin{theoremqed}[Carlson~\cite{Carlson:1981c}, Theorem~3.7]\label{th:Carlson}
There exists a constant $B$, which depends only on $p$ and $G$,
such that if $M$ is a $kG$-module then
\[ \dim \core_G(M) \le B \cdot \max_{E\le G}\dim\core_E(M) \]
where the maximum is taken over the set of elementary 
abelian $p$-subgroups $E$ of $G$.
\end{theoremqed}

\begin{lemma}\label{le:subgroup}
If $H$ is a subgroup of $G$ and $M$ is a $kG$-module then
$\npj_H(M)\le \npj_G(M)$.
\end{lemma}
\begin{proof}
We have $\cc_n^H(M) \le \cc_n^G(M)$ for all $n$.
\end{proof}

\begin{theorem}\label{th:E}
Let $M$ be a $kG$-module. Then
$\npj_G(M) = \displaystyle\max_{E\le G} \npj_E(M)$,
where the maximum is taken over the set of elementary abelian
$p$-subgroups $E$ of $G$.
\end{theorem}
\begin{proof}
By Theorem \ref{th:Carlson} and Lemma \ref{le:subgroup} we have
\begin{equation*}
\max_{E\le G}\sqrt[n]{\cc_n^E(M)}\le
\sqrt[n]{\cc_n^G(M)} \le
\sqrt[n]{B}.\max_{E\le G}\sqrt[n]{\cc_n^E(M)} .
\end{equation*}
Taking $\displaystyle\limsup_{n\to\infty}$, the factor of $\sqrt[n]{B}$ tends to $1$.
\end{proof}

\begin{eg}
Let $G$ be a generalised quaternion group%
\index{generalised!quaternion group} and let $k$ be a field of 
characteristic two. Then $G$ has only one elementary abelian $2$-subgroup
$E=\langle z\rangle$, where $z$ is the central element of order two.
Let $X=1+z$, an element of $kG$ satisfying $X^2=0$.
If $M$ is a $kG$-module then the restriction to $kE$ is a direct sum
of $\dim(\Ker(X,M)/\im(X,M))$ copies of the trivial module plus a free module.
It follows that
\[ \npj_G(M)=\dim(\Ker(X,M)/\im(X,M)). \]
In particular, this is an integer.
\end{eg}

\begin{prop}[Dade \cite{Dade:1978a,Dade:1978b}]\label{pr:Dade}
If $E$ is an elementary abelian $p$-group, then the only endo\-trivial $kE$-modules
are the syzygies $\Omega^n(k)$ ($n\in\bZ$) of the trivial module.\qed
\end{prop}

\begin{lemma}\label{le:Omega-k}
We have $\npj_G(\Omega^n k) =1$ for $n\in\bZ$, provided that $p$ divides $|G|$.
\end{lemma}

\begin{proof}
First suppose that $n\ge 0$. 
We have $\core_G((\Omega k)^{\otimes n}) \cong \Omega^nk$,
and $\dim \Omega^n k$ grows polynomially in $n$ (see for example \cite{Benson:1991b}\,  \S 5.3). 
Therefore $\npj_G(\Omega^n k)=1$. Since $(\Omega^n k)^* \cong \Omega^{-n} k$,
the lemma is also true for $n<0$.
\end{proof}

For a general representation ideal $\fX$ in a representation ring $\fa$, we saw that
it was possible for an $\fX$-endotrivial element $x\in\fa_{\cge 0}$ to
have $\npj_\fX(x)>1$, and so we needed to introduce two different
versions of the endotrivial group, see Section~\ref{se:endotriv}. 
For $\fa=a(G)$ and $\fX=\fX_\proj$, this
cannot happen, as shown by the following theorem. 

\begin{theorem}
A $kG$-module $M$ is endotrivial if and only if $\npj_G(M)=1$.%
\index{endotrivial!module}\index{module!endotrivial}
\end{theorem}
\begin{proof}
If $M$ is neither projective
nor endotrivial then it follows from Theorem~\ref{th:sqrt2} 
that $\npj_G(M)\ge \sqrt{2}$.
If $M$ is projective then $\npj_G(M)=0$. 
If $M$ is endotrivial then its restriction to 
every elementary abelian $p$-subgroup of $G$ is endotrivial.
So by Theorem \ref{th:E}, we may assume that $G=E$ is
an elementary abelian $p$-group. By Proposition \ref{pr:Dade},
$M$ is a syzygy of the trivial module, so
by Lemma~\ref{le:Omega-k} we have $\npj_E(M)=1$.
\end{proof}

\begin{question}
For choices of representation ideals $\fX$ in $a(G)$ 
other than $\fX=\fX_\proj$, is it true
that an $\fX$-endotrivial module $M$ necessarily 
satisfies $\npj_{G,\fX}(M)=1$?
\end{question}

\section{Induced modules}

In view of Theorem~\ref{th:E}, the following proposition is
important in finding the gamma invariants of 
induced modules.\index{induced module}\index{module!induced}
For calculational purposes, 
this proposition should be combined with the Mackey
decomposition formula\index{Mackey decomposition formula} 
for restricting induced modules
to elementary abelian subgroups.

\begin{prop}
If $E'$ is a subgroup of an elementary abelian $p$-group $E$ and
$M$ is a $kE'$-module then
\begin{enumerate}
\item
$\core_E(M{\uparrow^E})\cong (\core_{E'}(M)){\uparrow^E}$
\item
 $\npj_E(M{\uparrow^E})=|E:E'|\npj_{E'}(M)$.
\end{enumerate}
\end{prop}
\begin{proof}
We have $M{\uparrow^E}{\downarrow_{E'}}\cong|E:E'|M$, a direct sum of
$|E:E'|$ copies of $M$. So $M$ has a projective summand if and only if
$M{\uparrow^E}$ has a projective summand. This proves (i).
Applying this to $M^{\otimes n}$, we have
\begin{equation}\label{eq:induced-core}
\dim\core_E(M^{\otimes n}{\uparrow^E}) = \dim\core_E(M^{\otimes
    n}){\uparrow^E}=|E:E'|\dim\core_{E'}(M^{\otimes n}). 
\end{equation}

We have
\[ M{\uparrow^E}\otimes M{\uparrow^E} \cong (M\otimes
  M{\uparrow^E}{\downarrow_{E'}}){\uparrow^E}
\cong |E:E'|(M\otimes M){\uparrow^E}. \]
Applying induction, we then obtain
\[ (M{\uparrow^E})^{\otimes n} \cong |E:E'|^{n-1}(M^{\otimes
    n}){\uparrow^E} \]
and so using~\eqref{eq:induced-core} we have
\begin{align*} 
\cc_n^E(M{\uparrow^E})&=\dim\core_E((M{\uparrow^E})^{\otimes n})\\
&=|E:E'|^{n-1}\dim\core_E(M^{\otimes  n}{\uparrow^E})\\
&=|E:E'|^n\dim\core_{E'}(M^{\otimes n})=|E:E'|^n\cc_n^{E'}(M). 
\end{align*}
Thus we have 
\[ \sqrt[n]{\cc_n^E(M{\uparrow^E})}=|E:E'|\sqrt[n]{\cc_n^{E'}(M)}, \] 
and taking limits as $n$ tends to infinity, it follows that 
$\npj_E(M{\uparrow^E})=|E:E'|\npj_{E'}(M)$, and (ii) is proved.
\end{proof}

\section{Relatively projective modules and relative syzygies}\label{se:rel-proj}

In this section, we consider a generalisation of Example~\ref{eg:p}\,(iii). 

\begin{defn}
Let $N$ be a $kG$-module. 
We say that a $kG$-module $M$ is 
\emph{projective relative to $N$}\index{projective!relative to $N$}
or \emph{relatively $N$-projective}\index{relatively!projective module}
if $M$ is isomorphic to a direct summand of $N\otimes U$ for some
$kG$-module $U$. 
\end{defn}

\begin{lemma}
For $kG$-modules $M$ and $N$, the following are equivalent:
\begin{enumerate}
\item
$M$ is relatively $N$-projective, 
\item the natural map
$N\otimes N^*\otimes M \to M$ given by $n\otimes f \otimes m\mapsto
f(n)m$ splits,
\item the natural map
$M \to N^* \otimes N\otimes M$ given by $m \mapsto \sum_i f_i\otimes
n_i$ splits.
\end{enumerate}
\end{lemma}
\begin{proof}
(i) $\Rightarrow$ (ii):
It suffices to prove this for $M=N\otimes U$.
Let $n_i$ and $f_i$ be dual bases of $N$ and $N^*$. Then for
$n\in N$ we have $\sum_i f_i(n)n_i=n$.
Then the
map $N \otimes N^*\otimes N \otimes U \to N \otimes U$ is
split by the map sending $n\otimes u$ to $\sum_i n \otimes f_i \otimes
n_i \otimes u$.

(ii) $\Rightarrow$ (i):
If the natural map $N \otimes N^* \otimes
M \to M$ splits then $M$ is isomorphic to a direct summand of
$N\otimes U$ where $U=N^*\otimes M$.

The implications (i) $\Rightarrow$ (iii) and (iii) $\Rightarrow$ (i)
are proved similarly.
\end{proof}

\begin{rk}\label{rk:rel-proj}
If $N=kG$, this is the definition of a projective $kG$-module. If $N$ is the 
permutation module\index{permutation module}\index{module!permutation}
on the cosets of a subgroup $H\le G$,
this is the definition of a relatively $H$-projective module. More
generally, let $\fH$ be a collection of subgroups of $G$ with
the property that if $H$ is contained in a conjugate of $H'$ and
$H'\in\fH$ then $H\in \fH$. Let $N$ be the direct sum over $H\in\fH$
of the permutation modules on the cosets of $H$. Then a module
is relatively $\fH$-projective if and only if it is relatively $N$-projective.

The condition that $M$ is injective relative to $N$%
\index{injective relative to $N$} is equivalent
to the condition that it is projective relative to $N$, so we shall
not need to use this terminology.
\end{rk}

\begin{defn}
A short exact sequence 
\[ 0 \to M_3\to M_2 \to M_1 \to 0 \]
is 
\emph{relatively $N$-split}\index{relatively!split squence}
if the sequence
\[ 0 \to N\otimes M_3 \to N\otimes M_2 \to N \otimes M_1 \to 0 \]
splits.
\end{defn}

\begin{prop}
For a short exact sequence 
$0 \to M_3\to M_2 \to M_1\to 0$
the following are equivalent:
\begin{enumerate}
\item
The sequence
$0 \to N\otimes M_3\to N\otimes M_2 \to N\otimes M_1\to 0$
splits.
\item
The sequence
$0 \to N^* \otimes M_3 \to N^*\otimes M_2 \to N^*\otimes M_1 \to  0$
splits.
\item
The sequence
$ 0 \to N \otimes N^*\otimes M_3 \to N\otimes N^* \otimes M_2 \to
  N\otimes N^*\otimes M_1 \to 0$
splits.
\end{enumerate}
\end{prop}
\begin{proof}
This follows from the fact that $N$ is isomorphic to a direct summand
of $N\otimes N^*\otimes N$, and $N^*$ is a direct summand of
$N^*\otimes N \otimes N^*$, see Proposition~\ref{pr:AC}\,(i).
\end{proof}

\begin{cor}\label{co:dual-rel-split}
A short exact sequence
\[ 0 \to M_3\to M_2\to M_1\to 0 \]
is relatively $N$-split if and only if the dual sequence
\[ 0 \to M_1^*\to M_2^*\to M_3^*\to 0 \]
is relatively $N$-split.\qed
\end{cor}

\begin{lemma}\label{le:rel-split-proj}
If $0\to M_3\to M_2\to M_1 \to 0$ is a relatively $N$-split short
exact sequence, and either $M_1$ or $M_3$ is relatively $N$-projective, then the
sequence splits.
\end{lemma}
\begin{proof}
Suppose that $M_1$ is relatively $N$-projective.
We have a diagram\medskip
\[ \xymatrix{0 \ar[r] & N \otimes N^*\otimes M_3 \ar[r]\ar[d] &
N\otimes N^*\otimes M_2 \ar[r]\ar[d] & N \otimes N^*\otimes M_1
\ar[r]\ar@/_3ex/[l]\ar[d] & 0 \\
0 \ar[r] &M_3\ar[r] & M_2 \ar[r] & M_1 \ar[r]\ar@/_2ex/[u] & 0} \]
Composing the two splittings and the map $N\otimes N^*\otimes M_2\to
M_2$ gives a splitting for the map $M_2\to M_1$.

The proof with $M_3$ relatively projective is dual (cf.\ Remark~\ref{rk:rel-proj}).
\end{proof}

\begin{defn}
A \emph{relative $N$-syzygy}\index{relative!N@$N$-syzygy} 
of $M$, or \emph{syzygy of $M$ relative to $N$}%
\index{syzygy relative to a module}
is the module $M_3$ in a relatively $N$-split sequence with
$M_2$ relatively $N$-projective and $M_1=M$. This exists, because
the natural epimorphism $N^*\otimes N \otimes M \to M$ is relatively
split. We write $\Omega_N(M)$ for the relative $N$-syzygy of $M$.

Dually, we write $\Omega^{-1}_N(M)$ for the module $M_1$
in a relatively $N$-split sequence with $M_2$ relatively
$N$-projective and $M_3=M$. This exists because the
natural map $M \to N^*\otimes N \otimes M$ is relatively split.
\end{defn}

The following is a relative version of Schanuel's
lemma.\index{relative!Schanuel lemma}\index{Schanuel's lemma, relative version}
It shows that relative $N$-syzygies are well defined up to adding and removing relatively
$N$-projective summands.

\begin{lemma}\label{le:rel-Schan}
Let $M$ and $N$ be $kG$-modules. If 
\begin{align*}
0\to A \to B &\to M\to 0\\
0 \to A' \to B' &\to M \to 0
\end{align*} 
are two relatively $N$-split sequences with
$B$ and $B'$ relatively $N$-projective then $A\oplus B' \cong A'\oplus B$.
\end{lemma}
\begin{proof}
We form the pullback $X$ of $B\to M$ and $B'\to M$. Then we have a
commutative diagram
\[ \xymatrix{&&0\ar[d]&0\ar[d] \\
&&A'\ar@{=}[r]\ar[d]&A' \ar[d] \\
0\ar[r] &A\ar[r]\ar@{=}[d] &X\ar[r] \ar[d]&B' \ar[r]\ar[d] & 0 \\
0 \ar[r] & A \ar[r] & B \ar[r]\ar[d] & M \ar[r]\ar[d] & 0\\
&&0&0} \]
We claim that the sequence ending in $B'$ given by the middle row
of this diagram is relatively $N$-split. To see this, we use the fact
that the bottom row is relatively $N$-split. Then we compose
\[ N\otimes B' \to N \otimes M \to N\otimes B. \]
Since $N\otimes X$ is a pullback of $N \otimes B \to N \otimes M$
and $N\otimes B'\to N \otimes M$, we can use this composite
together with the identity map on $N\otimes B'$ to obtain the required
map $N\otimes B' \to N \otimes X$. Similarly, the sequence ending
in $B$ given by the middle column is relatively $N$-split.

Since $B$ and $B'$ are relatively $N$-projective, 
using Lemma~\ref{le:rel-split-proj} we see that the sequences ending in
$B$ and $B'$ split. So we have $A\oplus B' \cong X \cong A'\oplus B$.
\end{proof}

\begin{defn}
If $N$ is a $kG$-module, let $\fX_N$\index{X@$\fX_N$} be the collection of
indecomposable modules that are projective relative to $N$.
Then $\fX_N$ is an ideal of $kG$-modules.
\end{defn}

\begin{prop}\label{pr:OmegaNM}
We have the following.
\begin{enumerate}
\item
The module $\core_{G,\fX_N}(\Omega_N(M))$ is well defined, and
isomorphic to
\[ \core_{G,\fX_N}(\Omega_N(k)\otimes M). \]
\item
$\core_{G,\fX_N}(\Omega_N^{-1}\Omega_N(M))\cong
\core_{G,\fX_N}(M)$.
\item
The module $\core_{G,\fX_N}(\Omega_N^{-1}(M))$ is well
defined, and is isomorphic to 
\[ \core_{G,\fX_N}(\Omega^{-1}(k)\otimes M). \]
\item
$\core_{G,\fX_N}(\Omega_N\Omega_N^{-1}(M))\cong
\core_{G,\fX_N}(M)$.
\end{enumerate}
\end{prop}
\begin{proof}
The fact that $\core_{G,\fX_N}(\Omega_N(M))$ is well defined follows
from Lemma \ref{le:rel-Schan}. Examining the short exact sequence
\[ 0 \to \Omega_N(k)\otimes M \to N\otimes N^* \otimes M \to M \to 0, \]
we see that $\Omega_N(k)\otimes M$ is one candidate for $\Omega_N(M)$.
It follows by examining the same sequence, that $M$ is one candidate
for $\Omega^{-1}_N\Omega_N(M)$. This proves parts (i) and (ii), and
(iii) and (iv) are dual.
\end{proof}

\begin{lemma}
We have $\Omega^{-1}_N(k) \cong \Omega_N(k)^*$. Thus
\[ \npj_{G,\fX_N}(\Omega^{-1}_N(k))=\npj_{G,\fX_N}(\Omega_N(k)). \]
\end{lemma}
\begin{proof}
The first statement follows by applying
Corollary~\ref{co:dual-rel-split} to the sequence
\begin{equation*}
0 \to \Omega_N(k) \to N\otimes N^* \to k \to 0.
\qedhere
\end{equation*}
\end{proof}

\begin{rk}
It often happens that $\npj_{G,\fX_N}(\Omega_N(k))=1$. For example,
the finite generation of finite group cohomology implies that in the
case $N=kG$, we obtain $\npj_G(\Omega(k))=1$. More generally,
a theorem of Brown~\cite{Brown:1994a} shows that if $N$ is a
permutation module then $\npj_{G,\fX_N}(\Omega_N(k))=1$.
We do not know of an example of a $kG$-module $N$ for which
$\npj_{G,\fX_N}(\Omega_N(k))>1$.
\end{rk}

\begin{theorem}
We have 
\[
\displaystyle\frac{\npj_{G,\fX_N}(M)}{\npj_{G,\fX_N}(\Omega_N(k))}\le
\npj_{G,\fX_N}(\Omega_N(M))
\le\npj_{G,\fX_N}(\Omega_N(k))\npj_{G,\fX_N}(M). \]
\end{theorem}
\begin{proof}
By Proposition~\ref{pr:OmegaNM} and Theorem~\ref{th:npj_G}\,(ix) we have
\begin{equation*} 
\npj_{G,\fX_N}(\Omega_N(M))=\npj_{G,\fX_N}(\Omega_N(k)\otimes M) 
\le \npj_{G,\fX_N}(\Omega_N(k))\npj_{G,\fX_N}(M) 
\end{equation*}
Similarly, $\npj_{G,\fX_N}(M)=
\npj_{G,\fX_N}(\Omega^{-1}_N(k)\otimes\Omega_N(M))\le 
\npj_{G,\fX_N}(\Omega^{-1}_N(k))\npj_{G,\fX_N}(\Omega_N(M))$.
\end{proof}

\begin{cor}
If $N$ is a permutation module then
\[ \npj_{G,\fX_N}(\Omega_N(M))=\npj_{G,\fX_N}(M). \]
\end{cor}
\begin{proof}
This follows from the theorem and the preceding remark.
\end{proof}

\begin{cor}
We have $\npj_G(\Omega(M))=\npj_G(M)$.
\end{cor}
\begin{proof}
This is the case $N=kG$.
\end{proof}

\section{Trivial source modules}

In this section, we compute $\npj_G(M)$ for trivial source modules
$M$. This is also discussed in Upadhyay~\cite{Upadhyay:2021a}.

\begin{defn}
Let $G$ be a finite group and $k$ a field of characteristic $p$.
A $kG$-module $M$ is said to be a 
\emph{trivial source module}\index{trivial!source module}%
\index{module!trivial source}
if its restriction to a Sylow $p$-subgroup is a permutation module.
This is equivalent to the condition that $M$ is isomorphic to
a direct summand of a permutation $kG$-module.\index{permutation module}%
\index{module!permutation}
\end{defn}

\begin{lemma}
If $Q$ is a subgroup of a finite $p$-group $P$ then the
permutation module $P/Q$ is indecomposable. If $M$ is a
permutation $kP$-module then either $M$ is projective or 
the action of $P$ on $M$ has a kernel.
\end{lemma}

\begin{lemma}\label{le:Etensors}
Let $E$ be an elementary abelian $p$-group, and let
$M_1,\dots,M_n$ be indecomposable permutation $kE$-modules.
Then $M_1\otimes \cdots \otimes M_n$ is either projective or
the action of $E$ has a kernel.
\end{lemma}
\begin{proof}
If $E_1$ and $E_2$ are subgroups of $E$ then by the Mackey
decomposition theorem the tensor
product of permutation modules $k(E/E_1) \otimes k(E/E_2)$
is a direct sum of copies of $kE/(E_1\cap E_2)$. 
By induction, it follows
that $k(E/E_1)\otimes\cdots \otimes k(E/E_n)$ is a direct sum
of copies of $kE/(E_1\cap\dots\cap E_n)$.
This is projective
if and only if the kernel of the action, $E_1\cap \dots \cap E_n$ is trivial.
\end{proof}

\begin{prop}\label{pr:Efaithless}
Let $E$ be an elementary abelian $p$-group and $M$ a permutation
$kE$-module. Then there is a direct summand $N$ of $M$
such that the action of $E$ on $N$ is not faithful and 
$\npj_E(M)=\npj_E(N)=\dim(N)$.
\end{prop}
\begin{proof}
If the action of $E$ on $M$ is not faithful then $\npj_E(M)=\dim(M)$
and we are done. On the other hand, if the action is faithful then
$M\cong M'\oplus M_1\oplus \cdots \oplus M_n$
with $M_i\cong k(E/E_i)$ and $E_1\cap\cdots\cap E_n=1$.
Then by Lemma~\ref{le:Etensors}, $M_1\otimes \dots\otimes M_n$
is projective. So for $m\ge n$ the module
$M^{\otimes m}$ is isomorphic to a summand of
\[ \bigoplus_{i=1}^n (M' \oplus M_1 \oplus \cdots
\overset{i}{\uparrow} \cdots \oplus M_n)^{\otimes m} \oplus (\proj) \]
where the upward arrow indicates omission of the term indexed by $i$.
It follows that
\[ \npj_E(M^{\otimes m}) \le n\max_{1\le i\le n}\npj_E(M'\oplus M_1\oplus \cdots 
\overset{i}{\uparrow} \cdots \oplus M_n)^{\otimes m}. \]
Taking $m$th roots, we have
\[ \npj_E(M) \le \sqrt[m]{n}\max_{1\le i\le n}\npj_E(M'\oplus M_1\oplus \cdots 
\overset{i}{\uparrow} \cdots \oplus M_n). \]
Letting $m$ tend to infinity, we can ignore the term $\sqrt[m]{n}$.
Then since each of the terms on the right hand side is less than or
equal to the left hand side, we have equality:
\[ \npj_E(M) =\max_{1\le i\le n}\npj_E(M'\oplus M_1\oplus \cdots 
\overset{i}{\uparrow} \cdots \oplus M_n). \]
Thus there is a non-zero summand $M_i$ of $M$ which may be
removed from $M$ without altering the value of $\npj_E(M)$.
Arguing by induction, there is a summand $N$ of $M$ which
is not faithful, with $\npj_E(M)=\npj_E(N)=\dim(N)$.
\end{proof}

\begin{theorem}
Let $M$ be a trivial source $kG$-module. Then $\npj_G(M)$ is
equal to the maximum over pairs $(E,N)$ of $\dim(N)$,
where $E$ runs over the elementary abelian $p$-subgroups of 
$G$ and $N$ runs over the summands of $M$ as $kE$-modules 
which are not faithful.
\end{theorem}
\begin{proof}
This follows from Theorem~\ref{th:E} and Proposition~\ref{pr:Efaithless}.
\end{proof}

\section{\texorpdfstring{The norms on $a(G)$}
{The norms on a(G)}}

Definition~\ref{def:norm} gives us
the following.

\begin{defn}
The \emph{weighted $\ell^1$ norm}%
\index{weighted!$\ell^1$ norm}\index{norm!weighted $\ell^1$} on
$a_\bC(G)$ is given by
\[ \left\|\sum_{i\in\fI}a_i[M_i]\right\|=\sum_{i\in\fI}|a_i|\dim
  M_i. \]
We define $\hat a(G)$\index{aa@$\hat a(G)$} 
to be the completion of $a_\bC(G)$ with respect to
this norm.
\end{defn}

If $\fX$ is a representation ideal in $a(G)$ then by
Lemma~\ref{le:quotient}\,(iii), the quotient 
$\hat a_\fX(G)$\index{aa@$\hat a_{\fX}(G)$|textbf} of $\hat a(G)$ by
$\langle\fX\rangle_\bC$ is isometrically isomorphic to the completion
of $a_{\bC,\fX}(G)=a_\bC(G)/a_\bC(G,\fX)$ with respect to
the quotient norm. By
\eqref{eq:normA/X}, the quotient norm on
$a_{\bC,\fX}(G)$ is given by
\[  \left\|\sum_{i\in\fI}a_i[M_i]\right\|_\fX
=\sum_{i\in\fI}|a_i|\dim\core_\fX(M_i) 
=\sum_{i\in\fI\setminus\fX}\dim M_i. \]
Elements of the completion $\hat a_\fX(G)\cong\widehat{a_{\bC,\fX}(G)}$ may be regarded as
infinite, but countably supported, linear combinations
$\sum_{i\in\fI\setminus\fX}a_i[M_i]$ where
\[ \sum_{i\in\fI\setminus\fX}|a_i|\dim M_i<\infty. \]

Recall that if $k$ is algebraically closed then the modules $M_i$ with 
$i\in\fX_{\max}$ are the indecomposables of dimension divisible by
$p$, so that $a(G,\max)=a(G;p)$ and $a_{\max}(G)=a(G)/a(G;p)$.

In general, we have the weighted $\ell^2$ 
norm\index{weighted!$\ell^2$ norm}\index{norm!weighted $\ell^2$} on
$a_{\bC,\max}(G)=a_\bC(G)/\langle\fX_{\max}\rangle_\bC$ described in 
Definition~\ref{def:weighted-ell2}. 
\[ \left|\sum_{i\in\fI}a_i[M_i]\right|=\sqrt{\sum_{i\in\fI}n_i|a_i|^2}. \]
The completion of
$a_{\bC,\max}(G)$ with respect to this norm is a Hilbert space\index{Hilbert space}
$H(G)$.\index{H@$H(G)$} By Theorem~\ref{th:fa-in-Lin}, left
multiplication induces a continuous map of normed $*$-algebras
$a_{\bC,\max}(G)\to\Lin(H(G))$, and extends to an injective map
$\hat a_{\max}(G)\to \Lin(H(G))$.  The $C^*$-algebra
$C^*_{\max}(G)$ is defined to be the closure of the image of this
map. This is a commutative $C^*$-algebra, and is the completion of
$a_{\max}(G)$ with respect to the sup norm
\[ \|x\|_{\sup}=\sup_{|y|=1}|xy|. \]

\section{\texorpdfstring{Species of $a(G)$}
{Species of a(G)}}\index{species!of $a(G)$}

We say that a species $s$ of $a(G)$ is \emph{$\fX$-core bounded} if
for all $kG$-modules $M$ we have 
\[ |s([M])|\le\dim\core_{G,\fX}(M). \]
In particular, the extension of of an $\fX$-core bounded species to 
$a_\bC(G)$ vanishes on $a_\bC(G,\fX)$, and so defines an algebra homomorphism
$a_{\bC,\fX}(G)\to \bC$.
If $\fX=\varnothing$, we say that $s$ is \emph{dimension bounded},
and if $\fX$ is the ideal of projective indecomposable modules,
we just say that $s$ is \emph{core bounded}.

\begin{theorem}\label{th:core-bounded-aG}
For a species $s\colon a_\bC(G)\to \bC$, the following are equivalent:
\begin{enumerate}
\item $s$ is $\fX$-core bounded.
\item For all $x\in a_\bC(G)$ we have $|s(x)|\le\|x\|_\fX$.
\item $s$ is continuous with respect to the norm on $a_{\bC,\fX}(G)$.
\item $s$ vanishes on $\langle\fX\rangle_\bC$ and extends to an
  algebra homomorphism $\hat a_\fX(G)$.
\end{enumerate}
\end{theorem}
\begin{proof}
This is Theorem~\ref{th:core-bounded} in this context.
\end{proof}

\begin{theorem}\label{th:spec-radius-aG}
For $x\in a(G)$, the spectrum\index{spectrum} 
$\Spec_\fX(x)$\index{Spec@$\Spec_\fX(x)$} is the set of
values of $s(x)$ as $x$ runs over the $\fX$-core bounded species of
$a(G)$. The spectral radius\index{spectral!radius} is
\[ \npj_\fX(x)=\max_{\substack{s\colon a(G)\to\bC\\ \fX-\text{\rm core bounded}}}|s(x)|. \]
There is an $\fX$-core bounded species $s$ with $\npj_\fX(x)=s(x)$.
\end{theorem}
\begin{proof}
This is Theorem~\ref{th:spec-radius-fa} in this context.
\end{proof}

\begin{prop}
If $s$ is a dimension bounded species of $a(G)$ then either
$s$ is a Brauer species or $s$ is core bounded.
\end{prop}
\begin{proof}
  This follows from Theorem~\ref{th:core-bounded-aG}.
\end{proof}

\begin{defn}
If $\fX$ is an ideal of indecomposable $kG$-modules,
and $A=\hat a_\fX(G)$, we write 
$\Struct_\fX(G)$\index{DeltaXG@$\Struct_\fX(G)$} for $\Struct(\hat a_\fX(G))$.
Thus $\Struct_\fX(G)$ may be identified with the set of $\fX$-core bounded 
species of $a(G)$, with the weak* topology. It is a compact Hausdorff
topological space.
\end{defn}

\begin{prop}
If $\fY\subseteq\fX$ are ideals of indecomposable $kG$-modules, then
$\Struct_\fX(G)$ is homeomorphic to the subset of $\Struct_\fY(G)$ consisting
of the $\fX$-core bounded species.
\end{prop}
\begin{proof}
Every $\fX$-core bounded species is $\fY$-core bounded, so 
$\Struct_\fX(G)\subseteq\Struct_\fY(G)$. Now apply Lemma \ref{le:structure-subset}.
\end{proof}

The corresponding notion for the weighted $\ell^2$ norm is as follows.

\begin{defn}
A species $s\colon a(G)\to\bC$ is 
\emph{sup bounded}\index{sup!bounded species}\index{species!sup bounded}
if for all $x\in a(G)_{\cge 0}$ we have $|s(x)|\le
\|x\|_{\sup}$. 
\end{defn}

 The results of
Sections~\ref{se:idempotents} and~\ref{se:qn} give the following.

\begin{theoremqed}
There are no non-zero quasi-nilpotent elements in $\hat a_{\max}(G)$.
\end{theoremqed}

\begin{corqed}
\begin{enumerate}
\item
The Jacobson radical of $\hat a_{\max}(G)$ is zero.
\item
The Jacobson radical of $a_{\bC,\max}(G)$ is zero.
\qedhere
\end{enumerate}
\end{corqed}

\begin{theoremqed}
If $e\in \hat a_{\max}(G)$ then $0<\Tr(e)<1$.
\end{theoremqed}

\begin{corqed}
Let $K$ be a field of characteristic zero whose only totally real
subfield is $\bQ$, and let $\cO_K$ be its ring of integers. Then there
are no idempotents in $a_{\cO_K,\max}(G)$ other than $0$ and $1$.
\end{corqed}

\section{\texorpdfstring
{Restriction and induction on $\hat a(G)$}
{Restriction and induction on â(G)}}\index{restriction}\index{induction}

Let $H$ be a subgroup of a finite group $G$. Then we have a restriction
map $\res_{G,H}\colon a(G)\to a(H)$ and an induction map
$\ind_{H,G}\colon a(H)\to a(G)$. The map $\res_{G,H}$ is a ring
homomorphism, while $\ind_{H,G}$ is an $a(G)$-module homomorphism
via restriction.
These extend in an obvious way to
maps $\res_{G,H}\colon a_\bC(G)\to a_\bC(H)$ and $\ind_{H,G}\colon
a_\bC(H)\to a_\bC(G)$ with the same properties.

\begin{prop}\label{pr:res-conts}
Let $H$ be a subgroup of a finite group $G$. Then 
restriction 
\[ \res_{G,H}\colon a_\bC(G) \to a_\bC(H) \] 
is continuous with respect
to the weighted $\ell^1$ norm, and therefore extends to give a homomorphism
of commutative Banach algebras $\res_{G,H}\colon \hat a(G)\to \hat
a(H)$.
\end{prop}
\begin{proof}
By Lemma~\ref{le:bd=cts}, we must show that this map is bounded.
If $x=\sum_{i\in\fI} a_i[M_i]$ with the $M_i$ indecomposable
$kG$-modules, then
\begin{align*} 
\|\res_{G,H}(x)\|&=\left\|\sum_{i\in\fI} a_i\res_{G,H}([M_i])\right\| \\
&\le \sum_{i\in\fI}|a_i|\|\res_{G,H}([M_i])\| \\
&=\sum_{i\in\fI}|a_i|\dim(M_i) \\
&=\|x\|
\end{align*}
and so $\res_{G,H}$ is bounded, and hence continuous.
Since $\res_{G,H}$ is an algebra homomorphism from $a_\bC(G)$ to
$a_\bC(H)$, by continuity it is an algebra homomorphism from
$\hat a(G)$ to $\hat a(H)$.
\end{proof}

\begin{prop}\label{pr:ind-conts}
Let $H$ be a subgroup of a finite group $G$. Then induction
\[ \ind_{H,G}\colon a(H) \to a(G) \]
is continuous with respect to the weighted $\ell^1$ norm, and
therefore extends to a continuous map
\[ \ind_{H,G}\colon \hat a(H)\to \hat a(G). \] 
For $x\in\hat a(G)$ and $y\in\hat a(H)$, we have
\begin{equation}\label{eq:Frobenius} 
\ind_{H,G}(\res_{G,H}(x)y)=x\,\ind_{H,G}(y), 
\end{equation}
so regarding $\hat a(H)$ as an $\hat a(G)$-module via
$\res_{G,H}$, $\ind_{H,G}$ is a map of Banach $\hat a(G)$-modules.
\end{prop}
\begin{proof}
Again by Lemma~\ref{le:bd=cts}, we must show that the map is bounded.
If $x=\sum_{i\in\fI}a_i[M_i]$ with the $M_i$ indecomposable
$kH$-modules then 
\begin{align*}
\|\ind_{H,G}(x)\|&=\left\|\sum_{i\in\fI}a_i\ind_{H,G}([M_i])\right\|  \\
&\le \sum_{i\in\fI}|a_i|\|\ind_{H,G}([M_i])\| \\
&=\sum_{i\in\fI}|a_i||G:H|\dim(M_i)\\
&= |G:H|\|x\|,
\end{align*}
so $\ind_{H,G}$ is bounded, and hence continuous.
Frobenius reciprocity\index{Frobenius!reciprocity} implies that 
\eqref{eq:Frobenius}
holds for elements
$x\in a_\bC(G)$ and $y\in a_\bC(H)$. By the continuity of
$\ind_{H,G}$ and $\res_{G,H}$ (see Proposition~\ref{pr:res-conts})
it holds for $x\in\hat a(G)$ and $y\in\hat a(H)$.
\end{proof}

\begin{theorem}\label{th:Mackey}
If $H$ and $K$ are subgroups of $G$ then the 
Mackey decomposition formula\index{Mackey decomposition formula} holds for
the maps $\ind_{K,G}\colon\hat a(K) \to \hat a(G)$ and 
$\res_{G,H}\colon\hat a(G) \to\hat a(H)$. Namely, for $x\in \hat
a(K)$ we have
\[ \res_{G,H}\ind_{K,G}(x)=\sum_{g\in H\setminus G/K} 
\ind_{H\cap{^g\!K},H}\res_{^g\!K,H\cap{^g\!K}}({^gx}). \]
Here, $^g\!K$ is the conjugate $gKg^{-1}$ and $^gx$ is the image of $x$
under the natural isomorphism $\hat a(K)\to \hat a({^g\!K})$ given by conjugation
by $g$. The sum is over a set of double coset representatives for $H$
and $K$ in $G$.
\end{theorem}
\begin{proof}
For $x\in a(K)$ this is the usual Mackey decomposition theorem.
The proposition therefore follows from the continuity of restriction
and induction, which hold by Propositions~\ref{pr:res-conts} 
and~\ref{pr:ind-conts}.
\end{proof}

\begin{theorem}\label{th:integral}
If $H$ is a subgroup of $G$ then $\hat a(H)$ is integral over the
image of $\res_{G,H}\colon \hat a(G)\to\hat a(H)$.
\end{theorem}
\begin{proof}
Proposition~5.3  of~\cite{Benson/Parker:1984a} proves this for
$a(G)_\bC\to a(H)_\bC$, and the same proof works in this context.
For convenience, we repeat the proof here.

Let $\hat a(H)^{N_G(H)}$ be the fixed points in $\hat a(H)$ 
under the natural conjugation action of $N_G(H)$. If $\alpha\in\hat
a(H)$ let $\alpha_1,\dots,\alpha_n$ be the images of $\alpha$
under $N_G(H)$. Then coefficients of 
the monic polynomial $(x-\alpha_1)\dots(x-\alpha_n)$ are the symmetric 
functions of $\alpha_1,\dots,\alpha_n$, which are in $\hat a(H)^{N_G(H)}$.
Since $\alpha$ is a root of this polynomial, it is integral over
$\hat a(H)^{N_G(H)}$. So $\hat a(H)$ is integral over $\hat a(H)^{N_G(H)}$, and
it remains to prove that $\hat a(H)^{N_G(H)}$ is integral over the
image of $\res_{G,H}$.

Let $\alpha\in\hat a(H)^{N_G(H)}$. If $K$ is a proper subgroup of $H$, we let
$U_K$ denote the subalgebra of $\hat a(H)$ generated by
the image of $\res_{G,K}$ together with 
the elements $\res_{{^g\!H},K}(^g\alpha)$ with $K\le {^g\!H}$. By induction
on $|H|$, we may assume that $U_K$ is finitely generated as a module
over $\im(\res_{G,K})$. We claim that 
\begin{equation}\label{eq:res+UK}
\im(\res_{G,H})+\sum_{K<H}\ind_{K,H}(U_K) 
\end{equation}
is a subalgebra of $\hat a(H)$, finitely generated as a module for
the image of $\res_{G,H}$, containing $\alpha$. 

First we show that~\eqref{eq:res+UK} is
a subalgebra. By~\eqref{eq:Frobenius} and the transitivity
of restriction, for $x,y\in \hat a(G)$  we have
\begin{align*} 
\res_{G,H}(x)\ind_{K,H}(\res_{G,K}(y))&=\ind_{K,H}(\res_{G,K}(xy)), \\
\res_{G,H}(x)\ind_{K,H}(\res_{^g\!H,K}({^g\alpha})) &= 
\ind_{K,H}(\res_{G,K}(x)\res_{^g\!H,K}({^g\alpha})) 
\end{align*}
and so $\im(\res_{G,H}) U_K \subseteq U_K$. If $K<H$ and $L<H$ then
using the Mackey Decomposition Theorem~\ref{th:Mackey}, we have
\begin{align*}
\ind_{K,H}(U_K)\ind_{L,H}(U_L) &\subseteq \sum_{g\in H} 
\ind_{K\cap{^g\!L},H}(\res_{K,K\cap{^g\!L}}(U_K).\res_{^g\!L,K\cap{^g\!L}}(U_{^g\!L})) \\
&\subseteq \sum_{g\in H} \ind_{K\cap{^g\!L},H}(U_{K\cap{^g\!L}}).
\end{align*}

Next we show that~\eqref{eq:res+UK} is finitely generated as a module for
$\im(\res_{G,H})$. By induction on $|H|$, we may assume that
for every $K\le H$, $U_K$ is finitely generated as a module for
$\im(\res_{G,K})$, say by elements $b_i$. Then 
using~\eqref{eq:Frobenius}, $\ind_{K,H}(U_K)$
is finitely generated as a module for $\im(\res_{G,H})$
by the elements $\ind_{K,H}(b_i)$.

Finally, we show that~\eqref{eq:res+UK} contains $\alpha$. Since
$\alpha$ is $N_G(H)$-invariant, the Mackey Decomposition
Theorem~\ref{th:Mackey} gives
\begin{equation*}
\res_{G,H}(\ind_{H,G}(\alpha)) =
|N_G(H):H|\alpha + \sum_{\substack{
g\in H\setminus G / H \\
H\cap{^g\!H}<H}}\ind_{H\cap{^g\!H},H}(\res_{{^g\!H},H\cap{^g\!H}} (\alpha^g)).
\qedhere
\end{equation*}
\end{proof}

\begin{cor}
If $s$ is a species of $\hat a(G)$ which vanishes on the kernel of
$\res_{G,H}$ then there exists a species $s'$ of $\hat a(H)$ such that
$s$ is the composite
\[ \hat a(G) \xrightarrow{\res_{G,H}} \hat a(H) \xrightarrow{s'}
  \bC. \]
\end{cor}
\begin{proof}
This follows from Theorem~\ref{th:integral} and the going-up theorem.
\end{proof}

\section{Adams psi operations}\index{Adams psi operations}%
\index{psi operations}

The Adams psi operations $\psi^n$ on $a(G)$ were defined
in~\cite{Benson:1984b} and further studied 
in~\cite{Bryant:2003a,Bryant/Johnson:2010a,Bryant/Johnson:2011a}. 
In this section, we show that for any
ideal of indecomposables $\fX$, $\psi^n$ acts on $a_\fX(G)$
via ring endomorphisms. This gives rise to an action on $\Struct_\fX(G)$.
For the purpose of this section, we assume that $k$ is algebraically
closed, although this could be avoided.

Let $M$ be a $kG$-module. Then there is an action of $\bZ/n$ on
$M^{\otimes n}$ permuting the tensor factors, and this commutes with
the action of $G$. If $n$ is coprime to $p$ 
then $M^{\otimes n}$
decomposes as a sum of eigenspaces of $\bZ/n$.
Let $\ep\colon \bZ/n\to k$ be a generator for the character group of 
$\bZ/n$ over $k$, and write $[M^{\otimes n}]_{\ep^i}$ for the
eigenspace corresponding to $\ep^i$. Let
$\zeta_n=e^{2\pi\bi/n}$ be
a primitive $n$th root of unity in $\bC$, and set
\[ \psi^n[M]=\sum_{i=1}^n \zeta_n^i[M^{\otimes n}]_{\ep^i}\in a_\bC(G). \]

\begin{prop}
If $M_1$ and $M_2$ are $kG$-modules then
\begin{enumerate}
\item $\psi^n([M_1]+[M_2])=\psi^n[M_1]+\psi^n[M_2]$, and
\item $\psi^n([M_1\otimes M_2]=\psi^n[M_1]\psi^n[M_2]$.
\end{enumerate}
\end{prop}
\begin{proof}
This is proved in Proposition~1 of \cite{Benson:1984b}.
\end{proof}

It follows from the proposition that $\psi^n$ can be extended to a
$\bC$-algebra homomorphism $a_\bC(G)\to a_\bC(G)$.

\begin{prop}
Let $n$ be coprime to $p$.
For $d|n$, let $\ep_d=\ep^{n/d}$, a character of order $d$ of
$\bZ/n$. Then
\[ \psi^n[M]=\sum_{d|n}\mu(d)[M^{\otimes n}]_{\ep_d}. \]
Here, $\mu$ is the M\"obius function.\index{Mo@M\"obius function}
\end{prop}
\begin{proof}
This is proved in Proposition~2 of~\cite{Benson:1984b}.
\end{proof}

\begin{cor}
For $n$ coprime to $p$, the homomorphism $\psi^n$ takes elements of $a(G)$ to elements of
$a(G)$, and therefore induces a ring homomorphism $\psi^n\colon
a(G)\to a(G)$.
\end{cor}
\begin{proof}
The formula given in the proposition has integer coefficients.
\end{proof}

\begin{theorem}\label{th:psi-composite}
If $m$ and $n$ are coprime to $p$ then
$\psi^m\circ\psi^n=\psi^{mn}=\psi^n\circ\psi^m$.
\end{theorem}
\begin{proof}
This is proved in Theorem~1 of~\cite{Benson:1984b}.
\end{proof}

\begin{defn}[Frobenius twist]\index{Frobenius!twist}
For $r\ge 0$ we let $k^{(p^r)}$ be a $k$-$k$-bimodule, where the 
left action of $\lambda\in k$ is multiplication by $\lambda$, 
and the right action is multiplication by $\lambda^{p^r}$.
If $M$ is a $kG$-module, we define the $r$th
\emph{Frobenius twist} of $M$ to be 
\[ M^{(p^r)}=k^{(p^r)}\otimes_k M. \]
Thus $\lambda^{p^r}(1\otimes m)=\lambda^{p^r}\otimes m =
1 \otimes \lambda m$, so $\lambda^{p^r}$ acts on $M^{(p^r)}$ in
the way that $\lambda$ acts on $M$. If $k$ is perfect,\index{perfect!field} then 
$\lambda$ acts on $M^{(p^r)}$ is the way that $\lambda^{p^{-r}}$ acts
on $M$.
\end{defn}

\begin{defn}
We define $\psi^p[M]=[M^{(p)}]$, the Frobenius twist of $M$. This gives us
ring homomorphisms $\psi^p\colon a(G)\to a(G)$ and $\psi^p\colon
a_\bC(G)\to a_\bC(G)$.
In general, if $n=p^am$ with $p\nmid m$, we define
$\psi^n=\psi^m\circ(\psi^p)^a=(\psi^p)^a\circ\psi^m$. 
Thus Theorem~\ref{th:psi-composite} remains true for all $m$ and $n$.

If $\fX$ is an ideal of indecomposable $kG$-modules, we say that $\fX$
is \emph{Frobenius stable}\index{Frobenius!stable} if $M\in\fX
\Rightarrow M^{(p)}\in\fX$. In Example~\ref{eg:p}, everything is
Frobenius stable except for $\fX_\cV$, and in this case it is
Frobenius stable precisely when the subset $\cV$ is 
stable under the Frobenius map,\index{Frobenius!map} 
namely the map induced by
the $p$th power map on $H^*(G,k)$.

If $s$ is a species of $a(G)$, we define  $\psi^n(s)$ by the
formula $\psi^n(s)(x)=s(\psi^n(x))$.
\end{defn}

\begin{prop}
Regarding $\dim\colon a(G)\to \bC$ as a species, we have
\[ \psi^n(\dim)=\dim. \]
\end{prop}
\begin{proof}
The species $\dim\colon a(G)\to \bC$ factors through restriction 
to the trivial subgroup $1$, $a(G)\to a(1)$. We have
$a(1)=\bZ$, with generator $[k]$. We have $[k^{\otimes n}]_{1}=[k]$
and $[k^{\otimes n}]_{\ep^i}=0$ if $\ep^i\ne 1$, and so 
we have $\psi^n[k]=[k]$. So
$\psi^n$ is the identity map on $a(1)$.
\end{proof}

\begin{theorem}\label{th:psi-bounded}
For $x\in a_\bC(G)$ we have $\|\psi^n(x)\|\le\|x\|^n$.
\end{theorem}
\begin{proof}
We first verify this with $x=[M]$ and $n$ a prime. If $n=p$ we have
$\|\psi^p[M]\|=\dim(M^{(p)}) = \dim(M)=\|[M]\|\le \|[M]\|^p$.
If $n$ is a prime $q$ not equal to $p$ then 
\[ \psi^q[M]=[M^{\otimes q}]_1 - [M^{\otimes q}]_\ep \] 
where $\ep$ is a faithful character of $\bZ/q$. On the other hand, 
\[ [M]^q = [M^{\otimes q}]_1 + (q-1)[M^{\otimes q}]_\ep. \]
Thus $\|\psi^q[M]\|\le (\dim M)^q=\|[M]\|^q$. 

If $x=\sum_i a_i[M_i]$ and $n$ is prime then $\psi^n(x)=\sum_i a_i\psi^n[M_i]$ and
\[ \|\psi^n(x)\| \le \sum_i |a_i|\|\psi^n[M_i]\|\le
  \sum_i|a_i|\|[M_i]\|^n\le \|x\|^n. \]
Finally, if $n$ is composite, we use
Theorem~\ref{th:psi-composite} and induction.
\end{proof}

\begin{cor}\label{co:psi-conts}
$\psi^n$ is continuous on $a_\bC(G)$ with respect to the 
weighted $\ell^1$ norm.
\end{cor}
\begin{proof}
By Theorem~\ref{th:psi-bounded}, if $\|x\|\le 1$ then
$\|\psi^n(x)\|\le 1$. So $\psi^n$ is bounded, and hence continuous by 
Lemma~\ref{le:bd=cts}.
\end{proof}

\begin{cor}
If $s$ is a dimension bounded species, then so is $\psi^n(s)$.
\end{cor}
\begin{proof}
We use Theorem~\ref{th:core-bounded} with $\fX=\varnothing$.
This says that $s$ is dimension bounded if and only if it is
continuous with respect to the weighted $\ell^1$ norm on $a_\bC(G)$.
If $s$ is dimension bounded then it is continuous, so using
Corollary~\ref{co:psi-conts}, the composite
$a_\bC(G)\xrightarrow{\psi^n}a_\bC(G)\xrightarrow{s}\bC$ is
continuous, and hence dimension bounded.
\end{proof}

\begin{theorem}
Let $\fX$ be a Frobenius stable ideal of indecomposable $kG$-modules.
If $s$ is an $\fX$-core bounded species of $a(G)$ then $\psi^n(s)$ is
also $\fX$-core bounded.
\end{theorem}
\begin{proof}
It follows from the definition of $\psi^n$ that 
$\psi^n\langle\fX\rangle_\bC\le\langle\fX\rangle_\bC$.
So the composite 
\[ a_\bC(G) \xrightarrow{\psi^n} a_\bC(G) \xrightarrow{s} \bC \] 
has $\langle\fX\rangle_\bC$ in its kernel. This composite is
continuous, by Corollary~\ref{co:psi-conts}, and so it induces a
continuous map $a_{\bC,\fX}(G)\to\bC$. By
Theorem~\ref{th:core-bounded}, it follows that $s$ is $\fX$-core bounded.
\end{proof}

It follows from the theorem that we have an action of $\psi^n$ on
$\Struct_\fX(G)$.

\section{The Burnside ring}

In this final section, we take a brief look at an example of a
representation ring where our theory does not say much. 
If $G$ is a finite group, the 
\emph{Burnside ring}\index{Burnside ring}  $b(G)$ is the Grothendieck
ring of finite $G$-sets. It is a free abelian group whose basis
elements $[G/H]$ correspond to the transitive $G$-sets $G/H$ up to
isomorphism. Here, $G/H$ and $G/H'$ are isomorphic as $G$-sets if and
only  if $H$ and $H'$ are conjugate subgroups. Non-negative elements
of $b(G)$ are interpreted as isomorphism classes of finite $G$-sets,
and multiplication is given by direct product of $G$-sets. The element
$\rho$ is the regular representation $[G/1]$, and is the only
projective basis element.

It is well known that $b(G)$ is semisimple.\index{semisimple} Its
species\index{species!of $b(G)$} are as
follows. If $H$ is a subgroup of $G$ then $s_H\colon b(G) \to \bC$
sends a $G$-set $X$ to $|X^H|$, the number of fixed points of $H$ on
$X$. This is obviously a ring homomorphism, and two subgroups
give rise to the same species if and only if they are conjugate. 

A collection $\fX$ of basis elements $[G/H]$ is a representation ideal
if and only if it corresponds to a
collection $\cH$ of subgroups $H\le G$, closed under conjugacy,
closed under taking subgroups, and not containing $G$. 
Thus $\fX_{\max}$ corresponds to the collection of all proper
subgroups, while $\fX_\proj$ corresponds to the collection just
consisting of the trivial subgroup. Since there are as many species as
there are basis elements, $b_\bC(G)=\bC\otimes_\bZ b(G)$ is semisimple.

Dress~\cite{Dress:1969a} investigated idempotents in $b(G)$, and
discovered that there is a one to one correspondence between the
primitive idempotents\index{primitive idempotent}%
\index{idempotent!primitive} and conjugacy classes of 
perfect subgroups\index{perfect!subgroups} of
$G$. So the statement that the only idempotents in $b(G)$ 
are $0$ and $1$ is equivalent to the statement that $G$ is
soluble.\index{soluble group}

\chapter{Examples and Problems}\label{ch:eg}

\section{\texorpdfstring{The two dimensional module for $\bZ/5$}
{The two dimensional module for ℤ/5}}\label{se:Z5}%
\index{cyclic!group of order $5$}\index{group!cyclic of order $5$}

We begin with the example in the introduction
to~\cite{Benson/Symonds:2020a}.
Consider the two dimensional module $J_2$ for 
$G=\bZ/5=\langle g\mid g^5=1\rangle$ over a field $k$ of characteristic
five given by
\[ g \mapsto \begin{pmatrix} 1 & 1 \\ 0 & 1 \end{pmatrix}. \]

There are five indecomposable $kG$-modules $J_1,\dots,J_5$
corresponding to Jordan blocks of lengths between $1$ and $5$ and
eigenvalue $1$, representing the element $g$; the Jordan block of
length two is illustrated above, and the one of length five is the
projective indecomposable module. The table of tensor
products is as follows.
\[ \begin{array}{c|ccccc}
& J_1 & J_2 & J_3 & J_4 & J_5 \\ \hline
J_1 & J_1 & J_2 & J_3 & J_4 & J_5 \\
J_2 & J_2 & J_1 \oplus J_3 & J_2 \oplus J_4 & J_3 \oplus J_5 & 2J_5 \\
J_3 & J_3 & J_2 \oplus J_4 & J_1 \oplus J_3 \oplus J_5 & J_2 \oplus
2J_5 & 3J_5  \\
J_4 & J_4 & J_3 \oplus J_5 & J_2 \oplus 2J_5 & J_1 \oplus 3J_5
& 4J_5 \\
J_5 & J_5 & 2J_5 & 3J_5 & 4J_5 & 5J_5
\end{array} \]
It is clearly visible from this table that the only non-trivial representation
ideal is the projective ideal $\{J_5\}$.
The tensor powers $J_2^{\otimes n}$ are given by the columns of the
following table.
\[ \begin{array}{c|cccccccccccccccccccc}
n \to & 0 & 1 & 2 & 3 & 4 & 5 & 6 & 7 & 8   & 9   & 10   & 11   & 12 &
     \dots \\ \hline
J_1 & 1 & 0 & 1 & 0 & 2 & 0 & 5 & 0   & 13 & 0   & 34   & 0     & 89 \\
J_2 & 0 & 1 & 0 & 2 & 0 & 5 & 0 & 13 & 0   & 34 & 0     & 89   & 0 \\
J_3 & 0 & 0 & 1 & 0 & 3 & 0 & 8 & 0   & 21 & 0   & 55   & 0     & 144  \\
J_4 & 0 & 0 & 0 & 1 & 0 & 3 & 0 & 8   & 0   & 21 & 0     & 55   & 0 \\ \hline
J_5 & 0 & 0 & 0 & 0 & 1 & 2 & 7 & 14 & 36 & 72 & 165 & 330 & 715 
& \dots \\
\end{array} \]
The Fibonacci pattern in the non-projective summands is clear. 
It should come as no surprise that the dimension of the non-projective
part of $J_2^{\otimes n}$ is given by
$\cc_{2n}^G(J_2) \sim \tau^{2n+1}$ and $\cc_{2n+1}^G(J_2)\sim
2\tau^{2n+1}$ where 
\[ \tau= (1+\sqrt{5})/2= 2\cos(\pi/5)\sim 1.618034 \]
is the golden ratio.\index{golden ratio|textbf}\index{tau@$\tau$|textbf}
More precisely, we have
\[ \cc_{2n}^G(J_2)=F_{2n-1}+3F_{2n}=F_{2n}+F_{2n+2} 
= \tau^{2n+1}+\bar\tau^{2n+1}, \qquad
\cc_{2n+1}^G(J_2)=2\cc_{2n}^G(J_2), \]
where $F_n$ is the $n$th Fibonacci number and
\[ \bar\tau=-1/\tau=1-\tau=(1-\sqrt{5})/2 = 2\cos(3\pi/5) \sim -0.618034. \]
Thus we have $\npj_G(J_2)=\tau$. In the next section, we put this in a
broader context.

\section{\texorpdfstring{The cyclic group of order $p$}
{The cyclic group of order p}}\label{se:cyclic}%
\index{cyclic!group of order $p$}\index{group!cyclic of order $p$}%
\index{species!of $a(\bZ/p)$}

The classification of modules for the cyclic group of order $p$ in characteristic
$p$ is well known. Let $G=\langle g \mid g^p=1$ and $k$ be a field of
characteristic $p$. We have $(g-1)^p=g^p-1=0$ in $kG$, and so
if $M$ is a $kG$-module then $g-1$ acts nilpotently. It follows that
the eigenvalues of $g$ on $M$ are all equal to one, and the action of
$g$ on $M$ may be put into Jordan canonical form without extending the
field. So the indecomposable $kG$-modules correspond to Jordan blocks
with eigenvalue one and dimension between one and $p$:
\[ g \mapsto \begin{pmatrix} 
1 & 1 &0 & \cdots & 0 & 0 \\
0 & 1 & 1 & & 0 & 0 \\ 
\vdots &&&\ddots&&\vdots \\
0 & 0 &  & &1  & 1 \\
0 & 0 &  & \cdots & 0 & 1
\end{pmatrix}  \]
We write $J_j$ for the indecomposable $kG$-module of dimension $j$ for
$1\le j\le p$. The projective modules are the direct sums of copies of
$J_p$. Tensor products are determined by 
\[ J_2\otimes J_j \cong \begin{cases} J_2 & j=1 \\ 
J_{j+1}\oplus J_{j-1} & 2\le j\le p-1 \\
J_p \oplus J_p & j=p.
\end{cases} \]
It follows from these relations that every element of $a(G)$ is a
polynomial in $[J_2]$. If $p=2$, this shows that $a(G)\cong
\bZ[X]/(X^2-2X)$, where $X$ corresponds to $[J_2]$. So we shall
assume for the rest of this section that $p\ge 3$, in which case we
have, for example,
$[J_3]=[J_2]^2-\one$.

Next, we discuss the species of
$a(G)$.  This discussion is based on the work of
Green~\cite{Green:1962a}, Srinivasan~\cite{Srinivasan:1964a},
and involves the Chebyshev polynomials\index{Chebyshev polynomials|textbf} 
of the second kind $U_j(X)$.\index{U@$U_j(X)$|textbf} Background material on
these polynomials can be found in Rivlin~\cite{Rivlin:1974a}.

\begin{defn}
The Chebyshev polynomials of the second kind $U_j(X)$ are defined inductively for $j\ge 0$ by
$U_0(X)=1$, $U_1(X)=2X$, and for $j\ge 2$,
\[ U_j(X)=2XU_{j-1}(X)-U_{j-2}(X). \]
The first few are as follows:
\begin{align*}
U_0(X)&=1 & U_4(X)&=16X^4-12X^2+1\\
U_1(X)&=2X & U_5(X)&=32X^5-32X^3+6X\\
U_2(X)&=4X^2-1 & U_6(X)&=64X^6-80X^4+24X^2-1\\
U_3(X)&=8X^3-4X & U_7(X)&=128X^7-192X^5+80X^3-8X.
\end{align*}
and in general
\begin{equation}\label{eq:Uj}
U_j(X)=\sum_{i=0}^{\lfloor\frac{j}{2}\rfloor}(-1)^i\binom{j-i}{i}(2X)^{j-2i}. 
\end{equation}
\end{defn}

\begin{lemma}\label{le:Uj}
We have 
\[ U_j(\cos\theta) = \frac{\sin(j+1)\theta}{\sin\theta}. \]
The roots of $U_j(X)$ are real and distinct, symmetric about $X=0$,
and given by
\[ X = \cos(k\pi/(j+1)), \qquad 1\le k \le j. \]
\end{lemma}
\begin{proof}
The verification of the formula for $U_j(\cos\theta)$ is an easy
inductive argument using standard trigonometric identities.
The vanishing for these values of $X$ then follows from 
the fact that $\sin(k\pi)=0$ while $\sin(k\pi/(j+1))\ne 0$. 
Since the number of such roots is $j$,
which is the degree of $U_j(X)$, these are the only roots.
\end{proof}

{\small
\[ \begin{picture}(210,130)(-105,-65)
\put(0,-65){\line(0,1){130}}
\put(-105,0){\line(1,0){210}}
\put(-100,-3){\line(0,1){6}}
\put(-104,-9){$-1$}
\put(100,-3){\line(0,1){6}}
\put(98,-9){$1$}
\multiput(-3,-60)(0,10){13}{\line(1,0){6}}
\put(-15,-52){$-5$}
\put(-9,48){$5$}
\put(40,-40){$U_4(X)$}
\qbezier(-100,50)(-95,30)(-90,17.776)
\qbezier(-90,17.776)(-85,6.5)(-80,-1.264)
\qbezier(-80,-1.264)(-75,-7.5)(-70,-10.384)
\qbezier(-70,-10.384)(-65,-13)(-60,-12.464)
\qbezier(-60,-12.464)(-55,-12)(-50,-10)
\qbezier(-50,-10)(-45,-8)(-40,-5.104)
\qbezier(-40,-5.104)(-35,-2)(-30,0.496)
\qbezier(-30,0.496)(-25,3)(-20,5.456)
\qbezier(-20,5.456)(-15,7.5)(-10,8.816)
\qbezier(-10,8.816)(-5,10)(0,10)
\qbezier(0,10)(5,10)(10,8.816)
\qbezier(10,8.816)(15,7.5)(20,5.456)
\qbezier(20,5.456)(25,3)(30,0.496)
\qbezier(30,0.496)(35,-2)(40,-5.104)
\qbezier(40,-5.104)(45,-8)(50,-10)
\qbezier(50,-10)(55,-12)(60,-12.464)
\qbezier(60,-12.464)(65,-13)(70,-10.384)
\qbezier(70,-10.384)(75,-7.5)(80,-1.264)
\qbezier(80,-1.264)(85,6.5)(90,17.776)
\qbezier(90,17.776)(95,30)(100,50)
\end{picture}\]
\[ \begin{picture}(210,130)(-105,-65)
\put(0,-65){\line(0,1){130}}
\put(-105,0){\line(1,0){210}}
\put(-100,-3){\line(0,1){6}}
\put(-104,-9){$-1$}
\put(100,-3){\line(0,1){6}}
\put(98,-9){$1$}
\multiput(-3,-60)(0,10){13}{\line(1,0){6}}
\put(-15,-52){$-5$}
\put(-9,48){$5$}
\put(40,-40){$U_5(X)$}
\qbezier(-100,-60)(-95,-25)(-90,-9.6768)
\qbezier(-90,-9.6768)(-85,6)(-80,10.9824)
\qbezier(-80,10.9824)(-75,15)(-70,13.9776)
\qbezier(-70,13.9776)(-65,12.5)(-60,8.2368)
\qbezier(-60,8.2368)(-55,4.7)(-50,0)
\qbezier(-50,0)(-45,-4)(-40,-6.7968)
\qbezier(-40,-6.7968)(-35,-9.5)(-30,-10.1376)
\qbezier(-30,-10.1376)(-25,-11)(-20,-9.5424)
\qbezier(-20,-9.5424)(-15,-8)(-10,-5.6832)
\qbezier(-10,-5.6832)(-5,-3)(0,0)
\qbezier(0,0)(5,3)(10,5.6832)
\qbezier(10,5.6832)(15,8)(20,9.5424)
\qbezier(20,9.5424)(25,11)(30,10.1376)
\qbezier(30,10.1376)(35,9.5)(40,6.7968)
\qbezier(40,6.7968)(45,4)(50,0)
\qbezier(50,0)(55,-4.7)(60,-8.2368)
\qbezier(60,-8.2368)(65,-12.5)(70,-13.9776)
\qbezier(70,-13.9776)(75,-15)(80,-10.9824)
\qbezier(80,-10.9824)(85,-6)(90,9.6768)
\qbezier(90,9.6768)(95,25)(100,60)
\end{picture} \]
}

\begin{defn}\label{def:fj}
We define $f_j(X)=U_{j-1}(X/2)$. So we have $f_1(X)=1$, $f_2(X)=X$ and
for $j\ge 2$,
\begin{equation}\label{eq:Xfj} 
Xf_j(X)=f_{j+1}(X)+f_{j-1}(X). 
\end{equation}
The first few are as follows:
\begin{align*}
f_1(X)&=1 & f_6(X)&=X^5-4X^3+3X\\
f_2(X)&=X & f_7(X)&=X^6-5X^4+6X^2-1\\
f_3(X)&=X^2-1 & f_8(X)&=X^7-6X^5+10X^3-4X\\
f_4(X)&=X^3-2X & f_9(X)&=X^8-7X^6+15X^4-10X^2+1\\
f_5(X)&=X^4-3X^2+1& f_{10}(X)&=X^9-8X^7+21X^5-20X^3+5X.
\end{align*}

 By Lemma~\ref{le:Uj} we have
\begin{equation}\label{eq:fj2costheta}
f_j\left(\frac{\sin 2\theta}{\sin\theta}\right)=\frac{\sin
    j\theta}{\sin\theta} 
\end{equation}
(note that $\sin 2\theta/\sin\theta=2\cos\theta$).
\end{defn}

We write $\zeta_n$\index{zeta@$\zeta_n=e^{2\pi\bi/n}$} for
$e^{2\pi\bi/n}$, a primitive $n$th root of unity in $\bC$.
For the purposes of this section, we do not need part~(vii) of the
following lemma, only part~(vi). 
But part~(vii) will be used in Section~\ref{se:SL2q}.

\begin{lemma}\label{le:fpX}\ 

\begin{enumerate}
\item
$f_j(X)$ is a polynomial of degree $j-1$ in $X$ with integer
coefficients.
\item
If $j$ is odd then $f_j(X)$ is a polynomial in $X^2$ of degree
$(j-1)/2$.
\item
If $j$ is even then $f_j(X)$ is $X$ times a polynomial in $X^2$ of
degree $(j-2)/2$.
\item The roots of $f_j(X)$ are
  $\zeta_{2j}^k+\zeta_{2j}^{-k}=2\cos(k\pi/j)$ with $1\le k \le j-1$,
each with multiplicity one.
\item
If $j_1$ divides $j_2$ then $f_{j_1}(X)$ divides $f_{j_2}(X)$.
\item If $p>2$ is a prime then $f_p(X)$ is an irreducible polynomial
  in $X^2$ of degree $(p-1)/2$, whose roots are
  $\zeta_{2p}^k+\zeta_{2p}^{-k}$ with $1\le k \le p-1$. The splitting
  field of $f_p(X)$ is the real subfield%
\index{real subfield}\index{subfield, real} of the 
cyclotomic field\index{cyclotomic!field}\index{field!cyclotomic} of
  $p$th roots of unity.
\item More generally, if 
$p^m>2$ is a prime power then $f_{p^m}(X)/f_{p^{m-1}}(X)$ is
  an irreducible polynomial in $X^2$ of degree $p^{m-1}(p-1)/2$, whose
  roots are $\zeta_{2p^m}^k+\zeta_{2p^m}^{-k}$ with $1\le k\le p^m-1$
  not divisible by $p$. The splitting field of
  $f_{p^m}(X)/f_{p^{m-1}}(X)$ is the real subfield of the cyclotomic
  field of $2p^m$th roots of unity.
\end{enumerate}
\end{lemma}
\begin{proof}
Parts (i)--(iii) follow from the definition and induction on $j$.
Part (iv) follows from Lemma~\ref{le:Uj}. It follows from (iv) that if
$j_1$ divides $j_2$ then the roots of $f_{j_1}$ are among the roots of
$f_{j_2}$, which proves (v). Since (vi) is a special case of (vii), it
remains to prove (vii). This follows from the fact that the real
subfield of the cyclotomic field of $2p^m$th roots of unity
is generated by the element $\zeta_{2p^m}+\zeta_{2p^m}^{-1}$, whose conjugates are
the $\zeta_{2p^m}^k+\zeta_{2p^m}^{-k}$ with $1\le k \le p^m-1$ not
divisible by $p$, and the degree of this real subfield is
$p^{m-1}(p-1)/2$ provided $p^m>2$.
Note that if $p$ is odd then the cyclotomic field of $2p^m$th roots of
unity is the same as that of $p^m$th roots of unity. 
\end{proof}

\begin{theorem}\label{th:aG/aG1}
For $1\le j \le p$ we have $f_j([J_2])=[J_j]$ in $a(G)$.
We have 
\[ a(G)/a(G,1)\cong \bZ[X]/(f_p(X)), \] 
where $X$ corresponds to the element $[J_2]$. 
\end{theorem}
\begin{proof}
We have $[J_2][J_j]=[J_{j-1}]+[J_{j+1}]$ for $1\le j\le p-1$,
and we also have $f_2(X)f_j(X)=f_{j-1}(X)+f_{j+1}(X)$. So by
induction on $j$ we deduce that $f_j([J_2])=[J_j]$ for $1\le j\le p$.
The ideal $a(G,1)$ is generated by $[J_p]$, so we have a surjective
map $\bZ[X]\to a(G)/a(G,1)$ sending $X$ to $[J_2]$ and $f_j(X)$ to
$[J_j]$ for $1\le j\le p-1$, and whose kernel contains $f_p(X)$. Comparing
ranks, we see that this map induces an isomorphism $\bZ[X]/(f_p(X))\to
a(G)/a(G,1)$. 
\end{proof}

\begin{theorem}\label{th:cyclic-species}
The species of $a_\proj(G)=a(G)/a(G,1)$ are given by
\[ s_k[J_2]=2\cos(k\pi/p) \]
for $1\le k \le p-1$. We have
\[ s_k[J_j]=\frac{\sin(jk\pi/p)}{\sin(k\pi/p)} = f_j(2\cos(k\pi/p)). \]
\end{theorem}
\begin{proof}
By Theorem~\ref{th:aG/aG1} we have $a(G)/a(G,1)\cong \bZ[X]/(f_p(X))$
with $X$ corresponding to $J_2$. By Lemma~\ref{le:fpX} with $m=1$, $f_p(X)$ is
irreducible, and its roots are $X=2\cos(k\pi/p)$ with $1\le k\le p-1$.
So the ring homomorphisms $a(G)/a(G,1)\to \bC$ are given by
$s_k[J_2]=2\cos(k\pi/p)$. Still referring to Theorem~\ref{th:aG/aG1},
we have $[J_j]=f_j[J_2]$, and so
\[ s_k[J_j]=f_j(s_k[J_2])=f_j(2\cos(k\pi/p)), \] 
which by \eqref{eq:fj2costheta} is equal to $\sin(jk\pi/p)/\sin(k\pi/p)$.
\end{proof}

\begin{corqed}
The species of $a(G)$ are the core bounded ones, which are given in
Theorem~\ref{th:cyclic-species}, together with $s_0=\dim\colon a(G)\to
\bC$. 
\end{corqed}

\begin{rk}\label{rk:aG}
Since $[J_2][J_p]=2[J_p]$, the proof of Theorem~\ref{th:aG/aG1} shows
that 
\[ a(G)\cong \bZ[X]/((X-2)f_p(X)), \] 
with $X$ corresponding to $[J_2]$. 
\end{rk}

\begin{eg}
For $p=3$ and $p=5$ the species table of $\bZ/p$ is as 
follows:\index{tau@$\tau$}\index{golden ratio}
\[ \begin{array}{c|ccc} 
p=3 & s_0 & s_1 & s_2 \\ \hline
{}[J_1] & 1 & 1 & 1 \\
{}[J_2] & 2 & 1 & -1\phm \\
{}[J_3] & 3 & 0 & 0 
\end{array}\qquad
\begin{array}{c|ccccc}
p=5 & s_0 & s_1 & s_2 & s_3 & s_4 \\ \hline
{}[J_1] & 1 & 1 & 1 & 1 & 1 \\
{}[J_2] & 2 & \tau & -\bar\tau\phm & \bar\tau & -\tau\phm \\
{}[J_3] & 3 & \tau & \bar\tau & \bar\tau & \tau \\
{}[J_4] & 4 & 1 & -1\phm & 1 & -1\phm \\
{}[J_5] & 5 & 0 & 0 & 0 & 0
\end{array} \]
\end{eg}

\begin{theorem}
We have $\npj_G(J_2)=2\cos(\pi/p)$, and more generally
\[ \npj_G(J_j)=\sin(j\pi/p)/\sin(\pi/p). \]
\end{theorem}
\begin{proof}
It follows from Theorem~\ref{th:spec-radius-aG} 
that $\npj_G(J_2)$ is the largest value of a core bounded species on
$[J_2]$. The core bounded species are described in
Theorem~\ref{th:cyclic-species} and its corollary. 
The largest of the numbers $s_k[J_2]=2\cos(k\pi/p)$ for $1\le k\le p-1$ 
is $s_1[J_2]=2\cos(\pi/p)$, and so $\npj_G(J_2)=2\cos(\pi/p)$. 
By Theorem~\ref{th:PF}, for all $kG$-modules $M$ we have
$\npj_G(M)=s_1[M]$. In particular, $\npj_G(J_j)=s_1[J_j]=\sin(j\pi/p)/\sin(\pi/p)$.
\end{proof}

{\small
\begin{center}
\renewcommand{\arraystretch}{1.3}
\begin{tabular}{|c|c|c|c|c|c|c|c|c|}
\multicolumn{9}{c}{Approximate values of $\npj_G(J_j)$}\\\hline
$p$ & 3 & 5 & 7 & 11 & 13 & 17 & 19 & 23 \\ \hline
$J_1$ & 1.00000 & 1.00000 & 1.00000 & 1.00000 & 1.00000 & 1.00000 & 1.00000 & 1.00000 \\
$J_2$ & 1.00000 & 1.61803 & 1.80194 & 1.91899 & 1.94188 & 1.96595 & 1.97272 & 1.98137  \\
$J_3$ & 0.00000 & 1.61803 & 2.24698 & 2.68251 & 2.77091 & 2.86494 & 2.89163 & 2.92583 \\
$J_4$ &               & 1.00000 & 2.24698 & 3.22871 & 3.43891 & 3.66638 & 3.73167 & 3.81579 \\
$J_5$ &               & 0.00000 & 1.80194 & 3.51334 & 3.90704 & 4.34296 & 4.46992 & 4.63467 \\
$J_6$ &               &               & 1.00000 & 3.51334 & 4.14811 & 4.87165 & 5.08623 & 5.36722 \\
$J_7$ &               &               & 0.00000 & 3.22871 & 4.14811 & 5.23444 & 5.56381 & 5.99978 \\
$J_8$ &               &               &               & 2.68251 & 3.90704 & 5.41898 & 5.88962 & 6.52058 \\
\hline 
\end{tabular}\medskip
\end{center}
}

The Adams operations\index{Adams psi operations}%
\index{psi operations} 
$\psi^n$ on the representation rings of cyclic
groups were investigated by Almkvist~\cite{Almkvist:1981a},
Kouwenhoven~\cite{Kouwenhoven:1987a}, 
Bryant and Johnson~\cite{Bryant/Johnson:2010a,Bryant/Johnson:2011a},
Nam and Oh~\cite{Nam/Oh:2011a,Nam/Oh:2012a}. 

\begin{theorem}
Suppose that $p$ does not divide $n$. Then we have the following.
\begin{enumerate}
\item
For $1\le n\le p-1$,
$\psi^n[J_2]=[J_{n+1}]-[J_{n-1}]$.
\item
$\psi^{2p-n}=\psi^{2p+n}=\psi^n=\psi^{p^m n}$ on $a(G)$ for all $m\ge 0$.
\item
$\psi^n(s_k)[J_2]=2\cos nk\pi/p$.
\item
If $nk\equiv \pm j \pmod{2p}$ with $1\le j\le p-1$ then $\psi^n(s_k)=s_j$.
\end{enumerate}
\end{theorem}
\begin{proof}
(i) This is proved in Section~5 of~\cite{Almkvist:1981a}, see also
Theorem~5.1 of~\cite{Bryant/Johnson:2010a}.

(ii) The first two equalities are 
proved in Theorem~3.3 of~\cite{Bryant/Johnson:2010a}.
The third follows from the fact that the modules $J_i$ are defined
over $\bF_p$, so the Frobenius map $\psi^{p^m}$ acts as the identity map
on $a(G)$ for all $m\ge 0$.

(iii) Using (i), we have
\begin{align*}
\psi^n(s_k)[J_2]&=s_k(\psi^n[J_2])\\
&=s_k[J_{n+1}]-s_k[J_{n-1}]\\
&=(\sin((n+1)k\pi/p)-\sin((n-1)k\pi/p))/\sin(k\pi/p)\\
&=(2\cos nk\pi/p\sin k\pi/p)/\sin k\pi/p\\
&=2\cos nk\pi/p.
\end{align*}

(iv)
A species $s_j$ is determined by its value on $[J_2]$. So
this follows from (ii), together with the fact that if $m\equiv
\pm j \pmod{2p}$ then $\cos m\pi/p = \cos j\pi/p$.
\end{proof}

\section{Some Frobenius groups}%
\index{Frobenius!group}\index{group!Frobenius}\label{se:Frobgroup}

Let $G$ be a Frobenius group 
with cyclic normal Sylow $p$-subgroup and cyclic quotient.
The structure of $a(G)$ was investigated by
O'Reilly~\cite{O'Reilly:1964a,O'Reilly:1965a}, 
Lam~\cite{Lam:1968a}, Lam and Reiner~\cite{Lam/Reiner:1969a},
Benson and Carlson~\cite{Benson/Carlson:1986a}. In this section we
shall concentrate on the case $\bZ/p\rtimes\bZ/m$ 
where the Sylow $p$-subgroup has order $p$. We allow the action to
have a kernel, so we do not assume that $m$ divides $p-1$.
In a later section we
shall examine $\bZ/p^n\rtimes\bZ/m$, so we set up the notation in that
generality, and then impose the assumption that $n=1$.

Let $p^n$ be a prime power, let $m$ be a positive integer not divisible
by $p$, and let $q$ be a positive integer satisfying $q^{m}\equiv 1
\pmod{p^n}$. Let
\[ G = \langle g,h\mid g^{p^n}=1,\ h^{2m}=1, hgh^{-1}=g^q\rangle \cong
  \bZ/p\rtimes\bZ/2m, \]
a Frobenius group of order $2p^nm$,
and let $P=\langle g\rangle$, $H=\langle h\rangle$ as subgroups of
$G$. We remark that the representation ring of $G/\langle
h^m\rangle\cong \bZ/p^n\rtimes\bZ/m$ is contained in $a(G)$, and we
shall identify the image. But it turns out to be convenient to have
this extra central involution. This only really matters when $m$ is
even, but we do it in all cases for uniformity.
Let $k$ be a field of characteristic $p$ containing a primitive $2m$th
root of unity $\eta$. Let $d$ be chosen so that $\eta^{2d}\equiv q
\pmod p$. 

The element 
\[ x=\sum_{\substack{1\le j \le p^n-1\\ (p,j)=1}}g^j/j \]
spans an $H$-invariant complement to $J^2(kP)$ in $J(kP)$, and
$hxh^{-1}=qx$. We have the presentation
\[ kG=k\langle h,x\mid x^{p^n}=0, h^{2m}=1, hx=qxh\rangle. \]
There are $2m$ isomorphism classes of simple modules $S_i$, $i\in\bZ/2m$, all one
dimensional, corresponding to the characters of $H$. Letting $v_i$ be
a basis vector for $S_i$, we have $hv_i=\eta^iv_i$ and $xv_i=0$.
The space $\Ext^1_{kG}(S_i,S_j)$ is one dimensional if $j=i+2d$ and
zero dimensional otherwise. 
The projective indecomposable modules are uniserial of
length $p^n$, with composition factors (from the top) of $P_i$ being
$S_i$, $S_{i+2d}$, $S_{i+4d}$, \dots, $S_i$. Every indecomposable module
is a quotient of a projective indecomposable module.
We write $J_j$ ($1\le j\le p^n$) for the indecomposable module of length
$j$ with composition factors $S_{-d(j-1)}$,
$S_{-d(j-3)},\dots,S_{d(j-1)}$.
Note that it is only because of the extra involution that we are able
to do this symmetrically about the middle. 
A complete list of the
$2mp$ isomorphism classes of
indecomposable $kG$-modules is given by the modules $J_j\otimes S_i$
with $1\le j\le p^n$, $0\le i < 2m$.

Now assume that $n=1$. We have
\[ J_2\otimes J_j \cong\begin{cases} J_2 & j=1, \\
J_{j+1} \oplus J_{j-1} & 2\le j\le p-1 \\
(J_p \otimes S_d) \oplus (J_p \otimes S_{-d}) & j=p.
\end{cases} \]
This should be compared with the relations given in
Section~\ref{se:cyclic}, where the modules $J_j$ are the restrictions
to $P$ of the ones here.

\begin{theorem}\label{th:Frobgroup}
Let $G=\bZ/p\rtimes\bZ/2m$ as above.
We have
\[ a(G) \cong \bZ[X,Y]/(Y^{2m}-1,(X-Y^d-Y^{-d})f_p(X)), \]
where $X$ corresponds to $J_2$ and $Y$ corresponds to $S_1$, and the
polynomials $f_i$ are described in Definition~\ref{def:fj}.
\end{theorem}
\begin{proof}
This is proved in essentially the same way as Theorem~\ref{th:aG/aG1}
and Remark~\ref{rk:aG}, with $f_j(X)$ corresponding to $[J_j]$ for
$1\le j \le p$.
\end{proof}

\begin{cor}
The ring $a(G)$ is semisimple.\index{semisimple} 
Its $2pm$ species $s_{i,j}$ 
(with $0\le i<p,\ 0\le j<2m$) are given by
\begin{align*}
X&\mapsto \begin{cases} 
\zeta_{2m}^{dj}+\zeta_{2m}^{-dj}=2\cos(dj\pi/m) & i=0,  \\
\zeta_{2p}^i + \zeta_{2p}^{-i}=2\cos(i\pi/p) & 0 < i < p.
\end{cases} \\
Y&\mapsto \zeta^j_{2m}.
\end{align*}
The Brauer species are the ones with $i=0$.
\end{cor}
\begin{proof}
The assignment $Y\mapsto \zeta_{2m}^{j}$ satisfies the relation
$Y^{2m}-1=0$. If $i=0$ then $X \mapsto
\zeta_{2m}^{dj}+\zeta_{2m}^{-dj}$ makes the factor $(X-Y^d-Y^{-d})$
vanish. If $0<i<p-1$ then $X \mapsto \zeta_{2p}^i+\zeta_{2p}^{-i}$
makes the factor $f_p(X)$ vanish by Lemma~\ref{le:fpX}.

There are $2mp$ species $s_{i,j}$, and the
$\bZ$-rank of $a(G)$ is $2mp$, so there can be no more species, and
$a_\bC(G)$ is semisimple.
\end{proof}

To identify the ring $a(G)/\langle h^m\rangle$, 
we define some polynomials $\Dick_j$\index{E@$\Dick_j(y,z)$} of two variables as follows.

\begin{defn}\label{def:Fiyz}
Define \emph{Dickson polynomials}\index{Dickson polynomials} (of the
second kind)
$\Dick_j(y,z)$ inductively as 
follows: $\Dick_0(y,z)=1$, $\Dick_1(y,z)=z$, and
\[ \Dick_j(y,z)=z\Dick_{j-1}(y,z)-y\Dick_{j-2}(y,z)\qquad (j\ge 3). \]
Thus $\Dick_j(y,z)$ is an inhomogeneous polynomial of degree $j$. 
These are related to the 
Chebyshev polynomials\index{Chebyshev polynomials} 
of the second kind $U_j(X)$\index{U@$U_j(X)$} introduced in Section~\ref{se:cyclic} via
\[ \Dick_j(y,z)=y^{\frac{j}{2}}f_{j+1}\bigl(\textstyle\frac{z}{y^{1/2}}\bigr)=
y^{\frac{j}{2}}U_{j}\bigl(\textstyle\frac{z}{2y^{1/2}}\bigr). \]
There is an illusory choice of square root of $y$ here: as long as we take
$y^{\frac{j}{2}}$ to be the $j$th power of $y^{1/2}$, there is
no ambiguity in the answer. Indeed, by Equation~\ref{eq:Uj} we have
\[ \Dick_j(y,z) =
  \sum_{i=0}^{\lfloor\frac{j}{2}\rfloor}(-1)^i\binom{j-i}{i}y^{j-2i}(-z)^i. \]
\end{defn}

\begin{theorem}\label{th:aZprtimesZm}
We have
\[ a(\bZ/p\rtimes \bZ/m) = 
a(G/\langle h^m\rangle)\cong
\bZ[y,z]/(y^m-1,(z-y^d-1)\Dick_{p-1}(y^d,z)) \]
with $y=Y^2$ and
$z=XY^d$. This is a complete intersection of $\bZ$-rank $pm$, with
a $\bZ$-basis consisting of the monomials $y^iz^j$ with $0\le i<m$,
$0\le j < p$.
\end{theorem}
\begin{proof}
The modules $S_2$ and $J_2\otimes S_d$ generate the representation
ring of $G/\langle h^m\rangle$, and correspond to the elements $Y^2$
and $XY^d$. We have 
\[ \Dick_{p-1}(y^d,z)=\Dick_{p-1}(Y^{2d},XY^d)=Y^{(p-1)d}f_p(X), \] 
and so 
\[ (z-y^d-1)\Dick_{p-1}(y,z)=(X-Y^d-Y^{-d})Y^{pd}f_p(X). \] 
Since $Y^{pd}$ is
invertible, this is equivalent to the relation given in $a(G)$.
\end{proof}

\begin{cor}
The ring $a_\bC(\bZ/p\rtimes\bZ/m)$ is semisimple. 
Its $pm$ species $s_{i,j}$ $(0\le i<p,\ 
0\le j<m)$ are given by
\begin{align*}
y&\mapsto\zeta_{2m}^{2j} \\
z&\mapsto \begin{cases} 
\zeta_{2m}^{2dj}+1 & i=0, \\
(\zeta_{2p}^i+\zeta_{2p}^{-i})\zeta_{2m}^{dj} & 0<i<p.
\end{cases}
\end{align*}
The Brauer species are the ones with $i=0$. \qed
\end{cor}

\begin{eg}
The smallest example is the symmetric group%
\index{group!symmetric $S_3$}\index{symmetric!group $S_3$}
of degree three, which is
a semidirect product $\bZ/3 \rtimes \bZ/2$, in characteristic three.
The species table is as follows.
\[ \renewcommand{\arraystretch}{1.2}
\begin{array}{c|cccccc}
& s_{0,0} & s_{0,1} & s_{1,0} & s_{1,1} & s_{2,0} & s_{2,1} \\ \hline
\qquad [S_0]& 1 & 1 & 1 & 1  & 1 & 1 \\
y=[S_2]& 1 & -1\phm & 1 &  -1\phm & 1 & -1\phm \\
z=[J_2\otimes S_1] & 2 & 0 & 1 & \bi & -1\phm  & -\bi\phm \\
yz=[J_2 \otimes S_3] & 2 & 0 & 1 & -\bi\phm & -1\phm & \bi \\
\qquad [J_3] & 3 & 1 & 0 & 0  & 0 & 0 \\
\qquad [J_3\otimes S_2] & 3 & -1\phm & 0 & 0 & 0 & 0
\end{array} \]
\end{eg}

\begin{eg}
The species table for $\bZ/5\rtimes \bZ/4$ in characteristic five is
as follows, with $\tau=(1+\sqrt 5)/2$ and 
$\zeta=\zeta_8=e^{\pi \bi/4}=(1+\bi)/\sqrt 2$.\index{tau@$\tau$}\par
{\tiny
\[ \renewcommand{\arraystretch}{1.2}
\setlength{\arraycolsep}{-0.15mm}
\begin{array}{c|cccccccccccccccccccc}
& \,s_{0,0} & s_{0,1} & s_{0,2} & s_{0,3} 
& s_{1,0} & s_{1,1} & s_{1,2} & s_{1,3} 
& s_{2,0} & s_{2,1} & s_{2,2} & s_{2,3} 
& s_{3,0} & s_{3,1} & s_{3,2} & s_{3,2} 
& s_{4.0} & s_{4,1} & s_{4,2} & s_{4,3} \\ \hline
1 & 1 & 1 & 1 & 1 & 1 & 1 & 1 & 1 & 1 & 1 & 
1 & 1 & 1 & 1 & 1 & 1 & 1 & 1 & 1 & 1 \\
y & 1 & \bi & -1\phm & -\bi\phm & 1 
& \bi & -1\phm & -\bi\phm & 1 & \bi &
-1\phm & -\bi\phm & 1 & \bi & -1\phm &
-\bi\phm & 1 & \bi & -1\phm & -\bi\phm \\
y^2 & 1 & -1\phm & 1 & -1\phm & 1& 
-1\phm & 1 & -1\phm & 1 & -1\phm &
1 & -1\phm & 1 & -1\phm & 1 &
-1\phm & 1 & -1\phm & 1 & -1\phm \\
y^3 & 1 & -\bi\phm & -1\phm & \bi & 1 &
-\bi\phm & -1\phm & \bi & 1 & -\bi\phm &
-1\phm & \bi & 1 & -\bi\phm & -1\phm &
\bi & 1 & -\bi\phm & -1\phm & \bi \\
z & 2 & 1+\bi & 0 & 1-\bi & \tau 
& \zeta\tau & \bi\tau & -\bar\zeta\tau\phm  & -\bar\tau\phm
& -\zeta\bar\tau\phm & -\bi\bar\tau\phm & \bar\zeta\bar\tau 
& \bar\tau & \zeta\bar\tau & \bi\bar\tau & -\bar\zeta\bar\tau\phm 
& -\tau\phm & -\zeta\tau\phm & -\bi\tau\phm  & \bar\zeta\tau   \\
 & 2 &-1+\bi\phm& 0 &-1-\bi\phm& 
\tau & -\bar\zeta\tau\phm & -\bi\tau\phm & \zeta\tau & 
-\bar\tau\phm &\bar\zeta\bar\tau & \bi\bar\tau & -\zeta\bar\tau\phm & 
\bar\tau & -\bar\zeta\bar\tau\phm & -\bi\bar\tau\phm & \zeta\bar\tau & 
-\tau\phm & \bar\zeta\tau & \bi\tau & -\zeta\tau\phm   \\
 & 2 & -1-\bi\phm& 0 & -1+\bi\phm & 
\tau &-\zeta\tau\phm & \bi\tau & \bar\zeta\tau & 
-\bar\tau\phm & \zeta\bar\tau & -\bi\bar\tau\phm &
-\bar\zeta\bar\tau\phm & 
\bar\tau &-\zeta\bar\tau\phm & \bi\bar\tau & \bar\zeta\bar\tau 
& -\tau\phm & \zeta\tau & -\bi\tau\phm & -\bar\zeta\tau\phm  \\
 & 2 & 1-\bi & 0 & 1+\bi &\tau 
& \bar\zeta\tau & -\bi\tau\phm & -\zeta\tau\phm & 
-\bar\tau\phm & -\bar\zeta\bar\tau\phm & \bi\bar\tau & \zeta\bar\tau & 
\bar\tau &\bar\zeta\bar\tau & -\bi\bar\tau\phm & -\zeta\bar\tau\phm & 
-\tau\phm  & -\bar\zeta\tau\phm & \bi\tau & \zeta\tau  \\
 & 3 & \bi & 1 & -\bi\phm & \tau & 
\bi\tau & -\tau\phm &  -\bi\tau\phm & \bar\tau & 
\bi\bar\tau & -\bar\tau\phm & -\bi\bar\tau\phm & 
\bar\tau & \bi\bar\tau & -\bar\tau\phm & -\bi\bar\tau\phm &
\tau & \bi\tau & -\tau\phm & -\bi\tau\phm \\
 & 3 & -1\phm & -1\phm & -1\phm & 
\tau & -\tau\phm & \tau & -\tau\phm & 
\bar\tau & -\bar\tau\phm & \bar\tau & -\bar\tau\phm &
\bar\tau & -\bar\tau\phm & \bar\tau & -\bar\tau\phm &
\tau & -\tau\phm & \tau & -\tau\phm \\
 & 3 & -\bi\phm & 1 & \bi & 
\tau & -\bi\tau\phm & -\tau\phm & \bi\tau  &
\bar\tau & -\bi\bar\tau\phm & -\bar\tau\phm & \bi\bar\tau &
\bar\tau & -\bi\bar\tau\phm & -\bar\tau\phm & \bi\bar\tau &
\tau & -\bi\tau\phm & -\tau\phm & \bi\tau \\
 & 3 & 1 & -1\phm & 1 & \tau & \tau & \tau & \tau &
\bar\tau & \bar\tau & \bar\tau & \bar\tau &
\bar\tau & \bar\tau & \bar\tau & \bar\tau &
\tau & \tau & \tau & \tau \\
 & 4 & 0 & 0 & 0 & 1 &
-\bar\zeta\phm & -\bi\phm & \zeta & 
-1\phm & \bar\zeta & \bi & -\zeta\phm &
1 & -\bar\zeta\phm & -\bi\phm & \zeta &
-1\phm & \bar\zeta & \bi & -\zeta\phm \\
 & 4 & 0 & 0 & 0 & 
1 & -\zeta\phm & \bi & \bar\zeta &
-1\phm & \zeta & -\bi\phm & -\bar\zeta\phm &
1 & -\zeta\phm & \bi & \bar\zeta & 
-1\phm & \zeta & -\bi\phm & -\bar\zeta\phm \\
 & 4 & 0 & 0 & 0 & 
1 & \bar\zeta & -\bi\phm & -\zeta\phm &
-1\phm & -\bar\zeta\phm & \bi & \zeta &
1 & \bar\zeta & -\bi\phm & -\zeta\phm &
-1\phm & -\bar\zeta\phm & \bi & \zeta \\
 & 4 & 0 & 0 & 0 & 
1 & \zeta & \bi & -\bar\zeta\phm &
-1\phm & -\zeta\phm & -\bi\phm & \bar\zeta &
1 & \zeta & \bi & -\bar\zeta\phm &
-1\phm & -\zeta\phm & -\bi\phm & \bar\zeta \\
 & 5 & 1 & 1 & 1 & 0 & 0 & 0 & 0 & 0 & 0
& 0 & 0 & 0 & 0 & 0 & 0 & 0 & 0 & 0 & 0 \\
 & 5 & \bi & -1\phm & -\bi\phm & 0 & 0 & 0 & 0 & 0 & 0
& 0 & 0 & 0 & 0 & 0 & 0 & 0 & 0 & 0 & 0 \\
 & 5 & -1\phm & 1 & -1\phm & 0 & 0 & 0 & 0 & 0 & 0
& 0 & 0 & 0 & 0 & 0 & 0 & 0 & 0 & 0 & 0 \\
 & 5 & -\bi\phm & -1\phm & \bi & 0 & 0 & 0 & 0 & 0 & 0
& 0 & 0 & 0 & 0 & 0 & 0 & 0 & 0 & 0 & 0 \\
\end{array} \]
}
\end{eg}

\section{Cyclic central subgroups}\label{se:cyclic-central}%
\index{cyclic!central subgroups}\index{central subgroup, cyclic}

Let $g$ be a central element of a group $G$ of order $p$, and 
$Z=\langle g\rangle\le G$ with $Z\cong\bZ/p$. In this section
we describe $p-1$ algebra homomorphisms 
$\hat s_\ell\colon a_\bC(G)\to a_\bC(G/Z)$ ($1\le \ell\le p-1$).
We show that these extend to Banach algebra homomorphisms
$\hat s_\ell\colon \hat a(G) \to \hat a(G/Z)$.
In the next section we use this construction to examine the
species of $a(\bZ/p^n)$.

We write $X$ for
$g-1\in kG$, so $X^p=0$. If $M$ is a $kG$-module,
and $1\le i\le p$, we let
\[ F_i(M)=\frac{\Ker(X)\cap\im(X^{i-1})}{\Ker(X)\cap \im(X^i)}. \]
This is annihilated by $X$, so $g$ acts as the identity,
and $F_i(M)$ therefore has the structure of a $kG/Z$-module. 

As in Section~\ref{se:cyclic}, we write $[J_i]$, $1\le i\le p$, for the
indecomposable $kZ$-modules, and we write $s_j\colon a(Z)\to\bC$ for
the species, with $0\le j\le p-1$, described in
Theorem~\ref{th:cyclic-species} and its Corollary. 
The interpretation of $F_i$ is that it picks out the socles of the Jordan
blocks of $X$ on $M$ of length exactly $i$. 

\begin{lemma}
The dimension of $F_i(M)$ is equal to $[\res_{G,Z}(M):J_i]$.
\end{lemma}
\begin{proof}
Applying $\Ker(X)\cap\im(X^{i-1})$ to $J_j$, we get zero if $j<i$ and
$\Soc(J_j)$ if $j\ge i$. So the quotient
$(\Ker(X)\cap\im(X^{i-1}))/(\Ker(X)\cap\im(X^i))$ is isomorphic to
$\Soc(J_i)$ if $j=i$ and zero otherwise.
\end{proof}

\begin{prop}\label{pr:cijkFiFj}
If $J_i\otimes J_j= \bigoplus_k c_{i,j,k}J_k$ then for $1\le k\le p-1$
we have
\[ F_k(M \otimes N) \cong \bigoplus_{i,j}c_{i,j,k}F_i(M)\otimes
  F_j(N). \]
\end{prop}
\begin{proof}
This amounts to the semisimplicity of the coproduct on the simple
functors $F_k\colon\mmod(k\bZ/p)\to\vect(k)$ for $1\le k\le p-1$:
\[ \Delta F_k\cong\bigoplus_{i,j}c_{i,j,k}F_i \otimes F_j\colon
  \mmod(k\bZ/p) \times \mmod(k\bZ/p) \to \vect(k). \]
The easiest proof of semisimplicity is to notice that this functor is
isomorphic to the
contragredient dual functor sending the pair $(M,N)$ to
$(F_k(M^*\otimes N^*))^*$, and all the multiplicities
$c_{i,j,k}$ with $1\le k\le p-1$ are zero or one. More precisely, if
$i\le j$ and $i+j\le p$ then (Feit~\cite{Feit:1982a}, Theorem~VIII.2.7)
$J_i\otimes J_j \cong
\bigoplus_{s=1}^{i}J_{j-i+2s-1}$
and 
$J_{p-i}\otimes J_{p-j} \cong
(p-i-j)J_p\oplus (J_i\otimes J_j)$. 
Using the commutativity of tensor product, all cases are covered in
these two statements, since either $i+j\le p$ or $(p-i)+(p-j)\le p$.
Note that some of the multiplicities
$c_{i,j,p}$ are greater than one, and $\Delta F_p$ is not semisimple.
\end{proof}

\begin{theorem}\label{th:sellhat}
For $1\le \ell\le p-1$ the map
\[ \hat s_\ell\colon [M] \mapsto \sum_{k=1}^{p-1}s_\ell([J_k])[F_k(M)] \]
defines an algebra homomorphism $\hat s_\ell\colon a_\bC(G) \to
a_\bC(G/Z)$, which is the identity on the sub\-algebra $a_\bC(G/Z)\subseteq a_\bC(G)$.
This is continuous with respect to the norm, and so it extends to give a map
of Banach algebras $\hat s_j\colon \hat a(G) \to \hat a(G/Z)$.
\end{theorem}
\begin{proof}
To see that $\hat s_j$ is a ring homomorphism, we just have to check
multiplicativity. Using Proposition~\ref{pr:cijkFiFj} we have
\begin{align*}
\hat s_\ell([M])\hat s_\ell([N]) 
&=\sum_{i=1}^{p-1}s_\ell([J_i])[F_i(M)]\sum_{j=1}^{p-1}s_\ell([J_j])[F_j(N)] \\
&=\sum_{i,j=1}^{p-1}s_\ell([J_i\otimes J_j])[F_i(M)][F_j(N)] \\
&=\sum_{k=1}^{p-1}\sum_{i,j}c_{i,j,k}s_\ell([J_k])[F_i(M)\otimes F_j(N)] \\
&=\sum_{k=1}^{p-1}s_\ell([J_k])[F_k(M\otimes N)]\\
&=\hat s_\ell([M\otimes N]).
\end{align*}
To see that $\hat s_\ell$ is the identity on $a_\bC(G/Z)$,
we note that if $M$ is a $k(G/Z)$-module, regarded 
as a $kG$-module via inflation, then $F_1(M)=M$ and
$F_k(M)=0$ for $2\le k\le p$. Thus $\hat s_\ell([M])=
s_\ell([J_1])[M]=[M]$.

To prove continuity, we prove boundedness and use
Lemma~\ref{le:bd=cts}. If $x=\sum_{i\in\fI}a_i[M_i]$ then
\begin{align*} 
\|\hat s_\ell(x)\|&\le
\sum_{k=1}^{p-1}\sum_{i\in\fI}
|a_i||s_\ell(J_k)|\dim F_k(M_i) \\
&\le \sum_{i\in\fI}|a_i|\sum_{k=1}^{p-1}
k\dim F_k(M_i)\\
&=\sum_{i\in\fI}|a_i|\dim M_i =\|x\|.
\qedhere
\end{align*}
\end{proof}

\begin{rk}
When $p=2$, there is only one relevant value of $\ell$ in
Proposition~\ref{pr:cijkFiFj} and Theorem~\ref{th:sellhat}. In this case,
they both amount to the statement that $F_1(M\otimes N)\cong F_1(M)\otimes F_1(N)$.
This is known as the K\"unneth theorem.\index{Kunneth@K\"unneth theorem}
\end{rk}

If $s\colon a_\bC(G/Z)\to \bC$ is a species, then we compose to give $p-1$
species 
\[ s\hat s_\ell\colon a_\bC(G) \to \bC\qquad (1\le\ell\le p-1). \]

\section{\texorpdfstring{Cyclic $p$-groups}
{Cyclic p-groups}}%
\index{cyclic!group of order $p^n$}\index{group!cyclic of order $p^n$}%
\index{species!of $a(\bZ/p^n)$}

The structure of $a(G)$ with $G$ a cyclic group of order a power of a prime
was investigated by Almkvist and Fossum~\cite{Almkvist/Fossum:1978a}, 
Green~\cite{Green:1962a}, Renaud~\cite{Renaud:1979a}, 
Srinivasan~\cite{Srinivasan:1964a}. An abstract proof of
semisimplicity of $a(G)$ avoiding computations 
was given in Benson and Carlson~\cite{Benson/Carlson:1986a}.
In this section we
provide a new computational approach to the description of the species
of $a_\bC(\bZ/p^n)$,
and prove that it is semisimple.\index{semisimple}

It turns out that there is
one more ring homomorphism $\hat s_0\colon a(\bZ/p^{n+1})\to a(\bZ/p^n)$
that we need to use. It would be nice to have a more ``functorial''
way to describe this, but that does not seem so easy.

\begin{lemma}\label{le:s0hat}
The linear map defined by
\[ \hat s_0 \colon [J_{2bp^n\pm r}] \mapsto 2b[J_{p^n}]\pm [J_r] \]
is a ring homomorphism $a(\bZ/p^{n+1})\to a(\bZ/p^n)$, which is
the identity on the subring $a(\bZ/p^n)\subseteq a(\bZ/p^{n+1})$.
Here, $0\le b\le p/2$ and $0\le r \le p^n$. 

This formula covers all possible lengths of Jordan
blocks, and if $r=p^n$ the two different values of $b$ describing
the same module give the same answer.
\end{lemma}
\begin{proof}
The ring $a(\bZ/p^{n+1})$ is generated by the subring $a(\bZ/p^n)$
together with one more element $[J_{p^n+1}]$.
The effect of multiplying $[J_{p^n+1}]$ with a basis
element $[J_j]$ (copied from (2.8c)--(2.8e) of
Green~\cite{Green:1962a}, see also 
Section~I.1 of~\cite{Almkvist/Fossum:1978a}) is as
follows. We write $j=j_0p^n+j_1$ with $0\le j_1 \le p^n-1$. Then
\begin{align*}
[J_{p^n+1}][J_j]&=\begin{cases}
[J_{p^n+j_1}] + (j_1-1)[J_{p^n}] & j \le p^n \\
[J_{(j_0+1)p^n+j_1}]+(j_1-1)[J_{(j_0+1)p^n}]
& \\
\qquad {}+[J_{(j_0+1)p^n-j_1}]+(p^n-j_1-1)[J_{j_0p^n}]&\\
\qquad {}+[J_{(j_0-1)p^n+j_1}]
& p^n<j\le (p-1)p^n\\
(j_1+1)[J_{p^{n+1}}]+(p^n-j_1-1)[J_{(p-1)p^n}] 
\\
\qquad {}+[J_{(p-2)p^n+j_1}]& (p-1)p^n<j\le p^{n+1}
\end{cases}
\end{align*}
Now the map $\hat s_0$ preserves dimension, and takes $[J_{j_0p^n}]$
to $j_0[J_{p^n}]$ for $1\le j_0\le p$, so it suffices to work modulo the
basis elements of the form
$[J_{j_0p^n}]$. This makes the above relations easier to read:
\begin{align}
[J_{p^n+1}][J_j]&\equiv
\begin{cases}
[J_{p^n+j_1}] & j\le p^n \\
[J_{(j_0+1)p^n+j_1}]+[J_{(j_0+1)p^n-j_1}]&\\
\qquad {}+[J_{(j_0-1)p^n+j_1}]
&p^n< j\le (p-1)p^n \\
[J_{(p-2)p^n+j_1}] & (p-1)p^n<j\le p^{n+1}
\end{cases}
\label{eq:JJ}
\end{align}
Under the map $\hat s_0$, we must divide into
two cases according to the parity of $j_0$.
If $j_0$ is even, both sides go to $-[J_{p^n-j_1}]$,
whereas if $j_0$ is odd, both sides go to $[J_{j_1}]$.
\end{proof}

\begin{theorem}\label{th:species-of-Zpn}
The species $s_{\ell_0}\hat s_{\ell_1}\dots\hat
s_{\ell_{n-1}}\colon a_\bC(\bZ/p^n)\to\bC$ with $0\le \ell_i\le p-1$
for $0\le i\le n-1$ are all distinct. 
\end{theorem}
\begin{proof}
Consider the elements $\chi_0,\dots,\chi_{n-1}$ of $a_\bC(\bZ/p^n)$ given by 
\[ \chi_i=\begin{cases}[J_2] &i=0, \\ 
[J_{p^i+1}]-[J_{p^i-1}]&1\le i \le n-1.
\end{cases} \]
These generate $a_\bC(\bZ/p^n)$ as an algebra.
The element $\chi_i$ is in $a_\bC(\bZ/p^{i+1})\subseteq a_\bC(\bZ/p^n)$, and
the map $\hat s_\ell\colon a_\bC(\bZ/p^{i+1}) \to a_\bC(\bZ/p^i)$
evaluated at $\chi_i$ gives
\[ \hat s_\ell(\chi_i)=\begin{cases} 
2([J_{p^i}] - [J_{p^i-1}]) & \ell = 0 \\
2\cos \ell\pi/p\,[J_1] & 1\le \ell \le p-1.
\end{cases} \]
Furthermore, we have
\[ \hat s_\ell([J_{p^{i+1}}]-[J_{p^{i+1}-1}])=\begin{cases}
[J_{p^i}]-[J_{p^i-1}] & \ell = 0 \\
(-1)^\ell[J_1] & 1\le \ell \le p-1.
\end{cases} \]
Applying these formulas inductively, we have
\[ s_{\ell_0}\hat s_{\ell_1} \dots \hat s_{\ell_{n-1}} (\chi_i) =
\begin{cases}
\pm 2 & \ell_i = 0 \\
2\cos\ell_i\pi/p & 1\le \ell_i \le p-1.
\end{cases} \]
If $1\le \ell_i\le p-1$ then $|2\cos \ell_i\pi/p|<2$,
and since $\cos\colon[0,\pi]\to\bR$ is injective, 
$2\cos\ell_i\pi/p$ determines $\ell_i$.
Thus $\ell_i$ is determined by the value of 
$s_{\ell_0}\hat s_{\ell_1} \dots \hat s_{\ell_{n-1}}$
 on $\chi_i$.
\end{proof}

\begin{cor}
The algebra $a_\bC(\bZ/p^n)$ is
semisimple.\index{semisimple}
\end{cor}
\begin{proof}
By Lemma~\ref{le:species-indep}, the $p^n$ distinct species of 
the algebra $a_\bC(\bZ/p^n)$
described in Theorem~\ref{th:species-of-Zpn} are
linearly independent. This algebra has
dimension $p^n$. These species therefore give us an
isomorphism between $a_\bC(\bZ/p^n)$ and a product of
$p^n$ copies of $\bC$.
\end{proof}

\begin{eg}
Let $G=\bZ/4$.\index{cyclic!group of order $4$}%
\index{group!cyclic of order $4$}
The tensor product table is as follows, with the
obvious abbreviated notation.\par
{\tiny
\[ \begin{array}{|cccc|} \hline
1 & 2 & 3 & 4 \\
2 & 2^2 & 42 & 4^2 \\
3 & 42 & 4^21 & 4^3 \\
4 & 4^2 & 4^3 & 4^4\\ \hline
\end{array} \]
}
Then $F_j\colon a(G) \to a(\bZ/2)$, the $\hat s_j\colon a_\bC(G)\to
a_\bC(\bZ/2)$ and the $s_i\hat s_j$ are as follows.\par
{\tiny
\[ \renewcommand{\arraystretch}{1.2}
\begin{array}{|c||c|c||c|c||c|c|c|c|} \hline
&F_1&F_2&\hat s_0&\hat s_1&s_0\hat s_0&s_1\hat s_0 
& s_0\hat s_1&s_1\hat s_1\\ \hline
{}[J_1]&[J_1]&0&[J_1]&[J_1]&1&1&1&1\\
{}[J_2]&[J_2]&0&[J_2]&[J_2]&2&0&2&0\\
{}[J_3]&[J_1]&[J_1]&2[J_2]-[J_1]&[J_1]&3&-1\phm&1&1\\
{}[J_4]&0&[J_2]&2[J_2]&0&4&0&0&0\\ \hline
\end{array} \]
}
The last four columns of this table give the species table for $a(\bZ/4)$.
\end{eg}

\begin{eg}
Let $G=\bZ/8$.%
\index{cyclic!group of order $8$}\index{group!cyclic of order $8$}
The tensor product table is as follows.\par
{\tiny
\[ \setlength{\arraycolsep}{3.5mm}
\renewcommand{\arraystretch}{1.2}
\begin{array}{|cccccccc|} \hline
1&2&3&4&5&6&7&8 \\
2&2^2&42&4^2&64&6^2&86&8^2\\
3&42&4^21&4^3&74^2&864&8^25&8^3\\
4&4^2&4^3&4^4&84^3&8^24^2&8^34&8^4\\
5&64&74^2&84^3&8^24^21&8^342&8^43&8^5\\
6&6^2&864&8^24^2&8^342&8^42^2&8^52&8^6\\
7&86&8^25&8^34&8^43&8^52&8^61&8^7\\
8&8^2&8^3&8^4&8^5&8^6&8^7&8^8\\ \hline
\end{array} \]
}
Then the $F_j\colon a(G)\to a(\bZ/4)$, 
the $\hat s_j\colon a_\bC(G)\to a_\bC(\bZ/4)$ and the $s_i\hat s_j\hat
s_k$ are as follows.\par
{\tiny
\[ \setlength{\arraycolsep}{1mm}
\renewcommand{\arraystretch}{1.2}
\begin{array}{|c||cc||c|c||c|c|c|c|c|c|c|c|} \hline
&F_1&F_2&\hat s_0 & \hat s_1 & s_0\hat s_0\hat s_0 & 
s_1\hat s_0\hat s_0 & s_0\hat s_1\hat s_0 &
s_1\hat s_1\hat s_0 & s_0\hat s_0\hat s_1 &
s_1\hat s_0\hat s_1 & s_0\hat s_1\hat s_1 &
s_1\hat s_1\hat s_1 \\ \hline
{}[J_1]&[J_1]&0&[J_1]&[J_1]&1&1&1&1&1&1&1&1\\
{}[J_2]&[J_2]&0&[J_2]&[J_2]&2&0&2&0&2&0&2&0\\
{}[J_3]&[J_3]&0&[J_3]&[J_3]&3&-1\phm&1&1&3&-1\phm&1&1\\
{}[J_4]&[J_4]&0&[J_4]&[J_4]&4&0&0&0&4&0&0&0\\
{}[J_5]&[J_3]&[J_1]&2[J_4]-[J_3]&[J_3]&5&1&-1\phm&-1\phm&3&-1\phm&1&1\\
{}[J_6]&[J_2]&[J_2]&2[J_4]-[J_2]&[J_2]&6&0&-2\phm&0&2&0&2&0\\
{}[J_7]&[J_1]&[J_3]&2[J_4]-[J_1]&[J_1]&7&-1\phm&-1\phm&-1\phm&1&1&1&1\\
{}[J_8]&0&[J_4]&[2J_4]&0&8&0&0&0&0&0&0&0\\ \hline
\end{array} \]
}
The last eight columns of this table give the species table for
$a(\bZ/8)$.
\end{eg}

\begin{eg}
Let $G=\bZ/9$.\index{cyclic!group of order $9$}%
\index{group!cyclic of order $9$}
The tensor product table is as follows.\par
{\tiny
\[ \setlength{\arraycolsep}{3.5mm}
\renewcommand{\arraystretch}{1.2}
\begin{array}{|ccccccccc|} \hline
1&2&3&4&5&6&7&8&9\\
2&31&3^2&53&64&6^2&86&97&9^2\\
3 &3^2&3^3&63^2&6^23&6^3&96^2&9^26&9^3\\
4&53&63^2&7531&8642&96^23&9^264&9^35&9^4\\
5&64&6^23&8642&97531&9^263^2&9^353&9^44&9^5\\
6&6^2&6^3&96^23&9^263^2&9^33^3&9^43^2&9^53&9^6\\
7&86&96^2&9^264&9^353&9^43^2&9^531&9^62&9^7\\
8&97&9^26&9^35&9^44&9^53&9^62&9^71&9^8\\
9&9^2&9^3&9^4&9^5&9^6&9^7&9^8&9^9\\ \hline
\end{array} \]
}
Then the $F_j\colon a(G)\to a(\bZ/3)$, the $\hat
s_j\colon a_\bC(G) \to a_\bC(\bZ/3)$ and the $s_i\hat s_j$ are as
follows.\par
{\tiny
\[ \renewcommand{\arraystretch}{1.5}
\setlength{\arraycolsep}{0.6mm}
\begin{array}{|c||c|c|c||c|c|c||c|c|c|c|c|c|c|c|c|} \hline
& F_1 & F_2 & F_3  & \hat s_0&\hat s_1 & \hat s_2 &  s_0\hat s_0&
s_1\hat s_0 & s_2\hat s_0 & s_0\hat s_1 & s_1\hat s_1 & s_2\hat s_1
& s_0\hat s_2 & s_1\hat s_2 & s_2\hat s_2 \\ \hline
{}[J_1] & [J_1] & 0 & 0 & [J_1]& [J_1] & [J_1] 
& 1 & 1 & 1 & 1 & 1 & 1 & 1 & 1 & 1  \\
{}[J_2] & [J_2] & 0 & 0 & [J_2]& [J_2] & [J_2] 
& 2 & 1 & -1\phm & 2 & 1 & -1\phm & 2 & 1 & -1\phm \\
{}[J_3] & [J_3] & 0 & 0 & [J_3]&[J_3] & [J_3] 
& 3 & 0 & 0 & 3 & 0 & 0 & 3 & 0 & 0 \\
{}[J_4] & [J_2] & [J_1] & 0 & 2[J_3]-[J_2]& [J_2]+[J_1] & [J_2] -[J_1]
& 4 & -1\phm & 1 & 3 & 2 & 0 & 1 & 0 & -2\phm \\
{}[J_5] & [J_1] & [J_2] & 0 & 2[J_3]-[J_1]& [J_1]+[J_2] & [J_1]-[J_2] 
& 5 & -1\phm & -1\phm & 3 & 2 & 0 & -1\phm & 0 & 2 \\
{}[J_6] & 0 & [J_3] & 0 & 2[J_3]& [J_3] & -[J_3] & 
6 & 0 & 0 & 3 & 0 & 0 & -3\phm & 0 & 0 \\
{}[J_7] & 0 & [J_2] & [J_1] &2[J_3]+[J_1]&[J_2] & -[J_2] 
& 7 & 1 & 1 & 2 & 1 & -1\ & -2\phm & -1\phm & 1 \\
{}[J_8] & 0 & [J_1] & [J_2] & 2[J_3]+[J_2]& [J_1] & -[J_1] 
& 8 & 1 & -1\phm & 1 & 1 & 1 & -1\phm & -1\phm & -1\phm \\
{}[J_9] & 0 & 0 & [J_3] & 3[J_3]& 0 & 0 
& 9 & 0 & 0 & 0 & 0 & 0 & 0 & 0 & 0 \\ \hline
\end{array} \]
}
The last nine columns of this table give the species table for
$a(\bZ/9)$. 
\end{eg}

\begin{eg}
Let $G=\bZ/25$.\index{cyclic!group of order $25$}%
\index{group!cyclic of order $25$}
A quarter of the tensor product table is as follows;
the rest may be deduced using $\Omega(M)\otimes N\cong 
\Omega(M\otimes N)\oplus{}$(projective) (Schanuel's lemma), and so on.\par
{\tiny
\[ \renewcommand{\arraystretch}{1.5}
\setlength{\arraycolsep}{1.2mm}
\renewcommand{\,}{\hspace{2pt}}
\newcommand{\x}{\hspace{0.5pt}}
\hspace{-4mm}
\begin{array}{|c|ccccccc|} \hline
1&2&3&4&5&6&7&8\\ \hline
2&3\,1&4\,2&5\,3&5^2&7\,5&8\,6&9\,7\\
3&4\,2&5\,3\,1& 5^2\x 2&5^3&8\,5^2&9\,7\,5&
10\,8\,6\\
4&5\,3&5^2\,2&5^3\x 1&5^4&9\,5^3&10\,8\,5^2&
10^2\x 7\,5\\
5&5^2&5^3&5^4&5^5&10\,5^4&10^2\x 5^3&10^3\x 5^2\\
6&7\,5&8\,5^2&9\,5^3&10\,5^4&11\,9\,5^3\,1&
12\,10\,8\,5^2\,2&13\,10^2\x 7\,5\,3\\
7&8\,6&9\,7\,5&10\,8\,5^2&10^2\x 5^3&12\,10\,8\,5^2\x 2&
13\,11\,9\,7\,5\,3\,1&14\,12\,10\,8\,6\,4\,2\\
8&9\,7&10\,8\,6&10^2\x 7\,5&10^3\x 5^2&
13\,10^2\x 7\,5\,3&14\,12\,10\,8\,6\,4\,2&
15\,13\,11\,9\,7\,5\,3\,1\\
9&10\,8&10^2\,7&10^3\x 6&10^4\x 5&14\,10^3\x 6\,4&
15\,13\,10^2\x 7\,5\,3&15^2\x 12\,10\,8\,5^2\x 2\\
10&10^2&10^3&10^4&10^5&15\,10^4\x 5&15^2\x 10^3\x 5^2&
15^3\x 10^2\x 5^3\\
11&12\,10&13\,10^2&14\,10^3&15\,10^4&16\,14\,10^3\,6&
17\,15\,13\,10^2\x 7\,5&18\,15^2\x 12\,10\,8\,5^2\\
12&13\,11&14\,12\,10&15\,13\,10^2&15^2\,10^3&17\,15\,13\,10^2\x 7\,5&
18\,16\,14\,12\,10\,8\,6\,4&19\,17\,15\,13\,11\,9\,7\,5\\ 
\hline
\end{array}\bigskip \]
\[ \renewcommand{\arraystretch}{1.5}
\setlength{\arraycolsep}{1.2mm}
\renewcommand{\,}{\hspace{2pt}}
\newcommand{\x}{\hspace{0.3pt}}
\begin{array}{|c|cccc|} \hline
1&9&10&11&12\\ \hline
2&10\,8&10^2&12\,10&13\,11\\
3&10^2\x 7&10^3&13\,10^2&14\,12\,10\\
4&10^3\x 6&10^4&14\,10^3&15\,13\,10^2\\
5&10^4\x 5&10^5&15\,10^4&15^2\x 10^3\\
6&14\,10^3\x 6\,4&15\,10^4\x 5&16\,14\,10^3\x 6&
17\,15\,13\,10^2\x 7\\
7&15\,13\,10^2\x 7\,5\,3&15^2\x 10^3\x 5^2&
17\,15\,13\,10^2\x 7\,5&18\,16\,14\,12\,10\,8\,6\\
8&15^2\x 12\,10\,8\,5^2\x 2&
15^3\x 10^2\x 5^3&18\,15^2\x 12\,10\,8\,5^2&
19\,17\,15\,13\,11\,9\,7\,5\\
9&15^3\x 11\,9\,5^3\x 1&15^4\x 10\,5^4&
19\,15^3\x 11\,9\,5^3&20\,18\,15^2\x 12\,10\,8\,5^2\\
10&15^4\x 10\,5^4&15^5\x 5^5&
20\,15^4\x 10\,5^4&20^2\x 15^3\x 10^2\x 5^3\\
11&19\,15^3\x 11\,9\,5^3&20\,15^4\x 10\,5^4&
21\,19\,15^3\x 11\,9\,5^3\x 1&
22\,20\,18\,15^2\x 12\,10\,8\,5^2\x 2\\
12&20\,18\,15^2\x 12\,10\,8\,5^2&20^2\x 15^3\x 10^2\x 5^3&
22\,20\,18\,15^2\x 12\,10\,8\,5^2\x 2&
23\,21\,19\,17\,15\,13\,11\,9\,7\,5\,3\,1\\ 
\hline
\end{array}\bigskip \]
}
Then the $F_j\colon a(G)\to a(\bZ/5)$, the $\hat
s_j\colon a_\bC(G) \to a_\bC(\bZ/5)$ and the $s_i\hat s_j$ are as
follows.\index{tau@$\tau$}\index{golden ratio}\par
{\tiny
\[ \renewcommand{\arraystretch}{1.1}
\setlength{\arraycolsep}{0.73mm}\hspace{-2cm}
\begin{array}{|c||ccccc||c|c|c|c|c||c|c|c|c|c|c|c|c|c|c|c|} \hline
&F_1&F_2&F_3&F_4&F_5&\hat s_0&\hat s_1& \hat s_2& \hat s_3 & \hat s_4 &
s_0\hat s_0&s_1\hat s_0&s_2\hat s_0&s_3\hat s_0&
s_4\hat s_0 \\ \hline
{}[J_1]&[J_1]&0&0&0&0&[J_1]&[J_1]&[J_1]&[J_1]&[J_1]
&1&\phm 1&\phm 1&\phm 1&\phm 1\\
{}[J_2]&[J_2]&0&0&0&0&[J_2]&[J_2]&[J_2]&[J_2]&[J_2]
&2&\phm\tau&-\bar\tau&\phm\bar\tau&-\tau\\
{}[J_3]&[J_3]&0&0&0&0&[J_3]&[J_3]&[J_3]&[J_3]&[J_3]
&3&\phm\tau&\phm\bar\tau&\phm\bar\tau&\phm\tau\\
{}[J_4]&[J_4]&0&0&0&0&[J_4]&[J_4]&[J_4]&[J_4]&[J_4]
&4&\phm 1&-1&\phm 1&-1\\
{}[J_5]&[J_5]&0&0&0&0&[J_5]&[J_5]&[J_5]&[J_5]&[J_5]
&5&\phm 0&\phm 0&\phm 0&\phm 0\\
{}[J_6]&[J_4]&[J_1]&0&0&0&2[J_5]-[J_4]&[J_4]+\tau[J_1]&
[J_4]-\bar\tau[J_1]&[J_4]+\bar\tau[J_1]&[J_4]-\tau[J_1]
&6&-1&\phm 1&-1&\phm 1\\
{}[J_7]&[J_3]&[J_2]&0&0&0&2[J_5]-[J_3]&[J_3]+\tau[J_2]&
[J_3]-\bar\tau[J_2]&[J_3]+\bar\tau[J_2]&[J_3]-\tau[J_2]
&7&-\tau&-\bar\tau&-\bar\tau&-\tau\\
{}[J_8]&[J_2]&[J_3]&0&0&0&2[J_5]-[J_2]&[J_2]+\tau[J_3]&
[J_2]-\bar\tau[J_3]&[J_2]+\bar\tau[J_3]&[J_2]-\tau[J_3]
&8&-\tau&\phm\bar\tau&-\bar\tau&\phm\tau\\
{}[J_9]&[J_1]&[J_4]&0&0&0&2[J_5]-[J_1]&[J_1]+\tau[J_4]&
[J_1]-\bar\tau[J_4]&[J_1]+\bar\tau[J_4]&[J_1]-\tau[J_4]
&9&-1&-1&-1&-1\\
{}[J_{10}]&0&[J_5]&0&0&0&2[J_5]&\tau[J_5]&
-\bar\tau[J_5]\phm&\bar\tau[J_5]&-\tau[J_5]
&10&\phm 0&\phm 0&\phm 0&\phm 0\\
{}[J_{11}]&0&[J_4]&[J_1]&0&0&2[J_5]+[J_1]&\tau([J_4]+[J_1])&
\bar\tau([J_1]-[J_4])&\bar\tau([J_4]+[J_1])&\tau([J_1]-[J_4])
&11&\phm 1&\phm 1&\phm 1&\phm 1\\
{}[J_{12}]&0&[J_3]&[J_2]&0&0&2[J_5]+[J_2]&\tau[J_3]+\tau[J_2]&
\bar\tau([J_2]-[J_3])&\bar\tau([J_3]+[J_2])&\tau([J_2]-[J_3])
&12&\phm\tau&-\bar\tau&\phm\bar\tau&-\tau\\
{}[J_{13}]&0&[J_2]&[J_3]&0&0&2[J_5]+[J_3]&\tau[J_2]+\tau[J_3]&
\bar\tau([J_3]-[J_2])&\bar\tau([J_2]+[J_3])&\tau([J_3]-[J_2])
&13&\phm\tau&\phm\bar\tau&\phm\bar\tau&\phm\tau\\
{}[J_{14}]&0&[J_1]&[J_4]&0&0&2[J_5]+[J_4]&\tau[J_1]+\tau[J_4]&
\bar\tau([J_4]-[J_1])&\bar\tau([J_1]+[J_4])&\tau([J_4]-[J_1])
&14&\phm 1&-1&\phm 1&-1\\
{}[J_{15}]&0&0&[J_5]&0&0&3[J_5]&\tau[J_5]&
\bar\tau[J_5]&\bar\tau[J_5]&\tau[J_5]
&15&\phm 0&\phm 0&\phm 0&\phm 0\\
{}[J_{16}]&0&0&[J_4]&[J_1]&0&4[J_5]-[J_4]&\tau[J_4]+[J_1]&
\bar\tau[J_4]-[J_1]&\bar\tau[J_4]+[J_1]&\tau[J_4]-[J_1]
&16&-1&\phm 1&-1&\phm 1\\
{}[J_{17}]&0&0&[J_3]&[J_2]&0&4[J_5]-[J_3]&\tau[J_3]+[J_2]&
\bar\tau[J_3]-[J_2]&\bar\tau[J_3]+[J_2]&\tau[J_3]-[J_2]
&17&-\tau&-\bar\tau&-\bar\tau&-\tau\\
{}[J_{18}]&0&0&[J_2]&[J_3]&0&4[J_5]-[J_2]&\tau[J_2]+[J_3]&
\bar\tau[J_2]-[J_3]&\bar\tau[J_2]+[J_3]&\tau[J_2]-[J_3]
&18&-\tau&\phm\bar\tau&-\bar\tau&\phm\tau\\
{}[J_{19}]&0&0&[J_1]&[J_4]&0&4[J_5]-[J_1]&\tau[J_1]+[J_4]&
\bar\tau[J_1]-[J_4]&\bar\tau[J_1]+[J_4]&\tau[J_1]-[J_4]
&19&-1&-1&-1&-1\\
{}[J_{20}]&0&0&0&[J_5]&0&4[J_5]&[J_5]&
-[J_5]\phm&[J_5]&-[J_5]\phm
&20&\phm 0&\phm 0&\phm 0&\phm 0\\
{}[J_{21}]&0&0&0&[J_4]&[J_1]&4[J_5]+[J_1]&[J_4]&
-[J_4]\phm&[J_4]&-[J_4]\phm
&21&\phm 1&\phm 1&\phm 1&\phm 1\\
{}[J_{22}]&0&0&0&[J_3]&[J_2]&4[J_5]+[J_2]&[J_3]&
-[J_3]\phm&[J_3]&-[J_3]\phm
&22&\phm\tau&-\bar\tau&\phm\bar\tau&-\tau\\
{}[J_{23}]&0&0&0&[J_2]&[J_3]&4[J_5]+[J_3]&[J_2]&
-[J_2]\phm&[J_2]&-[J_2]\phm
&23&\phm\tau&\phm\bar\tau&\phm\bar\tau&\phm\tau\\
{}[J_{24}]&0&0&0&[J_1]&[J_4]&4[J_5]+[J_4]&[J_1]&
-[J_1]\phm&[J_1]&-[J_1]\phm
&24&\phm 1&-1&\phm 1&-1\\
{}[J_{25}]&0&0&0&0&[J_5]&5[J_5]&0&
0&0&0&25&\phm 0&\phm 0&\phm 0&\phm 0\\\hline
\end{array}\hspace{-2cm} \]
\[ \renewcommand{\arraystretch}{1.1}
\setlength{\arraycolsep}{0.8mm}\hspace{-2cm}
\begin{array}{|c||c|c|c|c|c|c|c|c|c|c|c|c|c|c|c|c|c|c|c|c|}
  \hline &
s_0\hat s_1&s_1\hat s_1&s_2\hat s_1&s_3\hat s_1&s_4\hat s_1&
s_0\hat s_2&s_1\hat s_2&s_2\hat s_2&s_3\hat s_2&s_4\hat s_2&
s_0\hat s_3&s_1\hat s_3&s_2\hat s_3&s_3\hat s_3&s_4\hat s_3&
s_0\hat s_4&s_1\hat s_4&s_2\hat s_4&s_3\hat s_4&s_4\hat s_4 \\ \hline
{}[J_1]&
1&1&\phm 1&\phm 1&\phm 1&
1&\phm 1&\phm 1&\phm 1&\phm 1&
1&\phm 1&\phm 1&1&\phm 1&
1&\phm 1&\phm 1&\phm 1&\phm 1 \\
{}[J_2]&
2&\tau&-\bar\tau&\phm\bar\tau&-\tau&
2&\phm\tau&-\bar\tau&\phm\bar\tau&-\tau&
2&\phm\tau&-\bar\tau&\bar\tau&-\tau&
2&\phm\tau&-\bar\tau&\phm \bar\tau&-\tau\\
{}[J_3]&
3&\tau&\phm\bar\tau&\phm\bar\tau&\phm\tau&
3&\phm\tau&\phm\bar\tau&\phm\bar\tau&\phm\tau&
3&\phm\tau&\phm\bar\tau&\bar\tau&\phm\tau&
3&\phm\tau&\phm\bar\tau&\phm\bar\tau&\phm\tau\\
{}[J_4]&
4&1&-1&\phm 1&-1&
4&\phm 1&-1&\phm 1&-1&
4&\phm 1&-1&1&-1&
4&\phm 1&-1&\phm 1&-1\\
{}[J_5]&
5&0&\phm 0 &\phm 0 &\phm 0&
5&\phm 0&\phm 0&\phm 0&\phm 0&
5&\phm 0&\phm 0& 0 &\phm 0&
5&\phm 0&\phm 0&\phm 0&\phm 0\\
{}[J_6]&
4+\tau&\tau^2&-\bar\tau&\phm\tau^2&-\bar\tau&
4-\bar\tau&\phm\tau&-\bar\tau^2&\phm\tau&-\bar\tau^2&
4+\bar\tau&\phm\bar\tau^2&-\tau&\bar\tau^2&-\tau&
4-\tau&\phm\bar\tau&-\tau^2&\phm\bar\tau&-\tau^2\\
{}[J_7]&
3+2\tau&\tau^3&\phm\bar\tau^2&-\tau&-1&
3-2\bar\tau&\phm\tau^2&\phm\bar\tau^3&-1&-\bar\tau&
3+2\bar\tau&-\bar\tau&-1&\bar\tau^3&\phm\tau^2&
3-2\tau&-1&-\tau&\phm\bar\tau^2&\phm\tau^3\\
{}[J_8]&
2+3\tau&\tau^3&-\bar\tau^2&-\tau&\phm 1&
2-3\bar\tau&\phm\tau^2&-\bar\tau^3&-1&\phm\bar\tau&
2+3\bar\tau&-\bar\tau&\phm 1&\bar\tau^3&-\tau^2&
2-3\tau&-1&\phm\tau&\phm\bar\tau^2&-\tau^3\\
{}[J_9]&
1+4\tau&\tau^2&\phm\bar\tau&\phm\tau^2&\phm\bar\tau&
1-4\bar\tau&\phm\tau&\phm\bar\tau^2&\phm\tau&\phm\bar\tau^2&
1+4\bar\tau&\phm\bar\tau^2&\phm\tau&\bar\tau^2&\phm\tau&
1-4\tau&\phm\bar\tau&\phm\tau^2&\phm\bar\tau&\phm\tau^2\\
{}[J_{10}]&
5\tau&0&\phm 0&\phm 0&\phm 0&
-5\bar\tau&\phm 0&\phm 0&\phm 0&\phm 0&
5\bar\tau&\phm 0&\phm 0&0&\phm 0&
-5\tau&\phm 0&\phm 0&\phm 0&\phm 0\\
{}[J_{11}]&
5\tau&2\tau&\phm 0&\phm 2\tau&\phm 0&
-3\bar\tau&\phm 0&\phm 2\bar\tau&\phm 0&\phm 2\bar\tau&
5\bar\tau&\phm 2\bar\tau&\phm 0&2\bar\tau&\phm 0&
-3\tau&\phm 0&\phm 2\tau&\phm 0&\phm 2\tau\\
{}[J_{12}]&
5\tau&2\tau^2&\phm 0&-2&\phm 0&
-\bar\tau&\phm 0&-2\bar\tau^2&\phm 0&\phm 2&
5\bar\tau&-2&\phm 0&2\bar\tau^2&\phm 0&
-\tau&\phm 0&\phm 2&\phm 0&-2\tau^2\\
{}[J_{13}]&
5\tau&2\tau^2&\phm 0&-2&\phm 0&
\bar\tau&\phm 0&\phm 2\bar\tau^2&\phm 0&-2&
5\bar\tau&-2&\phm 0&2\bar\tau^2&\phm 0&
\tau&\phm 0&-2&\phm 0&\phm 2\tau^2\\
{}[J_{14}]&
5\tau&2\tau&\phm 0&\phm 2\tau&\phm 0&
3\bar\tau&\phm 0&-2\bar\tau&\phm 0&-2\bar\tau&
5\bar\tau&\phm 2\bar\tau&\phm 0&2\bar\tau&\phm 0&
3\tau&\phm 0&-2\tau&\phm 0&-2\tau\\
{}[J_{15}]&
5\tau&0&\phm 0&\phm 0&\phm 0&
5\bar\tau&\phm 0&\phm 0&\phm 0&\phm 0&
5\bar\tau&\phm 0&\phm 0&0&\phm 0&
5\tau&\phm 0&\phm 0&\phm 0&\phm 0\\
{}[J_{16}]&
1+4\tau&\tau^2&\phm\bar\tau&\phm\tau^2&\phm\bar\tau&
4\bar\tau-1&-\tau&-\bar\tau^2&-\tau&-\bar\tau^2&
1+4\bar\tau&\phm\bar\tau^2&\phm\tau&\bar\tau^2&\phm\tau&
4\tau-1&-\bar\tau&-\tau^2&-\bar\tau&-\tau^2\\
{}[J_{17}]&
2+3\tau&\tau^3&-\bar\tau^2&-\tau&\phm 1&
3\bar\tau-2&-\tau^2&\phm\bar\tau^3&\phm 1&-\bar\tau&
2+3\bar\tau&-\bar\tau&\phm 1&\bar\tau^3&-\tau^2&
3\tau-2&\phm 1&-\tau&-\bar\tau^2&\tau^3\\
{}[J_{18}]&
3+2\tau&\tau^3&\phm\bar\tau^2&-\tau&-1&
2\bar\tau-3&-\tau^2&-\bar\tau^3&\phm 1&\phm\bar\tau&
3+2\bar\tau&-\bar\tau&-1&\bar\tau^3&\phm \tau^2&
2\tau-3&\phm 1&\phm\tau&-\bar\tau^2&-\tau^3\\
{}[J_{19}]&
4+\tau&\tau^2&-\bar\tau&\phm\tau^2&-\bar\tau&
\bar\tau-4&-\tau&\phm\bar\tau^2&-\tau&\phm\bar\tau^2&
4+\bar\tau&\phm\bar\tau^2&-\tau&\bar\tau^2&-\tau&
\tau-4&-\bar\tau&\phm\tau^2&-\bar\tau&\phm\tau^2\\
{}[J_{20}]&
5&0&\phm 0&\phm 0&\phm 0&
-5&\phm 0&\phm 0&\phm 0&\phm 0&
5&\phm 0&\phm 0&0&\phm 0&
-5&\phm 0&\phm 0&\phm 0&\phm 0\\
{}[J_{21}]&
4&1&-1&\phm 1&-1&
-4&-1&\phm 1&-1&\phm 1&
4&\phm 1&-1&1&-1&
-4&-1&\phm 1&-1&\phm 1\\
{}[J_{22}]&
3&\tau&\phm\bar\tau&\phm\bar\tau&\phm\tau&
-3&-\tau&-\bar\tau&-\bar\tau&-\tau&
3&\phm\tau&\phm\bar\tau&\bar\tau&\phm\tau&
-3&-\tau&-\bar\tau&-\bar\tau&-\tau\\
{}[J_{23}]&
2&\tau&-\bar\tau&\phm\bar\tau&-\tau&
-2&-\tau&\phm\bar\tau&-\bar\tau&\phm\tau&
2&\phm\tau&-\bar\tau&\bar\tau&-\tau&
-2&-\tau&\phm\bar\tau&-\bar\tau&\phm\tau\\
{}[J_{24}]&
1&1&\phm 1&\phm 1&\phm 1&
-1&-1&-1&-1&-1&
1&\phm 1&\phm 1&1&\phm 1&
-1&-1&-1&-1&-1\\
{}[J_{25}]&
0&0&\phm 0&\phm 0&\phm 0&
\phm 0&\phm 0&\phm 0&\phm 0&\phm 0&
0&\phm 0&\phm 0&0&\phm 0&
\phm 0&\phm 0&\phm 0&\phm 0&\phm 0\\
\hline
\end{array}\hspace{-2cm} \]
}
The last $5$ columns of the first table above together with 
the second table give the species table for $a(\bZ/25)$.
\end{eg}

\section{Cyclic normal subgroups}%
\index{cyclic!normal subgroups}\index{normal subgroup, cyclic}

The methods of Sections~\ref{se:Frobgroup} and~\ref{se:cyclic-central}
can be combined to deal with cyclic normal subgroups, taking into
account the action of the normaliser. Let $P=\langle g\rangle \unlhd
G$ be a normal subgroup of order $p$, and let $C=C_G(P)$, a normal
subgroup of $G$ of index $m$ a divisor of $p-1$. Thus we have 
$G/C\cong \bZ/m$. Let $k$ be a field of
characteristic $p$ containing a primitive $2m$th root of unity.
And as in Section~\ref{se:Frobgroup}, it is convenient to make a central
extension of $G$ by an element of order two. In order to do this, we form the
pullback $\tilde G$ of $G\to\bZ/m$ and $\bZ/2m \to \bZ/m$:
\[ \xymatrix{&&1\ar[d] & 1\ar[d] \\
&& \bZ/2\ar[d]\ar@{=}[r] & \bZ/2 \ar[d] \\
1\ar[r] & C\ar[r]\ar@{=}[d] &\tilde G \ar[d]\ar[r] & \bZ/2m \ar[d]\ar[r] & 1 \\
1\ar[r] & C\ar[r] & G \ar[r]\ar[d] & \bZ/m \ar[r]\ar[d] & 1\\
&&1&1} \]
Let $h\in G$ be an element mapping to a generator of $G/C$,
let $hgh^{-1}=g^q$, and let 
\[ x=\sum_{1\le j\le p-1}g^j\in kP\le kG. \] 
Then we have $x^p=0$, $h^{2m}\in C\le \tilde G$, $hx=qxh$. Let $\eta$
be a square root of $q$ in $k$, and let $S_i$ ($i\in \bZ/2m$) be the
simple $kG$-module with a basis vector $v_i$ such that $C$ acts
trivially, and $hv_i=\eta^iv_i$. 
The functors $F_i$ of
Section~\ref{se:cyclic-central} is designed to pick out the socles of
the Jordan blocks of length $i$ of the action of 
$kP$. Tensoring with $S_2$ moves us down one radical layer of these
Jordan blocks, so to obtain a suitably symmetric definition, we should define
\[ F_i(M) = S_{-i} \otimes \frac{\Ker(x) \cap
    \im(x^{i-1})}{\Ker(x)\cap\im(x^i)}. \]
With this definition, as in Proposition~\ref{pr:cijkFiFj}, 
for $1\le k\le p-1$ we have
\[ F_k(M \otimes N) \cong \bigoplus_{i,j}c_{i,j,k}F_i(M)\otimes
  F_j(N). \]
with the same coefficients $c_{i,j,k}$ as before. 

\begin{theorem}
We have $p-1$ algebra homomorphisms $\hat s_i\colon a_\bC(\tilde G)\to
a_\bC(\tilde G/P)$
\[ \hat s_i \colon [M] \mapsto \sum_{k=1}^{p-1}
  2\cos(ik\pi/p)[F_k(M)] \]
with $0<i<p$. They are continuous with respect to the norm, and extend
to give maps of Banach algebras $\hat s_i\colon \hat a(\tilde G) \to
\hat a(\tilde G/P)$.
\end{theorem}
\begin{proof}
As in the proof of Theorem~\ref{th:sellhat},
if $s$ is a non-Brauer species of $a(\bZ/p\rtimes\bZ/2m)$, the map
\[ \hat s \colon [M] \mapsto \sum_{k=1}^{p-1} s([J_k])[F_k(M)] \]
defines an algebra homomorphism $\hat s\colon a_\bC(\tilde G) \to
a_\bC(\tilde G/P)$, which is the identity on the subalgebra
$a_\bC(\tilde G/P)\subseteq a_\bC(\tilde G)$. 
This is continuous with respect to the
norm, and so it extends to give a map of Banach algebras $\hat s\colon
\hat a(\tilde G) \to \hat a(\tilde G/P)$.

If $s$ is a non-Brauer species of $a(\bZ/p\rtimes\bZ/2m)$, then the
map $\hat s$ only depends on the value of $s$ on the elements $[J_k]$.
We are in the situation where $d=1$ in
Theorem~\ref{th:Frobgroup}, and so we have
\[  a(\bZ/p\rtimes\bZ/2m)\cong\bZ[X,Y]/(Y^{2m}-1,(X-Y-Y^{-1})f_p(X)). \]
The element $[J_k]$ corresponds to $f_k(X)$, which is in the subring
generated by $X$. The $2(p-1)m$ non-Brauer species $s_{i,j}$ are given by
$X \mapsto \zeta_{2p}^i+\zeta_{2p}^{-i}=2\cos(i\pi/p)$, $Y\mapsto \zeta_{2m}^j$,
with $0<i<p$, $0\le j<2m$. The value of $s_{i,j}$ on the elements
$[J_k]$ therefore only depends on $i$, and we write $\hat s_i$ for the
common value of the $\hat s_{i,j}$.
\end{proof}

This theorem may be used in order to construct all the species of the
Frobenius group\index{group!Frobenius}\index{Frobenius!group}
$a(\bZ/p^n \rtimes \bZ/m)$ with $m$ coprime to $p$, and show that it
is semisimple.\index{semisimple} 
As in that case, we do need one more ring homomorphism 
\[ \hat s_0\colon a(\bZ/p^{n+1}\rtimes \bZ/2m) \to a(\bZ/p^n\rtimes
\bZ/2m) \]
as in Lemma~\ref{le:s0hat}, constructed in a similar way.
Even though the tensor products are more complicated than in the cyclic
case, working modulo the ideal spanned by the
modules $J_{p^n} \otimes S_i$, the tensor product relations~\eqref{eq:JJ}
still hold. In the proof, instead of preserving dimension, we have to
preserve Brauer species, and so the map is given by
\begin{multline*} \hat s_0\colon [J_{2bp^n\pm r}] \to [S_{-d((2b-1)p^n\pm r)} \oplus 
S_{-d((2b-3)p^n\pm r)} \oplus \dots \oplus S_{-d(3p^n\pm r)} \oplus S_{-d(p^n\pm r)} \\ 
\oplus S_{d(p^n\pm r)} \oplus S_{d(3p^n\pm r)}\oplus \dots
\oplus S_{d((2b-3)p^n\pm r)} \oplus S_{d((2b-1)p^n\pm r)}][J_{p^n}]
\pm [J_r]
\end{multline*}
and $\hat s_0\colon [S_i] \mapsto [S_i]$.
The $2mp^n$ species of $a(\bZ/p^n\rtimes \bZ/2m)$ are then given by
\[ s_{\ell_0,j}\hat s_{\ell_1}\dots \hat s_{\ell_{n-1}} \]
with $0\le \ell_i\le p-1$ for $0\le i\le n-1$, and with $0\le j<2m$.
Restricting to the range $0\le j < m$ gives the species for $a(\bZ/p^n\rtimes\bZ/m)$.

\begin{rk}
All these species satisfy $s([M^*]) = \overline{s([M])}$ for all
modules $M$. It follows that $a(\bZ/p^n\rtimes \bZ/m)$ is a symmetric
Banach $*$-algebra,%
\index{symmetric!Banach $*$-algebra}%
\index{Banach!star@$*$-algebra!symmetric}
see Section~\ref{se:symm-rep-ring}.
\end{rk}

\section{An integral example}\label{se:Z2Z/4}%
\index{integral representation}\index{cyclic!group of order $4$}%
\index{group!cyclic of order $4$}

In modular representation theory of finite groups, finite
representation type implies that the representation ring is
semisimple.\index{semisimple} 
Here we give an example to show that this no longer holds
in integral representation theory. Let $\bZ_2$ denote the ring of $2$-adic
integers, and
consider the group ring
$\bZ_2 G$, where $G=\bZ/4$, the cyclic group of order four. 
Troy~\cite{Troy:1961a}, Roiter~\cite{Roiter:1960a} showed that there are nine
isomorphism classes of indecomposable finitely generated $\bZ_2$-free
$\bZ_2 G$-modules. Reiner~\cite{Reiner:1965a} denotes the basis
elements of the representation ring corresponding to these
indecomposable modules 
$c_1,\dots,c_9$, and computes the tensor products, which are as in the
following table.\par
{\tiny
\[ \begin{array}{c|cccccccc} 
c_1 & c_2 & c_3 & c_4 & c_5 & c_6 & c_7 & c_8 & c_9 \\ \hline
c_2 & c_1 & c_3 & c_4 & c_6 & c_5 & c_7 & c_8 & c_9 \\
c_3 & c_3 & 2c_4 & 2c_3 & c_4 + c_9 & c_4 + c_9 & 
c_3 + c_4 + c_9 & c_3 + c_4 + c_9 & 2c_9  \\
c_4 & c_4 & 2c_3 & 2c_4 & c_3 + c_9 & c_3 + c_9 &
c_3 + c_4 + c_9 & c_3 + c_4 + c_9 & 2c_9 \\
c_5 & c_6 & c_4 + c_9 & c_3 + c_9 & c_1 + 2c_9 & c_2 + 2c_9 &
c_8 + 2c_9 & c_7 + 2c_9 & 3c_9 \\
c_6 & c_5 & c_4 + c_9 & c_3 + c_9 & c_2 + 2c_9 & c_1 + 2c_9 &
c_8 + 2c_9 & c_7 + 2c_9 & 3c_9 \\
c_7 & c_7 & c_3 + c_4 + c_9 & c_3 + c_4 + c_9 & c_8 + 2c_9 & 
c_8 + 2c_9 & c_7 + c_8 + 2c_9  & c_7 + c_8 + 2c_9 & 4c_9 \\
c_8  & c_8 & c_3 + c_4 + c_9 & c_3 + c_4 + c_9 & c_7 + 2c_9 &
c_7 + 2c_9 & c_7 + c_8 + 2c_9 & c_7 + c_8 + 2c_9 & 4c_9 \\
c_9 & c_9 & 2c_9 & 2c_9 & 3c_9 & 3c_9 & 4c_9 & 4c_9 & 4c_9
\end{array} \]
}
The representation ring $a(\bZ_2G)$ and its representation ideals 
are displayed in the following diagram:
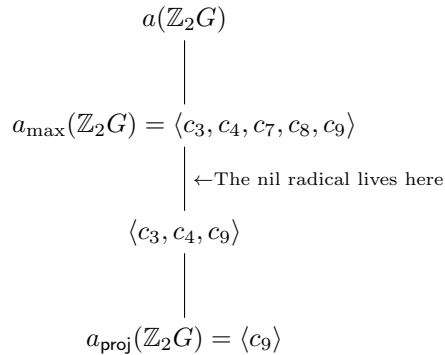
\[ \xymatrix@C=-4mm{a(\bZ_2G)\ar@{-}[d] \\
a_{\max}(\bZ_2G) =\langle c_3,c_4,c_7,c_8,c_9\rangle
\ar@{-}[d]^{\leftarrow\text{The nil radical lives here}} \\
\langle c_3,c_4,c_9\rangle\ar@{-}[d] \\
a_{\proj}(\bZ_2G)=\langle c_9\rangle} \]
The element $c_7 - c_8$ squares to zero, and generates the
nil radical. The quotient is semisimple, with eight species given by
the following table, where we have reordered the indecomposables to
reflect the structure of the representation ideals.
\[ \setlength{\arraycolsep}{1mm}
\renewcommand{\arraystretch}{1.3}
\begin{array}{|c|cccccccc|} \hline
c_1 & \phm 1\phm & \phm 1\phm & \phm 1\phm & 
\phm 1\phm & \phm 1\phm & \phm 1\phm & \phm 1\phm & \phm 1\phm \\
c_2 & 1 & 1 & 1 & 1 & 1 & 1 & -1\phm & -1\phm \\
c_5 & 3 & 1 & -1\phm & 1 & 1 & -1\phm & 1 & -1\phm \\
c_6 & 3 & 1 & -1\phm & 1 & 1 & -1\phm & -1\phm & 1 \\
c_7 & 4 & 2 & 0 & 2 & 0 & 0 & 0 & 0\\
c_8 & 4 & 2 & 0 & 2 & 0 & 0 & 0 & 0\\
c_3 & 2 & 2 & -2\phm & 0 & 0 & 0 & 0 & 0 \\
c_4 & 2 & 2 & 2 & 0 & 0 & 0 & 0 & 0 \\
c_9 & 4 & 0 & 0 & 0 & 0 & 0 & 0 & 0 \\ \hline
\end{array} \]
It is fairly easy to see that the role of the element $\rho$ is played
by $c_9$, which is the only projective module. 
The dimension function is the first column of numbers in the table,
and is the only Brauer species. Every module is self-dual, and since
all the entries in the table above are real, it follows that
$a(\bZ_2G)$ is a symmetric representation ring, 
see Section~\ref{se:symm-rep-ring}.%
\index{symmetric!representation ring}%
\index{representation!ring!symmetric}

\begin{rk}
In the same paper, Reiner~\cite{Reiner:1965a} shows that whenever $G$
is a cyclic group of order $p^n$ with $n\ge 2$, there is a non-zero
nilpotent element in $a(\bZ_pG)$. This result is extended in
Reiner~\cite{Reiner:1966a}.

The representation type of $\bZ_p G$ is finite if and only if the
Sylow $p$-subgroups of $G$ are trivial, or cyclic of order $p$ or
$p^2$, see Heller and
Reiner~\cite{Heller/Reiner:1962a,Heller/Reiner:1963a}. 
Integral representations of the dihedral group of order $2p$ 
are described in Lee~\cite{Lee:1964a}.%
\index{dihedral group!of order $2p$}\index{group!dihedral of order $2p$}
It would be interesting to know the tensor products of integral 
representations in cyclic and
dihedral cases of finite representation type.
\end{rk}

\section{\texorpdfstring{The group $SL(2,q)$}
{The group SL(2,q)}}\index{SL@$SL(2,q)$}\label{se:SL2q}

Let $q=p^m$ be a power of a prime $p$, and let $k$ be a field containing
$\bF_q$, and let $G=SL(2,q)$. 
In this section, we examine the two dimensional natural module $M$ for
$SL(2,q)$ over $k$. The goal is to show that
$\npj_G(M)=2\cos(\pi/q)$. 
On the way to this, we shall show that the
subring of $a(G)$ generated by summands of tensor powers of $M$ is
isomorphic to a ring of algebraic integers $\bZ[2\cos(\pi/q)]$, 
which in turn is the real
subring of the cyclotomic integers\index{cyclotomic!integers}%
\index{integers, cyclotomic} $\bZ[\zeta_q]$ where
$\zeta_q=e^{2\pi\bi/q}$.

We shall use the theory of tilting modules\index{tilting module}  $T(n)$
for $SL(2,\bar k)$, which turn out to be the direct summands of
the natural module $L(1)$.
A general discussion of tilting modules for reductive groups
may be found in Donkin~\cite{Donkin:1993a}, to which we
refer for general background.
There are also relevant discussions of
summands of tensor powers of the natural $SL(2,q)$-modules
in Alperin~\cite{Alperin:1976b} and
Craven~\cite{Craven:2013a}. 

The simple $SL(2,\bar k)$-modules $L(n)$ are indexed by their highest
weight, which in this case is an integer $n\ge 0$. In
particular, $L(0)$ is the trivial module, $L(1)$ is the natural
two dimensional module, and for $0\le n\le p-1$ we have $L(n)\cong
S^n(L(1))$, the symmetric 
powers\index{symmetric!powers of natural module} 
of the natural module.

Steinberg's tensor product 
theorem\index{Steinberg's tensor product theorem}%
\index{tensor product!theorem, Steinberg's}
states that if
$n=\sum_{j=0}^{m-1}n_jp^j$ with $0\le n_j\le p-1$
then 
\[ L(n)\cong \bigotimes_{j=0}^{m-1}F^j(L(n_j)). \]
The restriction of the modules $L(n)$ to $SL(2,q)$ for $0\le n<q$, 
which we continue to denote $L(n)$, form
a complete set of irreducible modules for $SL(2,q)$.

From~\cite{Donkin:1993a}, we know that the tilting module $T(n)$ 
(Donkin's notation is $M(\lambda)$) is the unique indecomposable
summand of $L(1)^{\otimes n}$ with $n$ as a highest weight. A module is a
direct sum of tilting modules if and only if it is a direct summand of
a direct sum of tensor powers of $L(1)$. Tilting modules are
determined by their weights.  

\begin{theorem}
Let $L(1)$ be the natural two dimensional module for $SL(2,q)$ as
above. Then $L(1)$ is algebraic, and we have 
\[ \npj(L(1))=2\cos\pi/q. \] 
More generally, if $1\le
j\le p-1$ then 
\[ \npj(L(j))=\sin((j+1)\pi/q)/\sin(\pi/q). \]
\end{theorem}
\begin{proof}
Let $a_\tilt(SL(2,\bar k))$ be the subring of the representation ring of
rational $SL(2,\bar k)$-modules generated by the tilting modules. Then 
$a_\tilt(SL(2,\bar k))$ is isomorphic to the subring of
$\bZ[t,t^{-1}]$ generated by $t+t^{-1}$, with the powers of $t$
representing the non-negative weights, and $L(1)$ corresponding to
$t+t^{-1}$. 

Let $f_j(t)$ be the polynomials defined in
Definition~\ref{def:fj}. Then we have 
\[ f_j(t+t^{-1})=t^{j-1}+t^{j-3}+\dots+t^{-j+3}+t^{-j+1} \]
and
\[ (t+t^{-1})f_j(t+t^{-1})=f_{j+1}(t+t^{-1})+f_{j-1}(t+t^{-1}). \]
In particular, 
\begin{align*} 
f_q(t+t^{-1})&=
t^{q-1}+t^{q-3}+\dots+t^{-q+3}+t^{-q+1} \\
&=\prod_{j=1}^m(t^{p^{j-1}(p-1)}+t^{p^{j-1}(p-3)}+\dots+t^{-p^{j-1}(p-1)})\\
&=\prod_{j=1}^m
f_p(t^{p^j}+t^{-p^j}) 
\end{align*}
is the character of the
Steinberg module $L(q-1)$ for $SL(2,q)$. This a projective module of dimension
$q$. It follows that in $a(G)/a(G,1)$, we have $f_q[L(1)]=0$. In
particular, by Lemma~\ref{le:alg-mod-proj}, $[L(1)]$ is algebraic in
$a(G)$. Now  by Lemma~\ref{le:fpX}, the irreducible factors of
$f_q(X)$ exactly correspond to the Steinberg tensor product factors
of 
\[ L(q-1)=\bigotimes_{j=0}^{m-1}F^j(L(p-1)). \] 
No smaller tensor
product of these modules is projective, so $f_q$ is the minimal
polynomial of $L(1)$ in $a(G)/a(G,1)$.  Again using
Lemma~\ref{le:fpX}, the 
largest of the roots of $f_q(X)$ is $2\cos(\pi/q)$.
Applying Theorem~\ref{th:PF}, it follows that $\npj(L(1))=2\cos\pi/q$.
If $1\le j\le p-1$ then $[L(j)]=f_{j+1}[L(1)]$ and so 
\begin{equation*}
\npj(L(j))=f_{j+1}(2\cos\pi/q)=\sin((j+1)\pi/q)/\sin(\pi/q).
\qedhere
\end{equation*}
\end{proof}

\begin{conj}
If $M$ is a $kG$-module with $\npj(M)<2$ then for some integer $q\ge
2$ we have $\npj(M)=2\cos(\pi/q)$.
\end{conj}

It is even plausible that if $\npj(M)=2\cos(\pi/q)$ then $q$ is a
power of the characteristic $p$ of the coefficient field $k$.
We know of no counterexamples to this statement.

\section{The Klein four group}\index{Klein four group}%
\index{elementary abelian!group of order $4$}%
\index{species!of $a(\bZ/2\times\bZ/2)$}

Let $G=\bZ/2\times\bZ/2$ and $k$ an algebraically closed field of characteristic two. 
It was shown in Conlon~\cite{Conlon:1965a} that elements of
he representation ring $a_\bC(G)$ are separated by species (which he calls
$G$-characters) $s\colon a_\bC(G)\to\bC$, and therefore $a_\bC(G)$ is
semisimple.  The species are described there, and more explicitly 
in Benson and Parker~\cite{Benson/Parker:1984a},
and we repeat the description here.

The set of species for $a_\bC(G)$ falls naturally into three subsets:
\begin{enumerate}
\item The dimension.
  \item A continuous set of species parametrised by the non-zero
    complex numbers $z\in\bC\setminus\{0\}$.
  \item A discrete set of species parametrised by the set of ordered
    pairs $(N,\lambda)$ with $N>0$ in $\bZ$ and $\lambda\in\bP^1(k)$,
the projective line over $k$.\index{projective!line}\index{P@$\bP^1(k)$}
\end{enumerate}

The set of indecomposable $kG$-modules also falls naturally into three
subsets:
\begin{enumerate}
\item The projective indecomposable module of dimension four.
\item The syzygies of the trivial module $\Omega^m(k)$, $m\in\bZ$,
of dimension $2|m|+1$.
\item A set of representations parametrised by the set of ordered
  pairs $(n,\lambda)$ with $n>0$ in $\bZ$ and $\lambda\in\bP^1(k)$, of
  dimension $2n$.
\end{enumerate}

Define infinite matrices $A$ and $B$ as follows.
\[ \renewcommand{\arraystretch}{1.2}
    \begin{array}{c|c|ccccccc}
     \multicolumn{2}{l}{A}& N&\to  & & & & \\ \cline{3-7}
     \multicolumn{1}{}{} &  & 1 & 2 & 3 & 4 & 5 & \cdots \\ \cline{2-7}
     n & 1& 2 & 0 & 0 & 0 & 0 & \\
     {\downarrow} &2 & 2 & 2 & 0 & 0 & 0 & \\
     & 3 & 2 & 2 & 2 & 0 & 0 & \\
     & 4 & 2 & 2 & 2 & 2 & 0 & \\
     & 5 & 2 & 2 & 2 & 2 & 2 & \\
     \multicolumn{2}{r}{\vdots\,} & & & & & & \ddots
\end{array}\qquad
\begin{array}{c|c|ccccccc}
     \multicolumn{2}{l}{B}& N&\to  & & & & \\ \cline{3-7}
     \multicolumn{1}{}{} &  & 1 & 2 & 3 & 4 & 5 & \cdots \\ \cline{2-7}
     n & 1& \sqrt 2 & -\sqrt 2 & 0 & 0 & 0 & \\
     {\downarrow} &2 & 2 & 2 & 0 & 0 & 0 & \\
     & 3 & 2 & 2 & 2 & 0 & 0 & \\
     & 4 & 2 & 2 & 2 & 2 & 0 & \\
     & 5 & 2 & 2 & 2 & 2 & 2 & \\
     \multicolumn{2}{r}{\vdots\,} & & & & & & \ddots
\end{array} \]
Then the representation table is as follows.\par
{\small
\[ \renewcommand{\arraystretch}{1.2}
\setlength{\arraycolsep}{1.3mm}
 \begin{array}{c|ccccccccc}
 \text{Parameters} & \dim & z & (N,\infty) & (N,0) & (N,1) &
 (N,\lambda_1) & (N,\lambda_2) & (N,\lambda_3) & \cdots \\ \cline{1-9}
 \text{(Projective)} & 4 & 0 & 0 & 0 & 0 & 0 & 0 & 0 \\
 m & 2|m|+1 & z^m & 1 & 1 & 1 & 1 & 1 & 1 \\
 (n,\infty) & 2n & 0 & A & 0 & 0 & 0 & 0 & 0 \\
 (n,0) & 2n & 0 & 0  & A & 0 & 0 & 0 & 0 \\
 (n,1) & 2n & 0 & 0 & 0 & A & 0 & 0 & 0 \\
 (n,\lambda_1) & 2n & 0 & 0 & 0 & 0 & B & 0 & 0 \\
 (n,\lambda_2) & 2n & 0 & 0 & 0 & 0 & 0 & B & 0 \\
 (n,\lambda_3) & 2n & 0 & 0 & 0 & 0 & 0 & 0 & B \\
 \multicolumn{7}{l}{\qquad\ \vdots }
   \end{array} \]
}

Thus there are three special points $\infty,0,1\in \bP^1$ where the matrix
$A$ is used, and for the rest of the points the matrix $B$ is used.
The members of the continuous family of species with $|z|\ne 1$ are not
dimension bounded; the rest of the species are. 
There is a single Brauer species\index{Brauer species}\index{species!Brauer} 
$\dim$, which is dimension bounded but not
core bounded; the rest of the dimension bounded species are core
bounded. So the set of core bounded species is $S^1 \cup
(\bP^1(k)\times\bZ_{>0})$. The weak* topology on this may be described as
follows. The subset $S^1$ has the usual topology inherited from $\bC$.
The subset $\bP^1(k)\times\bZ_{>0}$ is discrete, but its closure is the
one point compactification, using the point $1\in S^1$. So the space
$\Struct(G)$ is a wedge of a circle with the one point
compactification of the discrete space $\bP^1(k)\times\bZ_{>0}$.

\section{\texorpdfstring
{The alternating group $A_4$}
{The alternating group A₄}}%
\index{alternating group $A_4$}\index{group!alternating $A_4$}

Let $k$ be an algebraically closed field of characteristic two, and let $G$ be the
alternating group $A_4$. Let $V_4$ be the normal subgroup of $G$ of
index three, isomorphic to the Klein four group.\index{Klein four group}
The indecomposable $kG$-modules are described in
Conlon~\cite{Conlon:1965a}
in terms of those of $kV_4$; see also Conlon~\cite{Conlon:1966a}
and the appendix to Benson~\cite{Benson:1986a}.
Let $\bF_4=\{0,1,\omega,\bar\omega\}\subseteq k$, and
 write $k$, $\omega$ and $\bar\omega$
for the three one dimensional representations where a generator $h$ for
$G/V_4\cong\bZ/3$ goes to $1$, $\omega$, $\bar\omega$ respectively. 
We also have an action of $G/V_4$ on $\bP^1(k)$, in which $h$ sends
$\lambda$ to $\lambda^h=1/(1+\lambda)$. The fixed points of this action are
$\omega$ and $\bar\omega$.

The set of indecomposable
$kG$-modules falls naturally into three subsets:
\begin{enumerate}
\item The projective indecomposables $P_k$, $P_\omega$ and
  $P_{\bar\omega}$, each of dimension four.
\item The syzygies of the simple modules $\Omega^n(k)$,
  $\Omega^n(\omega)$ and $\Omega^n(\bar\omega)$, each of which
  restrict to $\Omega^n(k)$ as a $kV_4$-module.
\item For each orbit of
  $G/V_4$ on $\bP^1(k)\setminus\{\omega,\bar\omega\}$ and each $n\in\bZ_{>0}$ there is an
  indecomposable $kG$-module of dimension $6n$ which restricts to the
  sum of the indecomposable $kV_4$-modules corresponding to
  $(n,\lambda)$, $(n,1+1/\lambda)$, $(1,1/(1+\lambda))$; 
for each $\lambda\in\{\omega,\bar\omega\}$ and each $n\in\bZ_{>0}$  there are three
  $kG$-modules of dimension $2n$, 
restricting to the $kV_4$-module corresponding to $(n,\lambda)$.
\end{enumerate}

We define one more infinite matrix $C$ as follows.
\[ \renewcommand{\arraystretch}{1.2}
    \begin{array}{c|c|rrrrrrc}
     \multicolumn{2}{l}{C}& N&\to  & & & & & \\ \cline{3-8}
     \multicolumn{1}{}{} &  & 1 & 2 & 3 & 4 & 5 & 6 & \cdots \\ \cline{2-8}
     n & 1& \sqrt 2 & -\sqrt 2 & 0 & 0 & 0 & 0 & \\
     {\downarrow} &2 & 2 & 2 & 0 & 0 & 0 & 0 & \\
     & 3 & -1 & -1 & -1 & 0 & 0 & 0 & \\
     & 4 & -1 & -1 & -1 & -1 & 0 & 0 & \\
     & 5 & 2 & 2 & 2 & 2 & 2 & 0 \\
     & 6 & -1 & -1 & -1 & -1 & -1 & -1 \\
     \multicolumn{2}{r}{\vdots\,} & & & & & & & \ddots
\end{array} \]
After the first row, each column repeats with period three where it is non-zero.
Then the representation table is as follows. As in the appendix
to~\cite{Benson:1986a}, we have used the Atlas
conventions~\cite{Atlas} to illustrate the relationship with the
tables for the Klein four group.\par
{\tiny
\[ \setlength{\arraycolsep}{0.5mm}
\renewcommand{\arraystretch}{1.2}
 \begin{array}{|c|cc|ccccccccccccc|} \hline
 \text{Params} & \dim & z & (N,\infty) & (N,0) & (N,1) & 
 (N,\omega) & (N,\bar\omega) &
 (N,\lambda) & (N,\lambda^h) & (N,\lambda^{h^2}) & \text{fus} & h & z
   & (N,\omega) & (N,\bar\omega) \\ \hline
 \text{(Proj)} & 4 & 0 & 0 & 0 & 0 & 0 & 0 & 0 & 0 & 0 & :& 1&0&0&0\\
 m & 2|m|+1 & z^m & 1 & 1 & 1 & 1 & 1 & 1 & 1 & 1 & : & \ep & z^m&1&1
   \\ \hline
 (n,\infty) & 2n & 0 & A &  &  &&&  &  &  & 
\multirow{3}{*}{\rule[-4pt]{0.7pt}{8.5ex}} & 0 & 0 & 0 & 0 \\
 (n,0) & 2n & 0 &   & A &  &&&  & \text{\Large$0$} & & & & & &  \\
 (n,1) & 2n & 0 &  &  & A &&&  &  & & & &&&  \\
 (n,\omega) & 2n &0&& && B & & & && : & \ep & 0 & C & 0  \\
 (n,\bar\omega) & 2n & 0 &&&&&B &&&& : & \ep & 0 & 0 & C \\
 (n,\mu) & 2n & 0 &  &  &  &&& B\delta_{\lambda,\mu} & 
 &  & \multirow{3}{*}{\rule[-4pt]{0.7pt}{8.5ex}} & 0 & 0 & 0 & 0 \\
 (n,\mu^h) & 2n & 0 &  & \text{\Large$0$} &  &&&  &
B\delta_{\lambda,\mu} & & & & & & \\
 (n,\mu^{h^2}) & 2n & 0 &  &  &  &&&  &  & B\delta_{\lambda,\mu} & & &
   & & \\ \hline
   \end{array}\medskip \]
}
Here, $\ep\in\{-1,0,1\}$ is congruent to the dimension modulo three.
Just as in the case of the Klein four group, the structure space
$\Delta(G)$ has a discrete part and a continuous part. The continuous
part is $\Delta_{\max}(G)$, and consists of three disjoint circles, corresponding to the columns
headed ``$z$'' (the one on the right side is really two, one for each
non-trivial character of the quotient $G/V_4$).
All but the second row of this table (which is really an infinite set
of rows) represent elements of
$a(G,\max)$, on which these columns take the value zero.
The closure of the discrete part is again its one point
compactification, attached at the basepoint of one of the three circles.

Expanding out the table for the quotient $a_{\max}(G)$
from Atlas format to full notation, we obtain the following table:
\[ \begin{array}{|c|ccc|} \hline
\Omega^m(k) & z^m & z^m & z^m \\
\Omega^m(\omega) & z^m & \omega z^m & \bar\omega z^m \\
\Omega^m(\bar\omega) & z^m & \bar\omega z^m & \omega z^m \\ \hline
\end{array} \]
This is exactly the character table for $\bZ\times \bZ/3$. This is
because $a_{\max}(G)$ is isomorphic to the group ring of 
$\PPic_{\max}(G)\cong\bZ\times \bZ/3$, a group with generators
$\Omega(k)$ and $\omega$.

\section{\texorpdfstring{Dihedral $2$-groups}
{Dihedral 2-groups}}\index{dihedral group!of order $2^n$}%
\index{group!dihedral of order $2^n$}

The indecomposable modules for the dihedral groups $D_{2^n}$ ($n\ge 3$) were
classified by Ringel~\cite{Ringel:1975a}. Let $k$ be a field of
characteristic two, let
\[ G=D_{2^n}=\langle x,y\mid x^2=1,\,y^2=1,\,(xy)^{2^{n-1}}=1\rangle, \]
and let $X=x-1$, $Y=y-1$ as elements of $kG$. Then
\[ kG=k\langle X,Y\rangle/(X^2,Y^2,(XY)^{2^{n-2}}-(YX)^{2^{n-2}}). \]

The modules come in two types,
called \emph{strings}\index{string module}\index{module!string} and
\emph{bands}.\index{band module}\index{module!band} 
The ones of odd dimension are string modules. 

The string modules $M(C)$ correspond to words
$C=w_1 w_2 \dots w_{m}$ where the $w_i$
alternate between
$X^{\pm 1}$ and $Y^{\pm 1}$. The dimension of the module is $m+1$.
Thus for example the word $X^{-1}YXYX^{-1}Y^{-1}$ gives a
module with schema
\[ \xymatrix{\bul\ar[r]^X & \bul & \bul\ar[l]_Y & \bul\ar[l]_X &
\bul \ar[l]_Y \ar[r]^X & \bul \ar[r]^Y & \bul} \]
For a particular order of dihedral group, there is also a restriction
on the number of consectutive letters which are all direct or all
inverse. The module corresponding to a given word
$w_1 w_2 \dots w_m$ has a $k$-basis $v_0,v_1,v_2,\dots,v_m$.
The elements $X$ and $Y$ in $kG$ act in the manner indicated by the
schema, sending each basis either to an adjacent basis element or to
zero. In the example, we have 
\begin{gather*}
X \colon\quad v_0 \mapsto v_1,\quad 
v_1 \mapsto 0,\quad v_2\mapsto 0,\ \quad 
v_3\mapsto v_2,\quad v_4\mapsto v_5,\quad
v_5\mapsto 0,\ \quad v_6\mapsto 0, \\ 
 Y \colon\quad v_0\mapsto 0,\ \quad
v_1\mapsto 0,\quad v_2\mapsto v_1,\quad 
v_3\mapsto 0,\ \quad v_4\mapsto v_3,\quad
v_5\mapsto v_6,\quad v_6\mapsto 0.
\end{gather*}
Modules coming from two different words are isomorphic if and only if
one word is the inverse of the other. To invert a word, reverse the letters and invert
each one. So for example the inverse of the word above is $YXY^{-1}X^{-1}Y^{-1}X$.

The band modules $M(C,\phi)$ are similar, except that the word has to have even
length, and the beginning and end of
the word are linked to make a cycle. Instead of putting one basis
element at each vertex, we take a vector space $V$ and an
indecomposable automorphism $\phi\colon V\to V$, and we put a copy of
$V$ at each vertex. The arrows are identity maps, but the two end
vertices are identified using $\phi$. So for example the word above gives us a
schema
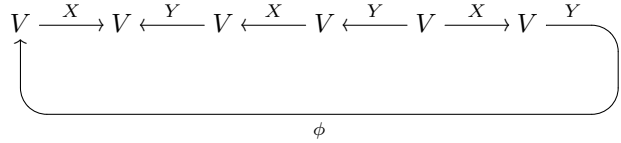
\[ \xymatrix{V\ar[r]^X & V & V \ar[l]_Y & 
V\ar[l]_X & V \ar[l]_Y\ar[r]^X &  V \ar `r[r]^{\quad Y} `d[llllld] `[lllll]^\phi  [lllll]
&\\ &&&&&&    } \]
The word $C$ is not allowed to be a power of a smaller word, as this would be
absorbed into making the vector space $V$ larger. Modules $M(C,\phi)$
and $M(C',\phi')$ are isomorphic if and only if either $C$ and $C'$
differ by a rotation and $\phi=\phi'$, or $C^{-1}$ and $C'$ differ by a
rotation and $\phi^{-1}=\phi'$.

The band modules all have even dimension. So the odd dimensional
modules are string modules for words of even length. Inverting the
word if necessary, we may assume that it starts with $X^{\pm 1}$ and
ends with $Y^{\pm 1}$, and then we don't need to bother about
equivalent words. Thus the odd dimensional modules are of the form
$M(C)$ with $C=X^{\pm 1}\dots Y^{\pm 1}$. This includes the empty
word, which we take to corrspond to the trivial module.

\begin{lemma}\label{le:D2nodd}
If $M$ is an odd dimensional indecomposable $kD_{2^n}$-module then
$M{\da_{\langle x\rangle}}$ is a direct sum of a one dimensional
trivial module and a projective module.
\end{lemma}
\begin{proof}
This follows immediately from the description above. The one
dimensional summand corresponds to the right hand vertex. The
remaining pairs of vertices give free summands as modules for
$k\langle x\rangle=k\langle X\rangle$.
\end{proof}

The next two theorems come from Archer~\cite{Archer:2008a}.

\begin{theorem}
If $M$ and $N$ are odd dimensional indecomposable $kD_{2^n}$-modules
then $M\otimes N$ has a unique odd dimensional indecomposable summand.
\end{theorem}
\begin{proof}
This follows by restricting to $\langle x\rangle$ and using Lemma~\ref{le:D2nodd}.
\end{proof}

It follows from this theorem that the isomorphism classes of 
odd dimensional indecomposable
$kD_{2^n}$-modules form an abelian group, equal to
$\PPic_{\max}(a(kD_{2^n}))$\index{Tmax@$\PPic_{\max}(a(kD_{2^n}))$}  
(see Section~\ref{se:endotriv}).
The product of $[M]$ and
$[N]$ in this group is the isomorphism class of the 
unique odd dimensional summand of $M\otimes N$. The inverse of $[M]$
is $[M^*]$.

\begin{theorem}
This group is torsion free.
\end{theorem}

\begin{rk}
Zemanek~\cite{Zemanek:1973a} showed that there are non-zero nilpotent elements in
$a(D_{2^n})$ ($n\ge 3$); see also Benson and
Carlson~\cite{Benson/Carlson:1986a}, 
Heldner~\cite{Heldner:1994a}. 
So we cannot hope to separate elements of
$a(D_{2^n})$ using species, as we did in the case of the Klein four
group.

The papers of Herschend~\cite{Herschend:2007a,Herschend:2010a} study a
different tensor product on representations of dihedral group algebras.
\end{rk}

\section{\texorpdfstring{Semidihedral $2$-groups}
{Semidihedral 2-groups}}\index{semidihedral groups}

Let $k$ be a field of characteristic two, and let
\[ G = SD_{2^n} =\langle x,y\mid x^2=1,\,y^{2^{n-1}}=1,\,
yx=xy^{2^{n-2}-1}\rangle. \]
In Section~3 of Bondarenko and Drozd~\cite{Bondarenko/Drozd:1982a}, 
an explicit isomorphism is given between the quotient by the one
dimensional socle, $kG/\Soc(kG)$,
and the algebra $\Lambda_{2^{n-1}-1}$ where
\begin{equation*} 
\Lambda_m = k\langle X,Y\rangle/(X^3,Y^2,X^2-(YX)^mY). 
\end{equation*}
Since every non-projective indecomposable $kG$-module has $\Soc(kG)$
in the kernel, classification of the indecomposable $kG$-modules
amounts to classification of the indecomposable $\Lambda_{2^{n-1}-1}$-modules.
The indecomposable $\Lambda_m$-modules for $m\ge 1$ were
classified by Crawley-Boevey~\cite{Crawley-Boevey:1989a}; see also Gei\ss~\cite{Geiss:1999a}.
They have a description in terms 
similar to the strings and bands described in the last section, but
more complicated. There
are four types, called asymmetric strings, symmetric strings,
asymmetric bands, and symmetric bands. The asymmetric and symmetric
bands have even dimension. If 
$M$ is an odd dimensional indecomposable $kG$-module 
then $M$ is an asymmetric
or symmetric string module corresponding to a word of even length.
In particular, just as in the dihedral case, the restriction of an odd
dimensional module to the two dimensional subalgebra generated by 
$Y$ is trivial plus free. So we get the following theorem.

\begin{theoremqed}
If $M$ and $N$ are odd dimensional indecomposable $kSD_{2^n}$-modules
then $M \otimes N$ has a unique odd dimensional indecomposable summand.
\end{theoremqed}

Again we deduce that the isomorphism classes of odd dimensional
$kSD_{2n}$-modules form an abelian group, equal to $\PPic_{\max}(a(kSD_{2^n}))$.
But this time it is not torsion free. There is a 
self-dual module $M=\Lambda_{2^{n-1}-1}/\langle YX\rangle$ of
dimension $2^{n-1}+1$, with simple socle, such that $M\otimes M\cong k\oplus kG$.
Thus $[M]^2=\one$ in this group.

\section{\texorpdfstring{Finite $2$-groups}{Finite 2-groups}}

The following conjecture appears in \cite{Benson:2020a}.

\begin{conj}\label{conj:char2tensors}
Let $G$ be a finite $2$-group and $k$ an algebraically closed field of
characteristic $2$. If $M$ is an indecomposable
$kG$-module of odd dimension then $M\otimes M^*$ is a direct sum of
$k$ and indecomposable modules of dimension divisible by four.
\end{conj}

The conjecture is true in the case of cyclic groups, dihedral groups,
and semidihedral groups.

Given this conjecture, the isomorphism classes of indecomposables of
odd dimension form a discrete abelian group, equal to $\PPic(a(G))$, 
and $a_{\max}(G)$ is
its group algebra, weighted with the function $M\mapsto\dim\core_{\max}(M)$. 

\section{Some Hopf algebras}%
\index{Sweedler's Hopf algebra}\index{Hopf algebra!Sweedler's}

Whenever the modular representation ring of a finite group is finite
dimensional, it is semisimple. We saw in Section~\ref{se:Z2Z/4} that
this is not the case for integral representation rings. We shall see
in this section that it is also
not true for
finite dimensional Hopf algebras over a field. To illustrate this, we 
examine some generalisations of 
Hopf algebras introduced by Earl Taft~\cite{Taft:1971a}, 
which are neither commutative nor
cocommutative. They are similar to the group algebras of the Frobenius
groups studied in Section~\ref{se:Frobgroup}, but sufficiently
different that we find it worthwhile to spell out the details.
The end result is that the radical is contained in the ideal of
projective modules, but is non-zero.

We begin with the smallest case, which is 
Sweedler's four dimensional 
Hopf algebra; see page 89--90 of Sweedler's book,
as well as Remark~5.8 in Cibils~\cite{Cibils:1993a}
and Remark~1.5.6 in Montgomery~\cite{Montgomery:1993a}.
Let $k$ be a field of characteristic not equal to two, and consider
the $k$-algebra with a vector
space basis consisting of elements $1$, $g$, $x$ and $gx$.
The multiplication is given by $g^2=1$,  $x^2=0$, $xg=-gx$.
The comultiplication\index{comultiplication} is given by
$\Delta(g)=g\otimes g$, $\Delta(x)=1\otimes
x+x\otimes g$. The counit\index{counit} is given by $\ep(g)=1$, $\ep(x)=0$,
and the antipode\index{antipode} 
is given by $S(g)=g$, $S(x)=gx$, $S(gx)=-x$, with $S^4=1$.

This Hopf algebra has four isomorphism classes of
indecomposables. There are two simples, $S_0$ and $S_1$, and their
projective covers $P_0$ and $P_1$. The tensor products are given by
the following table:
\[ \renewcommand{\arraystretch}{1.2}
\begin{array}{|c|cccc|} \hline
&[S_0]&[S_1]&[P_0]&[P_1] \\ \hline
{}[S_0]&[S_0]&[S_1]&[P_0]&[P_1] \\
{}[S_1]&[S_1]&[S_0]&[P_1]&[P_0] \\
{}[P_0]&[P_0]&[P_1]&[P_0]+[P_1]&[P_0]+[P_1] \\
{}[P_1]&[P_1]&[P_0]&[P_0]+[P_1]&[P_0]+[P_1] \\ \hline
\end{array}\medskip \]
Note, in particular, that the representation ring is commutative.
Its radical is generated by $[P_0]-[P_1]$. There are three species, given by the
following table.
\[ \renewcommand{\arraystretch}{1.2}
\begin{array}{|c|crc|} \hline
&s_1&s_2&s_3\\ \hline
{}[S_0]&1&1&1\\
{}[S_1]&1&-1&1\\
{}[P_0]&0&0&\sqrt{2}\\
{}[P_1]&0&0&\sqrt{2}\\ \hline
\end{array} \]
The representation ideals $\fX_{\max}$ and $\fX_{\proj}$ are equal, 
and consist of the projectives $[P_0]$ and $[P_1]$.
So there is only one possible non-zero choice for a representation
ideal $\fX$, namely $\fX=\fX_{\max}=\fX_{\proj}$.
We have $\npj(S_0)=\npj(S_1)=1$, $\npj(P_0)=\npj(P_1)=0$.\medskip

The family of generalised Taft algebras 
$H_{m,n}(q)$\index{H@$H_{m,n}(q)$}~\cite{Taft:1971a}%
\index{Taft algebra}\index{generalised!Taft algebra}%
\index{Hopf algebra!Taft}
has the Sweedler four dimensional algebra as the case
$H_{2,2}(-1)$.  Their representation rings were studied in
\cite{Chen/vanOystaeyen/Zhang:2014a,
Huang/vanOystaeyen/Yang/Zhang:2014a,
Li/Zhang:2013a,Radford:1975a,
Yuan/Yao/Li:2018a}. They
were constructed as examples of 
Hopf algebras whose antipode has arbitrarily large finite (even) order. 
Let $k$ be a field having a primitive 
$m$th root of unity $\eta$, with $m$ an integer at least two; in
particular, we assume that $k$ has characteristic coprime to $m$.
Let $n\ge 2$ be a divisor of $m$, and let $q=\eta^{d}$, a
primitive $n$th root of unity, where $d=m/n$.
The algebra $H_{m,n}(q)$ is generated over $k$ by elements
$h$ and $x$ satisfying $h^m=1$, $x^n=0$, $hx=qxh$.
The comultiplication is given by
$\Delta(h)=h\otimes h$, $\Delta(x)=1\otimes x + x \otimes h$, 
the counit is given by $\ep(h)=1$, $\ep(x)=0$. The
antipode is given by $S(h)=h^{-1}$, $S(x)=-xh^{-1}$, and has
order $2n$. For reasons similar to those given in
Section~\ref{se:Frobgroup}, we consider the double cover 
$H_{2m,n}(q)$ of $H_{m,n}(q)$ first and then
identify $a(H_{m,n}(q))$ as a subring of $a(H_{2m,n}(q))$.

So we now suppose that the field $k$ has a primitive $2m$th root of
unity $\eta$, the integer $n\ge 2$ divides $m$, and $q=\eta^{2d}$ is a primitive
$n$th root of unity, where $d=m/n$.
The Hopf algebra $H_{2m,n}(q)$ over $k$ has
generators $h$ and $x$ satisfying $h^{2m}=1$, $x^n=0$, $hx=qxh$.
The comultiplication, counit and antipode are as before.

The algebra $H_{2m,n}(q)$ has $2m$ isomorphism classes of simple modules $S_i$,
$i\in\bZ/2m$, all one dimensional,
corresponding to the characters of the subgroup
generated by $h$. Letting $v_i$ be a basis element for $S_i$,
we have $hv_i=\eta^iv_i$ and $xv_i=0$. The space
$\Ext^1_{H_{2m,n}(q)}(S_i,S_j)$ is one dimensional if $j=i+2d$ and zero
dimensional otherwise. The projective
indecomposable modules are uniserial of length $n$, with composition
factors (from top to bottom) of $P_i$ being $S_i$, $S_{i+2d}$, $S_{i+4d}$, \dots,
$S_{i-2d}$. So $H_{2m,n}(q)$ is a Frobenius algebra, but not a
symmetric algebra.
Every indecomposable module is a quotient of a projective
indecomposable module. We write $J_j$ ($1\le j\le n$)
for the indecomposable module of length $j$ with composition factors
$S_{-d(j-1)}$, $S_{-d(j-3)}$, \dots, $S_{d(j-1)}$. A complete list of the
$2mn$ isomorphism classes of indecomposable $\tilde
H_{m,n}(q)$-modules is given by the modules $J_j \otimes S_i$ with
$1\le j\le n$, $0\le i<2m$. The representation ring $a(H_{m,n}(q))$ is
commutative, even though this Hopf algebra is not quasitriangular.
As in Section~\ref{se:Frobgroup}, we have
\[ J_2\otimes J_j \cong \begin{cases}
J_2 & j=1 \\
J_{j+1} \oplus J_{j-1} & 2\le j\le n-1 \\
(J_n \otimes S_d) \oplus (J_n\otimes S_{-d}) & j=n.
\end{cases} \]

\begin{theorem}
We have
\[ a(H_{2m,n}(q)) \cong \bZ[X,Y]/(Y^{2m}-1,
  (X-Y^d-Y^{-d})f_n(X)) \]
where $X$ corresponds to $J_2$ and $Y$ corresponds to $S_1$, and the
polynomials $f_i$ are described in Definition~\ref{def:fj}. This ring
has a basis consisting of the $X^iY^j$ with $0\le i<n$ and $0\le j<2m$.
\end{theorem}
\begin{proof}
This follows from the above relations, as in Theorem~\ref{th:aG/aG1},
Remark \ref{rk:aG} and Theorem~\ref{th:Frobgroup}. The element
$f_j(X)$ again corresponds to $[J_j]$.
\end{proof}

Our next task is to identify the species and radical of this representation ring.

\begin{lemma}\label{le:fnY+Y^-1}
In $\bZ[Y,Y^{-1}]$, for $j\ge 0$ we have 
\[ (Y-Y^{-1})f_j(Y+Y^{-1})=Y^j-Y^{-j}. \]
\end{lemma}
\begin{proof}
We prove this by induction on $j$, the cases $j=0$ and $j=1$ being
trivial to verify. For the inductive step, with $j\ge 2$, we have 
\begin{align*}
(Y-Y^{-1})f_{j+1}(Y+Y^{-1})&=(Y-Y^{-1})((Y+Y^{-1})f_j(Y+Y^{-1})\\
&\qquad{}-f_{j-1}(Y+Y^{-1}))\\
&=(Y+Y^{-1})(Y^j-Y^{-j})-(Y^{j-1}-Y^{-j+1}) \\
&=Y^{j+1}-Y^{-j-1}. 
\qedhere
\end{align*}
\end{proof}

\begin{lemma}\label{le:div-by-X-Y-Y^-1}
In $\bZ[X,Y,Y^{-1}]$ the element $(Y^d-Y^{-d})f_n(X)-(Y^m-Y^{-m})$ is
divisible by $X-Y^d-Y^{-d}$.
\end{lemma}
\begin{proof}
It follows from Lemma~\ref{le:fnY+Y^-1} that 
\[ (Y^d-Y^{-d})f_n(Y^d+Y^{-d})=Y^m-Y^{-m}. \]
Now use the factor theorem.
\end{proof}

\begin{prop}\label{pr:JaH2mnq}
The element $(Y^d-Y^{-d})f_n(X)$ squares to zero in the representation
ring $a(H_{2n,m}(q))$.
\end{prop}
\begin{proof}
Since $Y^m-Y^{-m}$ is zero in $a(H_{2n,m}(q))$, 
it follows from Lemma~\ref{le:div-by-X-Y-Y^-1} that  the element
$(Y^d-Y^{-d})f_n(X)$ is divisible by $(X-Y^d-Y^{-d})$. Hence its square
is divisible by $(X-Y^d-Y^{-d})f_n(X)$, 
which is zero in $a(H_{2n,m}(q))$.
\end{proof}

\begin{theorem}\label{th:aH2nmq}
The nil radical of $a(H_{2m,n}(q))$ is generated by the element
$(Y^d-Y^{-d})f_n(X)$, and has $\bZ$-rank $2(m-d)$.
There are $2(mn-m+d)$ species $s_{i,j}$ of $a(H_{2m,n}(q))$, and they are
given by
\begin{align*} 
X & \mapsto 
\zeta_{2n}^i+\zeta_{2n}^{-i} = 2\cos (i\pi/n)  \\
Y & \mapsto \zeta_{2m}^j.
\end{align*}
where $0\le i \le n$, $0\le j < 2m$, and if $j$ is divisible by $n$
then $i \equiv j \pmod{2n}$. 

The ideal of projective modules is generated by $f_n(X)$. There are  
 $2d$ Brauer species, which are the species $s_{i,j}$ with $j$ divisible by $n$.
The quotient 
\[ a_{\proj}(H_{2m,n}(q))=a(H_{2m,n}(q))/(f_n(X)) \] 
is semisimple. 
\end{theorem}
\begin{proof}
By Proposition~\ref{pr:JaH2mnq}, $(Y^d-Y^{-d})f_n(X)$ is in the
radical. Consider the quotient
$a(H_{2m,n}(q))/((Y^d-Y^{-d})f_n(X))$. Since $X-Y^d-Y^{-d}$ annihilates 
$(Y^d-Y^{-d})f_n(X)$, we have
\[ X(Y^d-Y^{-d})f_n(X) = (Y^d+Y^{-d})(Y^d-Y^{-d})f_n(X). \]
We also have
\[ (Y^{2(m-d)} + Y^{2(m-2d)} + \dots + Y^{2d} + 1)
(Y^d-Y^{-d})f_n(X)=Y^{-d}(Y^{2m}-1)f_n(X)=0. \]
So the ideal generated by $(Y^d-Y^{-d})f_n(X)$ is the $\bZ$-span of
the elements 
\[ Y^i(Y^d-Y^{-d})f_n(X) \] 
with $0\le i < 2(m-d)$. The
quotient therefore has rank $2(mn-m+d)$, and has a $\bZ$-basis
consisting of the $X^iY^j$ with $0\le i < n$, $0\le j < 2m$, such
that if $i=n-1$ then $0\le j <2d$. 

For the species, we must satisfy the two relations $Y^{2m}=1$ and
\[ (X-Y^d-Y^{-d})f_n(X)=0. \] 
The first implies that $Y \mapsto
\zeta_{2m}^j$ with $0\le j < 2m$. Then the second relation becomes
$(X-\zeta_{2n}^j-\zeta_{2n}^{-j})f_n(X)=0$. The roots of $f_n(X)=0$
are $X \mapsto \zeta_{2n}^i + \zeta_{2n}^{-i}$ with $0<i<n$. So the
product has a repeated root unless $\zeta_{2n}^j=\pm 1$, namely $j$ is
divisible by $n$. In that case, $X \mapsto 2$ if $j$ is divisible by
$2n$ and $X\mapsto -2$ otherwise. This accounts for the cases $i=0$
and $i=n$. The element $f_n(X)=[J_n]$ generates the projectives, and
so the Brauer species are the ones where $f_n(X)$ does not go to
zero. This is the case where $j$ is divisible by $n$.
\end{proof}

To go down from $H_{2m,n}(q)$ to $H_{m,n}(q)$, we use the 
Dickson polynomials\index{Dickson polynomials}
$\Dick_i(y,z)$\index{E@$\Dick_j(y,z)$} given in Definition~\ref{def:Fiyz}.

\begin{theorem}
We have
\[ a(H_{m,n}(q)) = a(H_{2m,n}(q)/\langle h^m\rangle) \cong
\bZ[y,z]/(y^m-1,(z-y^d-1)\Dick_{n-1}(y^d,z)) \] 
with $y=Y^2$ and $z=XY^d$. 
This is a complete intersection of $\bZ$-rank $mn$, with a $\bZ$-basis
consisting of the monomials $y^iz^j$ with $0\le i<m$, $0\le j<n$.

The nil radical of $a(H_{m,n}(q))$ is generated by the element
$(y^d-1)\Dick_{n-1}(y,z)$, which squares to zero. There are $mn-m+d$ 
species $s_{i,j}$ of $a(H_{m,n}(q))$, and they are given by
\begin{align*}
y & \mapsto \zeta_m^j \\
z & \mapsto \zeta_{2n}^{j+i} + \zeta_{2n}^{j-i} 
\end{align*}
where $0\le i \le n$, $0\le j < m$, and if $j$ is divisible by $n$
then  $i\equiv j \pmod{2n}$.

The ideal of projective modules is generated by $\Dick_{n-1}(y^d,z)$, and
the $d$ Brauer species are the species $s_{i,j}$ with $j$ divisible by $n$. The
quotient 
\[ a_{\proj}(H_{m,n}(q))=a(H_{m,n}(q))/(\Dick_{n-1}(y^d,z)) \] 
is semisimple.
\end{theorem}
\begin{proof}
The proof of the first part is the same as the proof of
Theorem~\ref{th:aZprtimesZm}. The second part follows from Theorem~\ref{th:aH2nmq}.
\end{proof}

\begin{rk}
The Taft algebras $H_n(q)$ are the case of the generalised Taft
algebras $H_{m,n}(q)$ where $m=n$ and $d=1$.
\end{rk}

\section{Some open problems}

Throughout this section, $G$ is a finite group, $k$ is a field of
characteristic $p$, and $M$ is a $kG$-module. Some of the following questions
come from~\cite{Benson:2020a,Benson/Symonds:2020a}. 
Others have been around for a while.

\begin{question}\label{qu:MM*}
If $M$ is a $kG$-module, is $\npj_G(M\otimes M^*)=\npj_G(M)^2$?
\end{question}

\begin{question}
Is $\hat a(G)$ a symmetric 
Banach $*$-algebra?\index{symmetric!Banach $*$-algebra}%
\index{Banach!star@$*$-algebra!symmetric} 
If so, then by Proposition~\ref{pr:symm-rep-ring}, 
we can deduce that Question~\ref{qu:MM*} has a positive answer.
More generally, develop a good way of testing whether a representation
ring is symmetric.\index{symmetric!representation ring}%
\index{representation!ring!symmetric}
Example~\ref{eg:ordinary=>symmetric} shows that ordinary
representation rings are symmetric.
\end{question}

\begin{question}
If $1<\npj_G(M)< 2$, is $\npj_G(M)=2\cos\pi/n$ for some integer $n\ge 4$? If
this holds, is $n$ a power of $p$?
\end{question}

\begin{question}\label{qu:alg-int}
Is $\npj_G(M)$ an algebraic integer?
\end{question}

\begin{question}
Do the numbers $c_n^G(M)$ satisfy a linear recurrence relation with
constant coefficients, for all sufficiently large values of $n$? 
This is
equivalent to asking whether the generating fuction
\[ f(t)=\sum_{n=0}^\infty c_n^G(M)t^n \] 
is the power series expansion of a rational function of $t$.
If so, then Question~\ref{qu:alg-int} has a positive answer. 
\end{question}

\begin{question} 
Let $G$ be a finite $2$-group and $k$ an algebraically closed field of
characteristic $2$. If $M$ is an indecomposable
$kG$-module of odd dimension then 
Conjecture~\ref{conj:char2tensors} states that $M\otimes M^*$ is a direct sum of
$k$ and indecomposable modules of dimension divisible by four. Is this true?
\end{question}

\begin{question}
The indecomposable modules for the dihedral and semidihedral groups in
characteristic two are known. How do their tensor products decompose?
Compute the radical of the representation ring. Is it nilpotent?
\end{question}

\begin{question}
What are the indecomposable modules for the quaternion groups in
characteristic two? We know that the group algebra has tame
representation type, but a classification of the modules is not known.
\end{question}

\begin{question}
Are there nilpotent elements in the representation ring of 
an elementary abelian $2$-group $(\bZ/2)^r$
in characteristic two, when $r\ge 3$? The same question may be asked
of the representation ring of an exterior algebra of rank at least
three in any characteristic, regarded as a finite supergroup
scheme.\index{finite!supergroup scheme}
\end{question}

\bibliographystyle{amsplain}
\bibliography{../repcoh}
\printindex
\end{document}